\documentclass[oneside,a4paper,english, 11pt]{amsart}

\usepackage{adjustbox}
\usepackage[latin9]{inputenc}
\usepackage{hyperref}
\usepackage{lscape}
\usepackage{slashbox}
\usepackage{mathrsfs}  
\usepackage{enumerate}  
\makeatletter
 \theoremstyle{plain}
\newtheorem{thm}{Theorem}[section]
 \theoremstyle{plain}

 \theoremstyle{plain}
\newtheorem{conj}{Conjecture}[section]
\theoremstyle{plain}
  \newtheorem{prop}[thm]{Proposition}
\theoremstyle{plain}
 \newtheorem{lemma}[thm]{Lemma}
\theoremstyle{plain}

\theoremstyle{plain}
\newtheorem{cor}[thm]{Corollary}
\theoremstyle{plain}

\theoremstyle{definition}
  \newtheorem{defn}[thm]{Definition}
\theoremstyle{definition}
  
 \theoremstyle{definition}

  \newtheorem{warning}[thm]{Warning}
  \newtheorem{hypothesis}[thm]{Hypothesis}
 \newtheorem{property}[thm]{Property}

\theoremstyle{remark}
\newtheorem{rmk}[thm]{Remark}
\numberwithin{equation}{section}

\usepackage{amsmath}
\usepackage{amssymb}
\usepackage{amscd}
\usepackage{amsthm}
\usepackage{array,multirow}
\usepackage[arrow,matrix,tips,all,cmtip,poly,color]{xy}

\usepackage{stmaryrd}
\usepackage{url}
\usepackage{tikz-cd}
\usepackage{color}
\usepackage{caption}
\usepackage{subcaption}
\usepackage{float}

\pdfpageheight\paperheight
\pdfpagewidth\paperwidth
\newcommand{\Z}{\mathbb{Z}}

\newcommand{\Q}{\mathbb{Q}}

\newcommand{\Qp}{\mathbb{Q}_p}

\newcommand{\R}{\mathbb{R}}

\newcommand{\F}{\mathbb{F}}

\newcommand{\N}{\mathbb{N}}

\newcommand{\A}{\mathbb{A}}

\newcommand{\fM}{\mathfrak{M}}

\newcommand{\fP}{\mathfrak{P}}

\newcommand{\fm}{\mathfrak{m}}
\newcommand{\fp}{\mathfrak{p}}

\newcommand{\bA}{\mathbb{A}}

\newcommand{\bT}{\mathbb{T}}

\newcommand{\cA}{\mathcal{A}}

\newcommand{\cC}{\mathcal{C}}

\newcommand{\cJ}{\mathcal{J}}

\newcommand{\cO}{\mathcal{O}}
\newcommand{\cP}{\mathcal{P}}

\newcommand{\cW}{\mathcal{W}}
\newcommand{\cX}{\mathcal{X}}

\newcommand{\eps}{\varepsilon}

\newcommand{\phz}{\varphi}

\newcommand{\Zp}{\mathbb{Z}_p}

\newcommand{\Id}{\mathrm{id}}
\newcommand{\Gal}{\mathrm{Gal}}
\newcommand{\Hom}{\mathrm{Hom}}
\newcommand{\Ext}{\mathrm{Ext}}
\newcommand{\Tor}{\mathrm{Tor}}
\newcommand{\Res}{\mathrm{Res}}

\newcommand{\End}{\mathrm{End}}
\newcommand{\Aut}{\mathrm{Aut}}

\newcommand{\GL}{\mathrm{GL}}

\newcommand{\Spec}{\mathrm{Spec}\ }

\newcommand{\Frob}{\mathrm{Frob}}

\newcommand{\Fp}{\F_p}

\newcommand{\un}[1]{\underline{#1}}
\renewcommand{\bf}[1]{\mathbf{#1}}
\newcommand{\Rep}{\mathrm{Rep}}

\newcommand{\tld}[1]{\widetilde{#1}}

\newcommand{\JH}{\mathrm{JH}}

\newcommand{\supp}{\mathrm{Supp}}

\newcommand{\Adm}{\mathrm{Adm}}
\newcommand{\Trns}{\mathfrak{Tr}}

\newcommand{\rbar}{\overline{r}}
\newcommand{\rhobar}{\overline{\rho}}
\newcommand{\taubar}{\overline{\tau}}

\newcommand{\Spf}{\mathrm{Spf}}

\newcommand{\rG}{\mathrm{G}}

\newcommand{\rad}{\mathrm{rad}}
\newcommand{\cosoc}{\mathrm{cosoc}}
\newcommand{\soc}{\mathrm{soc}}

\newcommand{\defeq}{\stackrel{\textrm{\tiny{def}}}{=}}

\newcommand{\ovl}[1]{\overline{#1}}

\newcommand{\leqeta}{\leq\mkern-1.5mu\eta}
\newcommand{\orient}{\textrm{or}}

\newif\iffinalrun
\iffinalrun
  \newcommand{\mar}[1]{}
\else
  \newcommand{\mar}[1]{\marginpar{\raggedright\tiny #1}}

\DeclareMathOperator{\Mod}{Mod}
\DeclareMathOperator{\Coh}{Mod}

\DeclareMathOperator{\Fil}{Fil}

\DeclareMathOperator{\gr}{gr}
\DeclareMathOperator{\coker}{coker}

\newcommand{\ra}{\rightarrow}

\newcommand{\into}{\hookrightarrow}
\newcommand{\surj}{\twoheadrightarrow}
\newcommand{\onto}{\twoheadrightarrow}

\newcommand{\risom}{\buildrel\sim\over\rightarrow} \usepackage{enumitem}
\setlist{nolistsep}

\title{$K_1$-invariants in the mod $p$ cohomology of $U(3)$ arithmetic manifolds}

\author{Daniel Le}
\address{Department of Mathematics,
Purdue University,
150 N. University Street, 
West Lafayette, IN 47907-2067}
\email{ledt@purdue.edu}

\author{Bao V.~Le Hung}
\address{Department of Mathematics,
Northwestern University, 
2033 Sheridan Road\\
Evanston, IL 60208, USA}
\email{lhvietbao@googlemail.com}

\author{Stefano Morra}
\address{LAGA, UMR 7539, CNRS, Universit\'e Paris 13 - Sorbonne Paris Cit\'e, 
Universit\'e de Paris 8,
99 avenue Jean Baptiste Cl\'ement,
93430 Villetaneuse,
France }
\email{morra@math.univ-paris13.fr}
\begin{document}

\begin{abstract}
Let $F/F^+$ be a CM extension and $H_{/F^+}$ a definite unitary group in three variables that splits over $F$. 
We describe Hecke isotypic components of mod $p$ algebraic modular forms on $H$ at first principal congruence level at $p$ and ``minimal" level away from $p$ in terms of the restrictions of the associated Galois representation to decomposition groups at $p$ when these restrictions are tame and sufficiently generic. 
This confirms an expectation of local-global compatibility in the mod $p$ Langlands program. 
To prove our result, we develop a local model theory for multitype deformation rings and new methods to work with patched modules that are not free over their scheme-theoretic support. 
\end{abstract}

\maketitle

\tableofcontents

\section{Introduction}

\subsection{The main result}
In this paper, we describe some Hecke isotypic components of spaces of algebraic modular forms at first principal congruence level for definite unitary groups in three variables. 
We begin by motivating this problem. 
Let $p$ be a prime and $\F/\F_p$ be a (sufficiently large) finite extension. 
Let $F/F^+$ be a CM extension for which $p$ is inert in $F^+$ and splits in $F$. 
Let $n$ be a positive integer and $H_{/F^+}$ be an outer form of $\GL_n$ which splits over $F$ and is definite at infinity i.e.~$H(F^+ \otimes_{\Q}\R)$ is compact. 
Let $U^p \subset H(\A_{F^+}^{\infty p})$ be a compact open subgroup. 
We define a space of mod $p$ modular forms
\[
S(U^p,\F) \defeq \{f: H(F^+) \backslash H(\A_{F^+}^\infty)/U^p \ra \F \textrm{ locally constant}\}
\]
at infinite level at $p$. 
This has a faithful action of a Hecke algebra $\bT = \F[T_{\tld{v}}^{(j)}]_{1\leq j \leq n,v\in \cP}$ over $\F$ at ``good places" $v\in \cP$ where $\cP$ is a cofinite subset of the places of $F^+$ that split in $F$ (the indexing of the Hecke operators depends on a choice of a place $\tld{v}$ of $F$ lying over $v$). 
The ring $\bT$ is semilocal, and for a homomorphism $\alpha: \bT\ra \F$ with kernel $\fm$ there is a continuous semisimple representation $\rbar: G_F \ra \GL_n(\F)$ whose conjugacy class is characterized by the equations 
\[
\det(xI_n-\rbar(\Frob_{\tld{v}})) = \sum_{i=0}^n (-1)^i (\mathbf{N}\tld{v})^{{i}\choose{2}} \alpha(T_{\tld{v}}^{(j)}) x^{n-i} 
\]
for each $v\in\cP$ (see \cite[Proposition 3.4.2]{CHT}). 
Suppose from now on that $\rbar$ is irreducible. 
Motivated by the (classical) local Langlands correspondence and local-global compatibility, there is the following central conjecture in the mod $p$ Langlands program. 

\begin{conj} \label{conj:LLC}
For a finite extension $K/\Qp$, there is an injection 
\[
\mathrm{LLC}:\{G_K \ra \GL_n(\F)\}_{/\cong} \ \lhook\joinrel\longrightarrow  \ \{\textrm{finite length smooth admissible }\GL_n(K)\textrm{-representations}/\F\}_{/\cong}
\] 
such that for a place $w|p$ of $F$, $S(U^p,\F)[\fm] \cong \mathrm{LLC}(\rbar|_{G_{F_w}})^{\oplus d}$ as $H(F^+_p) \cong \GL_n(F_w)$-representations for some positive integer $d = d(U^p)$. 
\end{conj}

\noindent 
This injection should be the atom with which one builds a deformation theoretic $p$-adic local Langlands correspondence that mediates Langlands reciprocity in $p$-adic families. 
Such an injection has been constructed for $n=2$ and $K = \Q_p$, but the situation appears to be substantially more complicated if $n>2$ or $K \neq \Qp$. 
One reason is that there does not seem to be a simple classification of irreducible smooth admissible $\GL_n(K)$-representations, not to mention a (conjectural) local characterization of $\mathrm{LLC}$ or a simple (conjectural) description of the image of $\mathrm{LLC}$. 
Given these difficulties, it is natural to study the space $S(U^p,\F)[\fm]$ directly. 

At present, the strongest evidence towards Conjecture \ref{conj:LLC} is a description of the constituents of the $\GL_n(\cO_{F_w})$-socle of $S(U^p,\F)[\fm]$ when $\rbar$ is tamely ramified and sufficiently generic at $p$ \cite{MLM}. 
When $n=2$, there is stronger evidence---the invariants of $\pi(\rbar) \defeq S(U^p,\F)[\fm]$ under the pro-$p$ Iwahori and first principal congruence subgroups can be described in terms of $\rhobar|_{G_{F_w}}$ \cite{EGS,HW,LMS,Le} when $\rhobar|_{G_{F_w}}$ is sufficiently generic, $\rbar$ satisfies a Taylor--Wiles hypothesis, and $U^p$ is ``minimal'' (in particular, this implies $d=1$). 
There are a number of reasons to single out these two compact open subgroups, even for general $n$. 
First, they are pro-$p$ subgroups so that their invariants are necessarily nonzero. 
Second, their normalizers jointly generate the group $\GL_n(F_w)$ which is a central feature in the theory of coefficient systems on buildings \cite{SS,Paskunas,BP}. 
Finally, it is expected that the invariants under the first principal subgroup $U(p) \defeq \ker(\GL_n(\cO_{F_w})\ra \GL_n(k_w))$ generate $\pi(\rbar)$ (this is known for $n=2$ and $\rhobar$ sufficiently generic by \cite[Theorem 1.6]{HW2} and \cite[Theorem 1.3.8]{BHHMS2}). 
Far less is known when $n>2$, and current evidence suggests that the situation is very complicated. 
When $n=3$ and the level is ``minimal", the authors and Levin described the invariants of $\pi(\rbar)$ under the pro-$p$ Iwahori subgroup in terms of $\rhobar|_{G_{F_w}}$ when this representation is sufficiently generic \cite{LLLM2,GL3Wild}. 
In this paper, we build on these results to describe the invariants of $\pi(\rbar)$ under the first principal congruence subgroup $U(p)$ when $n=3$ and $\rbar$ is sufficiently generic and tamely ramified at $p$. 

\begin{thm}[Theorem \ref{main:glob:app}]\label{thm:intro:main}
Suppose that $n=3$. 
Moreover, suppose that 
\begin{itemize}
\item $\rbar$ satisfies a Taylor--Wiles hypothesis; 
\item $\rbar$ is tamely ramified and sufficiently generic at $p$; and 
\item $U^p$ is ``minimal". 
\end{itemize}
Then the $\GL_3(k_w)$-representation $\pi(\rbar)^{U(p)}$ is uniquely determined by $\rbar|_{G_{F_w}}$. 
\end{thm}

\begin{rmk}
\begin{enumerate}
\item In fact, we show that $\pi(\rbar)^{U(p)}$ is uniquely determined by the restriction $\rbar|_{I_{F_w}}$ to the inertial subgroup. 
We also explicitly determine the semisimplification of $\pi(\rbar)^{U(p)}$ (see Theorem \ref{thm:multiplicity}). 
\item The assumption that $p$ is inert in $F^+$ can be replaced by an assumption that $p$ is unramified in $F^+$. 
\item It seems likely that one could remove the hypothesis that $\rbar$ is tamely ramified at $p$ by breaking the problem into cases based on the set of extremal weights \cite{GL3Wild,OBW}.
\item In stark contrast to $n=2$, many constituents of $\soc\,\pi(\rbar)^{U(p)}$ reappear beyond the socle.
Theorem \ref{main:glob:app} thus has the flavor of characterizing a specific point in a (continuous) moduli of representations with the same discrete invariants.
This does not seem to have been previously anticipated in the literature, even though the appearance of constituents of $\soc\,\pi(\rbar)^{U(p)}$ beyond the socle was expected by F.~Herzig (private communication).
\end{enumerate}
\end{rmk}

\cite{BHHMS2} gives a number of refinements of Conjecture \ref{conj:LLC} and proves some of these when $n=2$ in part by using results on invariants of $\pi(\rbar)$. 
We expect that Theorem \ref{thm:intro:main} will have similar applications, and we hope to return to this in future work. 

\subsection{Methods}

The proof of Theorem \ref{thm:intro:main} can be broken into three steps. 
Let $n=3$ and $U^p$ and $\rbar$ be as in Theorem \ref{thm:intro:main}. 
Let $\rhobar \defeq \rbar|_{G_{F_w}}$. 
It was shown in \cite{LLLM2} that the $\GL_3(\cO_{F_w})$-socle of $\pi(\rbar)$ is $\oplus_{\sigma \in W^?(\rhobar)} \sigma$ where $W^?(\rhobar)$ is an explicit set of irreducible $\GL_3(\cO_{F_w})$-representations over $\F$ (equivalently irreducible $\GL_3(k_w)$-representations) predicted in \cite{herzig-duke,GHS}. 
First one calculates the dimension of $\Hom_{\GL_3(k_w)}(P_\sigma,\pi(\rbar)^{U(p)})$ for each $\sigma \in W^?(\rhobar)$ where $P_\sigma$ denotes a $\GL_3(k_w)$-projective cover of $\sigma$. 
This gives the multiplicity of each $\sigma \in W^?(\rhobar)$ as a Jordan--H\"older factor of $\pi(\rbar)^{U(p)}$. 
Next, we construct a certain $\GL_3(k_w)$-representation $D_{\mathrm{m}}$ depending on $\rbar|_{I_{F_w}}$ such that $\JH(D_{\mathrm{m}}) = W^?(\rhobar)$, $D_{\mathrm{m}}$ injects into $\pi(\rbar)^{U(p)}$, and both $\pi(\rbar)^{U(p)}$ and $D_{\mathrm{m}}$ contain each $\sigma \in W^?(\rhobar)$ as a Jordan--H\"older factor with the same multiplicity. 
Finally, a simple argument shows that $\pi(\rbar)^{U(p)}$ is the unique maximal representation containing $D_{\mathrm{m}}$ such that they have the same $\GL_3(k_w)$-socle and contain each $\sigma \in W^?(\rhobar)$ as a Jordan--H\"older factor with the same multiplicity. 

The most difficult step is calculating the dimensions of the multiplicity spaces $\Hom_{\GL_3(k_w)}(P_\sigma,\pi(\rbar)^{U(p)})$ for $\sigma \in W^?(\rhobar)$. 
The basic idea, going back to Taylor and Wiles, is to replace the dual of $\pi(\rbar)^{U(p)}$ with a space $M_\infty$ (we caution that this is slightly different than the meaning of $M_\infty$ later in the text) obtained by taking truncated Witt vectors as coefficients, adding auxiliary level (at Taylor--Wiles places), and then taking a noncanonical limit. 
This object, and more generally its multiplicity spaces (the ``patched modules''), are maximal Cohen--Macaulay modules over local Galois deformation spaces with $p$-adic Hodge theory conditions. While these objects are very non-canonical, taking fiber at the closed point recovers their original versions. Thus, dimensions of mod $p$ multiplicity spaces can be computed as minimal number of generators of patched modules.
Diamond and Fujiwara observed that when a patched module is supported over a local deformation space that happens to be regular, then it is free by theorems of Serre and Auslander--Buchsbaum. This shows a mod $p$ multiplicity agrees with a characteristic $0$ multiplicity, which are often known by automorphic methods. 
In \cite{EGS}, Emerton, Gee and Savitt pioneered methods for studying mod $p$ multiplicity questions when the relevant deformation spaces are not regular by taking advantage of the fact that $M_\infty$ is a projective $W(\F)[\GL_3(k_w)]$-module i.e.~$M_\infty(-)\defeq \Hom_{\GL_3(k_w)}(M_\infty,(-)^\vee)^\vee$ defines an exact functor where $(-)^\vee$ denotes Pontrjagin duality. This implies: 
\begin{enumerate}
\item (Nakayama gluing) If $V\subset W$ such that $M_\infty(V) \subset \fm M_\infty(W)$, then $M_\infty(W)$ and $M_\infty(W/V)$ have the same minimal number of generators. 
\item (Fiber product gluing) A quasi-isomorphism $V \ra P^\bullet$ of complexes of $W(\F)[\GL_3(k_w)]$-modules gives a quasi-isomorphism $M_\infty(V) \ra M_\infty(P^\bullet)$. 
\end{enumerate}
By combining these two properties, one has a potential strategy to compute the minimal number of generators of complicated patched modules from simpler ones. 
We will first illustrate the strategy for $n=2$ and then describe what changes for $n=3$. 

Suppose that $n=2$, that $\rhobar$ is tame and sufficiently generic, and that $\sigma \in W^?(\rhobar)$. 
Let $\ovl{P}_\sigma\defeq P_\sigma/\rad^2(P_\sigma)$ where $\rad^\bullet$ denotes the radical filtration. Then $\ovl{P}_\sigma$ sits in a short exact sequence 
\[
0 \ra \ovl{P}_\sigma \ra \bigoplus_i Q_i \ra (\bigoplus_i \sigma)/\Delta\sigma \ra 0,
\]
where $Q_i$ denote the different length two quotients of $\ovl{P}_\sigma$ and $\Delta\sigma \subset \bigoplus_i \sigma$ is the diagonally embedded copy. Using the Diamond--Fujiwara trick, one shows $M_\infty(Q_i)$ is cyclic. This allows us to compute $M_\infty(Q_i)$ and check that $M_\infty(\ovl{P}_\sigma)$ is also cyclic via fiber product gluing.
Finally, one shows that $M_\infty(P_\sigma/\rad^m P_\sigma)$ is cyclic for all $m$ inductively using Nakayama gluing, the crucial ingredient being that $M_\infty(\rad^{m-1}P_\sigma /\rad^{m+1} P_\sigma)$ and $M_\infty(\rad^{m-1}P_\sigma/\rad^m P_\sigma)$ have the same minimal number of generators. This last fact follows from the cyclicity of $M_\infty(\ovl{P}_\kappa)$ for \emph{all} $\kappa \in W^?(\rhobar) \cap \JH(\rad^{m-1}P_\sigma/\rad^m P_\sigma)$, and the following ``covering'' property of $\rad^{m-1}P_\sigma/\rad^m P_\sigma$ with respect to $W^?(\rhobar)$:
\begin{property}\label{property:covering}
The cokernel of the evaluation map
\[\bigoplus_{\kappa \in W^?(\rhobar) \cap \JH(\rad^{m-1}P_\sigma/\rad^m P_\sigma)} P_\kappa\otimes \Hom_{\GL_3(k_w)}(P_\kappa,\rad^{m-1}P_\sigma /\rad^{m+1} P_\sigma) \ra \rad^{m-1}P_\sigma /\rad^{m+1} P_\sigma\]
has no Jordan-Holder factors in $W^?(\rhobar)$.
\end{property}

There are two major obstacles to carrying out an analogous strategy when $n>2$. 
First, one would like to show an analogue of Property \ref{property:covering}. 
Unfortunately, the structure of projective $\F[\GL_n(k_w)]$-modules is poorly understood. 
While it is known that generic projective indecomposable $\F[\GL_n(k_w)]$-modules $P_\sigma$ are the restrictions of certain tilting modules for (a product of) $\GL_{n/\F}$, describing the submodule structure of such tilting modules appears to be hopeless in general. 
Fortunately, when $n=3$, tilting modules for $\GL_{3/\F}$ are fairly well-understood (in part because $\Ext^1$ groups for simple modules have dimension at most one) \cite{BDM}, from which we are able to deduce Property \ref{property:covering} (Proposition \ref{prop:W?-cover}). 
Much of the first part of \S \ref{sec:rt} is devoted to these considerations. 

The second obstacle is that $M_\infty(P_\sigma/\rad^2(P_\sigma))$ is often not cyclic, which prevents us from Nakayama gluing as when $n=2$. 
The remedy is to instead consider larger quotients of $P_\sigma$. 
We first elaborate when $F^+ = \Q$ as the general case involves substantial complication in a different direction. %
Here, the highest weight of a Serre weight is in one of two $p$-restricted alcoves, an upper alcove and a lower alcove. 
Moreover, the radical layers of a projective indecomposable $\F[\GL_3(\F_p)]$-module contain Serre weights in a single $p$-alcove of alternating type. 
One can compute $M_\infty(P_\sigma/\rad^2(P_\sigma))$ for $\sigma\in W^?(\rhobar)$ using fiber product gluing and see that it is not cyclic exactly when $\sigma$ is in the lower alcove, which precludes the verbatim inductive argument in the $n=2$ case. 
Fortunately, it turns out that $M_\infty(P_\sigma/\rad^3(P_\sigma))$ is cyclic when $\sigma$ is in the upper alcove. 
Then a modified inductive argument using Property \ref{property:covering} and Nakayama gluing along upper alcove layers shows that $M_\infty(P_\sigma)$ and $M_\infty(P_\sigma/\rad^2(P_\sigma))$ are minimally generated by the same number of elements. 

To show that $M_\infty(P_\sigma/\rad^3(P_\sigma))$ is cyclic, we decompose it as an iterated fiber product. 
Since some Jordan--H\"older factors appear with multiplicity greater than one, $P_\sigma/\rad^3(P_\sigma)$ cannot be presented in terms of the mod $p$ reductions of a single irreducible $\GL_3(\F_p)$-representation. 
Instead, we use fiber products of lattices in several different Deligne--Lusztig representations. 
In parallel, $M_\infty(P_\sigma/\rad^3(P_\sigma))$ is not supported on a single tame type Galois deformation space. 
This forces us to work in some highly singular ``multitype" deformation rings, which first appeared in print in \cite{DanWild} and were later used in \cite{BHHMS}, but were first developed in the course of this project.
We develop a local model theory for such multitype deformation rings building on \cite{MLM,GL3Wild}, cf.~\S \ref{subsec:MTdef}. 
With this in hand, we deduce the desired cyclicity from the transversality of certain scheme theoretic intersections in these models.

The $F^+ \neq \Q$ case brings many further complications. 
Now, our arguments require fiber product gluing situations involving noncyclic patched modules, which makes it hard to control the end result. However, we give new criteria for the minimal number of generators of a fiber product to be well-behaved in terms of the transversality of (only) infinitesimal properties of its factors. These criteria are essential because while the local models for single tame type deformation rings have tensor product structures over embeddings $F_w \into W(\F)[1/p]$, the local models for multitype stacks do not seem to have such product structures. A key observation here is that the product structure persists infinitesimally, yielding the sought-after transversality statements.

Finally, we come to the construction of $D_{\mathrm{m}}$, which is the minimal subrepresentation of $\pi(\rbar)^{U(p)}$ containing all Jordan--H\"older factors in $W^?(\rhobar)$. 
When $n=2$, cyclicity of certain patched modules implies that $D_{\mathrm{m}} \cong \oplus_{\sigma \in W^?(\rhobar)} \sigma$---in particular it is determined uniquely from multiplicity information. 
When $n=3$, there is a moduli of representations with the same $\F[\GL_3(k_w)]$-socle and multiset of Jordan--H\"older factors as $D_{\mathrm{m}}$, but only $D_{\mathrm{m}}$ injects into $\pi(\rbar)^{U(p)}$. 
To single out this point, we give ``coordinates" on this moduli space using categories of lattices in various Deligne--Lusztig representations. 
Then we compute $M_\infty(D)$ for $\F[\GL_3(k_w)]$-modules $D$ in (a part of) this moduli space and characterize $M_\infty(D_{\mathrm{m}})$ among them. 

\begin{rmk} Unlike when $n=2$, our approach falls short of giving a purely local characterization of $M_\infty(P_{\sigma})$ in general, since there is a non-trivial moduli of non-cyclic Cohen--Macaulay modules with prescribed support.  
In principle, one could attempt such a characterization by presenting indecomposable projective $W(\F)[\GL_3(k_w)]$-modules $\tld{P}_\sigma$ in terms of sums of lattices in Deligne--Lusztig representations and use the fact that patched modules for such lattices were computed in \cite{LLLM2}. 
However, it seems challenging to find explicit presentations and worse, the non-cyclicity of some patched lattices makes it unclear how to compute the effect of patching on the \emph{maps} that arise in such a presentation.
\end{rmk}
\subsection{Outline}
\S \ref{sec:prelim} is a brief summary of background on tame inertial types, Deligne--Lusztig representations, the inertial local Langlands, and Serre weights largely following \cite{LLLM2,GL3Wild}. 
\S \ref{sec:multitype} and \S \ref{sec:rt} collect a number of technical results needed for the proof of the main result which spans \S \ref{sec:patch} and \S \ref{sec:locality}. 
We defer several tedious ideal intersection computations to Appendix \ref{appendix:IC}, which can be ignored by a reader willing to take these computations on faith.
We urge the reader skip \S \ref{sec:multitype} and \S \ref{sec:rt} on a first reading, only referring back as needed. 
\S \ref{subsec:MTdef} develops a local model theory for some multitype deformation rings building on \cite{MLM,GL3Wild}. 
\S \ref{sec:ca} collects a number of commutative algebra lemmas. 
\S \ref{subsec:idealrelations} contains technical computations with deformation rings used in \S \ref{sec:patch} and \ref{sec:locality}. 
\S \ref{sec:rt} establishes results in modular and integral representation theory building on \cite{BDM,LLLM2} needed to prove the main result. 
\S \ref{sec:patch} computes multiplicities of Serre weights in an axiomatic context by combining results in \S \ref{sec:multitype} and \ref{sec:rt}. 
\S \ref{sec:locality} proves Theorem \ref{thm:intro:main} first in an axiomatic context and then in a global context. 

\indent
\paragraph{\emph{Acknowledgements}}
This work has been developed over several years starting in the fall of 2016. 
We apologize for the long delay in writing up this.
The authors thank the various institutions where they have been hosted while working on this project (University of Toronto, Institute for Advanced Studies, Universit\'e de Montpellier, Universit\`a degli Studi di Padova, the Hausdorff Institute of Mathematics).
B.LH. acknowledges support from the National Science Foundation under grant Nos.~DMS-1128155, DMS-1802037, DMS-2302619 and the Alfred P. Sloan Foundation.
D.L. was supported by the National Science Foundation under agreements Nos.~DMS-1128155, DMS-1703182, and DMS-2302623, an AMS-Simons travel grant, and a start-up grant from Purdue University.
S.M. is member of the Institut Universitaire de France and acknowledges its support.

\subsection{Notation}
\label{sec:notation}

For any given field $K$ we fix once and for all a separable closure $\ovl{K}$ and define $G_K \defeq \Gal(\ovl{K}/K)$. 
If $K$ is a nonarchimedean local field, we let $I_K \subset G_K$ denote the inertial subgroup.
We fix a prime $p$.
Fix an algebraic closure $\ovl{\Q}_p$ of $\Q_p$. Let $\ovl{\F}_p$ denote its residue field. 
Let $\F\subset\ovl{\F}_p$ be a finite subfield which we will assume is sufficiently large for our purposes. %
Let $\cO$ be $W(\F)$ and $E$ be the fraction field of $\cO$. %
Then we have a natural embedding $E \subset \ovl{\Q}_p$.  

Let $G\defeq{\GL_3}_{/\Z}$, $B\subset G$ the subgroup of upper triangular matrices, $T \subset B$ the split torus of diagonal matrices and $Z \subset T$ the center of $G$.  
Let $\Phi^{+} \subset \Phi$ denote the subset of positive roots in the set of roots for $(G, B, T)$. 
Let $X^*(T)$ be the group of characters of $T$ which we identify with $\Z^3$ in the standard way.

We write $W$ (resp.~$W_a$, resp.~$\tld{W}$) for the Weyl group (resp.~the affine Weyl group, resp.~the extended affine Weyl group) of $G$.
If $\Lambda_R \subset X^*(T)$ denotes the root lattice for $G$ we then have
\[
W_a = \Lambda_R \rtimes W, \quad \tld{W} = X^*(T) \rtimes W
\]
and use the notation $t_{\nu} \in \tld{W}$ to denote the image of $\nu \in X^*(T)$. 
Let $\eta=(2,1,0)\in X^*(T)$.
We define the \emph{$p$-dot action} by $t_\lambda w \cdot \mu = p\lambda+ w (\mu+\eta) - \eta$.
Let $w_0$ denote the longest element in $W$ and define $\tld{w}_h\defeq w_0 t_{-\eta}$.

Let $\langle \ ,\,\rangle$ denote the duality pairing on $X^*(T)\times X_*(T)$.
A weight $\lambda\in X^*(T)$ is \emph{dominant} if $0\leq \langle\lambda,\alpha^\vee\rangle$  for all simple root $\alpha\in \Phi$.
Set $X^0(T)$ to be the subgroup consisting $\lambda\in X^*(T)$ such that $\langle\lambda,\alpha^\vee\rangle=0$ for all $\alpha\in \Phi$, and $X_1(T)$ to be the set of $\lambda\in X^*(T)$ such that $0\leq \langle\lambda,\alpha^\vee\rangle< p$ for all simple root $\alpha\in \Phi$.

A $p$-alcove is a connected component of  
\[
X^*(T)\otimes_{\Z}\R\ \setminus\ \Big(\bigcup_{(\alpha,n)}\{\lambda\in X^*(T)\otimes_{\Z}\R\ :\ \langle\lambda+\eta,\alpha^\vee\rangle=np\}\Big)
\]
where $(\alpha,n)$ runs over $\Phi^+\times \Z$.
A $p$-alcove $C$ is $p$-restricted (resp.~dominant) if $0<\langle\lambda+\eta,\alpha^\vee\rangle<p$ (resp.~$0<\langle\lambda+\eta,\alpha^\vee\rangle$) for all simple roots $\alpha\in \Phi$ and $\lambda\in C$.
The group $W_a$ acts simply transitively on the set of alcoves and throughout this paper the dot action of $\tld{W}$ on the alcoves will always be the $p$-dot action. 
We let $C_0 \subset X^*({T})\otimes_{\Z}\R$ denote the dominant base alcove (i.e.~the alcove defined by $0<\langle\lambda+\eta,\alpha^\vee\rangle<p$ for all $\alpha\in \Phi^+$), and set
\begin{equation*}
\tld{W}^+\defeq\{\tld{w}\in \tld{W}:\tld{w}\cdot C_0 \textrm{ is dominant}\},\quad
\tld{W}^+_1\defeq\{\tld{w}\in \tld{{W}}^+:\tld{w}\cdot {C}_0 \textrm{ is } p\textrm{-restricted}\}.
\end{equation*}
We sometimes refer to $C_0$ as the \emph{lower alcove} and $\tld{w}_h\cdot C_0$ as the \emph{upper alcove}.
Given $N\in\N$ we say that $\lambda\in X^*(T)$ is \emph{$N$-deep in alcove $C_0$} if $N<\langle\lambda+\eta,\alpha^\vee\rangle< p-N$ for all $\alpha\in \Phi^+$.

Let $\cO_p$ be a finite \'etale $\Z_p$-algebra and fix an isomorphism $\cO_p\cong\prod\limits_{v\in S_p} \cO_{v}$ where $S_p$ is a finite set and $\cO_{v}$ is the ring of integers of a finite unramified extension $F^+_{v}$ of $\Q_p$.
We define $G_0 \defeq \Res_{\cO_p/\Z_p} G_{/\cO_p}$, and similarly $T_0\subset B_0\subset G_0$. 
We assume that $\cO$ contains the image of any ring homomorphism $\cO_p \ra \ovl{\Z}_p$ and write $\cJ\defeq \Hom_{\Zp}(\cO_p,\cO)$.
We define $\un{G} \defeq (G_0)_{/\cO}$, and similarly $\un{T}\subset \un{B}\subset \un{G}$.
We use underlined notations for the objects introduced above for $G$ and now relative to $\un{G}$ (hence $\un{\Phi}^+ \subset \un{\Phi}$, $\un{W}$, $\un{W}_a$, $\tld{\un{W}}$, $\tld{\un{W}}^+$, $\tld{\un{W}}^+_1$, $\un{C}_0$ and the like) and we have a notion of $N$-deepness in alcove $\un{C}_0$ for elements of $X^*(\un{T})$.
The natural isomorphism $\un{G} \cong G_{/\cO}^{\cJ}$ induces compatible isomorphisms $X^*(\un{T}) = X^*(T)^{\cJ}$ and the like.  
Given an element $j\in\cJ$, we use a subscript notation to denote $j$-components obtained from the isomorphism $\un{G}\cong G_{/\cO}^{\cJ}$ (e.g.~given $\tld{w}\in \tld{\un{W}}$ we write $\tld{w}_j$ to denote its $j$-th component via the induced isomorphism $\tld{\un{W}}\cong \tld{W}^{\cJ}$).

For sake of readability, we abuse notation and still write $w_0$ to denote the longest element in $\un{W}$, and $\eta\in X^*(\un{T})$ for the element $\sum_{j\in\cJ}(2,1,0)_j$.
The meaning of $w_0$, $\eta$ and $\tld{w}_h\defeq w_0t_{-\eta}$ should be clear from the context.

The absolute Frobenius automorphism on $\cO_p/p$ lifts canonically to an automorphism $\varphi$ of $\cO_p$. We define an automorphism $\pi$ of  $X^*(\un{T})$ by $\pi(\lambda)_\sigma = \lambda_{\sigma \circ \varphi^{-1}}$ for all $\lambda\in X^*(\un{T})$ and $\sigma: \cO_p \ra \cO$.
We similarly define an automorphism $\pi$ of $\un{W}$ and $\tld{\un{W}}$.

When $S_p=\{v\}$ is a singleton we simplify notation and let $K\defeq F^+_p$, $f\defeq [K:\Qp]$, ring of integers $\cO_K$, residue field $k$. 
Let $W(k)$ be ring of Witt vectors of $k$, which is also the ring of integers of $K$.  
We denote the arithmetic Frobenius automorphism on $W(k)$ by $\phz$ (it acts as raising to $p$-th power on the residue field).

Recall that we fixed a separable closure $\ovl{K}$ of $K$.
We choose $\pi \in \ovl{K}$ such that $\pi^{p^f-1} = -p$ and let $\omega_K : G_K \ra \cO_K^\times$ be the character defined by $g(\pi) = \omega_K(g) \pi$, which is independent of the choice of $\pi$. 
We fix an embedding $\sigma_0: K \into E$ and define $\sigma_j = \sigma_0 \circ \phz^{-j}$, which identifies $\cJ = \Hom(k, \F) = \Hom_{\Qp}(K, E)$ with $\Z/f \Z$. 
We write $\omega_f:G_K \ra \cO^\times$ for the character $\sigma_0 \circ \omega_K$.

Let $\eps$ denote the $p$-adic cyclotomic character. 
We normalize the definitions of labelled Hodge--Tate weights so that $\eps$ has Hodge--Tate weight $\{1\}$ for every embedding $K\into E$ (this convention is opposite of that of \cite{EGstacks,CEGGPS} and agrees with that of \cite{GHS}), and the functor from potentially semistable representation to Weil--Deligne representations is \emph{covariant}.

A potentially semistable representation $\rho:G_K \ra \GL_3(E)$ has type $(\mu, \tau)$ if $\rho$ has labeled Hodge--Tate weights $\mu\in X^*(\un{T})$ (note that this differs from the conventions of \cite{GHS} via a shift by $\eta$) and the restriction to $I_K$ of the Weil-Deligne representation attached to $\rho$ is isomorphic to $\tau$.

Let $\Gamma$ be a group.
If $V$ is a finite length $\Gamma$-representation, we let $\JH(V)$ be the (finite) set of Jordan--H\"older factors of $V$.
If $V^\circ$ is a finite $\cO$-module with a $\Gamma$-action, we write $\ovl{V}^\circ$ for the $\Gamma$-representation $V^\circ\otimes_{\cO}\F$ over $\F$.
\section{Preliminaries}
\label{sec:prelim}
\subsection{Tame inertial types, inertial local Langlands, Serre weights}
\label{subsec:Prelim}

Unless otherwise stated, we assume throughout this section that $S_p=\{v\}$. 
We write $\cO_p=\cO_K$ (the ring of integers of a finite unramified extension $K$ of $\Qp$ of degree $f$) and let $G_0 = \Res_{\cO_K/\Z_p} G_{/\cO_K}$.
We drop subscripts $v$ from notation and we identify $\cJ=\Hom_{\Qp}(K,E)$ with $\Z/f\Z$ via $\sigma_{j} \defeq \sigma_0 \circ \phz^{-j} \mapsto j$.

\subsubsection{Tame inertial types, Deligne--Lusztig representations and inertial local Langlands}

\label{subsubsec:TIT}
An inertial type (for $K$, over $E$) is the $\GL_3(E)$-conjugacy class of an homomorphism $\tau:I_K\ra\GL_3(E)$ with open kernel and which extends to an homomorphism $W_K\rightarrow \GL_3(E)$.
We will identify a tame inertial type with a fixed choice of a representative in its $\GL_3(E)$-conjugacy class, and say that the type is \emph{tame} if the homomorphism $\tau$ factors through the tame quotient of $I_K$ (this notion is independent of the choice of the representative in the conjugacy class).
Given $s\in\un{W}$, $\mu\in X^*(\un{T})\cap \un{C}_0$ and $j'\in\{0,\dots,6f-1\}$ we set $\boldsymbol{\alpha}'_{j'}\defeq s_{6f-1}^{-1}s_{6f-2}^{-1}\dots s_{6f-j'}^{-1}(\mu_{6f-j'}+\eta_{6f-j'})\in X^*(T)$, where the subscripts are taken modulo $f$, and let $\mathbf{a}^{\prime(j')}$ be $\sum_{i=0}^{6f-1}\boldsymbol{\alpha}'_{-j+i}p^i\in X^*(T)$.
We thus define the tame inertial type $\tau(s,\mu+\eta)$ to be $(\omega_{6f})^{\mathbf{a}^{\prime(0)}}$, and say that $(s,\mu)$ is a \emph{lowest alcove presentation} for $\tau(s,\mu+\eta)$.
Given $N\in\N$ we say that a tame inertial type $\tau$ is $N$-generic if there exists a pair $(s,\mu)$ as above such that $\tau\cong \tau(s,\mu+\eta)$ and $\mu$ is $N$-deep in alcove $\un{C}_0$.
In this case, we say that $\tau$ has an $N$-generic lowest alcove presentation $(s,\mu)$ with associated element $\tld{w}(\tau)\defeq t_{\mu+\eta}s\in\tld{\un{W}}$.

Replacing $E$ with $\F$ in the preceding paragraph we have the notion of inertial $\F$-type, tame inertial $\F$-type (typically denoted with the overlined notation $\taubar$), lowest alcove presentation and associated element $\tld{w}(\taubar)$ for a tame inertial $\F$-type $\taubar$, and $N$-genericity for $N\in\N$.

Given a pair $(s,\mu-\eta)\in \un{W}\times X^*(\un{T})\cap\un{C}_0$ we can attach to it a virtual $G_0(\Fp)$-representation $R_s(\mu)$ over $E$ by \cite[Definition 9.2.2]{GHS} and the paragraph above \emph{loc.~cit}., where in \emph{loc.~cit}.~the representation $R_s(\mu)$ is denoted by $R(s,\mu)$.
If $\mu-\eta$ is moreover $1$-deep in $\un{C}_0$ then $R_s(\mu)$ is an irreducible representation of $G_0(\Fp)$ and the pair $(s,\mu)$ is said to be a lowest alcove presentation of $R_{s}(\mu)$.
Finally, given $N\in\N$ we say that a Deligne--Lusztig representation $R$ is $N$-generic if there exist $(s,\mu-\eta)\in \un{W}\times X^*(\un{T})\cap\un{C}_0$ such that $\mu-\eta$ is $N$-deep in $\un{C}_0$ and $R\cong R_s(\mu)$, in which case $(s,\mu-\eta)$ is a lowest alcove presentation for $R$ with corresponding element $\tld{w}(R)\defeq t_\mu s\in\tld{\un{W}}$.

By \cite[Theorem 3.7]{CEGGPS} we can attach to a tame inertial type $\tau:I_K\ra\GL_3(E)$ an irreducible $G_0(\Zp)$-representation $\sigma(\tau)$ over $E$, satisfying results towards the inertial local Langlands correspondence. 
In this paper we take $\sigma(\tau)$  to be the inflation of  $R_s(\mu+\eta)$ (\cite[Corollary 2.3.5]{LLL}).

If $\mu\in\un{C}_0$ is $1$-deep and  $\tau\cong \tau(s,\mu+\eta)$ then let $\sigma(\tau)$ be the inflation of  $R_s(\mu+\eta)$ to $G_0(\Zp)$.
This representation $\sigma(\tau)$ satisfies the properties described in \cite[Theorem 3.7]{CEGGPS}, see \cite[Corollary 2.3.5]{LLL} (this is a form of inertial local Langlands).

\subsubsection{Serre weights}
\label{subsubsec:SW}
A \emph{Serre weight} for $G_0(\Fp)$ is (an isomorphism class of) an absolutely irreducible $G_0(\Fp)$-representation over $\F$.
Any Serre weight is the restriction to $G_0(\Fp)$ of an irreducible algebraic representation of $G_0$ of highest weight in $X_1(\un{T})$, and given $\lambda\in X_1(\un{T})$ we write $F(\lambda)$ for the Serre weight associated to the irreducible algebraic representation $L(\lambda)$ of $G_0$ over $\Fp$.
The assignment $\lambda\mapsto F(\lambda)$ gives a bijection between $X_1(\un{T})/(p-\pi)X^0(\un{T})$ and the set of Serre weight for $G_0(\Fp)$.

Let $\cA$ denote the set of $p$-restricted alcoves in $X^*(\un{T})\otimes_\Z\R$.
As in \cite{GL3Wild,LLLM2} it will be more convenient to parametrize (regular) Serre weights by a subset of $\un{\Lambda}_W\times\cA$.
Specifically let $\lambda\in X^*(\un{T})\cap \un{C_0}+\eta$ and define $\un{\Lambda}_W^{\lambda}$ to be the subset of $X^*(\un{T})/X^0(\un{T})$ whose elements $\ovl{\omega}$ satisfy $\omega+\lambda-\eta\in \un{C}_0$ for some (equivalently, for all) lift $\omega\in X^*(\un{T})$ of $\ovl{\omega}$.
Then \emph{loc.~cit}.~produces an injection $\Trns_\lambda:\un{\Lambda}_W^{\lambda}\times\cA\into X_1(\un{T})/(p-\pi)X^0(\un{T})$ whose image consists of regular Serre weights with central character $(\lambda-\eta)|_{\un{Z}}$ modulo $(p-\pi)X^*(\un{Z})$ (see \emph{loc.~cit}.~Proposition 2.1.3 and 2.1.4.).
The reason this parametrization is preferred is that it decouples contributions of different elements of $\cJ$ in the representation theory of $G_0$.

\subsubsection{Combinatorics of types and weights}
\label{subsubsec:combinatorics_TandW}

We identify $\cA$ with $\{0,1\}^{\cJ}$ where $0$ and $1$ stand for the lower and upper alcove, respectively.
The standard action of $\tld{\un{W}}$ on $X^*(\un{T})$ descends to an action on $X^*(\un{T})/X^0(\un{T})$ hence on $X^*(\un{T})/X^0(\un{T})\times\cA$ by letting $\tld{\un{W}}$ act trivially on $\cA$.
Let $\eps_1$ (resp.~$\eps_2$) denote the image of $(1,0,0)\in X^*(T)$ (resp.~$(0,0,-1)\in X^*(T)$) in $X^*(T)/X^0(T)$ and set
\[
\Sigma_0\defeq \begin{Bmatrix} (\eps_1+ \eps_2, 0), (\eps_1 - \eps_2, 0), (\eps_2 - \eps_1, 0) \\
 (0, 1) , (\eps_1, 1), (\eps_2, 1) \\
 (0, 0), (\eps_1, 0), (\eps_2, 0) 
 \end{Bmatrix}
\]
and $\Sigma\defeq \Sigma_0^{\cJ}\subset X^*(\un{T})/X^0(\un{T})\times\cA$.
We define $(\omega,a),(\nu,b)\in\Sigma_0$ to be \emph{adjacent} if  $\omega-\nu\in\{0,\pm\eps_1,\,\pm\eps_2,\,\pm\eps_1-\eps_2\}$ and $a\neq b$.
With this notion of adjacency $\Sigma_0$ is a connected graph with a distance function $d$.

Assume that $\lambda-\eta$ is $0$-deep in $\un{C}_0$ and let $(s,\mu-\eta)\in \un{W}\times X^*(\un{T})\cap\un{C}_0$ be a pair such that $\mu-\eta$ is $2$-deep and $\mu+\eta-\lambda \in\un{\Lambda}_R$.
Then the set set $\JH(R_s(\mu))$ is given by $F(\Trns_\lambda(t_{\mu-\eta}s(\Sigma))$, see \cite[Proposition 2.1.1]{GL3Wild}.

Moreover, we \emph{define} $W^?(\taubar(s,\mu+\eta))$ to be the set of Serre weights $F\Big(\mathfrak{Tr}_{\lambda}\Big(t_{{\mu+\eta-\lambda}}s\big(r(\Sigma)\big)\Big)\Big)$, where $r(\Sigma)$ is defined by swapping the digits of $a\in\{0,1\}^{\cJ}$ in the elements $(\eps,a)\in\Sigma$.

\subsubsection{$L$-groups and $L$-parameters}
\label{subsubsec:L_parameters}

In this subsection we let $S_p$ have arbitrary finite cardinality.

Let $F^+_p$ be $\cO_p[1/p]$ so that $F^+_p\cong\prod\limits_{v \in S_p} F^+_{v}$ where $F^+_{v} \defeq \cO_{v}[1/p]$ for each $v \in S_p$.
Let 
\[
\un{G}^\vee_{/\Z} \defeq \prod_{F^+_p \ra E} G^\vee_{/\Z} 
\]
be the dual group of $\un{G}$ so that the Langlands dual group of $G_0$ is $^L \un{G}_{/\Z} \defeq \un{G}^\vee\rtimes \Gal(E/\Q_p)$ where $\Gal(E/\Q_p)$ acts on the set of homomorphisms $F^+_p \ra E$ by post-composition.

An \emph{$L$-homomorphism} (over $E$) is a continuous homomorphism $\rho:G_{\Qp}\ra{}^L\un{G}(E)$ which is compatible with the projection to $\Gal(E/\Qp)$.
The $\un{G}^\vee(E)$-conjugacy class of an $L$-homomorphism is called $L$-parameter.
An inertial $L$-parameter is a $\un{G}^\vee(E)$-conjugacy class of a homomorphism $\tau:I_{\Qp}\ra\un{G}^\vee(E)$ with open kernel, and which admits an extension to an $L$-homomorphism.
An (inertial) $L$-parameter is \emph{tame} if some (equivalently, any) representative in its equivalence class factors through the tame quotient of $I_{\Qp}$.
Fixing isomorphisms $\ovl{F^+_v}\stackrel{\sim}{\ra}\ovl{\Q}_p$ for all $v\in S_p$, we have a bijection between $L$-parameters (resp.~tame inertial $L$-parameters) and collections $(\rho_v)_{v\in S_p}$ (resp.~$(\tau_v)_{v\in S_p}$) where $\rho_v:G_{F^+_v}\ra\GL_3(E)$ is a continuous Galois representation (resp.~$\tau_v:I_{F^+_v}\ra\GL_3(E)$ is a tame inertial type for $F^+_v$) for all $v\in S_p$.

We have similar notions when $E$ is replaced by $\F$.
Abusing terminology, we dentify (tame inertial) $L$-parameters with a fixed choice of a representative in its class.
Nothing in what follows will depend on this choice.

The definitions and results of \S \ref{subsubsec:TIT}--\ref{subsubsec:combinatorics_TandW} generalize  for tame inertial $L$-parameters and $L$-homomorphism.
In particular, given a tame inertial $L$-parameter $\tau$ corresponding to the collection $(\tau_v)_{v\in S_p}$, we let $\sigma(\tau)$ be $\otimes_{v\in S_p} \sigma(\tau_v)$ (an irreducible smooth $G_0(\Z_p)$-representation over $E$).

\section{Geometry of multi-type deformation rings}
\label{sec:multitype}

\subsection{Multi-type deformation rings}
\label{subsec:MTdef}
The goal of this section is to introduce --and give the first tools to analyze-- deformation rings for $\rhobar:G_{K}\rightarrow \GL_3(\F)$, with $p$-adic Hodge theory conditions defined by a finite set of tame inertial types.

Let $\rhobar:G_{K}\rightarrow \GL_3(\F)$ be a continuous Galois representation, and let $R_{\rhobar}^\Box$ denote the universal lifting ring for $\rhobar$. 

If $\tau$ is an inertial type for $K$, let $R^{\leqeta,\tau}_{\rhobar}$ (resp.~$R^{\eta,\tau}_{\rhobar}$) be the reduced quotient of $R_{\rhobar}^\Box$ parametrizing potentially crystalline representations of type $\tau$ and Hodge--Tate weights $\leq (2,1,0)$ (resp.~equal to $(2,1,0)$). %
\begin{defn}
\label{defn:multi-type}
Let $T\defeq\{\tau:I_K\ra\GL_3(E)\}$ be a finite set of inertial types for $K$. 
We define $R^{\leqeta,T}_{\rhobar}$ to be the image of $R_{\rhobar}^\Box\to \prod_T R^{\leqeta,\tau}_{\rhobar}$. We say that $R^{\leqeta,T}_{\rhobar}$ is the $(\leqeta,T)$ \emph{multi-type deformation ring}. We similarly define $R^{\eta,T}_{\rhobar}$ by replacing ``$\leqeta$'' with $\eta$ everywhere.
\end{defn}

From now on we assume $\rhobar:G_{K}\rightarrow \GL_3(\F)$ is tame, and that $\rhobar|_{I_K}=\ovl{\tau}(s,\mu)$ where $\mu$ is $6$-deep in alcove $\un{C}_0$. Set $\tld{w}(\rhobar)\defeq t_{\mu+\eta}s$.

\subsubsection{Presentations of single type deformation rings} 
\label{subsub:single}
Suppose $T=\{\tau\}$ consists of a single tame inertial type. In this case, the local model theory of \cite{GL3Wild}, \cite{MLM} produces an explicit presentation of $R^{\leqeta,\tau}$ which we now recall.
In order for $R^{\leqeta,\tau}$ to be non-zero, $\tau$ must admit a lowest alcove presentation such that
\[\tld{w}^*(\rhobar,\tau)\defeq (\tld{w}(\tau)^{-1}\tld{w}(\rhobar))^*\]
belongs to $\Adm^{\vee}(\eta)$ by \cite[Corollary 5.5.8]{MLM} (in particular, this gives $\tau$ a $4$-generic lowest alcove presentation). We abbreviate $\tld{z}=(z_jt_{\nu_j})_{j\in\cJ}\defeq \tld{w}^*(\rhobar,\tau)$ for the remainder of this subsection.

We recall the following objects from \cite{GL3Wild},\cite{MLM}: 
\begin{itemize}
\item The Emerton--Gee stack $\cX^{\leqeta,\tau}$ parametrizing potentially crystalline representations with Hodge--Tate weights $\leqeta$ and inertial type $\tau$ (see \cite[\S 3.2]{GL3Wild}).
\item The moduli stack $Y^{\leqeta,\tau}$ of Breuil-Kisin modules for $K_\infty\defeq \bigcup_{n\in\N}K(\sqrt[p^n]{-p})$ 
with tame descent data of type $\tau$ and elementary divisors $\leqeta$. It admits an open substack $Y^{\leqeta,\tau}(\tld{z})$ 
(see \cite[Proposition 3.1.1]{GL3Wild}).
\item We also have the stack $\Phi\text{-}\Mod^{\text{\'et},3}_{K}$ parametrizing rank $3$ \'{e}tale $\phz$ modules over $K_\infty$ (see \cite[\S 5.4.1]{MLM}).
\item (cf.~\cite[\S 3.1]{GL3Wild}) $M(\tld{z})\defeq\prod_{j\in\cJ} M_j(\tld{z}_j)$ the affine scheme over $\cO$ parametrizing $\cJ$-tuples of matrices $3\times 3$ matrices $A\defeq (A^{(j)})_{j\in\cJ}$ whose entries are polynomials in $v$ with the following constraints for all $j\in\cJ$:
\begin{itemize}
\item $A^{(j)}_{ik}$ is divisible by $v$ for $i>k$.
\item $\deg A^{(j)}_{ik}\leq \nu_{j,k}-\delta_{i<z_j(k)}$, with equality when $i=z_j(k)$.
\item The leading coefficient of $A^{(j)}_{z_j(k)k}$ is a unit.
\item $\det A^{(j)}$, which is cubic with unit top coefficient, is a unit times $(v+p)^3$.
\end{itemize}
We furthermore have the closed subscheme $\tld{U}(\leqeta,\tld{z})\subset M(\tld{z})$ obtained as the $\cO$-flat closure of the locus where for all $j\in\cJ$ the $2\times 2$ minors of $A^{(j)}$ are divisible by $(v+p)$ (i.e.~they vanish when setting $v=-p$).
\end{itemize}
The relationship between these objects is summarized in the following diagram with cartesian squares \cite[Theorem 3.2.2]{GL3Wild}:
\begin{equation}
\label{eq:basic:diag}
\xymatrix{
\tld{U}(\tld{z},\leqeta)^{\wedge_p}\ar[r]& Y^{\leqeta,\tau}(\tld{z})\ar@{^{(}->}[r]& Y^{\leqeta,\tau}\ar@{^{(}->}[r] &\Phi\text{-}\Mod^{\text{\'et},3}_{K}\\
\tld{\cX}^{\leqeta,\tau}(\tld{z})
\ar[r]\ar@{^{(}->}[u]\ar@{}[ur]|\Box& \cX^{\leqeta,\tau}(\tld{z})\ar@{^{(}->}[r]\ar@{^{(}->}[u]\ar@{}[ur]|\Box&  \cX^{\leqeta,\tau}
\ar@{^{(}->}[u]\ar@{^{(}->}[ur]}
\end{equation}
whose relevant features for us are:
\begin{itemize}
\item The top left horizontal arrow identifies
\[Y^{\leq\eta,\tau}(\tld{z})=[\tld{U}(\tld{z},\leqeta)^{\wedge_p}/_\tau T^{\vee,\cJ}]\]
where the action of $(t_j)\in T^{\vee,\cJ}$ is the $\tau$-twisted shifted conjugation action defined by:
\[A^{(j)}\mapsto t_j A^{(j)} \sigma^{-1}_j(t_{j-1})\]
where $(\sigma,\kappa)$ is the lowest alcove presentation of $\tau$ determined at the beginning of \S \ref{subsub:single}.
\item $\tld{\cX}^{\leqeta,\tau}(\tld{z})$ identifies with an explicit closed, $T^{\vee,\cJ}$-stable formal subscheme $\tld{U}(\tld{z},\leqeta,\nabla_{\tau,\infty})$ which is obtained as the $p$-saturation of a natural ``monodromy condition'' on the Breuil--Kisin module living over $\tld{U}(\tld{z},\leqeta)^{\wedge_p}$ cf.~\cite[Definition 7.1.2 \& Corollary 7.1.5]{MLM}.  We also recall that there is a closed immersion $\tld{U}_{\mathrm{reg}}(\tld{z},\leqeta,\nabla_{\tau,\infty})\into \tld{U}(\tld{z},\leqeta,\nabla_{\tau,\infty})$ \cite[Theorem 7.3.2]{MLM} which correspond to the part $\cX^{\eta,\tau}(\tld{z})\into \cX^{\leqeta,\tau}(\tld{z})$ with Hodge-Tate weights exactly $\eta$.
\item The edges of the triangle are closed immersions.
\item The composite of the top horizontal arrow assigns to $A=(A^{(j)})$ the free \'{e}tale $\phz$-module with matrix of Frobenius given by $A\tld{w}^*(\tau)$.
\end{itemize}
Our chosen tame $\rhobar$ gives a point in $\cX^{\leqeta,\tau}(\tld{z})(\F)$ and a point $\ovl{\fM}_{\tau}\in Y^{\leqeta,\tau}(\F)$, whose lifts to $\tld{U}(\tld{z},\leqeta)(\F)$ has the form $\ovl{A}=\ovl{D}_\tau\tld{z}$ for $\ovl{D}_\tau\in T^{\vee,\cJ}(\F)$ well-defined up to $\tau$-shifted conjugation. Then $R^{\leqeta,\tau}_{\rhobar}$ is a versal ring to $\cX^{\leqeta,\tau}$ at $\rhobar$, thus diagram \eqref{eq:basic:diag} identifies $R^{\leqeta,\tau}_{\rhobar}$, up to formal variables, with the completion of the explicit (formal) scheme $\tld{U}(\tld{z},\leqeta,\nabla_{\tau,\infty})$ at $\ovl{A}$.

\subsubsection{Presentations of multi-type deformation rings}
\label{subsub:gpd}
We now let $T$ be a finite set of tame inertial types endowed with lowest alcove presentations such that $R^{\leqeta,\tau}_{\rhobar}\neq 0$ for all $\tau\in T$, thus we get $\tld{w}^*(\rhobar,\tau)\in  \Adm^{\vee}(\eta)$ for each $\tau\in T$. For each $j\in\cJ$ let $T^{(j)}\defeq\{\tld{w}^*(\rhobar,\tau)_j, \tau\in T\}$.
Throughout this paper, we restrict ourselves to $T$ satisfying the following:
\begin{hypothesis}\label{hypothesis:T}For each $j\in \cJ$, either:
\begin{enumerate}[label=(\Roman*)]
\item
\label{it:prop:T:1}
$T^{(j)}\subseteq\{t_{\un{1}},\alpha\beta t_{\un{1}},\beta\alpha t_{\un{1}},\alpha\beta\alpha t_{\un{1}},{\alpha\beta\alpha\gamma t_{\un{1}}}\}$; or
\item
\label{it:prop:T:1'}
 $T^{(j)}\subseteq\{t_{w_0(\eta)},  t_{w_0(\eta)}\alpha,  t_{w_0(\eta)}\beta, t_{w_0(\eta)}w_0\}$; or
\item 
\label{it:prop:T:2}
$T^{(j)}$ is a singleton {which does not fall in either of the cases above}, $\tld{w}^*(\rhobar,\tau)_j\in\Adm^\vee(\eta_j)$ and $\ell(\tld{w}(\rhobar,\tau)_j)\geq 2$.
\end{enumerate}
\end{hypothesis}
The basic idea to probe $R^{\leqeta,T}_{\rhobar}$ is to glue diagrams \eqref{eq:basic:diag} along $\Phi\text{-}\Mod^{\text{\'et},3}_{K}$ as $\tau$ varies over $T$. To this end, we consider the following:
\begin{defn}
\begin{enumerate}
\item Let $M^T\defeq \prod_{j\in\cJ} M^{T,(j)}$ be the affine scheme such that
\begin{itemize}
\item For $T^{(j)}$ as in \ref{it:prop:T:1}, $M^{T,\,(j)}$ consists of
\[
\begin{pmatrix}
d_{11}^{*(j)}(v+p)+c^{(j)}_{11}{+\frac{e_{11}^{(j)}}{v}}&c^{(j)}_{12}&c^{(j)}_{13}\\
d^{(j)}_{21}(v+p)+c^{(j)}_{21}&d^{*(j)}_{22}(v+p)+c^{(j)}_{22}&c^{(j)}_{23}\\
d^{(j)}_{31}(v+p)+c^{(j)}_{31}&d^{(j)}_{32}(v+p)+c^{(j)}_{32}
&d_{33}^{*(j)}(v+p)+c^{(j)}_{33}
\end{pmatrix}
\]
with determinant $(v+p)^3d^{*(j)}_{11}d^{*(j)}_{22}d_{33}^{*(j)}$ and $d^{*(j)}_{ii}$ invertible. 
\item For $T^{(j)}$ as in \ref{it:prop:T:1'}, $M^{T,\,(j)}$ consists of
\[
\begin{pmatrix}
d_{11}^{*(j)}&c^{(j)}_{12}&c_{13}^{(j)}(v+p)+e^{(j)}_{13}\\
d^{(j)}_{21}&d^{*(j)}_{22}(v+p)+c_{22}^{(j)}&c^{(j)}_{23}(v+p)+e^{(j)}_{23}\\
d_{31}^{(j)}&d^{(j)}_{32}(v+p)+c^{(j)}_{32}
&d_{33}^{*(j)}(v+p)^2+c_{33}^{(j)}(v+p)+e_{33}^{(j)}
\end{pmatrix}
\]
with determinant $(v+p)^3d^{*(j)}_{11}d^{*(j)}_{22}d_{33}^{*(j)}$ and $d^{*(j)}_{ii}$ invertible.
\item For $T^{(j)}$ as in \ref{it:prop:T:2}, $M^{T,\,(j)}=M_j(\tld{w}^*(\rhobar,\tau)_j)$ as in \S \ref{subsub:single}.
\end{itemize} 
\item Define the element $\tld{w}^{*, T}(\rhobar)\in\tld{W}^{\cJ}$ as 
\[
\tld{w}^{*, T}(\rhobar)_j=
\begin{cases}
t_{-\un{1}}\tld{w}^{*}(\rhobar)_j& \text{ if $T^{(j)}$ is as in item \ref{it:prop:T:1}}\\
t_{-w_0(\eta)}\tld{w}^{*}(\rhobar)_j& \text{ if $T^{(j)}$ is as in item \ref{it:prop:T:1'}}
\\
\tld{w}^{*}(\tau)_j& \text{else}.
\end{cases}
\]

\end{enumerate}
\end{defn}
The significance of the above definition is explained by:

\begin{prop}\label{prop:factor_mono}
For each $\tau\in T$, we have the diagram 
\begin{equation}
\label{eq:morp}
\xymatrix{
\tld{U}(\tld{w}^*(\tau,\rhobar),\leqeta)^{\wedge_p}\ar@{^{(}->}^-{r_{\tld{w}^*(\tau)}}[r]\ar^-{f.s.}[d]
&(M^{T})^{\wedge_p}\tld{w}^{*, T}(\rhobar)
\ar^-{}[r]\ar^{f.s.}[d]&
\Phi\text{-}\Mod^{\textnormal{{\'et}},3}_{K}\\
[\tld{U}(\tld{w}^*(\tau,\rhobar),\leqeta)^{\wedge_p}/_\tau T^{\vee,\cJ}]
\ar@{^{(}->}^-{r_{\tld{w}^*(\tau)}}[r]
&
\left[(M^{T})^{\wedge_p}\tld{w}^{*, T}(\rhobar)/T^{\vee,\cJ}\textnormal{{-sh.cnj}}\right]\ar[ur] }
\end{equation}
where 
\begin{itemize}
\item The vertical arrows are quotient maps by $T^{\vee,\cJ}$.
\item The top right arrow assigns to the tuple $(A^{(j)}\tld{w}^{*,T}(\rhobar)_j)_{j\in\cJ}$ the corresponding free \'{e}tale $\phz$-module with corresponding matrix of Frobenius.
\item The hooked horizontal arrows are induced by right multiplication by $\tld{w}^*(\tau)$ on the family of matrices. They are closed immersions, equivariant for the $\tau$-shifted conjugation action on the source and shifted conjugation action on the target.
\item The diagonal arrow is a monomorphism.
\end{itemize}
\end{prop}
\begin{rmk}\label{rmk:Noetherian approximation} By \cite[Theorem 5.4.20]{EGimage}, $\Phi\text{-}\Mod^{\textnormal{{\'et}},3}_{K}$ can be written as $\mathrm{colim}_h\Phi\text{-}\Mod^{\leq h,3}_{K}$ of Noetherian $p$-adic formal substacks, where $\Phi\text{-}\Mod^{\leq h,3}_{K}$ is the moduli stack of \'{e}tale $\phz$-modules with height $\leq h$ with respect to the choice of polynomial $F(v)=v(v+p)$. Then in \eqref{eq:basic:diag}, we may replace $\Phi\text{-}\Mod^{\textnormal{{\'et}},3}_{K}$ by $\Phi\text{-}\Mod^{\leq h,3}_{K}$ so that all the objects involved are now $p$-adic formal algebraic stacks, and notions such as scheme theoretic images behave as expected.
\end{rmk}
\begin{proof}
The last item follows from the proof of \cite[Proposition 5.4.4]{MLM}, while all the other items follow from the definitions.
\end{proof}
Recall from \S \ref{subsub:single} the closed subscheme $\tld{U}(\tld{w}^*(\tau,\rhobar),\leqeta,\nabla_{\tau,\infty})\subset \tld{U}
(\tld{w}^*(\tau,\rhobar),\leqeta)^{\wedge_p}$ characterized by 
$\cX^{\leqeta,\tau}=[\tld{U}(\tld{w}^*(\tau,\rhobar),\leqeta,\nabla_{\tau,\infty})/_\tau T^{\vee,\cJ}]$.
\begin{defn} 
\label{def:M_infty}
Let $M^{T,\nabla_\infty}\subset (M^T)^{\wedge_p}$ so that $M^{T,\nabla_\infty}\tld{w}^{*, T}(\rhobar)$ is the scheme theoretic image of 
\[\coprod_{\tau\in T} r_{\tld{w}^*(\tau)}:\coprod_{\tau\in T} \tld{U}(\tld{w}^*(\tau,\rhobar),\leqeta,\nabla_{\tau,\infty})\to M^T\tld{w}^{*, T}(\rhobar).\]
\end{defn}
We now turn to formal completions. Our given $\rhobar$ gives rise to a point $\cX^{\leqeta,\tau}(\F)$ for each $\tau\in T$, whose image $\rhobar_\infty\defeq\rhobar|_{G_{K_\infty}}$ in $\Mod^{\textnormal{{\'et}},3}_{K}(\F)$ is independent of $\tau$. By Proposition \ref{prop:factor_mono}, $\rhobar_\infty$ arises from a point $\ovl{A}\in M^T(\F)$ which is unique up to the $T^{\vee,\cJ}(\F)$ action. We fix the choice of such a point $\ovl{x}_{\rhobar}$. We note that the semisimplicity of $\rhobar$ is equivalent to the fact that $\ovl{A}\tld{w}^{*,T}(\rhobar)=\ovl{D}\tld{w}^*(\rhobar)$ for $\ovl{D}\in T^{\vee,\cJ}(\F)$, so that its pre-image under $r_{\tld{w}^*(\tau)}$ is $\ovl{D}\tld{w}^*(\rhobar,\tau)$. In particular the entries of $\ovl{D}$ correspond to the numbers $\ovl{d}_{ii}^{*(j)}\in \F$.

In the following Proposition, we adopt the convention that if $X$ is a stack and $x\in X(\F)$ then $X_x$ stands for the formal completion of $X$ at $x$, i.e.~the restriction of $X$ to Artinian test rings:
\begin{prop} \label{prop:image}
Let $\cX^{\leqeta,T}\defeq\bigcup_{\tau\in T}\cX^{\leqeta,\tau}$ be the scheme theoretic union of the $\cX^{\leqeta,\tau}$ for $\tau\in T$ inside $\Phi\text{-}\Mod^{\textnormal{{\'et}},3}_{K}$. Then 
\begin{itemize}
\item $R^{\leqeta,T}_{\rhobar}$ is a versal ring of $\cX^{\leqeta,T}$ at $\rhobar_\infty$;
\item $\cX^{\leqeta,T}_{\rhobar_\infty}=[M^{T,\nabla_\infty}_{\ovl{x}_{\rhobar}}\tld{w}^{*, T}(\rhobar)/T_{1}^{\vee,\cJ}\textnormal{{-sh.cnj}}]$ as subfunctors of $\Phi\text{-}\Mod^{\textnormal{{\'et}},3}_{K,\rhobar_\infty}$.
\end{itemize}
\end{prop}
\begin{proof} 
As in \cite[\S 3.6]{EGstacks}, the $G_{K_\infty}$ lifting ring $R^{\Box}_{\rhobar_\infty}$ is a versal ring for $\Phi\text{-}\Mod^{\textnormal{{\'et}},3}_{K}$ at $\rhobar_\infty$. 
After pulling back to $\Spf\, R^{\Box}_{\rhobar_\infty}$, $\cX^{\leqeta,\tau}$ becomes $\Spf\, R^{\leqeta,\tau}_{\rhobar}$, and hence $\cX^{\leqeta,T}$ becomes the formal spectrum of $\mathrm{im}(R^{\Box}_{\rhobar_\infty}\to \prod_{\tau\in T} R^{\leqeta,\tau}_{\rhobar})$. However by \cite[Proposition 3.12]{LLLM}, restriction to $G_{K_\infty}$ induces a surjection $R^\Box_{\rhobar_\infty}\onto R^{\Box}_{\rhobar}$, hence this image ring coincides with $R^{\leqeta,T}_{\rhobar}$. This gives the first item.

For the second item, by Proposition \ref{prop:factor_mono}, after pulling back to $\Spf\, R^{\Box}_{\rhobar_\infty}$, $[M^{T}\tld{w}^{*, T}(\rhobar)/T^{\vee,\cJ}\text{{-sh.cnj}}]$, $[M^{T,\nabla_\infty}\tld{w}^{*, T}(\rhobar)/T^{\vee,\cJ}\text{{-sh.cnj}}]$ becomes $\Spf\, R'$, $\Spf\, R$, where $R^{\Box}_{\rhobar_\infty}\onto R'\onto R^{\leqeta,\tau}_{\rhobar}$ for each $\tau\in T$, and by definition $R$ is the image of the natural map $R'\to \prod_{\tau\in T} R^{\leqeta,\tau}_{\rhobar}$. 
But this shows $R=R^{\leqeta,T}_{\rhobar}$.
\end{proof}

\begin{prop}
\label{prop:big_diagram}
Recall that $T$ satisfies Hypothesis \ref{hypothesis:T} and $\rhobar$ is tame, $6$-generic. Then there is an isomorphism
\[R^{\eta,T}_{\rhobar}[\![x_1,\cdots, x_{3f}]\!]\cong \tld{S}/\tld{I}_{T,\nabla_\infty}[\![\{y_i\}_{1\leq i\leq 9}]\!]  \]
where
\begin{itemize}
\item
$\tld{S}\defeq \widehat{\otimes}_{j\in\cJ}\tld{S}^{(j)}$, where $\tld{S}^{(j)}$ is the ring in Tables \ref{Table_Ideals}, \ref{Table_Ideals_1} ,  \ref{Table_Ideals_2} in cases \ref{it:prop:T:1}-\ref{it:prop:T:1'}, and Table \cite[Table 1]{GL3Wild} in case \ref{it:prop:T:2}. 
\item For each $\tau\in T$, $\tld{I}_{\tau,\nabla_\infty}\defeq\sum_{j\in\cJ}\tld{I}^{(j)}_{\tau,\nabla_\infty}$ where for each $j\in\cJ$ the ideal $\tld{I}^{(j)}_{\tau,\nabla_\infty}\subseteq \tld{S}$ has the form described in Tables \ref{Table_Ideals}, \ref{Table_Ideals_1} and \ref{Table_Ideals_2} if $T^{(j)}$ is as in item \ref{it:prop:T:1}--\ref{it:prop:T:1'}, and are described in \cite[Table 2]{GL3Wild}\footnote{and moreover adding a term $O(p^{N-1})$ to each generator of the  ideals $\tld{I}^{(j)}_{\tau,\nabla_\infty}$ of \cite[Table 2]{GL3Wild}} if $T^{(j)}$ is as in item \ref{it:prop:T:2}.
\item $\tld{I}_{T,\nabla_\infty}\defeq \bigcap_{\tau\in T} \tld{I}_{\tau,\nabla_\infty}$.
\end{itemize}
\end{prop}
\begin{warning}We emphasize that:
\begin{itemize} 
\item
the $O(p^{N-4})$-tails appearing in the $\tld{I}^{(j)}_{\tau,\nabla_\infty}$-rows of tables \ref{Table_Ideals}, \ref{Table_Ideals_1}, \ref{Table_Ideals_2} \emph{involve variables in all embedding};
\item the structure constants $b_{\tau,1},b_{\tau,2},b_{\tau,3}\in \Zp$ appearing in the $\tld{I}^{(j)}_{\tau,\nabla_\infty}$-rows of tables \ref{Table_Ideals}, \ref{Table_Ideals_1}, \ref{Table_Ideals_2} \emph{depend on the whole $f$-tuple $\tld{w}(\rhobar,\tau)$}.
\end{itemize}
In particular $\tld{I}^{(j)}_{\tau,\nabla_\infty}$ is \emph{not} an ideal of $\tld{S}^{(j)}$, and $\tld{I}^{(j)}_{\tau,\nabla_\infty}+O(p^{N-4})$ \emph{does not} depend only on $\tld{w}(\rhobar,\tau)_j$ in general.
\end{warning}
\begin{proof} This follows from Proposition \ref{prop:image}, and the following observations:
\begin{itemize}
\item $\tld{S}$ is the formal power series ring on the coefficients of the entries of the matrices that $M^T$ parametrizes (where the variables corresponding to unit entries are characterized by $d^*_{ii}=[\ovl{d}^*_{ii}](1+x^*_{ii})$), and thus $M^T_{\ovl{x}_{\rhobar}}$ is the locus in $\Spf\tld{S}$ where the determinant conditions on the matrices are imposed.
\item By \cite[Theorem 3.2.2]{GL3Wild} and the discussion in \cite[\S 3.6.1]{LLLM2}, quotiening by $\tld{I}_{\tau,\nabla_\infty}$ corresponds to $\tld{U}_{\mathrm{reg}}(\tld{w}^*(\rhobar,\tau),\leqeta,\nabla_{\tau_\infty})$ via the closed immersion $r_{\tld{w}^*(\tau)}$ in diagram \eqref{eq:morp}. Note that a priori we only know $\tld{I}_{\tau,\nabla_\infty}$ is the $p$-saturation of $\sum_{j\in\cJ}\tld{I}^{(j)}_{\tau,\nabla_\infty}$. However by \cite[\S 3.6.1]{LLLM2}, the first bullet point at page 54, the natural surjection $\tld{S}/\sum_{j\in\cJ}\tld{I}^{(j)}_{\tau,\nabla_\infty}\onto \tld{S}/\tld{I}_{\nabla_{\tau,\infty}}$ is an isomorphism, and hence we don't need to $p$-saturate. 
\end{itemize}
\end{proof}
\begin{rmk} (Compatibility under shrinking $T$)
\label{rmk:compatibility:multitype}
Suppose $T'\subset T$ satisfy Hypothesis \ref{hypothesis:T}. 
Then we have a commutative diagram
\[
\xymatrix{
\Spf\, R_{\rhobar}^{\leqeta,T'}\ar@/_2.0pc/@{^{(}->}[dd]\ar@{^{(}->}[d]\ar[r]&
{\text{\footnotesize{$\left[M^{T',\nabla_{\infty}}\cdot\tld{w}^{*,T'}(\rhobar)/T^{\vee,\cJ}\text{{-sh.cnj}}\right]$}}}\ar@/^6.0pc/@{^{(}->}[dd]\ar@{^{(}->}[d]
\\
\Spf\, R_{\rhobar}^{\leqeta,T}\ar@{^{(}->}[d]\ar[r]&
{\text{\footnotesize{$\left[M^{T,\nabla_{\infty}}\cdot\tld{w}^{*,T}(\rhobar)/T^{\vee,\cJ}\text{{-sh.cnj}}\right]$}}}\ar@{^{(}->}[d]
\\
\Spf\, R^{\Box}_{\rhobar_\infty}\ar[r]&
\Phi\text{-}\Mod^{\textnormal{{\'et}},3}_{K}
}
\]
with formally smooth horizontal arrows. It follows that the isomorphism in Proposition \ref{prop:big_diagram} for $T'\subset T$ can be chosen to be compatible with the inclusion of ideals $\tld{I}_{T,\nabla_\infty}\subset \tld{I}_{T',\nabla_\infty}$.
\end{rmk}
We also record the following result, which will be later used to analyze $\tld{I}_{T,\nabla_\infty}$:
\begin{prop}\label{prop:ideal2type}
\label{prop:ideals_coprimes}
Let $\tld{I}_\tau=\sum_{j\in\cJ}\tld{I}^{(j)}_{\tau}$ where for each $j\in\cJ$ the ideal $\tld{I}_\tau^{(j)}\subseteq \tld{S}^{(j)}$  is described in Tables \ref{Table_Ideals}, \ref{Table_Ideals_1}, \ref{Table_Ideals_2}.
Then
\begin{enumerate}
\item $\tld{I}_\tau$ correspond to the closed immersion $r_{\tld{w}^*(\tau)}:\tld{U}(\tld{w}^*(\rhobar,\tau),\leqeta)\into M^T\tld{w}^{*, T}(\rhobar)$.
\item If $\tld{w}(\rhobar,\tau)_j\neq \tld{w}(\rhobar,\tau')_j$ then
\[
p^2\in \tld{I}^{(j)}_{\tau}+\tld{I}^{(j)}_{\tau'}.
\]
Moreover if $\ell\big(\tld{w}(\rhobar,\tau)_j)\geq 2$, $\ell(\tld{w}(\rhobar,\tau')_j\big)\geq 2$ and $\{\tld{w}(\rhobar,\tau)_j,\tld{w}(\rhobar,\tau')_j\}\neq \{\alpha\beta\alpha\gamma t_{\un{1}},\alpha\beta\alpha t_{\un{1}}\}$ we have 
\[
p\in \tld{I}^{(j)}_{\tau}+\tld{I}^{(j)}_{\tau'}.
\]
\end{enumerate}
\end{prop}
\begin{proof}
The first item is due to the fact the top group of generators of $\tld{I}^{(j)}_{\tau}$ (appearing in the corresponding row of Table \ref{Table_Ideals}, \ref{Table_Ideals_1}, \ref{Table_Ideals_2}) cut out the degree bound conditions, while the bottom group cut out the elementary divisor conditions of $\tld{U}(\tld{w}^*(\rhobar,\tau),\leqeta)$.

The second item follows from inspecting the tables. We give a sample computation for the case $\tld{w}^*(\rhobar,\tau)_j=\alpha\beta t_{\un{1}}$, $\tld{w}^*(\rhobar,\tau')_j=t_{\un{1}}$ (for other cases see \S \ref{appendix:multi:sp:fiber}). 
In the ring $\tld{S}^{(j)}/(\tld{I}^{(j)}_{\tau}+\tld{I}^{(j)}_{\tau'})$ we have:
\[c_{12}d_{21}c_{33}=c_{12}c_{23}d_{31}=\frac{c_{13}d_{32}c_{23}d_{31}}{d_{33}^*}=-pc_{13}d_{31}d_{22}^*=-pd_{22}^*d_{33}^*\]
Using this, and the relations $c_{22}=-pd_{22}^*$, $c_{33}=-pd_{33}^*$, the last generator in $\tld{I}^{(j)}_{t_{\un{1}}}$ becomes
\[c_{12}d_{21}c_{33} - c_{11}c_{22}d_{33}^* - c_{11}d_{22}^*c_{33} - d_{11}^*c_{22}c_{33} - p(c_{11}d_{22}^*d_{33}^*  +d_{11}^*c_{22}d_{33}^* +d_{11}^*d_{22}^*c_{33})=2p^2d_{11}^*d_{22}^*d_{33}^*.\]
\end{proof}

\subsection{Commutative algebra}\label{sec:ca}

In this section, we collect various commutative algebra results that we require in later sections. 

\begin{lemma}\label{lemma:disjointass}
Let $R$ be a Noetherian ring and $M$ a finitely generated $R$-module. 
Suppose that we have an exact sequence 
\[
0 \ra L \ra M \ra N \ra 0
\]
of $R$-modules such that $\supp(L)$ and $\mathrm{Ass}(N)$ are disjoint. 
Then $L$ is the kernel of the natural map
\[
\lambda: M \ra \bigoplus_{\fp \in \mathrm{Ass}(N)} M_{\fp}. 
\]
\end{lemma}
\begin{proof}
Let $L'$ be $\ker\lambda$. 
Then $L \subset L'$ since $L_{\fp} = 0$ for all $\fp \in \mathrm{Ass}(N)$ by assumption. 
Now let $a \in L'$ and denote by $\ovl{a}$ the image in $N$. 
For any $\fp \in \mathrm{Ass}(N)$, the image of $a$ in $M_{\fp}$ is $0$, and so the image of $\ovl{a}$ in $N_{\fp}$ is $0$. 
We conclude from \cite[\href{https://stacks.math.columbia.edu/tag/0311}{Tag 0311}]{stacks-project} that $\ovl{a} = 0$ or equivalently that $a \in L$. 
Thus, $L = L'$. 
\end{proof}

Throughout the rest of this section $(R,\fm)$ denotes a local ring with residue field $\F$. 
(In later applications, $\F$ is as in \S \ref{sec:notation}, though we do not require this here.) 

\begin{cor}\label{cor:disjointass}
Let $R$ be a Noetherian local ring and $M$ a finitely generated $R$-module. 
Suppose that we have an exact sequence 
\[
0 \ra L \ra M \ra N \ra 0
\]
of $R$-modules where $N$ is a finitely generated maximal Cohen--Macaulay $R$-module and $\supp(L)\cap \supp(N)$ contains no minimal primes of $R$. 
Then the image of $L$ in $M$ is the kernel of the natural map
\[
M \ra \bigoplus_{\fp \in \supp(N) \textrm{ minimal}} M_{\fp}. 
\]
\end{cor}
\begin{proof}
By \cite[Theorem 17.3(i)]{matsumura}, $N$ has no embedded primes so that $\mathrm{Ass}(N)$ is precisely the set of minimal primes of $R$ in $\supp(N)$. 
Thus $\mathrm{Ass}(N)$ and $\supp(L)$ are disjoint by assumption. 
The result then follows from Lemma \ref{lemma:disjointass}. 
\end{proof}

\begin{lemma}\label{lemma:CAfusion}
Let $R$ be a local ring with residue field $\F$. 
Let $I,J\subset K \subset R$ be ideals with $K$ finitely generated. 
Let $M$ be the kernel of the map
\[
R/I \oplus R/J \ra R/K
\]
which is the difference of the natural surjections. 
Then the following are equivalent:
\begin{enumerate}
\item \label{item:cyclic} $M$ is a cyclic $R$-module; 
\item \label{item:idealsum} $I+J = K$; and
\item \label{item:torsum} the induced map $\Tor^R_1(\F,R/I)\oplus \Tor^R_1(\F,R/J) \ra \Tor^R_1(\F,R/K)$ is surjective. 
\end{enumerate}
If these equivalent conditions hold then:
\begin{enumerate}[label=(\roman*)]
\item \label{item:idealintersect} $M \cong R/I\cap J$; and
\item \label{item:torintersect} the image of $\Tor^R_1(\F,M) \ra \Tor^R_1(\F,R/K)$ induced by the composition $M \ra R/I \ra R/K$ is the intersection of the images of $\Tor^R_1(\F,R/I) \ra \Tor^R_1(\F,R/K)$ and $\Tor^R_1(\F,R/J) \ra \Tor^R_1(\F,R/K)$. 
\end{enumerate}
\end{lemma}
\begin{proof}
We first show that \eqref{item:idealsum} implies \eqref{item:cyclic} and \ref{item:idealintersect}. 
Suppose that $I+J = K$. Then the natural map $\psi: R/(I\cap J) \ra R/I\oplus R/J$ is injective with cokernel $R/(I+J)$ using that $R/I\oplus R/J = (0\oplus R/J) + \mathrm{im}(\psi)$. 
This identifies $M$ with the cyclic $R$-module $R/(I\cap J)$. 

The inclusion $I+J \subset K$ is an equality if and only the map $I\oplus J \ra K$ obtained by taking the difference is a surjection. 
By Nakayama's lemma, this is equivalent to the surjectivity of the induced map $I \otimes_R \F \oplus J \otimes_R \F \ra K \otimes_R \F$. 
As there is a functorial identification of $\Tor^R_1(\F,R/L) \cong L\otimes_R \F$ for ideals $L \subset R$, we see that \eqref{item:idealsum} and \eqref{item:torsum} are equivalent. 

Finally we show that \eqref{item:cyclic} implies \eqref{item:torsum} and \ref{item:torintersect}. 
If $M$ is a cyclic $R$-module, then the sequence 
\[
0 \ra M\otimes_R \F \ra (R/I)\otimes_R \F \oplus (R/J)\otimes_R \F \ra (R/K)\otimes_R \F \ra 0
\]
is exact. 
The $\Tor^R$-long exact sequence gives the exact sequence 
\[
\Tor^R_1(\F,M) \ra \Tor^R_1(\F,R/I)\oplus \Tor^R_1(\F,R/J) \ra \Tor^R_1(\F,R/K) \ra 0
\]
where the second map is the difference of the natural maps $\Tor^R_1(\F,R/I) \ra \Tor^R_1(\F,R/K)$ and $\Tor^R_1(\F,R/J) \ra \Tor^R_1(\F,R/K)$. 
This gives \eqref{item:torsum}. 
As the image of $\Tor^R_1(\F,M)$ in $\Tor^R_1(\F,R/I)\oplus \Tor^R_1(\F,R/J)$ consists of pairs $(a,b)$ such that the images of $a$ and $b$ in $\Tor^R_1(\F,R/K)$ coincide, \ref{item:torintersect} follows. 
\end{proof}

\begin{lemma}\label{lemma:CAfinalglue}
Let $R$ be a local ring with residue field $\F$. 
Let $1\leq n$ and $I_i\subset K$ be ideals of $R$ for $i=1,\dots,n$. 
Let $M$ be a finitely generated $R$-module with a fixed surjection to $R/K$. 
Let $N$ be the kernel of the map 
\[
M\oplus\oplus_{i=1}^n R/I_i  \ra (R/K)^{\oplus (n+1)}/\Delta(R/K)
\]
induced by the sum of the natural maps $R/I_i \onto R/K$ and the fixed map $M \onto R/K$, and where $\Delta(R/K)$ denotes the diagonally embedded copy of $R/K$. 

Write $V_i$ $($resp.~$W$$)$ for the image of the induced map
\[
\Tor^R_1(\F,R/I_i)\ra\Tor^R_1(\F,R/K)
\]
for $1\leq i\leq n$ $($resp.~$\Tor^R_1(\F,M)\ra\Tor^R_1(\F,R/K)$$)$. 
Assume that for all $1\leq j\leq n$, 
\[
V_j+W\cap \cap_{i\neq j} V_i = \Tor^R_1(\F,R/K).
\]
Then the projection map $N \ra M$ induces an isomorphism $N\otimes_R \F \ra M\otimes_R \F$. 
\end{lemma}
\begin{proof}
From the $\Tor$-exact sequence and using that $\oplus_{i=1}^n R/I_i \ra (R/K)^{\oplus (n+1)}/\Delta(R/K)$ is an isomorphism after applying $-\otimes_R \F$, it suffices to show that the map 
\begin{equation}\label{eqn:tor}
\Tor^R_1(\F,M\oplus \oplus_{i=1}^n R/I_i) \ra \Tor^R_1(\F,R/K)^{\oplus (n+1)}/\Delta(\Tor^R_1(\F,R/K))
\end{equation}
is surjective. 
Writing an element of $\Tor^R_1(\F,R/K)^{\oplus (n+1)}$ as $(x_i)_{i=0}^n$, it suffices to show that for any $0<j\leq n$ and $a_j \in \Tor^R_1(\F,R/K)$ and setting $a_i = 0$ for $i\neq j$, $(a_i)_{i=0}^n+\Delta(\Tor^R_1(\F,R/K))$ is in the image of \eqref{eqn:tor}. 
By assumption, we can write $a_j = b_j+c_j$ where $b_j \in V_j$ and $c_j \in W\cap \cap_{i\neq j} V_i$. 
Then $(a_i - c_j)_{i=0}^n \in (a_i)_{i=0}^n+\Delta(\Tor^R_1(\F,R/K))$ and $(a_i - c_j)_{i=0}^n$ is in the image of the map 
\[
\Tor^R_1(\F,\oplus_{i=1}^n R/I_i \oplus M) \ra \Tor^R_1(\F,(R/K)^{\oplus (n+1)}).
\]
\end{proof}

In a similar fashion using the Tor-exact sequence we obtain
\begin{lemma}\label{lemma:glue:lower}
Let $R$ be a local ring with residue field $\F$. 
Let $1\leq n$ and $I_i\subset K$ be ideals of $R$ for $i=1,\dots,n$. 
Let $N$ be the kernel of the map 
\[
\oplus_{i=1}^n R/I_i  \ra (R/K)^{\oplus n}/\Delta(R/K)
\]
induced by the sum of the natural maps $R/I_i \onto R/K$, where $\Delta(R/K)$ denotes the diagonally embedded copy of $R/K$. 

Then
\[
\dim_{\F}(N\otimes_R\F)=1+\dim_{\F}\Bigg(\coker\bigg(\oplus_{i=1}^n\Tor_1^R(\F,R/I_i)\ra \Tor_1^R(\F,(R/K)^{\oplus n}/\Delta(R/K))\bigg)\Bigg).
\]
\end{lemma}

\begin{lemma}[``Distortion'' Lemma]
\label{lem:distortion}
Let $(R,\fm)$ be a local $\cO$-algebra with residue field $\F$ with $\cO$ and $\F$ as in \S \ref{sec:notation}.
Let $k\geq 2$ and $\{I_1,\dots,I_k\}$ be $p$-saturated ideals of $R$.

Let $f\in R$ and assume that:
\begin{enumerate}
\item for each $\ell=1,\dots,k$ there exists $\eps_\ell\in \mathfrak{m}$ such that $f+p\eps_\ell\in I_\ell$
\item
\label{hyp:distorsion:2}   we have
\[p\in \sum_{\ell=1}^k \bigcap_{i\neq \ell}I_{i}.
\]
\end{enumerate}
Then there exist $a_\ell\in \bigcap_{i\neq \ell}I_{i}$ for all $\ell=1,\dots,k$ such that 
\begin{itemize}
\item
$f+\sum_{\ell}a_{\ell}\eps_\ell\in \bigcap_{i=1}^kI_i$
\item
$\sum_{\ell}a_{\ell}\eps_\ell\in \mathfrak{m}\Big(\bigcap_{i=1}^k (p)+I_i\Big)$.
\end{itemize}
\end{lemma}
\begin{proof}
By \eqref{hyp:distorsion:2} we can find  $a_\ell\in \bigcap_{i\neq \ell}I_{i}$ such that $p=\sum_\ell a_\ell$.
Then $p(f+\sum_{\ell}a_{\ell}\eps_\ell)=\sum_\ell a_\ell(f+p\eps_\ell)\in \cap_{i=1}^kI_i$ hence $f+\sum_{\ell}a_{\ell}\eps_\ell\in \cap_{i=1}^kI_i$ (because $\cap_{i=1}^kI_i$ is $p$-saturated).

Furthermore, for each $\ell$, $a_\ell\in \bigcap_{i=1}^k\big((p)+I_i\big)$, hence the second item follows.
\end{proof}

\begin{lemma}\label{lemma:unitlift}
Let $R$ and $S$ be complete local Noetherian $\F$-algebras.
Let $I$ be a proper ideal of $R$ and $M$ be a finitely generated $S$-module.
Then for any $\psi \in \Aut_{R\widehat{\otimes} S}((R/I) \widehat{\otimes} M)$, there exists $\tld{\psi} \in \Aut_{R\widehat{\otimes} S}(R\widehat{\otimes} M)$ such that the diagram
\[
\xymatrix{
R\widehat{\otimes} M \ar^{\tld{\psi}}[r]\ar@{->>}[d]& R\widehat{\otimes} M\ar@{->>}[d]\\
(R/I) \widehat{\otimes} M\ar^{\psi}[r]& 
(R/I) \widehat{\otimes} M
}
\]
commutes where the vertical maps are the natural projections.
(All completed tensor products are taken over $\F$.)
\end{lemma}
\begin{proof}
Observe that $\Hom_{R\widehat{\otimes} S}(R\widehat{\otimes}M,R\widehat{\otimes}-)\cong R\widehat{\otimes}\Hom_{S}(M,-)$ as   functors on finitely generated $S$-modules, since they are left exact in $M$ and are isomorphic when $M=S$.
Hence $\End_{R\widehat{\otimes} S}(R\widehat{\otimes} M) \cong R \widehat{\otimes} \End_S(M)$ surjects onto $\End_{(R/I)\widehat{\otimes} S}((R/I) \widehat{\otimes} M) \cong (R/I) \widehat{\otimes} \End_S(M)$.

Thus any element $a \in \Aut_{R\widehat{\otimes} S}((R/I) \widehat{\otimes} M)$ can be lifted to an element $\widetilde{a} \in \End_{R\widehat{\otimes} S}(R\widehat{\otimes} M)$.
We claim that any such lift is in fact in $\Aut_{R\widehat{\otimes} S}(R\widehat{\otimes} M)$.
Consider the left ideal $\End_{R\widehat{\otimes} S}(R\widehat{\otimes} M) \cdot \widetilde{a}$.
Since this ideal surjects onto $\End_{R\widehat{\otimes} S}((R/I) \widehat{\otimes} M)$, we have that 
\[\End_{R\widehat{\otimes} S}(R\widehat{\otimes} M) \cdot \widetilde{a}+I\widehat{\otimes} S \cdot \End_{R\widehat{\otimes} S}(R\widehat{\otimes} M) = \End_{R\widehat{\otimes} S}(R\widehat{\otimes} M).\]
Since every term in this equation is a finitely generated left $R\widehat{\otimes} S$-module, Nakayama's lemma implies that $\End_{R\widehat{\otimes} S}(R\widehat{\otimes} M) \cdot \widetilde{a} = \End_{R\widehat{\otimes} S}(R\widehat{\otimes} M)$ so that $\widetilde{a} \in \Aut_{R\widehat{\otimes} S}(R\widehat{\otimes} M)$.
\end{proof}

\subsection{Special fiber}
\label{sub:SF}

Recall from \ref{subsec:MTdef} that $(s,\mu)$ is a fixed lowest alcove presentation for $\rhobar$ and we assume from now on that $\mu$ is $N\geq 6$-deep.
We write $S$, $S^{(j)}$, $I_{T,\nabla_\infty}$ etc.~for the mod $p$ reduction of $\tld{S}$, $\tld{S}^{(j)}$, $\tld{I}_{T,\nabla_\infty}$ etc.

Let $\tau\in T$. Then by \cite[Theorem 3.3.2]{GL3Wild}, the minimal primes of $\ovl{R}^{\eta,\tau}_{\rhobar}$ are in bijection with the Serre weights $\sigma\in W^?(\rhobar)\cap \JH(\ovl{\sigma}(\tau))$. Since the underlying topological space of $\Spec \ovl{R}^{\eta,T}_{\rhobar}$ and $\bigcup_{\tau\in T}\Spec \ovl{R}^{\eta,\tau}_{\rhobar}$ coincide, and $\Spec({S}/{I}_{T,\nabla_{\infty}})$ is a formally smooth modification of $\Spec \ovl{R}^{\eta,T}_{\rhobar}$ by Proposition \ref{prop:big_diagram}, we learn that the minimal primes of ${S}/{I}_{T,\nabla_{\infty}}$ are in bijection with $W^?(\rhobar)\cap \bigcup_{\tau\in T} \JH(\ovl{\sigma}(\tau))$. We denote by $\ovl{\fP}_{\sigma}$ the miminal prime of ${S}/{I}_{T,\nabla_{\infty}}$ corresponding to $\sigma$.

We now give an approximation of $\ovl{R}^{\eta,T}_{\rhobar}$.
Let $\tld{I}^{(j)}_{\tau,\nabla_{\textnormal{alg}}}$ be the ideal of $\tld{S}^{(j)}$ generated by the elements listed in in row $\tld{I}^{(j)}_{\tau,\nabla_{\infty}}$ of Tables \ref{Table_Ideals}, \ref{Table_Ideals_1}, \ref{Table_Ideals_2} \emph{without their $O(p^{N-4})$-tails} if $T^{(j)}$ is as in item \ref{it:prop:T:1}--\ref{it:prop:T:1'} (resp.~the ideal appearing in \cite[Table 2]{GL3Wild} row $\tld{z}_j$ if $T^{(j)}=\{\tld{z}_j\}$ is as in item \ref{it:prop:T:2}), and write $\tld{I}_{T,\nabla_{\textnormal{alg}}}$ for the ideal $\bigg(\bigcap_{\tau\in T}\Big(\sum_{j\in\cJ}\tld{I}^{(j)}_{\tau,\nabla_{\textnormal{alg}}}\tld{S}\Big)\bigg)$ of $\tld{S}$.
Write $I_{T,\nabla_{\textnormal{alg}}}$ for the image of $\tld{I}_{T,\nabla_{\textnormal{alg}}}$ in ${S}$, and similarly define $I^{(j)}_{\tau,\nabla_{\textnormal{alg}}}$ for $\tau\in T$, $j\in \cJ$.

\begin{prop}
\label{prop:special_fiber}
Fix $j\in\cJ$ and assume $T$ has the following form: 
\begin{enumerate}
\item 
\label{hyp:special_fiber:1}
$\#T^{(j')}=1$ if $j'\neq j$;
\item 
\label{hyp:special_fiber:2}
either $T^{(j)}\subseteq \{\alpha\beta\alpha t_{\un{1}},\beta\alpha t_{\un{1}},\alpha\beta t_{\un{1}}, t_{\un{1}}\}$ or $T^{(j)}\subseteq\{t_{w_0(\eta)}, \alpha t_{w_0(\eta)}, \beta t_{w_0(\eta)}, w_0t_{w_0(\eta)}\}$.
\end{enumerate}
Assume that $\mu$ is $N>{10}$ deep in $\un{C}_0$.
Then we have a surjection
\begin{equation}
\label{eq:iso:SPF}
{S}/ {I}_{T,\nabla_{\textnormal{alg}}}\onto {S}/{I}_{T,\nabla_{\infty}}.
\end{equation}
For each $\sigma\in W^{?}(\rhobar)\cap \bigcup_{\tau\in T}\JH(\ovl{\sigma}(\tau))$, $\ovl{\fP}_{\sigma}$ pulls back  to the prime ideal $\sum_{j=0}^{f-1}{\fP}^{(j)}_{(\eps_j,a_j)} {S}$ of ${S}$ where:
\begin{itemize}
\item 
$(\eps,a)=((\eps_j,a_j))_{j\in\cJ}\in r(\Sigma)$ is such that $\sigma=F(\mathfrak{Tr}_{\mu+2\eta}(s(\eps,a)))$;
\item the ideal ${\fP}^{(j)}_{(\eps_j,a_j)}$ is the prime ideal of ${S}^{(j)}$ described in Table \ref{Table:components} (resp.~Table \ref{Table:components:F2}) if $T^{(j)}\subseteq \{\alpha\beta\alpha t_{\un{1}},\beta\alpha t_{\un{1}},\alpha\beta t_{\un{1}}, t_{\un{1}}\}$ (resp.~if $T^{(j)}\subseteq\{t_{w_0(\eta)}, \alpha t_{w_0(\eta)}, \beta t_{w_0(\eta)}, w_0t_{w_0(\eta)}\}$).
\end{itemize}
\end{prop}
\begin{proof}
Letting $\tld{I}_{\tau,\nabla_{\textnormal{alg}}}\defeq \sum_j \tld{I}_{\tau,\nabla_{\textnormal{alg}}}^{(j)}\tld{S}$, the existence of the surjection \eqref{eq:iso:SPF} is equivalent to the inclusion
\[
\bigcap_{\tau\in T}\tld{I}_{\tau,\nabla_{\textnormal{alg}}}+(p)\subseteq\bigcap_{\tau\in T}\tld{I}_{\tau,\nabla_{\infty}}+(p)
\]
For each $j\in \cJ$ and $\tau\in T$ let $\big\{g_{\tau,i,\textnormal{alg}}^{(j)}\big\}_i$ be the set of generators of $\tld{I}_{\tau,\nabla_{\textnormal{alg}}}^{(j)}$ listed in row $\tld{w}^*(\rhobar,\tau)_j$ of Tables \ref{Table_Ideals}, \ref{Table_Ideals_1}, \ref{Table_Ideals_2} without their $\cO(p^{N-4})$-tails, so that we can write
$g_{\tld{z}_j,i,\textnormal{alg}}^{(j)}=g_{\tau,i,\infty}^{(j)}-\cO(p^{N-4})$, where $\cO(p^{N-4})$ is an element of $p^{N-4}\tld{S}$ (depending  on $\tau$) and $\big\{g_{\tau,i,\infty}^{(j)}\big\}_{i}\subset \tld{I}_{\tau,\nabla_{\infty}}$.

By Proposition \ref{prop:ideals_coprimes} and the fact that $\tld{I}_{\tau,\nabla_{\textnormal{alg}}}^{(j)}\subseteq \tld{S}^{(j)}$ for any $j\in\cJ$ we have
\begin{equation}
\label{eq:inc:crucial}
p^{6}\bigg(\bigcap_{\tau\in T}\tld{I}_{\tau,\nabla_{\textnormal{alg}}}\bigg)=
p^{6}\bigg(\sum_{j\in\cJ} \bigg(\bigcap_{\tau\in T}\tld{I}_{\tau,\nabla_{\textnormal{alg}}}^{(j)}\bigg)\tld{S}\bigg)\subseteq\sum_{j\in\cJ} \bigg(\prod_{\tau\in T}\tld{I}_{\tau,\nabla_{\textnormal{alg}}}^{(j)}\bigg)\tld{S}.
\end{equation}
Given $f\in \bigcap_{\tau\in T}\tld{I}_{\tau,\nabla_{\textnormal{alg}}}$ we can write $p^{6}f$ as a $\tld{S}$-linear combination of multiples of $\prod_{\tau\in T} g_{\tau,i,\textnormal{alg}}^{(j)}$. 
Setting $\tld{f}_\infty$ to be the same linear combination as $p^6f$ but replacing $\prod_{\tau\in T} g_{\tau,i,\textnormal{alg}}^{(j)}$ by $\prod_{\tau\in T} g_{\tau,i,\infty}^{(j)}$, we thus have 
\[
p^{6}f+\cO(p^{N-4})=\tld{f}_\infty\in \bigcap_{\tau\in T} \tld{I}_{\tau,\nabla_{\infty}}.
\]
yielding $p^{6}(f+\cO(p^{N-4-6}))\in \bigcap_{\tau\in T} \tld{I}_{\tau,\nabla_{\infty}}$ by the assumption on $N$. 
As the intersection of the $p$-saturated ideals $\tld{I}_{\tau,\nabla_{\infty}}$ is again $p$-saturated we conclude that $f+\cO(p^{N-4-6})\in \bigcap_{\tau\in T} \tld{I}_{\tau,\nabla_{\infty}}$.
The labeling and the explicit equations for the ideals ${\fP}_{(\eps_j,a_j)}^{(j)}$ follow from in \cite[Theorem 3.3.2]{GL3Wild}.
\end{proof}

\begin{rmk}
\label{rmk:iso:nose}
Fix $\tau\in T$.
The statement of Proposition \ref{prop:special_fiber} can be improved, replacing condition \eqref{hyp:special_fiber:1} by \begin{enumerate}
\item[(1')] if $j'\neq j$ then
either $T^{(j')}\subseteq \{\alpha\beta\alpha\gamma t_{\un{1}},\alpha\beta\alpha t_{\un{1}},\beta\alpha t_{\un{1}},\alpha\beta t_{\un{1}}, t_{\un{1}}\}$ or $T^{(j')}\subseteq\{t_{w_0(\eta)}, t_{w_0(\eta)}\alpha ,  t_{w_0(\eta)}\beta, t_{w_0(\eta)}w_0\}$ 
\end{enumerate}
and with the following more precise condition on $N$:
\begin{equation}
\label{eq:improved_bound_1}
N-4>\max_{j\in\cJ}\Big\{\big(\#\{\tau'\neq \tau\mid\ell\big(\tld{w}^*(\rhobar,\tau)_j^{-1} \tld{w}(\rhobar,\tau')_j\big)\leq 1\}\big)+2\big(\#\{\tau'\neq \tau\mid\ell\big(\tld{w}^*(\rhobar,\tau)_j^{-1} \tld{w}(\rhobar,\tau')_j\big)> 1\}\big)\Big\}.
\end{equation}
This is because in the proof of Proposition \ref{prop:special_fiber} the inclusion \eqref{eq:inc:crucial} and the reasoning following it still holds true when replacing $6$ by the right hand side of \eqref{eq:improved_bound_1}.

Moreover, under the stronger assumption
\begin{equation}
\label{eq:improved_bound_2}
N-4>\Big(\#\big\{\tau'\neq \tau\mid \max_j\ell\big(\tld{w}^*(\rhobar,\tau)_j^{-1} \tld{w}(\rhobar,\tau')_j\big)\leq 1\big\}\Big)+2\Big(\#\big\{\tau'\neq \tau\mid \max_j\ell\big(\tld{w}^*(\rhobar,\tau)_j^{-1} \tld{w}(\rhobar,\tau')_j\big)>1\big\}\Big).
\end{equation}
 the surjection \eqref{eq:iso:SPF} is actually an isomorphism.
Indeed, again Proposition  \ref{prop:ideals_coprimes} gives the inclusion $p^{(\text{RHS of \eqref
{eq:improved_bound_2}})}\bigg(\bigcap_{\tau\in T}\tld{I}_{\tau,\nabla_{\infty}}\bigg)\subseteq \prod_{\tau\in T}\tld{I}_{\tau,\nabla_{\infty}}$ and the argument of Proposition \ref{prop:special_fiber} can be now performed by reversing the roles of $\tld{I}_{\tau,\textnormal{alg}}$ and $g_{\tau,i,\textnormal{alg}}^{(j)}$ with $\tld{I}_{\tau,\infty}$ and $g_{\tau,i,\infty}^{(j)}$, and again replacing $6$ by the RHS of \eqref{eq:improved_bound_2}.

In particular whenever $N-4$ satisfies condition \eqref{eq:improved_bound_1} we have a commutative diagram
\begin{equation}
\label{eq:T=1}
\xymatrix{
{S}/ {I}_{T,\nabla_{\textnormal{alg}}}\ar@{->>}[r]\ar@{->>}[d]& {S}/ {I}_{T,\nabla_{\infty}}\ar@{->>}[d]\\
{S}/ {I}_{\tau,\nabla_{\textnormal{alg}}}\ar^{\sim}[r]& 
{S}/{I}_{\tau,\nabla_\infty}
}
\end{equation}
where horizontal arrow is an isomorphism (by applying the previous paragraph to the case $\#T=1$) and the vertical maps are the canonical surjections.
\end{rmk}

\subsection{Ideal relations in multi-type deformation rings}
\label{subsec:idealrelations}

Let $\rhobar:G_{K}\rightarrow \GL_3(\F)$ be a continuous semisimple Galois representation together with a lowest alcove presentation $(s,\mu)$ where $\mu$ is $N>10$-deep.
In this section we record facts about $R^{\eta,T}_{\rhobar}$ that we will need later.

Fix $j\in\cJ$. Recall the running assumption that $T$ satisfies Hypothesis \ref{hypothesis:T}. We now assume additionally that
\begin{enumerate}[label=(\Roman*)]\setcounter{enumi}{3}
\item\label{it:prop:T:4} 
{$\alpha\beta\alpha\gamma t_{\un{1}}\notin T^{(j)}$;}
\item\label{it:prop:T:3} 
For $j'\neq j$, $\#T^{(j')}=1$ and $\tld{w}^*(\rhobar,\tau)_{j'}\in\{\alpha\beta\alpha t_{\un{1}},\,\beta\gamma\beta t_{\un{1}},\,\gamma\alpha\gamma t_{\un{1}},\,\}$.
\end{enumerate}
{By \ref{it:prop:T:4} we can thus replace the ring $\tld{S}^{(j)}$ appearing in Tables \ref{Table_Ideals}, \ref{Table_Ideals_1} by $\tld{S}^{(j)}/(e_{11})$, and omit the variable $e_{11}$ in the computations of these sections.}
Throughout this section let $(a,b,c)\in\Fp^3$ be the mod $p$ reduction of $-(s_{j}^{-1}(\mu_j+\eta))$.

\subsubsection{Analysis of ${S}^{(j)}/{I}^{(j)}_{T,\nabla_{\textnormal{alg}}}$.}\label{subsec:specialfiber}

In what follows, we assume that $T^{(j)}$ is as in item \ref{it:prop:T:1}.
Given $b_j\in\{B,F_s,E_s,F_o,E_o\}$ we define $I_j^{b_j}\subset {S}^{(j)}$ to be the intersection ${\fP}^{(j)}_{(0,0)}\cap {\fP}^{(j)}_{(\omega,a)}$, where $(\omega,a)=(0,1)$, $(\eps_1,0)$, $(\eps_2,0)$, $ (\eps_2-\eps_1,1)$, $ (\eps_1-\eps_2,1)$ respectively if $b_j=B$, $F_s$, $E_s$, $F_o$, $E_o$.
\begin{lemma}
\label{lem:intsc}
Let $b_j\in\{B,E_o,F_o,E_s,F_s\}$.
Then the ideal $I_j^{b_j}$ is given by:
\begin{enumerate}[noitemsep]
\item
\label{eq:lem:intsc:1}
$(c_{33},c_{32},c_{31},c_{23},c_{22},c_{21},c_{13}d_{32}-c_{12}d_{33}^*,c_{13}d_{31}-c_{11}d_{33}^*,c_{12}d_{31}-c_{11}d_{32},(b-c)c_{12}d_{21}-(a-c)c_{11}d_{22}^*))$ 
\newline 
\indent if $b_j=B$;
\item
\label{eq:lem:intsc:2}
$c_{33},c_{32},c_{31},c_{21},c_{11},c_{23}d_{31},c_{22}d_{31},c_{13}d_{31},c_{12}d_{31},c_{13}c_{22}-c_{12}c_{23},c_{13}d_{21}-c_{23}d_{11}^*,c_{12}d_{21}-c_{22}d_{11}^*,(a-c-1)c_{23}d_{32}-(a-b-1)c_{22}d_{33}^*,(a-c-1)c_{13}d_{32}-(a-b-1)c_{12}d_{33}^*$
\newline
\indent if $b_j=F_s$;
\item
\label{eq:lem:intsc:3}
$(c_{32},c_{31},c_{22},c_{21},c_{12},c_{11},c_{23}d_{32}-c_{33}d_{33}^*,c_{23}d_{31}-d_{21}c_{33},(a-b)c_{13}d_{31}+(b-c-1)c_{33}d_{11}^*,(a-b)c_{13}d_{21}+(b-c-1)c_{23}d_{11}^*,c_{13}d_{21}d_{32}-c_{13}d_{31}d_{22}^*)$
\newline
\indent
if $b_j=E_s$;
\item
\label{eq:lem:intsc:4}
$(c_{33},c_{32},c_{31},c_{22},c_{13},c_{12},c_{11},c_{23}d_{32},c_{21}d_{32},(b-c-1)c_{23}d_{31}+(a-b+1)c_{21}d_{33}^*)$
\newline
\indent
if $b_j=F_o$;
\item
\label{eq:lem:intsc:5}
$(c_{33},c_{31},c_{23},c_{22},c_{21},c_{13},c_{11},d_{21}c_{32},c_{12}d_{21},(a-c)c_{12}d_{31}+(-b+c-1)c_{32}d_{11}^*)$
\newline
\indent
if $b_j=E_o$.
\end{enumerate}
\end{lemma}
\begin{proof}
See \S \ref{appendix:multi:sp:fiber}.
\end{proof}

For $b_j\in\{B,F_s,E_s,E_o,F_o\}$ let $M^{b_j}$ be the ${S}^{(j)}$-module ${S}^{(j)}/I^{b_j}_j$.
Let $M^{\emptyset}$ be the ${S}^{(j)}$-module ${S}^{(j)}/{\fP}^{(j)}_{(0,0)}$.
For $a_j=\{B,F_s,E_s,E_o,F_o\}$ (resp.~$a_j=\{B,F_s,E_s\}$) let $M^{a_j}$ be the kernel of the natural surjective map
\begin{align*}
\oplus_{b_j\in a_j}M^{b_j}\ra (M^{\emptyset})^{\oplus\, \# a_j}/\Delta(M^{\emptyset})\ra 0
\end{align*}
where $\Delta(M^{\emptyset})$ denotes the diagonally embedded copy of $M^{\emptyset}$ in $(M^{\emptyset})^{\oplus\, \# a_j}$.
\begin{prop}\label{prop:lowermult}
For either $a_j=\{B,F_s,E_s,E_o,F_o\}$ or $a_j=\{B,F_s,E_s\}$ we have
\begin{equation*}
\label{eq:dim:lowerpatched}
\dim_{\F}\big(M^{a_j}\otimes_{{S}^{(j)}}{S}^{(j)}/\fm_{{S}^{(j)}}\big)=3.
\end{equation*}
\end{prop}
\begin{proof}
By Lemma \ref{lemma:glue:lower} it is enough to check that the cokernel of the natural map
\begin{equation}
\label{eq:dim:lowerpatched}
\oplus_{b_j\in a_j}\Tor_1^{{S}^{(j)}}(\F,M^{b_j})\ra \Tor_1^{{S}^{(j)}}(\F,(M^\emptyset)^{\oplus \# a_j}/\Delta(M^\emptyset))
\end{equation}
is two dimensional.
Note that given an ideal $I\subseteq {S}^{(j)}$ we have $\Tor_1^{{S}^{(j)}}(\F,{S}^{(j)}/I)\cong I/(\fm_{S^{(j)}}\cdot I)$, which allows us to write elements of $\Tor_1^{{S}^{(j)}}(\F,{S}^{(j)}/I)$ in terms of generators of $I$.
In particular we see that $\Tor_1^{{S}^{(j)}}(\F,M^\emptyset)$ has a basis consisting of the image of the elements $c_{ik}$ for $1\leq i,k\leq 3$.
An immediate check on the generators of $I_j^{b_j}$ (described in Lemma \ref{lem:intsc}) shows that the image of the natural map $\Tor_1^{{S}^{(j)}}(\F,M^{b_j})\ra \Tor_1^{{S}^{(j)}}(\F,M^\emptyset)$ has the following description according to $b_j$:
\begin{enumerate}
\item 
\label{eq:ind:S:1}
if $b_j\in\{B,F_s,E_s\}$ it is the subspace of $\Tor_1^{{S}^{(j)}}(\F,M^\emptyset)$ generated by $c_{ik}$, $1\leq i,k\leq 3$, $(ik)\neq (13)$;
\item 
\label{eq:ind:S:2}
if $b_j=F_o$ it is the subspace of $\Tor_1^{{S}^{(j)}}(\F,M^\emptyset)$ generated by $c_{ik}$ for $1\leq i,k\leq 3$, $(ik)\neq (23)$; and
\item 
\label{eq:ind:S:3}
if $b_j=E_o$ it is the subspace of $\Tor_1^{{S}^{(j)}}(\F,M^\emptyset)$ generated by $c_{ik}$ for $1\leq i,k\leq 3$, $(ik)\neq (12)$.
\end{enumerate}
We deduce that the cokernel of the natural map
\begin{equation}
\label{eq:dim:lowerpatched:1}
\oplus_{b_j\in a_j}\Tor_1^{{S}^{(j)}}(\F, M^{b_j})\ra \Tor_1^{{S}^{(j)}}(\F,(M^\emptyset)^{\# a_j})
\end{equation}
has dimension $\# a_j$, with basis given by (the image of) the elements $\{c_{13}^{b_j}\}_{b_j\in\{B,F_s,E_s,\}}$, and $c_{23}^{F_o}$, $c_{12}^{E_o}$ if $\{E_o,F_o\}\subset a_j$ (the superscripts on the $c_{ik}$ denote which copy of $\Tor_1^{{S}^{(j)}}(\F,M^\emptyset)\subseteq  \Tor_1^{{S}^{(j)}}(\F,(M^\emptyset)^{\oplus\# a_j})$ the element $c_{ik}$ lives in).
On the other hand the image of $\Delta(\Tor_1^{{S}^{(j)}}(\F,M^\emptyset))$ in the cokernel of \eqref{eq:dim:lowerpatched:1} is generated by $c_{13}^{B}+c_{13}^{E_s}+c_{13}^{F_s}$, $c_{23}^{F_o}$, $c_{12}^{E_o}$ if $\# a_j=5$ and by $c_{13}^{B}+c_{13}^{E_s}+c_{13}^{F_s}$
 if $\# a_j=3$.
In both cases we conclude that the cokernel of \eqref{eq:dim:lowerpatched} has dimension 2.
\end{proof}

\subsubsection{Surgery for $T^{(j)}=\{t_{\un{1}}\}$.}
\label{subsub:surgeries}

Suppose that $T^{(j)}=\{t_{\un{1}}\}$ for all $j\in \cJ$. 
We write $\tau$ for the unique element in $T$. 
We now fix $j\in \cJ$. 
We abbreviate $\tld{R} \defeq \tld{S}^{(j)}/\tld{I}^{(j)}_{\tau,\nabla_{\textnormal{alg}}}$, and $R\defeq \tld{R}/(p)$ so that a presentation for $\tld{R}$ is given in row $t_{\un{1}}$ of Table \ref{Table_Ideals_1}.
The ring $\tld{R}$ is normal and Cohen--Macaulay by \cite[Corollary 8.9]{LLLM}.
Let $\mathfrak{j}: U \into \Spec \tld{R}$ be the complement of the vanishing locus of
\[
\prod_{(\omega,a) \neq (\nu,b) \in \Sigma_0} {\fP}^{(j)}_{(\omega,a)}+{\fP}^{(j)}_{(\nu,b)}. 
\]
Then $U$ is a regular scheme by the proof of \cite[Lemma 5.2.1]{LLLM2}. 
For a coherent reflexive (i.e.~coherent, $S_2$, and torsion-free) $\tld{R}$-module $\tld{M}$ and an effective divisor $D = \sum_{(\omega,a) \in \Sigma_0} n_{(\omega,a)} [{\fP}^{(j)}_{(\omega,a)}]$ of $U$ supported in the special fiber (if ${\fP}^{(j)}_{(\omega,a)} = R$, then take $[{\fP}^{(j)}_{(\omega,a)}] = 0$), define 
\begin{equation}
\label{eqn:divisor:def}
\tld{M}(-D) \defeq \mathfrak{j}_* \mathfrak{j}^* \prod_{(\omega,a) \in \Sigma_0} ({\fP}^{(j)}_{(\omega,a)})^{n_{(\omega,a)}} \tld{M}. 
\end{equation}
Since $\mathfrak{j}^* \prod_{(\omega,a) \in \Sigma_0} ({\fP}^{(j)}_{(\omega,a)})^{n_{(\omega,a)}}$ is locally free on $U$, $\mathfrak{j}^* \prod_{(\omega,a) \in \Sigma_0} ({\fP}^{(j)}_{(\omega,a)})^{n_{(\omega,a)}} \tld{M}$ is coherent and reflexive so that $\tld{M}(-D) = \mathfrak{j}_*\mathfrak{j}^* \tld{M}(-D)$ is its unique (up to isomorphism) coherent and reflexive extension by \cite[\href{https://stacks.math.columbia.edu/tag/0EBJ}{Tag 0EBJ}]{stacks-project}. 

\begin{lemma}\label{lemma:reducedcycle}
If $D = \sum_{(\omega,a) \in \Sigma_0} n_{(\omega,a)} [{\fP}^{(j)}_{(\omega,a)}]$ with all $0\leq n_{(\omega,a)}\leq 1$, then 
\[
\tld{R}(-D) = \cap_{\substack{(\omega,a)\in \Sigma_0 \\ n_{(\omega,a)}= 1}} {\fP}^{(j)}_{(\omega,a)}.
\]
\end{lemma}
\begin{proof}
Since $\tld{R}/\tld{R}(-D)$ is $S_1$ (as $\tld{R}$ and $\tld{R}(-D)$ are both $S_2$) and $R_0$, it is reduced. 
Thus $\tld{R}(-D)$ is the intersection of prime ideals. 
Since $\mathfrak{j}^* \tld{R}(-D) = \mathfrak{j}^* \cap_{\substack{(\omega,a)\in \Sigma_0 \\ n_{(\omega,a)}= 1}} {\fP}^{(j)}_{(\omega,a)}$, the result follows. 
\end{proof}

Let $\tld{M}^{(j)}_{(0,1)}$ be a free $\tld{R}$-module of rank one. 
Recall from \S \ref{subsubsec:combinatorics_TandW} that $\Sigma_0$ is a connected graph endowed with a distance function which we denote $d$.
For $(\omega,a) \in \Sigma_0$, we define  
\begin{equation}\label{eqn:divisor}
\tld{M}^{(j)}_{(\omega,a)} = \tld{M}^{(j)}_{(0,1)}\Big(-\sum_{(\nu,b) \in \Sigma_0} \frac{1}{2}\Big(d((\omega,a),(0,1)) +d((\omega,a),(\nu,b))- d((0,1),(\nu,b))\Big)[{\fP}^{(j)}_{(\nu,b)}]\Big). 
\end{equation}

\begin{lemma}\label{lemma:cartier}
We have $\tld{M}^{(j)}_{(\eps_1,1)} = (-c_{23}d_{32}+c_{22}d_{33}^*+c_{33}d_{22}^*+pd_{22}^*d_{33}^*)\tld{M}^{(j)}_{(0,1)}$ and $\tld{M}^{(j)}_{(\eps_2,1)} = (c_{33}+pd_{33}^*)\tld{M}^{(j)}_{(0,1)}$, where we have omitted $(j)$-superscripts in the variables. 
\end{lemma}
\begin{proof}
We will show that $M_1 \defeq \tld{M}^{(j)}_{(\eps_2,1)}$ and $M_2 \defeq (c_{33}+pd_{33}^*)\tld{M}^{(j)}_{(0,1)}$ are equal as the proof of the other claim is similar. 
Let $M$ be $\tld{M}^{(j)}_{(0,1)}$. 
As $M_2$ is free of rank one, it is coherent and reflexive.
By \cite[\href{https://stacks.math.columbia.edu/tag/0EBJ}{Tag 0EBJ}]{stacks-project}, it suffices to show that the locally free sheaves $\mathfrak{j}^* M_1$ and $\mathfrak{j}^*M_2$ of rank $1$ on $U$ are equal, or that $M/M_1$ and $M/M_2$ have the same length after localization at the minimal primes of $R$. 
First, for a minimal prime $\fP$ of $\Spec R$, $M_{\fP}/(c_{33},p)M_{\fP}$ has length $0$ or $1$ with length equal to $1$ if and only if $c_{33} \in \fP$ since the special fiber is reduced. 
We see that the length is $0$ if and only $\fP = \fP^{(j)}_{(\eps_2,1)}$. 

Second, we claim that
\begin{equation}\label{eqn:U2}
(c_{33}+pd_{33}^*)(-c_{12}d_{21}+c_{11}d_{22}^*+c_{22}d_{11}^*+pd_{11}^*d_{22}^*) = p^2d_{11}^*d_{22}^*d_{33}^*
\end{equation}
so that $p$ annihilates $(c_{33},p)M_{\fP}/(c_{33}+pd_{33}^*)M_{\fP}$. 
Indeed, since the determinant of $M^{T,(j)}$ 
equals $(v+p)^3d_{11}^*d_{22}^*d_{33}^*$, specializing at $v=0$ gives 
\[
\prod_{i=1}^3 (c_{ii}+pd_{ii}^*) = p^3d_{11}^*d_{22}^*d_{33}^*. 
\]
Moreover, the top-left $2 \times 2$ minor is divisible by $v+p$ and the $(2,1)$-entry is divisible by $v$ so that specializing at $v = -p$ gives $c_{11}c_{22} = -pc_{12}d_{21}$. 
Combining these and dividing both sides by $p$ gives the claim. 

By the claim, $(c_{33},p)M_{\fP}/(c_{33}+pd_{33}^*)M_{\fP}$ has length $0$ or $1$ with length equal to $0$ if and only if $\frac{p}{c_{33}+pd_{33}^*} \in \tld{R}_{\fP}$ or equivalently, by \eqref{eqn:U2}, 
\[
-c_{12}d_{21}+c_{11}d_{22}^*+c_{22}d_{11}^* \in \fP. 
\]
We see from  Table \ref{Table:components} that the length is $1$ if and only if $\fP = \fP^{(j)}_{(0,1)}$. 

It is now easy to check the desired length statement from \eqref{eqn:divisor}. 
\end{proof}

We define a \emph{path} to be a sequence of elements $\gamma = (\gamma_k)_{k\geq 1}^{\ell(\gamma)}$ in $\{(0,0), (\eps_1,0), (\eps_2,0),(0,1), (\eps_1,1), (\eps_2,1)\}$ of length $\ell(\gamma)=2$ or $3$ such that 
\begin{itemize}
\item the $\gamma_k$ are distinct; 
\item $\gamma_k$ and $\gamma_{k+1}$ are adjacent for $1\leq k \leq \ell(\gamma)-1$; and
\item $\gamma_1 = (\eps_1,1)$ or $(\eps_2,1)$. 
\end{itemize}
For paths $\beta$ and $\gamma$, we write $\beta \leq \gamma$ if $\ell(\beta) \leq \ell(\gamma)$ and $\beta_k = \gamma_k$ for $1\leq k \leq \ell(\beta)$.   
For a path $\gamma$, we define subsets $\Sigma_\gamma \subset \Sigma_0$ as follows: 
\begin{itemize}
\item If $\ell(\gamma)=2$, then $\Sigma_\gamma = \{\gamma_2,(0,1),(\eps_1,1),(\eps_2,1)\} \setminus \{\gamma_1\}$. 
\item If $\ell(\gamma)=3$, then $\Sigma_\gamma = \{\gamma_3\}$. 
\end{itemize}
Given a path $\gamma$, we define $M_\gamma^{(j)}$ to be $\tld{M}_{\gamma_1}^{(j)}(-D_\gamma)/p\tld{M}_{\gamma_1}^{(j)}$ where 
\[
D_\gamma \defeq \sum_{(\eps,a)\notin \Sigma_\gamma}[{\fP}^{(j)}_{(\eps,a)}]. 
\]
Then by Lemmas \ref{lemma:reducedcycle} and \ref{lemma:cartier}, $M_\gamma^{(j)}$ is naturally identified with $I_\gamma (\tld{M}_{\gamma_1}^{(j)}\otimes_{\cO} \F)$ where we define $I_\gamma$ to be the image of $\bigcap_{(\eps,a)\notin \Sigma_\gamma}{\fP}^{(j)}_{(\eps,a)}$ in $R$.
Note that $I_\gamma\supseteq I_\beta$ if $\beta\geq \gamma$.
Moreover, note that different paths of length three can give rise to the same ideal $I_\gamma$ (e.g.~$I_\gamma$  depends only on $\gamma_3$ and not on $\gamma_2$). 
In fact, if $\ell(\beta) = 3 = \ell(\gamma)$ and $\beta_3 = \gamma_3$, then $M_\beta^{(j)}$ and $M_\gamma^{(j)}$ are naturally identified. 
This follows easily from the definition \eqref{eqn:divisor:def} if $\beta_3 = \gamma_3 \neq (0,1)$. 
If $\beta_3 = \gamma_3 = (0,1)$, then $M_\beta^{(j)}$ and $M_\gamma^{(j)}$ are both identified with $p^2 \tld{M}^{(j)}_{(0,1)}/p^2 \tld{M}^{(j)}_{(0,1)}(-[{\fP}^{(j)}_{(0,1)}])$. 
Indeed, there are inclusions $p^2 \tld{M}^{(j)}_{(0,1)} \subset \tld{M}^{(j)}_{(\eps_i,1)}$ but $p^2 \tld{M}^{(j)}_{(0,1)} \not\subset p\tld{M}^{(j)}_{(\eps_i,1)}$ for $i = 1$ and $2$ by \eqref{eqn:divisor}. 
Then for $\alpha = \beta$ and $\gamma$, we claim that the natural map $p^2 \tld{M}^{(j)}_{(0,1)} \ra M_\alpha^{(j)}$ is surjective and induces an isomorphism $p^2 \tld{M}^{(j)}_{(0,1)}/p^2 \tld{M}^{(j)}_{(0,1)}(-[{\fP}^{(j)}_{(0,1)}]) \cong M_\alpha^{(j)}$. 
We explain for $\alpha_1 = (\eps_1,1)$, the other case being similar. 
The surjectivity follows from the fact that $p^2 \tld{M}^{(j)}_{(0,1)} = (c_{11} + pd_{11}^*)\tld{M}^{(j)}_{(\eps_1,1)}$ and Lemma \ref{lemma:id:intersect} below. 
The kernel of this surjection is determined by the fact that the image of the map is isomorphic to $R/{\fP}^{(j)}_{(0,1)}$. 

The computations for Lemma \ref{lemma:id:intersect}, similar to those of \cite[\S 3.6.3]{LLLM2}, are recorded in \S  \ref{subsub:surgery}.

\begin{lemma}
\label{lemma:id:intersect}
Let $\gamma$ be a path.
Then $I_\gamma$ is minimally generated by $4-\ell(\gamma)$ elements, and a minimal set of generators for $I_\gamma$ is given by:
\begin{align*}
&({c}_{22}\frac{{d}_{11}^*}{{d}_{22}^*},\,{d_{31}d_{22}^*-d_{21}d_{32}})&&\text{if $\gamma=\big((\eps_2,1),(0,0)\big)$;}\\
&({c}_{22}\frac{{d}_{11}^*}{{d}_{22}^*},{c}_{12})&&\text{if $\gamma=\big((\eps_2,1),(\eps_2,0)\big)$;}\\
&({c}_{22}\frac{{d}_{11}^*}{{d}_{22}^*}, {(a-b)c_{13}d_{21}+(b-c-1)c_{23}d_{11}^*})&&\text{if $\gamma=\big((\eps_2,1),(\eps_1,0)\big)$;}\\
&({c}_{33}\frac{{d}_{11}^*}{{d}_{33}^*},\,{d}_{31})&&\text{if $\gamma=\big((\eps_1,1),(0,0)\big)$;}\\
&({c}_{33}\frac{{d}_{11}^*}{{d}_{33}^*},{(c+1-a)c_{13}d_{32}+(a-b-1)c_{12}d_{33}^*})&&\text{if $\gamma=\big((\eps_1,1), (\eps_2,0)\big)$;}\\
&({c}_{33}\frac{{d}_{11}^*}{{d}_{33}^*}, {c_{13}d_{21} - c_{23}d_{11}^*})&&\text{if $\gamma=\big((\eps_1,1),(\eps_1,0)\big)$;}\\
&({c}_{11})&&\text{if $\ell(\gamma)=3, \gamma_3=(0,1)$;}
\\
&({c}_{22})&&\text{if $\ell(\gamma)=3, \gamma_3=(\eps_1,1)$;}
\\
&({c}_{33})&&\text{if $\ell(\gamma)=3, \gamma_3=(\eps_2,1)$.}
\end{align*}
\end{lemma}
If $\gamma$ is a path of length $3$, $\gamma >\beta$, and $\kappa_\gamma\in\F^\times$, we define the map 
\begin{align}
\label{map:surgeries}
{\iota}_{\kappa_\gamma}^{\gamma}:M^{(j)}_\gamma &\longrightarrow M^{(j)}_{\beta}
\end{align}
to be $\kappa_\gamma$ times the natural inclusion. 
If one chooses a generator of $\tld{M}^{(j)}_{\gamma_1} \otimes_{\cO}\F$, then \eqref{map:surgeries} is identified with the map 
\begin{align}
\label{map:idealsurgeries}
{\iota}_{\kappa_\gamma}^{\gamma}: {I}_\gamma&\longrightarrow {I}_{\beta}
\end{align}
defined in an analogous way. 
For $\kappa=(\kappa_\gamma)_{\gamma, \ell(\gamma)=3}\in (\F^\times)^{12}$, we get a collection of maps ${\iota}_{\kappa_\gamma}^{\gamma}$ which induce a map 
\[
\iota_\kappa: \bigoplus_{\gamma,\ell(\gamma)=3}M^{(j)}_\gamma \ra \bigoplus_{\beta,\ell(\beta)=2} M^{(j)}_\beta. 
\]
We let $(\bigoplus_{\gamma,\ell(\gamma)=3}M^{(j)}_\gamma)^0$ be the subspace cut out by the conditions
\[
\sum_{\substack{\gamma,\ell(\gamma)=3,\\\gamma_3=(\omega,1)}}a_{\gamma}=0
\]
for $\omega\in\{0,\eps_1,\eps_2\}$ (recall that the $M_\gamma$ with $\ell(\gamma)=3$ and $\gamma_3=(\omega,1)$ are identified). 
We define $\iota_\kappa^0$ to be the restriction of $\iota_\kappa$ to $(\bigoplus_{\gamma,\ell(\gamma)=3}M^{(j)}_\gamma)^0$, and let $M_\kappa$ be the cokernel of $\iota_\kappa^0$. 

In order to to compute $M_\kappa$ in certain instances, we choose generators for $\tld{M}^{(j)}_{(\eps_i,1)}$. 
Fix a generator $m\in \tld{M}^{(j)}_{(0,1)}$ and generators $(-c_{23}d_{32}+c_{22}d_{33}^*+c_{33}d_{22}^*+pd_{22}^*d_{33}^*)m$ and $\frac{b-a}{b-c}d_{22}^*(c_{33}+pd_{33}^*)m$ of $\tld{M}^{(j)}_{(\eps_1,1)}$ and $\tld{M}^{(j)}_{(\eps_2,1)}$, respectively. 
Using these generators, we get identifications of $M^{(j)}_\gamma = I_\gamma$ such that if $\ell(\gamma') = 3 = \ell(\gamma)$ and $\gamma'_3 = \gamma_3$, then the identification of $M^{(j)}_{\gamma'}$ and $M^{(j)}_\gamma$ is compatible with the equality $I_{\gamma'} = I_\gamma$. 
(When $\gamma_3 = (0,1)$, then $I_\gamma = (c_{11},p)\tld{R}/p\tld{R}$. One checks from Table \ref{Table_Ideals_1} that 
\begin{equation*}
\tiny{\frac{b-a}{b-c}c_{11}d_{22}^* = -c_{12}d_{21}+c_{11}d_{22}^*+c_{22}d_{11}^*}
\end{equation*}
in $R$ so that modulo $p(c_{33}+pd_{33}^*)\tld{M}_{(0,1)}^{(j)}$, we have
\begin{align*}
(c_{11}+pd_{11}^*)\frac{b-a}{b-c}d_{22}^*(c_{33}+pd_{33}^*)m &\equiv p^2 d_{11}^*d_{22}^*d_{33}^* m\\
&= (c_{11}+pd_{11}^*)(-c_{23}d_{32}+c_{22}d_{33}^*+c_{33}d_{22}^*+pd_{22}^*d_{33}^*)m
\end{align*} 
and the equality is an analogue of \eqref{eqn:U2}.)
Then $M_\kappa$ is isomorphic to the cokernel of the map 
\begin{align}
\label{eq:map:surgery}
\iota^0_\kappa:\Bigg(\bigoplus_{\gamma,\ell(\gamma)=3}{I}_\gamma\Bigg)^0\ra \bigoplus_{\beta, \ell(\beta)=2}{I}_\beta
\end{align}
where the notation $(-)^0$ denotes the subspace cut out by the conditions
\[
\sum_{\substack{\gamma,\ell(\gamma)=3,\\\gamma_3=(\omega,1)}}a_{\gamma}=0
\]
for $\omega\in\{0,\eps_1,\eps_2\}$.

In what follows $\kappa=(\kappa_\gamma)_{\gamma, \ell(\gamma)=3}\in (\F^\times)^{12}$ is a tuple such that $\kappa_\gamma=1$ if $\gamma_2 = (0,0)$ or $\gamma_3 = (0,1)$.
We abbreviate $\lambda_1\defeq -\kappa_{((\eps_2,1),(\eps_2,0),(\eps_1,1))}$, $\lambda_2\defeq-\kappa_{((\eps_2,1),(\eps_1,0),(\eps_1,1))}$, $\lambda_3\defeq -\kappa_{((\eps_1,1),(\eps_2,0),(\eps_2,1))}$, $\lambda_4\defeq -\kappa_{((\eps_1,1),(\eps_1,0),(\eps_2,1))}$.
(There are 12 paths of length $3$ and 8 of these paths have the property that $\gamma_2 = (0,0)$ or $\gamma_3 = (0,1)$.) 
We now analyze the module $M_\kappa$ and certain maps  $I_\gamma \ra M_\kappa$ (Proposition \ref{prop:surgery} and Corollary \ref{cor:surgery}). 
If $M$ is a finitely generated $S^{(j)}$-module and $m\in M$ we use overlined notations $\ovl{M}$ and $\ovl{m}$ for $M\otimes_{S^{(j)}}S^{(j)}/\fm_{{S}^{(j)}}$ and the image of $m$ in $\ovl{M}$, respectively.
Similar notation apply for $S^{(j)}$-linear maps between finitely generated $S^{(j)}$-modules.

We define the following $\F$-basis on $\bigoplus_{\gamma,\ell(\gamma)=3}\ovl{I}_\gamma$:
\begin{equation*}
\un{e}\defeq\Bigg(
\underbrace{\ovl{c}_{11},\quad\ovl{c}_{22}}_{\hspace{-.3cm}\underset{\substack{\gamma>\beta,\\\beta=((\eps_2,1),(0,0)),}}{\bigoplus}\hspace{-.7cm}\ovl{I}_\gamma},\quad\ 
\underbrace{\ovl{c}_{11},\quad\ovl{c}_{22}}_{\hspace{-.3cm}\underset{\substack{\gamma>\beta,\\\beta=((\eps_2,1),(\eps_2,0)),}}{\bigoplus}\hspace{-.7cm}\ovl{I}_\gamma},\quad\ 
\underbrace{\ovl{c}_{11},\quad\ovl{c}_{22}}_{\hspace{-.3cm}\underset{\substack{\gamma>\beta,\\\beta=((\eps_2,1),(\eps_1,0)),}}{\bigoplus}\hspace{-.7cm}\ovl{I}_\gamma},\quad\ 
\underbrace{\ovl{c}_{11},\quad\ovl{c}_{33}}_{\hspace{-.3cm}\underset{\substack{\gamma>\beta,\\\beta=((\eps_1,1),(0,0)),}}{\bigoplus}\hspace{-.7cm}\ovl{I}_\gamma},\quad\ 
\underbrace{\ovl{c}_{11},\quad\ovl{c}_{33}}_{\hspace{-.3cm}\underset{\substack{\gamma>\beta,\\\beta=((\eps_1,1),(\eps_2,0)),}}{\bigoplus}\hspace{-.7cm}\ovl{I}_\gamma},\quad\ 
\underbrace{\ovl{c}_{11},\quad\ovl{c}_{33}}_{\hspace{-.3cm}\underset{\substack{\gamma>\beta,\\\beta=((\eps_1,1),(\eps_1,0))}}{\bigoplus}\hspace{-.7cm}\ovl{I}_\gamma}
\Bigg)
\end{equation*}
and, on  $\Bigg(\bigoplus_{\gamma,\ell(\gamma)=3}{I}_\gamma\Bigg)^0$, the basis
\[
\un{e}^0\defeq\un{e}\cdot\tiny{ \begin{pmatrix}
0&0&1&1&1&1&1&0&0\\
1&1&0&0&0&0&0&0&0\\
0&0&-1&0&0&0&0&0&0\\
-1&0&0&0&0&0&0&0&0\\
0&0&0&-1&0&0&0&0&0\\
0&-1&0&0&0&0&0&0&0\\
0&0&0&0&-1&0&0&0&0\\
0&0&0&0&0&0&0&1&1\\
0&0&0&0&0&-1&0&0&0\\
0&0&0&0&0&0&0&-1&0\\
0&0&0&0&0&0&-1&0&0\\
0&0&0&0&0&0&0&0&-1\\
\end{pmatrix}}
\]
Finally, on $\bigoplus_{\beta, \ell(\beta)=2}\ovl{I}_\beta$ we consider the basis $\un{f}$ deduced from Lemma \ref{lemma:id:intersect}.

\begin{prop}
\label{prop:surgery}
Using the bases $\un{e}^0$ and $\un{f}$ described above, the matrix $A_\kappa$ associated to $\ovl{\iota}^0_\kappa$ is given by:
\begin{align*}
\hspace{-.7cm}\tiny{\begin{pmatrix}
1&1&\frac{(b-c)(a-b-1)}{(a-b)(c+1-a)}&\frac{(b-c)(a-b-1)}{(a-b)(c+1-a)}&\frac{(b-c)(a-b-1)}{(a-b)(c+1-a)}&\frac{(b-c)(a-b-1)}{(a-b)(c+1-a)}&\frac{(b-c)(a-b-1)}{(a-b)(c+1-a)}&0&0\\
0&0&0&0&0&0&0&0&0\\
\lambda_1&0&\frac{(b-c)}{(a-c)}&0&0&0&0&0&0\\
0&0&0&0&0&0&0&0&0\\
0&\lambda_2&0&\frac{(a-b-1)(b-c)}{(b-a)(a-c)}&&0&0&0&0\\
0&0&0&0&0&0&0&0&0\\
0&0&0&0&\frac{(b-c-1)}{(b-a)}&0&0&1&1\\
0&0&0&0&0&0&0&0&0\\
0&0&0&0&0&\frac{(b-c-1)(c-a)}{(b-a)(a-c-1)}&0&\lambda_3&0\\
0&0&0&0&0&0&0&0&0\\
0&0&0&0&0&0&\frac{(a-c-1)(b-c)}{(b-a)(a-c)}&0&\lambda_4\\
0&0&0&0&0&0&0&0&0
\end{pmatrix}}
\end{align*}
\end{prop}
\begin{proof}
By Lemmas \ref{lemma:id:intersect} and \ref{lemma:id:relations}, the following congruences hold in  ${S}^{(j)}$:
\begin{align*}
&c_{11}\equiv -\frac{(b-c)(1-a+b)}{(a-b)(1-a+c)}\frac{d_{11}^*}{d_{22}^*}c_{22} &\textnormal{ modulo } \fm_{{S}^{(j)}}\cdot I_\beta&&&\textnormal{ if $\beta=(\eps_2,1),(0,0)$};
\\
&c_{11}\equiv -\frac{(b-c)}{(a-c)}\frac{d_{11}^*}{d_{22}^*}c_{22}& \textnormal{ modulo }\fm_{{S}^{(j)}}\cdot I_\beta&&&\textnormal{ if $\beta=(\eps_2,1),(\eps_2,0)$};
\\
&c_{11}\equiv -\frac{(b-c)(1-a+b)}{(a-c)(a-b)}\frac{d_{11}^*}{d_{22}^*}c_{22}&\textnormal{ modulo }\fm_{{S}^{(j)}}\cdot I_\beta&&&\textnormal{ if $\beta=(\eps_2,1),(\eps_1,0)$};\\
&c_{11}\equiv \frac{(b-c-1)}{(a-b)}\frac{d_{11}^*}{d_{33}^*}c_{33}& \textnormal{ modulo } \fm_{{S}^{(j)}}\cdot I_\beta &&&\textnormal{ if $\beta=(\eps_1,1),(0,0)$};
\\
&c_{11}\equiv \frac{(b-c-1)(c-a)}{(a-b)(a-c-1)}\frac{d_{11}^*}{d_{33}^*}c_{33}& \textnormal{ modulo } \fm_{{S}^{(j)}}\cdot I_\beta &&&\textnormal{ if $\beta=(\eps_1,1),(\eps_2,0)$};
\\
&c_{11}\equiv \frac{(b-c)(a-c-1)}{(a-c)(a-b)}\frac{d_{11}^*}{d_{33}^*}c_{33}& \textnormal{ modulo }\fm_{{S}^{(j)}}\cdot I_\beta &&&\textnormal{ if $\beta=(\eps_1,1),(\eps_1,0)$}.
\end{align*}
Thus, the matrix associated to $\ovl{\iota}_\kappa$, in the bases $\un{e}$, $\un{f}$ is 
\begin{equation*}
\hspace{-2cm}\tiny{\begin{pmatrix}
\frac{(b-c)(a-b-1)}{(a-b)(c+1-a)}&1&0&0&0&0&0&0&0&0&0&0\\
0&0&0&0&0&0&0&0&0&0&0&0\\
0&0&\frac{(c-b)}{(a-c)}&-\lambda_1&0&0&0&0&0&0&0&0\\
0&0&0&0&0&0&0&0&0&0&0&0\\
0&0&0&0&\frac{(a-b-1)(b-c)}{(a-b)(a-c)}&-\lambda_2&0&0&0&0&0&0\\
0&0&0&0&0&0&0&0&0&0&0&0\\
0&0&0&0&0&0&\frac{(b-c-1)}{(a-b)}&1&0&0&0&0\\
0&0&0&0&0&0&0&0&0&0&0&0\\
0&0&0&0&0&0&0&0&\frac{(b-c-1)(c-a)}{(a-b)(a-c-1)}&-\lambda_3&0&0\\
0&0&0&0&0&0&0&0&0&0&0&0\\
0&0&0&0&0&0&0&0&0&0&\frac{(a-c-1)(b-c)}{(a-b)(a-c)}&-\lambda_4\\
0&0&0&0&0&0&0&0&0&0&0&0
\end{pmatrix}}
\end{equation*}
and the conclusion follows by the definition of $\un{e}^0$.
\end{proof}

The following result will be used in \S \ref{sec:locality}:
\begin{cor}
\label{cor:surgery}
Assume that $\kappa=\kappa_{\textnormal{min}}(a,b,c)\in (\F^\times)^{12}$ is such that $\kappa_\gamma=1$ if $\gamma_2 = (0,0)$ or $\gamma_3 = (0,1)$, and 
\begin{align*}
&\left(
-\lambda_1,-\lambda_2,-\lambda_3,-\lambda_4
\right)=\\
&\qquad=\left(\frac{(a-b)(a-c-1)}{(a-b-1)(a-c)},\frac{a-c-1}{a-c},-\frac{(a-c)}{(a-c-1)},-\frac{(b-c)(a-c-1)}{(b-c-1)(a-c)}\right).
\end{align*}
Then for any two paths $\gamma>\beta$ the composite
\[
\ovl{I}_\gamma\ra \ovl{I}_{\beta}\ra \ovl{M}^{(j)}_{\kappa}
\]
is nonzero.
\end{cor}
\begin{proof}
By Proposition \ref{prop:surgery} and Lemma \ref{lemma:id:intersect} it is equivalent to show that for any $i\in\{1,3,5,7,9,11\}$ the linear system $A_\kappa\un{X}=\un{e}_i$ has no solutions, where $\un{e}_i=(\dots,0,1,0\dots)$ has a unique non-zero entry at position $i$.
Equivalently, we show that for any $i=1,\dots,6$ the linear system 
\begin{equation*}
\tiny{\begin{pmatrix}
1&1&\frac{(b-c)(a-b-1)}{(a-b)(c+1-a)}&\frac{(b-c)(a-b-1)}{(a-b)(c+1-a)}&\frac{(b-c)(a-b-1)}{(a-b)(c+1-a)}&\frac{(b-c)(a-b-1)}{(a-b)(c+1-a)}&\frac{(b-c)(a-b-1)}{(a-b)(c+1-a)}&0&0\\
\lambda_1&0&\frac{(b-c)}{(a-c)}&0&0&0&0&0&0\\
0&\lambda_2&0&\frac{(a-b-1)(b-c)}{(b-a)(a-c)}&&0&0&0&0\\
0&0&0&0&\frac{(b-c-1)}{(b-a)}&0&0&1&1\\
0&0&0&0&0&\frac{(b-c-1)(c-a)}{(b-a)(a-c-1)}&0&\lambda_3&0\\
0&0&0&0&0&0&\frac{(a-c-1)(b-c)}{(b-a)(a-c)}&0&\lambda_4
\end{pmatrix}}\un{X}=\un{e}_i
\end{equation*}
has no solutions, where again $\un{e}_i=(\dots,0,1,0\dots)$ has the unique nonzero entry at position $i$.
Writing $\un{e}_i= (e_{i,j})_{1\leq j\leq 6}$, and letting $a_{ij}$ denote the relevant entries of the previous matrix, the linear system above is equivalent after row reduction to 
\begin{align}
\label{eq:system:final}
&\tiny{\begin{pmatrix}
0&0&a_{13}-\lambda_1^{-1}a_{23}&a_{14}-\lambda_2^{-1}a_{34}&0&a_{16}+\frac{a_{15}a_{56}}    {\lambda_3a_{45}}&a_{17}+\frac{a_{15}a_{67}}{\lambda_4a_{45}}&0&0\\
\lambda_1&0&a_{23}&0&0&0&0&0&0\\
0&\lambda_2&0&a_{34}&&0&0&0&0\\
0&0&0&0&a_{45}&0&0&1&1&\\
0&0&0&0&0&a_{56}&0&\lambda_3&0\\
0&0&0&0&0&0&a_{57}&0&\lambda_4
\end{pmatrix}}\un{X}=\\
&\qquad=
\tiny{\begin{pmatrix}
e_{i,1}-\lambda_1^{-1}e_{i,2}-\lambda_2^{-1}e_{i,3}-\frac{a_{15}}{a_{45}}(e_{i,4}-\lambda_3^{-1}e_{i,5}-\lambda_4^{-1}e_{i,6})\\
e_{i,2}\\
e_{i,3}\\
e_{i,4}\\
e_{i,5}\\
e_{i,6}
\end{pmatrix}}
\nonumber
\end{align}
By definition of $\kappa_{\textnormal{min}}(a,b,c)$, the first row of the matrix in the RHS of \eqref{eq:system:final} is zero, which implies that the linear system \eqref{eq:system:final} has no solution since $e_{i,j}\neq 0$ exactly when $j=i$, and $\lambda_i\neq 0$ for $i=1,2,3,4$.
\end{proof}

We conclude with a result on the support cycle of $M_{\kappa}$.
We employ the terminology and notations of \cite[\S 2.2]{EG-perspective}, in particular Definition 2.2.5 from \emph{loc.~cit}.~where we take $\cX$ to be $\Spec (R)$ (so that $d=6$ in the notation of \emph{loc.~cit}.~).
\begin{cor}
\label{cor:cycle:Mkappa}
For any $\kappa=(\kappa_\gamma)_{\gamma, \ell(\gamma)=3}\in (\F^\times)^{12}$ such that $\kappa_\gamma=1$ if $\gamma_2 = (0,0)$ or $\gamma_3 = (0,1)$ we have 
\[
Z_d(M_{\kappa})=\sum_{\omega\in\{0,\eps_1,\eps_2\}}\fP_{(\omega,1)}+2\sum_{\omega\in\{0,\eps_1,\eps_2\}}\fP_{(\omega,0)}.
\]
\end{cor}
\begin{proof}
Let $\gamma$ be a path. 
By definition of $I_\gamma$, the minimal prime ideals of $R$ in the support of $I_\gamma$ are exactly $\fP_{(\omega,a)}$ for $(\omega,a)\in \Sigma_\gamma$.
Hence, for any path $\gamma$ we have $Z_d(I_\gamma)=\sum_{(\omega,a)\in \Sigma_\gamma}\fP_{(\omega,a)}$.%

The conclusion now follows from the additivity of cycles in short exact sequences (\cite[Lemma 2.2.7]{EG-perspective}), the fact that ${M}_{\kappa}$ is the cokernel of the injective map $\iota^0_\kappa$ and noting that we have an exact sequence
\[
0\ra \Bigg(\bigoplus_{\gamma,\ell(\gamma)=3}{I}_\gamma\Bigg)^0\ra \bigoplus_{\gamma,\ell(\gamma)=3}{I}_\gamma\ra \oplus_{\gamma\in F}{I}_\gamma\ra0
\]
with $F\defeq \Big\{((\eps_2,1),(0,0),(\eps_1,1)), ((\eps_1,1),(0,0),(\eps_2,1)), ((\eps_2,1),(0,0),(0,1))\Big\}$.
\end{proof}

\subsubsection{Analysis for $T^{(j)}=\{w_0 t_{\un{1}}, \alpha\beta t_{\un{1}}, \beta\alpha t_{\un{1}}\}$}
\label{sec:ideal3types}
We now have $\# T=3$ and write $\tau_{w_0}, \tau_{\alpha\beta}, \tau_{\beta\alpha}$ for the elements of $T$ distinguished by the conditions $\tld{w}^*(\rhobar,\tau_{w_0})_j=\alpha\beta\alpha t_{\un{1}}$, $\tld{w}^*(\rhobar,\tau_{\alpha\beta})_j=\alpha\beta t_{\un{1}}$ and $\tld{w}^*(\rhobar,\tau_{\beta\alpha})_j=\beta\alpha t_{\un{1}}$ respectively.

Note that  the canonical surjections $\tld{S}/\big( \tld{I}_{\tau_{\ast},\nabla_{\infty}}\cap \tld{I}_{\tau_{w_0},\nabla_{\infty}}\big)\onto \tld{S}/\tld{I}_{\tau_{w_0},\nabla_{\infty}}$ induce canonical maps
\begin{equation}
\label{eq:map:tor1}
\Tor_1^{{S}}(\F,(\tld{S}/\big( \tld{I}_{\tau_{\ast},\nabla_{\infty}}\cap \tld{I}_{\tau_{w_0},\nabla_{\infty}}\big)) \otimes \F)
\rightarrow
\Tor_1^{{S}}(\F,(\tld{S}/ \tld{I}_{\tau_{w_0},\nabla_{\infty}}) \otimes \F)
\end{equation}
for $\ast\in\{\alpha\beta,\beta\alpha\}$.
\begin{lemma}\label{lemma:ideal3types}
The union of the images of $\Tor^{{S}}_1(\F,(\tld{S}/\big( \tld{I}_{\tau_{\alpha\beta},\nabla_{\infty}}\cap \tld{I}_{\tau_{w_0},\nabla_{\infty}}\big)) \otimes \F)$ and $\Tor_1^{{S}}(\F,(\tld{S}/\big( \tld{I}_{\tau_{w_0},\nabla_{\infty}}\cap \tld{I}_{\tau_{\beta\alpha},\nabla_{\infty}}\big)) \otimes \F)$ in $\Tor^{{S}}_1(\F,(\tld{S}/\tld{I}_{\tau_{w_0},\nabla_{\infty}}) \otimes \F)$ is spanning. 
\end{lemma}
\begin{proof}
By the first row of diagram \eqref{eq:T=1} it is enough to prove the statement with $\nabla_\infty$ replaced by $\nabla_{\textnormal{alg}}$.
We have
\[
\Tor_1^{{S}}(\F,(\tld{S}/\tld{I}_{\tau_{w_0},\nabla_{\infty}}) \otimes \F)\cong \Tor^{{S}}_1(\F,{S}/{I}_{\tau_{w_0},\nabla_{\textnormal{alg}}})\cong  \bigoplus_j {I}^{(j)}_{\tau_{w_0},\nabla_{\textnormal{alg}}}/({I}^{(j)}_{\tau_{w_0},\nabla_{\textnormal{alg}}}\cdot \fm_{{S}^{(j)}}).
\]
where the first equality is from the second row of diagram \eqref{eq:T=1}, and the second equality follows from the fact that ${I}_{\tau_{w_0},\nabla_{\textnormal{alg}}}=\sum_{j'\in\cJ}{I}^{(j)}_{\tau_{w_0},\nabla_{\textnormal{alg}}}{S}$ and ${I}^{(j')}_{\tau_{w_0},\nabla_{\textnormal{alg}}}\subseteq {S}^{(j')}$ for all $j'\in\cJ$.
Let $j'\neq j$ and $\ast\in\{\alpha\beta,\beta\alpha\}$.
We thus have $\tld{w}(\rhobar,\tau_{\ast})_{j'}=\tld{w}(\rhobar,\tau_{w_0})_{j'}$ and the elements listed in column 4, row $\tld{w}(\rhobar,\tau_{w_0})_{j'}$ of Table \ref{Table_Ideals} are independent of $\tau$ modulo $p\cdot \fm_{S^{(j')}}$, as ,  $b^{(j')}_{\tau_{\ast}} \equiv b^{(j')}_{\tau_{w_0}}$ modulo $p$ and $\rhobar$ is tame (so that $\fm_{S^{(j')}}\supseteq(c^{(j')}_{ik}, 1\leq i,k\leq 3, d^{(j')}_{21}, d^{(j')}_{31}, d^{(j')}_{32})$).

Hence applying repeatedly Lemma \ref{lem:distortion} with $k=2$, $I_1=\tld{I}_{\tau_{\ast},\nabla_{\infty}}$, $I_2=\tld{I}_{\tau_{w_0},\nabla_{\infty}}$ and $f$ an element of $\tld{I}^{(j')}_{\tau_{w_0},\nabla_{\textnormal{alg}}}$ with $j'\neq j$, we conclude that the image of the map \eqref{eq:map:tor1} contains $\sum_{j'\neq j} {I}^{(j')}_{\tau_{w_0},\nabla_{\textnormal{alg}}}/({I}^{(j')}_{\tau_{w_0},\nabla_{\textnormal{alg}}}\cdot \fm_{{S}^{(j')}})$.
Thus the desired statement will follow once we prove that the union of the images of 
\[
\Tor_1^{{S}}(\F,(\tld{S}^{(j)}/\big( \tld{I}^{(j)}_{\tau_{\ast},\nabla_{\textnormal{alg}}}\cap \tld{I}^{(j)}_{\tau_{w_0},\nabla_{\textnormal{alg}}}\big)) \otimes \F)
\rightarrow
\Tor_1^{{S}}(\F,(\tld{S}/ \tld{I}^{(j)}_{\tau_{w_0},\nabla_{\infty}}) \otimes \F)
\]
for $\ast\in\{\alpha\beta,\beta\alpha\}$ is spanning in $\Tor_1(\F,(\tld{S}^{(j)}/ \tld{I}^{(j)}_{\tau_{w_0},\nabla_{\infty}}) \otimes \F)$.
This follows from an inspection of Table \ref{Table_Ideals_mod_p} (see \S \ref{subsubsec:Tor:cmpt} for details).
\end{proof}

\begin{lemma}\label{lemma:p3type}
We have $p \in  \tld{I}^{(j)}_{\tau_{\alpha\beta}}\cap \tld{I}^{(j)}_{\tau_{w_0}} +  \tld{I}^{(j)}_{\tau_{\beta\alpha}}\cap \tld{I}^{(j)}_{\tau_{w_0}} +  \tld{I}^{(j)}_{\tau_{\alpha\beta}}\cap \tld{I}^{(j)}_{\tau_{\beta\alpha}}$.

In particular, $p \in  \tld{I}_{\tau_{\alpha\beta},\nabla_{\infty}}\cap \tld{I}_{\tau_{w_0},\nabla_{\infty}} +  \tld{I}_{\tau_{\beta\alpha},\nabla_{\infty}}\cap \tld{I}_{\tau_{w_0},\nabla_{\infty}} +  \tld{I}_{\tau_{\alpha\beta},\nabla_{\infty}}\cap \tld{I}_{\tau_{\beta\alpha},\nabla_{\infty}}$. 
\end{lemma}
\begin{proof}
From Table \ref{Table_Ideals} we have $c_{22}\in \tld{I}^{(j)}_{\tau_{\beta\alpha}}\cap \tld{I}^{(j)}_{\tau_{w_0}}$ and $c_{11}d_{33}^*-d_{31}c_{13}-pd_{11}^*d_{33}^*\in \tld{I}^{(j)}_{\tau_{\alpha\beta}}\cap \tld{I}^{(j)}_{\tau_{w_0}}$.
Moreover, using the first and last equation in row $(\alpha\beta t_{\un{1}}, \tld{I}_{\tau}^{(j)})$ of Table \ref{Table_Ideals} we obtain $c_{11}d_{22}^*d_{33}^*-d_{31}c_{13}d_{22}^*-c_{22}d_{11}^*d_{33}^*\in \tld{I}^{(j)}_{\tau_{\alpha\beta}}$ and hence  $c_{11}d_{22}^*d_{33}^*-d_{31}c_{13}d_{22}^*-c_{22}d_{11}^*d_{33}^*\in \tld{I}^{(j)}_{\tau_{\alpha\beta}}\cap \tld{I}^{(j)}_{\tau_{\beta\alpha}}$ since $c_{22}, c_{11}d_{33}^*-d_{31}c_{13}$ are both in $\tld{I}^{(j)}_{\tau_{\beta\alpha}}$.
Thus, 
\begin{align*}
pd_{11}^*d_{22}^*d_{33}^*&=
(c_{11}d_{22}^*d_{33}^*-d_{31}c_{13}d_{22}^*-c_{22}d_{11}^*d_{33}^*)+c_{22}d_{11}^*d_{33}^*-d_{22}^*(c_{11}d_{33}^*-d_{31}c_{13}-pd_{11}^*d_{33}^*)\\
&\in  \tld{I}^{(j)}_{\tau_{\alpha\beta}}\cap \tld{I}^{(j)}_{\tau_{\beta\alpha}} +  \tld{I}^{(j)}_{\tau_{\beta\alpha}}\cap \tld{I}^{(j)}_{\tau_{w_0}}+ \tld{I}^{(j)}_{\tau_{\alpha\beta}}\cap \tld{I}^{(j)}_{\tau_{w_0}}
\end{align*}
which implies the statement as $d_{11}^*d_{22}^*d_{33}^*$ is a unit in $\tld{S}^{(j)}$.
\end{proof}
\begin{rmk}
Even though $p\notin \tld{I}_{\tau_{\alpha\beta}} +  \tld{I}_{\tau_{\beta\alpha}}$, the situation changes after imposing monodromy, i.e. ~we do have $p \in  \tld{I}_{\tau_{\alpha\beta},\nabla_{\infty}} +  \tld{I}_{\tau_{\beta\alpha},\nabla_{\infty}}$.
To see this, for $\tau\in\{\tau_{\alpha\beta},\tau_{\beta\alpha}\}$ write $\textnormal{Mon}_{\tau,1}$ for the first element in row $\textnormal{Mon}_{\tau}$ of Table \ref{Table_Ideals}.
Then $\textnormal{Mon}_{\tau_{\beta\alpha},1}-\textnormal{Mon}_{\tau_{\alpha\beta},1}\in \tld{I}^{(j)}_{\tau_{\alpha\beta},\nabla_{\infty}} +  \tld{I}^{(j)}_{\tau_{\beta\alpha},\nabla_{\infty}}$ and, by inspection of the equations $\textnormal{Mon}_{\tau_{\alpha\beta},1},\textnormal{Mon}_{\tau_{\beta\alpha},1}$ and using that $(a,b,c)\defeq (b_{\tau,1},b_{\tau,2},b_{\tau,3}) \mod{p}$ is independent of $\tau\in\{\tau_{\alpha\beta},\tau_{\beta\alpha}\}$, we deduce
\begin{align*}
p\left((b-c)d_{11}^*d_{22}^*+pd_{11}^*d_{22}^*+xd_{21}c_{21}+yc_{11}d_{22}^*+O(p^{N-5})
\right)
\end{align*}
for some $x,y\in\Z_p$.
The factor in parenthesis is a unit in $\tld{S}^{(j)}$ since $p,d_{21}c_{21}, c_{11}\in \fm_{\tld{S}^{(j)}}$, $N>5$ and $b-c\not\equiv 0$ modulo $p$ and hence $p \in \tld{I}^{(j)}_{\tau_{\alpha\beta},\nabla_{\infty}} +  \tld{I}^{(j)}_{\tau_{\beta\alpha},\nabla_{\infty}}$.
\end{rmk}

\subsubsection{Analysis for $T^{(j)}=\{t_{w_0(\eta)}, t_{w_0(\eta)}\alpha ,  t_{w_0(\eta)}\beta\}$}
\label{sec:ideal3types:prime}
The analysis is similar to that of \S \ref{sec:ideal3types}, replacing $w_0t_{\un{1}}, \alpha\beta t_{\un{1}}$ and $\beta\alpha t_{\un{1}}$ by $t_{w_0(\eta)}$, $t_{w_0(\eta)}\alpha$ and $t_{w_0(\eta)}\beta$ respectively.
A proof analogous to that of Lemma \ref{lemma:ideal3types}, using now Table \ref{Table_Ideals_mod_p_F2} instead of Table \ref{Table_Ideals_mod_p}, gives us the following result:
\begin{lemma}\label{lemma:ideal3types:prime}
The union of the images of $\Tor^{{S}}_1(\F,(\tld{S}/\big( \tld{I}_{\tau_{t_{w_0(\eta)}\alpha },\nabla_{\infty}}\cap \tld{I}_{\tau_{t_{w_0(\eta)}},\nabla_{\infty}}\big)) \otimes \F)$ and of $\Tor_1^{{S}}(\F,(\tld{S}/\big( \tld{I}_{\tau_{t_{w_0(\eta)}},\nabla_{\infty}}\cap \tld{I}_{\tau_{t_{w_0(\eta)}\beta },\nabla_{\infty}}\big)) \otimes \F)$ in $\Tor^{{S}}_1(\F,(\tld{S}/\tld{I}_{\tau_{t_{w_0(\eta)}},\nabla_{\infty}}) \otimes \F)$ is spanning. 
\end{lemma}
(Note that, for consistency with \ref{sec:ideal3types}, we should replace $\eta$ by $(1,0,-1)$ in the statement of Lemma \ref{lemma:ideal3types:prime}; we used $\eta$ instead for ease of notation.)

\subsubsection{Analysis for $T^{(j)}=\{w_0 t_{\un{1}}, \alpha\beta t_{\un{1}}, \beta\alpha t_{\un{1}}, t_{\un{1}}\}$}
\label{subsub:4types}
We now have $\# T=4$ and write $\tau_{w_0}, \tau_{\alpha\beta}, \tau_{\beta\alpha}, \tau_{\Id}$ for the elements of $T$ distinguished by the conditions $\tld{w}^*(\rhobar,\tau_{w_0})_j=w_0t_{\un{1}}$, $\tld{w}^*(\rhobar,\tau_{\alpha\beta})_j=\alpha\beta t_{\un{1}}$, $\tld{w}^*(\rhobar,\tau_{\beta\alpha})_j=\beta\alpha t_{\un{1}}$ and $\tld{w}^*(\rhobar,\tau_{\beta\alpha})_j= t_{\un{1}}$ respectively.

Recall that $\tld{I}_{T,\nabla_\infty}=\tld{I}_{\tau_{\alpha\beta},\nabla_{\infty}}\cap \tld{I}_{\tau_{w_0},\nabla_{\infty}}\cap \tld{I}_{\tau_{\beta\alpha},\nabla_{\infty}}\cap \tld{I}_{\tau_{\Id},\nabla_{\infty}}$.
Define ${I}^{(j)}_{\Lambda}$ to be the intersection ${\fP}_{(0,1)}^{(j)}\cap{\fP}_{(0,0)}^{(j)}\cap{\fP}^{(j)}_{(\eps_1,0)}\cap{\fP}^{(j)}_{(\eps_1,0)}$ in ${S}^{(j)}/{I}^{(j)}_{\tau_{\Id},\nabla_{\mathrm{alg}}}$, and let $\tld{I}_\Lambda$ be the pullback in $\tld{S}$ of the ideal $\sum_{j'\in\cJ,j'\neq j}{I}^{(j')}_{\tau_{\Id},\nabla_{\mathrm{alg}}}{S}+{I}^{(j)}_{\Lambda}{S}\subseteq {S}$ via $\tld{S}/\tld{I}_{T,\nabla_\infty}\onto{S}/{I}_{\tau_{\Id},\nabla_{\infty}}\cong  {S}/{I}_{\tau_{\Id},\nabla_{\mathrm{alg}}}$, where the isomorphism follows from the bottom line of  \eqref{eq:T=1}.
(Note that in the setting of this subsection the ideal ${I}^{(j')}_{\tau,\nabla_{\mathrm{alg}}}$ is independent of $\tau\in T$ when $j'\neq j$, in particular in the definition of $\tld{I}_\Lambda$ we can replace $\sum_{j'\in\cJ,j'\neq j}{I}^{(j')}_{\tau_{\Id},\nabla_{\mathrm{alg}}}{S}$ with $\sum_{j'\in\cJ,j'\neq j}{I}^{(j')}_{\tau,\nabla_{\mathrm{alg}}}{S}$ for any choice of $\tau\in T$.)
The proof of the following Lemma is analogous to that of Lemma \ref{lem:intsc} (see also \S \ref{appendix:multi:sp:fiber}).

\begin{lemma}
\label{lem:broom:other}
Under the current assumption we have
\[
I_\Lambda^{(j)}=(c_{22},\,c_{33},\,d_{32}c_{23}, c_{13}d_{31}-c_{11}d_{33}^*).
\]
\end{lemma} 

We have $\tld{I}^{(j)}_{\tau_{\beta\alpha},\nabla_\infty}\subseteq{I}^{(j)}_{\tau_{\beta\alpha},\nabla_{\textnormal{alg}}}\subseteq {\fP}_{(0,1)}^{(j)}\cap{\fP}_{(0,0)}^{(j)}\cap{\fP}^{(j)}_{(\eps_1,0)}$ by Proposition \ref{prop:special_fiber} and \cite[Table 3]{LLLM2}, and similarly $\tld{I}^{(j)}_{\tau_{\alpha\beta},\nabla_\infty}\subseteq{I}^{(j)}_{\tau_{\alpha\beta},\nabla_{\textnormal{alg}}}\subseteq {\fP}_{(0,1)}^{(j)}\cap{\fP}_{(0,0)}^{(j)}\cap{\fP}^{(j)}_{(\eps_2,0)}$.
Hence we have canonical surjections $\tld{S}/\big( \tld{I}_{\tau_{\alpha\beta},\nabla_{\infty}}\cap \tld{I}_{\tau_{w_0},\nabla_{\infty}},p\big)\cap (\tld{I}_{\tau_{w_0},\nabla_{\infty}}\cap \tld{I}_{\tau_{\beta\alpha},\nabla_{\infty}},p)\onto \tld{S}/\tld{I}_\Lambda$ and $\tld{S}/(\tld{I}_{\tau_{\Id},\nabla_{\infty}},p)\onto \tld{S}/\tld{I}_\Lambda$ which induce canonical maps
\begin{align*}
\label{eq:map:tor:1-4T}
\Tor^{{S}}_1(\F,\tld{S}/\big( \tld{I}_{\tau_{\alpha\beta},\nabla_{\infty}}\cap \tld{I}_{\tau_{w_0},\nabla_{\infty}},p\big)\cap (\tld{I}_{\tau_{w_0},\nabla_{\infty}}\cap \tld{I}_{\tau_{\beta\alpha},\nabla_{\infty}},p))
&\rightarrow
\Tor_1^{{S}}(\F,(\tld{S}/ \tld{I}_\Lambda) \otimes \F)
\end{align*}
and
\begin{align*}
\Tor^{{S}}_1(\F,(\tld{S}/I_{\tau_{\Id},\nabla_{\infty}}) \otimes \F)
&\rightarrow
\Tor_1^{{S}}(\F,(\tld{S}/ \tld{I}_\Lambda) \otimes \F).
\end{align*}

\begin{lemma}\label{lemma:ideal4types}
The union of the images of $\Tor^{{S}}_1(\F,(\tld{S}/\tld{I}_{\tau_{\Id},\nabla_{\infty}}) \otimes \F)$ and $\Tor^{{S}}_1(\F,\tld{S}/\big( \tld{I}_{\tau_{\alpha\beta},\nabla_{\infty}}\cap \tld{I}_{\tau_{w_0},\nabla_{\infty}},p\big)\cap (\tld{I}_{\tau_{w_0},\nabla_{\infty}}\cap \tld{I}_{\tau_{\beta\alpha},\nabla_{\infty}},p))$ in $\Tor_1^{{S}}(\F,(\tld{S}/\tld{I}_\Lambda)\otimes\F)$ is spanning. 
\end{lemma}
\begin{proof}
The proof is very similar to that of  Lemma \ref{lemma:ideal3types}.
As in \emph{loc.~cit}.~it suffices to prove the statement with $\nabla_{\infty}$ replaced by $\nabla_{\textnormal{alg}}$ everywhere. 
We have
\[
\Tor^{{S}}_1(\F,\tld{S}/\tld{I}_{\tau_{\Id},\nabla_{\textnormal{alg}}}) \otimes \F)= \bigoplus_j {I}^{(j)}_{\tau_{\Id},\nabla_{\textnormal{alg}}}/({I}^{(j)}_{\tau_{\Id},\nabla_{\textnormal{alg}}}\cdot \fm_{{S}^{(j)}})
\]
since ${I}_{\tau_{\Id},\nabla_{\textnormal{alg}}}=\sum_{j'\in\cJ}{I}^{(j')}_{\tau_{\Id},\nabla_{\textnormal{alg}}}{S}$ with ${I}^{(j')}_{\tau_{\Id},\nabla_{\textnormal{alg}}}\subseteq {S}^{(j')}$ for all $j'\in\cJ$, and we have a similar decompositions
\[
\Tor^{{S}}_1(\F,(\tld{S}/\tld{I}_{\Lambda})\otimes\F)={I}^{(j)}_{\Lambda}/({I}^{(j)}_{\tau_{\Id},\nabla_{\textnormal{alg}}}\cdot \fm_{{S}^{(j)}})\oplus \bigoplus_{j'\neq j} {I}^{(j')}_{\tau_{\Id},\nabla_{\textnormal{alg}}}/({I}^{(j')}_{\tau_{\Id},\nabla_{\textnormal{alg}}}\cdot \fm_{{S}^{(j')}})
\]
and for $\Tor_1^{{S}}(\F,S/\big( \tld{I}_{\tau_{\alpha\beta},\nabla_{\textnormal{alg}}}\cap \tld{I}_{\tau_{w_0},\nabla_{\infty}},p\big)\cap (\tld{I}_{\tau_{w_0},\nabla_{\textnormal{alg}}}\cap \tld{I}_{\tau_{\beta\alpha},\nabla_{\infty}},p))$.

Hence the desired statement will follow once we prove that union of the images of 
\[
\Tor_1^{{S}^{(j)}}(\F,{S}^{(j)}/\big( \tld{I}^{(j)}_{\tau_{\alpha\beta},\nabla_{\textnormal{alg}}}\cap \tld{I}^{(j)}_{\tau_{w_0},\nabla_{\infty}},p\big)\cap (\tld{I}^{(j)}_{\tau_{w_0},\nabla_{\textnormal{alg}}}\cap \tld{I}^{(j)}_{\tau_{\beta\alpha},\nabla_{\infty}},p))
\rightarrow
\Tor^{{S}^{(j)}}_1(\F,{S}^{(j)}/ {I}^{(j)}_{\Lambda})
\]
and
\[
\Tor_1^{{S}^{(j)}}(\F,{S}^{(j)}/\big( \tld{I}^{(j)}_{\tau_{\Id},\nabla_{\textnormal{alg}}},p))
\rightarrow
\Tor^{{S}^{(j)}}_1(\F,{S}^{(j)}/ {I}^{(j)}_{\Lambda})
\]
 is spanning.
This follows from the last row in Table \ref{Table_Ideals_mod_p}  (see \S \ref{subsubsec:Tor:cmpt} for details).
\end{proof}

\subsubsection{Analysis for $T^{(j)}=\{t_{w_0(\eta)}, t_{w_0(\eta)}\alpha ,  t_{w_0(\eta)}\beta, t_{w_0(\eta)}w_0\}$}
\label{subsub:4types:prime}
The analysis is similar to that of \S \ref{subsub:4types}, replacing $
w_0 t_{\un{1}}$, $\alpha\beta t_{\un{1}}$, $\beta\alpha t_{\un{1}}$ and $t_{\un{1}}$ by $t_{w_0(\eta)}$, $t_{w_0(\eta)}\alpha$, $t_{w_0(\eta)}\beta$ and $t_{w_0(\eta)}w_0$ respectively, and  ${\fP}^{(j)}_{(0,0)}\cap{\fP}^{(j)}_{(0,1)}$ by ${\fP}^{(j)}_{(\eps_1+\eps_2,1)}$.
The following Lemma is proved in \S \ref{appendix:multi:sp:fiber}.

\begin{lemma}
\label{lem:broom:other:prime}
In the current assumptions we have:
\[
{I}^{(j)}_{\Lambda}=(c_{22},c_{33},c_{32},e_{33},e_{23}, d_{31},(a-b)c_{12}c_{23}-(a-c)e_{13}d_{22}^*,d_{21}d_{32},c_{23}d_{32},d_{21}c_{12}).
\]
\end{lemma}

A proof analogous to that of Lemma \ref{lemma:ideal4types}, using now Table \ref{Table_Ideals_mod_p_F2} instead of Table \ref{Table_Ideals_mod_p}, yields the following:
\begin{lemma}\label{lemma:ideal4types:prime}
The union of the image of $\Tor^{{S}}_1(\F,(\tld{S}/\tld{I}_{\tau_{t_{w_0(\eta)}w_0},\nabla_{\infty}}) \otimes \F)$ and of $\Tor^{{S}}_1(\F,\tld{S}/\big( \tld{I}_{\tau_{t_{w_0(\eta)}\alpha },\nabla_{\infty}}\cap \tld{I}_{\tau_{t_{w_0(\eta)}},\nabla_{\infty}},p\big)\cap (\tld{I}_{\tau_{t_{w_0(\eta)}},\nabla_{\infty}}\cap \tld{I}_{\tau_{t_{w_0(\eta)}\beta },\nabla_{\infty}},p))$ in $\Tor_1^{{S}}(\F,(\tld{S}/\tld{I}_\Lambda)\otimes\F)$ is spanning. 
\end{lemma}

\section{Representations of $\mathrm{GL}_3$}
\label{sec:rt}

\subsection{Some tilting modules for $\GL_3$}\label{sec:modrep}

We label some alcoves for $\GL_3$ as in Figure \ref{AlcoveLabels}.
\begin{figure}[htb] 
\centering
\includegraphics[scale=.3]{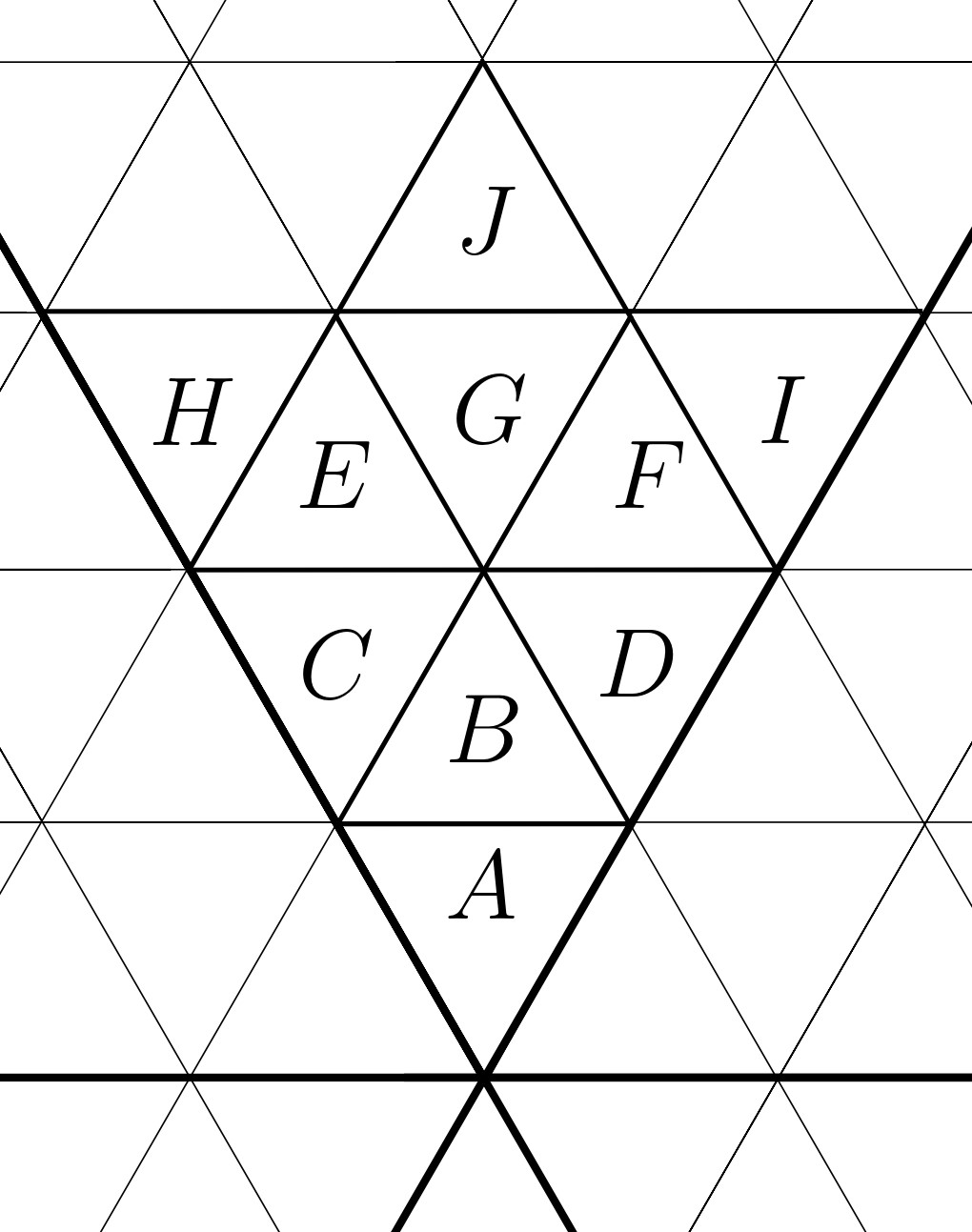}
\caption{Labelling of alcoves for $\GL_3$}
\label{AlcoveLabels}
\end{figure}
For a dominant alcove $X$, let $\tld{w}_X\in W_a$ denote the corresponding element of the affine Weyl group. 
We write $\tld{w}_X = t_{\omega_X}\tld{w}^0_X$ for some $\tld{w}^0_X \in \tld{W}_1^+$ and dominant $\omega_X \in X^*(T)$ which are unique up to $X^0(T)$. 
If $\lambda\in X^*(T)$ is $0$-deep in an alcove, $\lambda_X$ denotes the unique weight in alcove $X$ that is linked to $\lambda$ (cf.~\cite[II.6.5]{RAGS}). 
Then $\lambda_X=\lambda_X^0+p\omega_X$ with $\lambda_X^0\in X_1(T)$ and $\omega_X$ as above. 
Moreover, if $\lambda\in C_0$ then $\lambda_X=\tld{w}_X\cdot \lambda = \tld{w}^0_X\cdot \lambda+p\omega_X$ so that $\lambda_X^0 = \tld{w}^0_X\cdot \lambda$. 

Let $\F$ be a finite extension of $\F_p$. 
For a dominant $\lambda \in X^*(T)$, recall from \cite[\S II.2.4]{RAGS} the simple $\GL_{3/\F}$-representation $L(\lambda)$ with highest weight $\lambda$. 
In this section we write $\Ext^1$ for $\Ext^1_{\GL_{3/\F}}$.
From \cite[Theorem 4.2.3(i)]{yehia} we have:

\begin{prop}\label{prop:ext}
For any $\lambda, \mu \in X^+(T)$ we have $\dim_{\F} \Ext^1(L(\lambda),L(\mu)) \leq 1$.  
\end{prop}
\noindent Proposition \ref{prop:ext} implies that if a nonsplit extension of two simple $\GL_{3/\F}$-modules exists, then it is unique up to isomorphism. 
Recall from \cite[\S II.6.17]{RAGS} that $\Ext^1(L(\lambda),L(\mu)) = 0$ unless $\lambda$ and $\mu$ are linked i.e.~$\mu = \tld{w}\cdot \lambda$ for some $\tld{w} \in W_a$. 

For $\lambda \in X^*_1(T)$, recall from \cite[Theorem 4.2.1]{LLLM2} the $\GL_{3/\F}$-representation $Q_1(\lambda)$ (see also \cite[\S II.11.3 and \S II.11.11]{RAGS}). 
As the following results explain, the module $Q_1(\lambda)$ acts like an injective and projective module in the full subcategory of modules whose Jordan--H\"older factors have $p$-bounded highest weight. 

\begin{prop} \label{prop:Q1inj}
If $\mu \in X^+(T)$ is $p$-bounded, i.e.~$\langle \mu, \alpha^\vee\rangle < 4p$ for all roots $\alpha$, then $\Ext^1(L(\mu),Q_1(\lambda)) = 0$ and $\Ext^1(Q_1(\lambda),L(\mu)) = 0$. 
\end{prop}
\begin{proof}
The second vanishing statement follows from the first by duality and the first vanishing statement follows from the proof of \cite[(4.8)]{LLLM2}. 
\end{proof}

By d\'evissage, Proposition \ref{prop:Q1inj} yields the following. 

\begin{cor}\label{cor:Q1inj}
If $M$ is a (finite length) module with only $p$-bounded Jordan--H\"older factors, then $\Ext^1(M,Q_1(\lambda)) = 0$ and $\Ext^1(Q_1(\lambda),M) = 0$. 
\end{cor}

For $\lambda \in X^+(\un{T})$, we define $V(\lambda)$ as in \cite[II.2.13(1)]{RAGS}, and write $W(\lambda)$ for its dual.
By \cite[II.2.14(1)]{RAGS} $V(\lambda)$ (resp.~$W(\lambda)$) has ireducible cosocle (resp.~socle) isomorphic to $L(\lambda)$, and we call it the Weyl module (resp.~dual Weyl module) associated to $\lambda$.
A Weyl (resp.~dual Weyl) filtration on a $\GL_{3/\F}$-module $M$ is an exhaustive filtration whose graded pieces are direct sum of Weyl (resp.~dual Weyl) modules
(cf.~\cite[II.4.19]{RAGS}).

The following Proposition is a reformulation of \cite[Theorem B(b)]{BDM}.
\begin{prop}[\cite{BDM}]
\label{prop:filtrations}
\begin{enumerate}
\item The module $Q_1(\lambda)$ is rigid with Loewy layers given by the rows in Figures \ref{fig:weyl:QB}, \ref{fig:weyl:QA} (a row with alcove labels $X_1, \cdots, X_n$ denotes a direct sum with multiplicity of the simple $\GL_{3/\F}$-representations $L(\lambda_{X_1}),\cdots,L(\lambda_{X_n})$). 
\item The module $Q_1(\lambda)$ has a Weyl (resp.~dual Weyl) filtration such that $\textnormal{gr}(Q_1(\lambda))$ is a direct sum of multiplicity free Weyl modules (resp.~dual Weyl modules) with Jordan--H\"older factors given by the connected components of the graphs in Figures \ref{fig:weyl:QB} (resp.~\ref{fig:weyl:QA}). 
\item The socle and cosocle filtrations on these Weyl modules coincide with the filtrations induced from the Loewy filtration on $Q_1(\lambda)$ (up to shift). 
\item Each edge in Figure \ref{fig:weyl:QB} (resp.~\ref{fig:weyl:QA}) indicates the existence of a subquotient of a Weyl module (resp.~dual Weyl module) in $\textnormal{gr}(Q_1(\lambda))$, which is a nonsplit extension (unique up to isomorphism by Proposition \ref{prop:ext}) of the indicated simple modules. 
\end{enumerate}
\end{prop}

\begin{figure}[htb]
\caption{Weyl and dual Weyl filtrations for $Q_1(\lambda)$, case $\lambda\in B$}
\label{fig:weyl:QB}
  \centering
\adjustbox{max width=\textwidth}{
\begin{tabular}{| c | c |}
\hline
\hline
&\\
{\includegraphics[scale=.4]{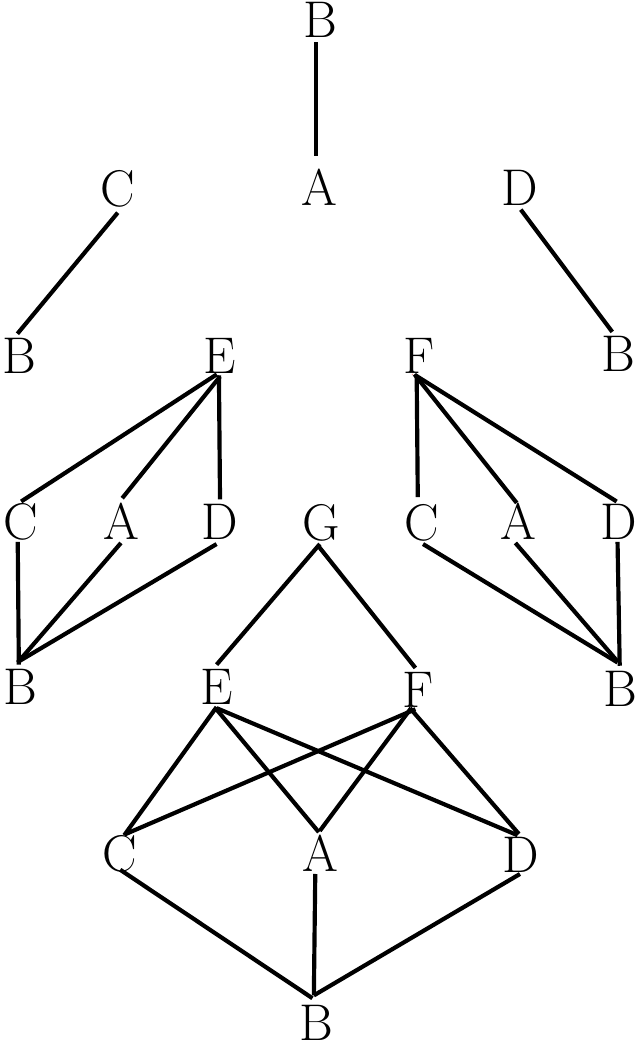}}
  &
  {\includegraphics[scale=0.4]{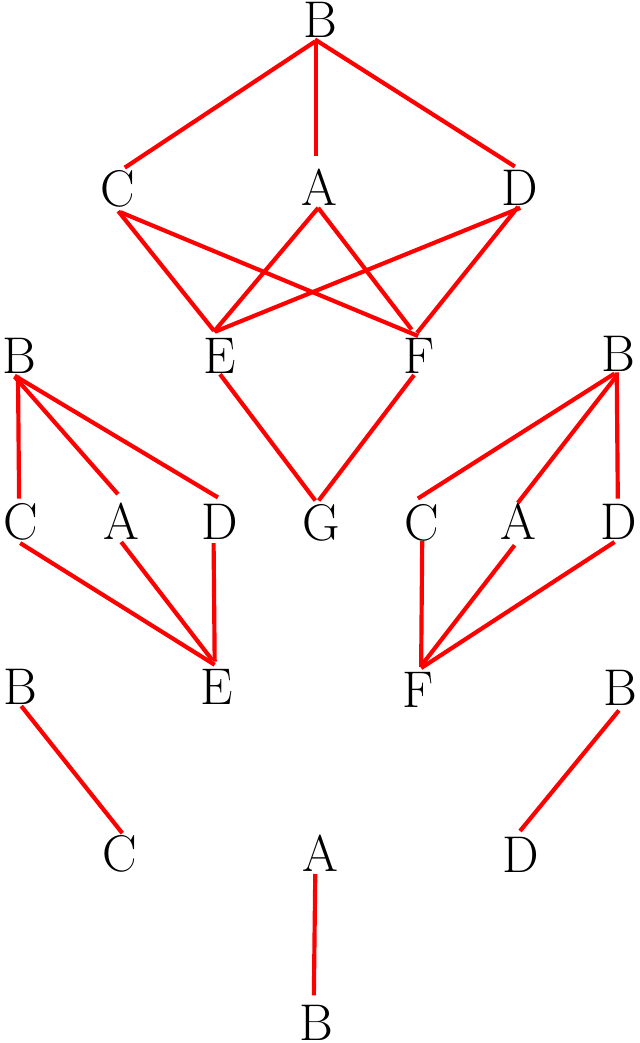}}\\
  &\\
    &\\
    \hline\hline
 \end{tabular}}
\end{figure}

\begin{figure}[htb]
\caption{Weyl and dual Weyl filtrations for $Q_1(\lambda)$, case $\lambda\in A$}
\label{fig:weyl:QA}
  \centering
\adjustbox{max width=\textwidth}{
\begin{tabular}{| c | c |}
\hline
\hline
&\\
{\includegraphics[scale=.4]{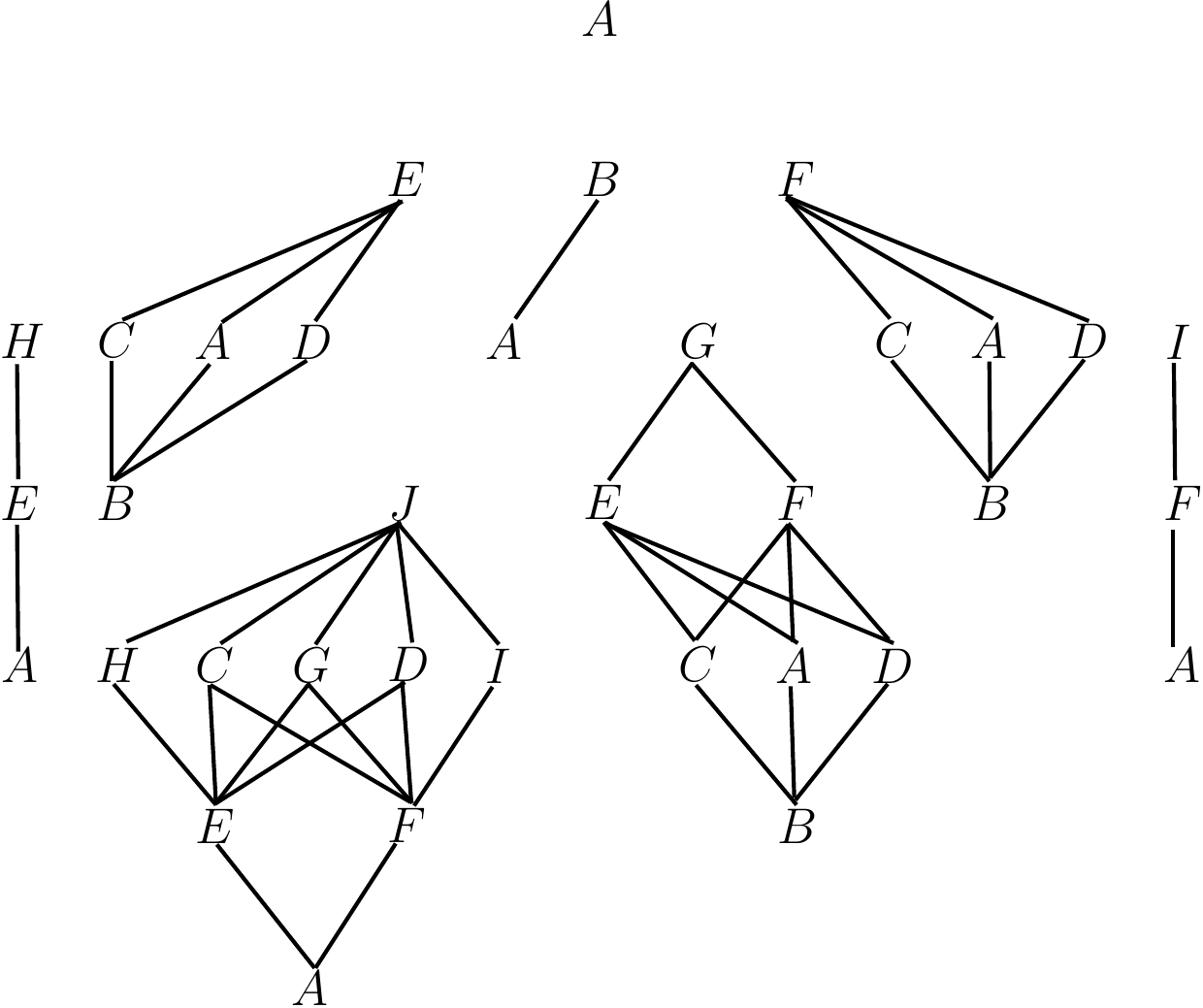}}
  &
  {\includegraphics[scale=0.4]{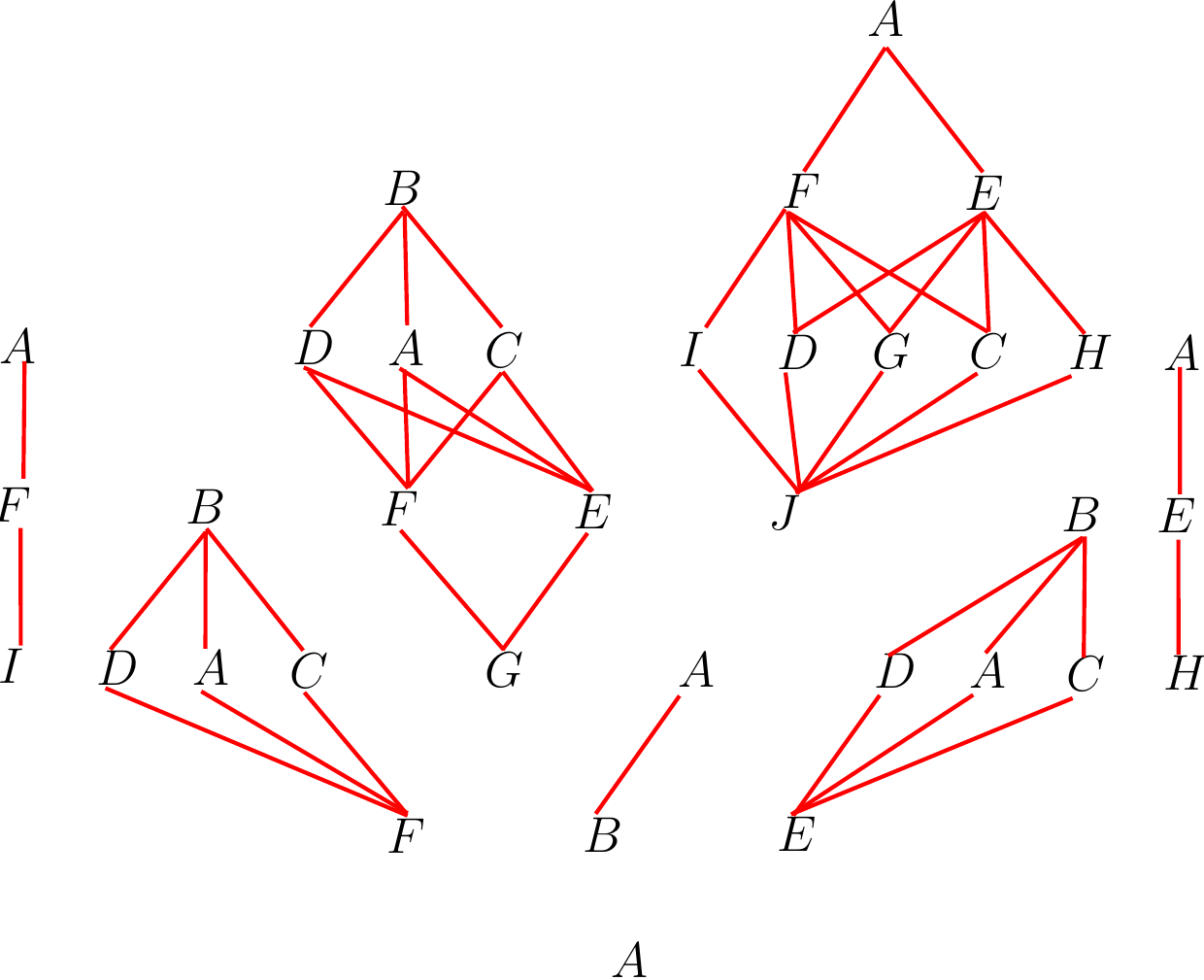}}\\
  &\\
    &\\
    \hline\hline
 \end{tabular}}
\end{figure}

Let $\rad^\bullet Q_1(\lambda)$ denote the (decreasing) radical filtration.
We let $Q_1(\lambda)^{\widehat{k}}$ denote $Q_1(\lambda)/\rad^k Q_1(\lambda)$. 
Informally speaking, $Q_1(\lambda)^{\widehat{k}}$ is the maximal quotient of $Q_1(\lambda)$ obtained by removing the $k$-th layer.

\begin{prop}[Translation principle]\label{prop:TP}
Let $\lambda \in X^*_1(T)$ be $m$-deep in its alcove, and $\eps \in X^*(T)$ such that $\langle \eps,\alpha^\vee\rangle\leq m$ for all $\alpha\in \Phi$.
Then 
\[
L(\eps)\otimes\rad^k Q_1(\lambda) \cong \bigoplus_{\nu \in L(\eps)}\rad^k Q_1(\lambda+\nu)
\]
(where $\nu \in L(\eps)$ means  that $\nu\in X^*({T})$ is in the weight space of $L(\eps)$, counted with multiplicity). 
\end{prop}
\begin{proof}
Noting that $L(\eps)\otimes L(\lambda)$ is semisimple by the assumption on $\eps$ (\cite[Proposition 6.4]{HumphreysBook}), the statement follows by d\'evissage from \cite[Lemma 4.6]{andersen-kaneda-rigidity} and Proposition \ref{prop:filtrations}.
\end{proof}

\subsubsection*{Restriction to rational points}\label{subsec:ratpts}

Recall from \S \ref{sec:notation} the algebraic group $G_0$ and the natural isomorphism $\un{G}_{/\F}\defeq G_0\times_{\Zp}\F\cong \prod_{j \in \cJ} \GL_{3/\F}$.
Let $\rG$ be the finite group $G_0(\F_p)$.

The following proposition expresses the Jordan--H\"older factors of $L(\mu)|_{\rG}$, for suitable $\mu\in X^*(\un{T})$, in terms of the extension graph.
\begin{prop}
\label{prop:alg_rep_ext_graph}
Let $m,t$ be nonnegative integers such that $4t\leq m$ and $\lambda \in A^\cJ$ be $m$-deep. Assume that for each $j\in\cJ$, $X_j$ is a dominant alcove such that $\max_{\mu\in X_j,\alpha\in \Phi} \langle \mu,\alpha^{\vee}\rangle<4pt$.
Then 
\[
\otimes_{j\in \cJ} L(\tld{w}_{X_j}\cdot \lambda_j)|_{\rG} \cong \bigoplus_{\omega \in L(\sum_{j\in \cJ} \omega_{X_j})} \Trns_{\lambda+\eta}(\pi\omega,\pi(\tld{w}_{X_j}^0(A))_{j\in\cJ}). 
\]
\end{prop}
\begin{proof}
By \cite[Lemma 4.2.4(1)]{LLLM2} (see also \cite[proof of Lemma 4.2.5]{LLLM2}) we have
\[
\otimes_{j\in \cJ} L(\tld{w}_{X_j}\cdot \lambda_j)|_{\rG} \cong
\bigoplus_{\omega\in L(\sum_j\omega_{X_j})} F\Big(\pi\omega+\sum_{j\in\cJ}\tld{w}^0_{X_j}\cdot \lambda_j \Big)
\]
and the conclusion follows directly from the definition of $\Trns_{\lambda+\eta}$ (see \cite[equation (2.3)]{GL3Wild}).
\end{proof}

\begin{figure}[htb] 
\centering
\begin{minipage}{0.45\textwidth}
\centering
\includegraphics[scale=.3]{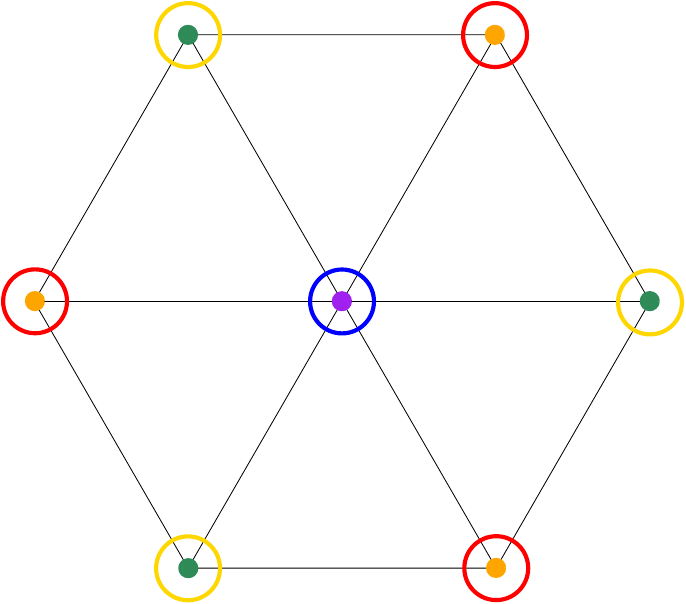}
\label{fig-res:1}
\end{minipage}\hfill
\begin{minipage}{0.45\textwidth}
\centering
\includegraphics[scale=.3]{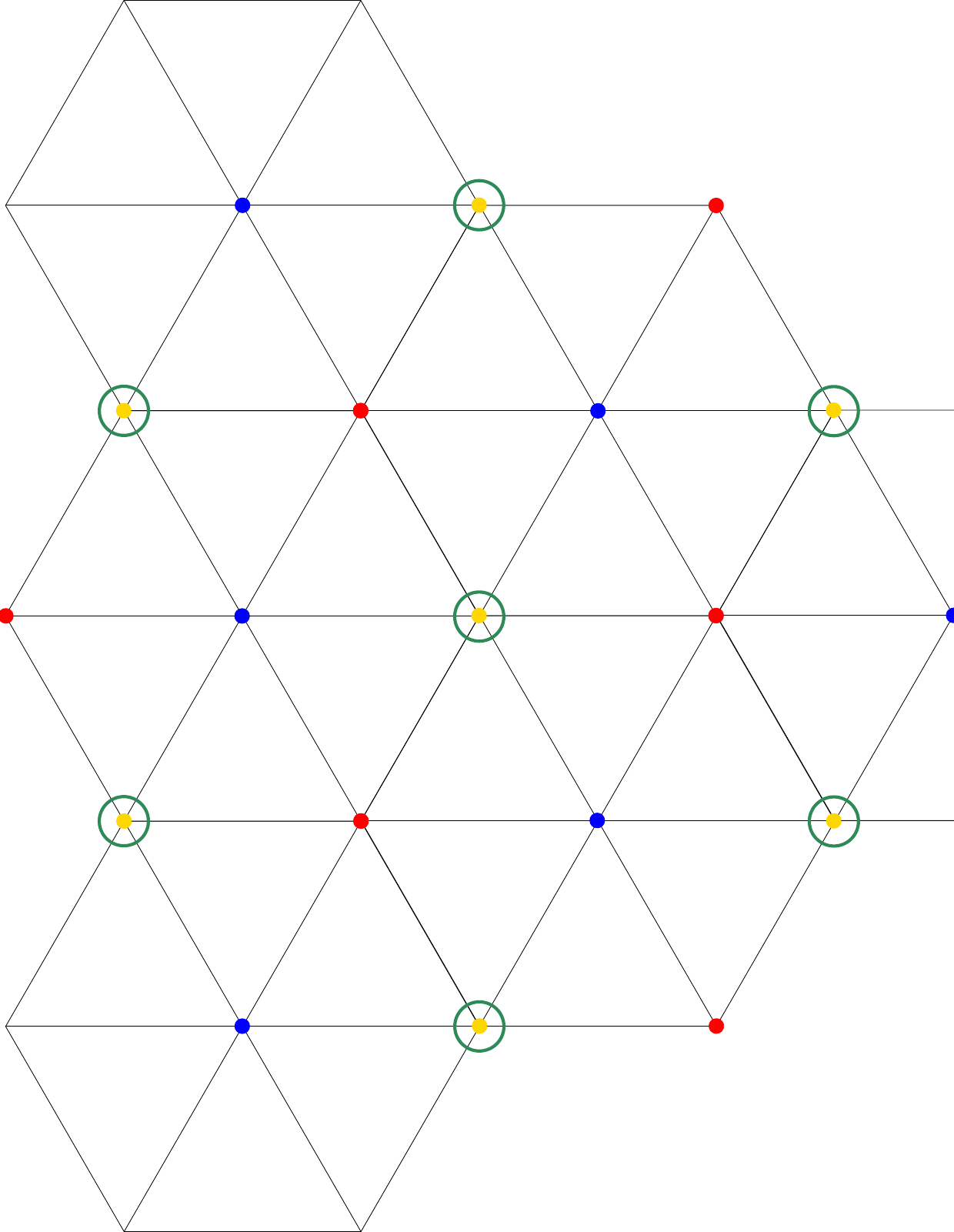}
\label{fig-res:2}
\end{minipage}
\caption{In this picture we represent the elements in the $\pi(j)$-th coordinate of $(\omega, (\tld{w}^0_{X_j}(A))_{j\in\cJ})$ for $\omega\in L(\sum_{j\in \cJ} \omega_{X_j})$, where 
$X_j\in\{A,B,C,D,E,F,G,J,I,H\}$.
(Informally speaking, these elements pictures the $\pi(j)$-th component of the restriction of  $L(\sum_{j\in \cJ} \omega_{X_j})$ to $\rG$.)
On the left we consider the case where $X_j\in\{A,B,C,D,E,F\}$.
The gold (resp.~red, resp.~blue) circles represent the case $X_j=E$ (resp.~$X_j=F$, resp.~$X_j=B$).
The green (resp.~orange, resp.~purple) dots represent the case $X_j=C$ (resp.~$X_j=D$, resp.~$X_j=A$).
On the right we consider the case where $X_j\in\{G,J,I,H\}$.
The green circles represent the case where $X_j=J$.
The gold (resp.~red, resp.~blue) dots represent the case $X_j=G$ (resp.~$X_j=H$, resp.~$X_j=I$).
}
\end{figure}

\subsection{$\rG$-projective covers and modular Serre weights} \label{subsec:Gproj}

For $\lambda \in X_1(\un{T})$, let $Q_1(\lambda)$ be the $\un{G}_{/\F}$-module $\otimes_{j\in \cJ} Q_1(\lambda_j)$. 
For $a = (a_j)_j \in \Z_{\geq 0}^\cJ$ and $\lambda \in X_1(\un{T})$, let $\rad^a Q_1(\lambda)$ {(resp.~$Q_1(\lambda)^{\widehat{a}}$)} be the tensor product $\otimes_{j\in \cJ} \rad^{a_j} Q_1(\lambda_j)$ (resp.~$\otimes_{j\in \cJ} Q_1(\lambda_j)^{\widehat{a_j}}$). 
We also let $\rad^{>a} Q_1(\lambda)$ be $\sum_{b\geq a, b\neq a} \rad^b Q_1(\lambda)$, where $\geq$ denotes the product partial order on $\Z_{\geq 0}^\cJ$, and let $\gr^a Q_1(\lambda)$ be $\rad^a Q_1(\lambda)/\rad^{>a} Q_1(\lambda)$. 
Note that $\rad^{>a} Q_1(\lambda)$ is $\rad (\rad^a Q_1(\lambda))$. 
For a tuple $a = (a_j)_{j\in \cJ}$ with $0\leq a_j \leq 7$, let $|a| = \sum_j a_j$.
If $n\in \N$ we define
\[
\rad^n Q_1(\lambda)\defeq \sum_{\substack{a\in \Z^{\cJ}_{\geq 0}\\ |a|=n}}\rad^a Q_1(\lambda).
\]
In fact, $\{\rad^n Q_1(\lambda)\}_{n\in\Z_{\geq 0}}$ coincides with the radical filtration for the $\F[\un{G}_{/\F}]$-module $Q_1(\lambda)$, though we will not use this. 

If $\sigma$ is a Serre weight, then let $\tld{P}_\sigma$ and $P_\sigma$ denote a $\cO[\rG]$-projective cover and $\F[\rG]$-projective cover of $\sigma$, respectively. 
Then $P_\sigma \cong \tld{P}_\sigma \otimes_{\cO}\F$. 
Recall that $\F[\rG]$ is a Frobenius algebra, so that $P_\sigma$ is isomorphic to the injective envelope of $\sigma$. 

\begin{prop}
\label{prop:Q1proj}
Assume that $p\geq 5$.
If $\lambda \in X_1(\un{T})$ is $1$-deep, then the projective $\rG$-module $P_{F(\lambda)}$ is isomorphic to $Q_1(\lambda)|_{\rG}$. 
{In particular, if $\lambda$ is $4$-deep, the Jordan--H\"older factors of $P_{F(\lambda)}$ are described by Proposition \ref{prop:alg_rep_ext_graph} (with $X_j\in\{A,B,C,D,E,F,G\}$ if $\lambda_j$ is in alcove $B$ and $X_j\in\{A,B,C,D,E,F,G,H,I,J\}$ if $\lambda_j$ is in alcove $A$).}
\end{prop}
\begin{proof}
This is a particular case of \cite[Theorem 4.2.1]{LLLM2}.
{As for the last statement, we note that the alcoves appearing in Figure \ref{AlcoveLabels} are all $p$-bounded.}
\end{proof}

Let $\lambda \in X_1(\un{T})$ be $1$-deep and $\sigma = F(\lambda)$. 
For $a = (a_j)_j \in \Z_{\geq 0}^\cJ$ and $n\in\Z_{\geq 0}$, let $\rad^a P_\sigma$, $\rad^{>a}P_\sigma$, $P^{\widehat{a}}_\sigma$, $\gr^a P_\sigma$ and $\rad^nP_\sigma$ be the subquotients of $P_\sigma$ corresponding to the restrictions $\rad^a Q_1(\lambda)|_{\rG}$, $\rad^{>a}Q_1(\lambda)|_{\rG}$, $Q_1(\lambda)^{\widehat{a}}|_{\rG}$, $\gr^a Q_1(\lambda)|_{\rG}$ and $\rad^n Q_1(\lambda)|_{\rG}$ under a fixed choice of isomorphism in Proposition \ref{prop:Q1proj} (the isomorphism classes of the subquotients do not depend on this choice of the isomorphism). 
Note that if $\sigma$ is $8$-deep then $\{\rad^n P_\sigma\}_{n\in\N}$ coincides with the radical filtration for the $\F[\rG]$-module $P_\sigma$ by \cite[Corollary 4.2.3]{LLLM2} and the analogous discussion for $Q_1(\lambda)$ before Proposition \ref{prop:Q1proj}. 
This description of the radical filtration of $P_\sigma$ also follows from the next result. 

We first introduce some notation.
If $a = (a_j)_{j\in \cJ}$ and $b = (b_j)_{j\in \cJ}$ are tuples with $a_j,b_j\in \Z$ for all $j\in \cJ$, let $a\pm b$ be the tuple $(a_j\pm b_j)_{j\in \cJ}$. 
For an integer $n$ and $i\in \cJ$, let $n_i\in \Z^{\cJ}$ be the tuple with $n_{i,j} = n\delta_{i=j}$. 

\begin{prop}\label{prop:cosoclefilter}
{Assume that $\sigma$ is $8$-deep.}
Let $S \subset \{a\in \Z_{\geq 0}^{\cJ} \mid |a| = n\}$ be a nonempty subset. 
Then 
\[
\rad \sum_{a\in S} \rad^a P_\sigma = \sum_{a\in S} \rad^{>a} P_\sigma. 
\]
\end{prop}
\begin{proof}
Since $\sum_{a\in S} \rad^a P_\sigma/\sum_{a\in S} \rad^{>a} P_\sigma \cong \oplus_{a\in S} \gr^a P_\sigma$, $\rad \sum_{a\in S} \rad^a P_\sigma\subset \sum_{a\in S} \rad^{>a} P_\sigma$. 

To show the reverse inclusion, we need an intermediate result. 
Let $\sigma = F(\lambda)$, $b\in \Z_{\geq 0}^{\cJ}$, and $i\in \cJ$. 
Then 
\begin{equation} \label{eqn:twolayer}
\rad^b P_\sigma/ \rad^{>b+1_i} P_\sigma 
\end{equation}
is the restriction to $\rG$ of 
\[
\rad^{b_i}Q_1(\lambda_i)/\rad^{b_i+2}Q_1(\lambda_i) \otimes \otimes_{j\neq i} \gr^{b_j} Q_1(\lambda). 
\]
By Proposition \ref{prop:filtrations} and \cite[Lemma 4.2.2]{LLLM2}, the radical of \eqref{eqn:twolayer} is $\rad^{>b} P_\sigma/\rad^{>b+1_i} P_\sigma$. 

We now show the reverse inclusion. 
If $a\in S$ and $b>a$, we will show that $\rad^b P_\sigma \subset \rad\, \rad^a P_\sigma$ by reverse induction on $|b|$. 
If $|b| > 6\#\cJ$, then $\rad^b P_\sigma = 0$ and we are done. 
Suppose now that $|b| \leq 6\#\cJ$. 
Since $b>a$, $b-1_i \geq a$ for some $i \in \cJ$. 
Then the kernel of the map $\rad^a P_\sigma\ra \cosoc \,\rad^a P_\sigma$ contains $ \rad^{>b}P_\sigma \subset \rad\,  \rad^a P_\sigma$ by the inductive hypothesis. 

Thus the kernel of the induced map $\rad^{b-1_i} P_\sigma \ra \cosoc \,\rad^a P_\sigma$ factors through $\rad^{b-1_i} P_\sigma/ \rad^{>b} P_\sigma$. 
By the above paragraph, $\rad^b P_\sigma$ is in the kernel of the map $\rad^{b-1_i} P_\sigma \ra \cosoc\, \rad^a P_\sigma$ which gives the desired inclusion. 
\end{proof}

\begin{cor}\label{cor:uniquewedge}
{Assume that $\sigma$ is $8$-deep.}
Let $a = (a_j)_{j\in \cJ}$ with $a_j = 2$ for all $j\in \cJ$. 
Suppose that $\Lambda$ is a multiplicity free $\F[\rG]$-module with cosocle isomorphic to $\sigma$ such that $\JH(\Lambda) \subset \JH(P^{\widehat{a}}_\sigma)$. 
Then $\Lambda$ is isomorphic to a unique quotient of $P^{\widehat{a}}_\sigma$. 
\end{cor}
\begin{proof}
Let $\Lambda$ be as in the statement of the corollary. 
Since $\cosoc\, \Lambda \cong \sigma$, we can and do fix a surjection $P_\sigma \onto \Lambda$. 
We claim that this map factors through $P^{\widehat{a}}_\sigma$. 
Suppose otherwise. 
Then the restriction $\rad^{2_i} P_\sigma \ra \Lambda$ is nonzero for some $i\in \cJ$. 
By Proposition \ref{prop:cosoclefilter}, this implies that $\JH(\gr^{2_i} P_\sigma) \cap \JH(\rad\, \Lambda)\neq \emptyset$. 
This contradicts the assumption that $\Lambda$ is a multiplicity free $\F[\rG]$-module and $\JH(\Lambda) \subset \JH(P^{\widehat{a}}_\sigma)$. 
The uniqueness of the quotient follows from the fact that $P^{\widehat{a}}_\sigma$ is multiplicity free. 
\end{proof}

The following proposition shows that maps between projective envelopes are compatible with the filtrations $\rad^a P_{F(\lambda)}$. 

\begin{prop}\label{prop:filtrationcompatible}
Let $\sigma$ and $\kappa$ be $8$-deep Serre weights and $P_\kappa \ra \rad^a P_\sigma$ be a map whose image is not contained in $\rad^{>a}P_\sigma$. 
Then the image of $\rad^b P_\kappa$ is contained in $\rad^{a+b} P_\sigma$ but not in $\rad^{>a+b}P_\sigma$. 
\end{prop}
\begin{proof}
We proceed by induction with respect to the partial ordering on $b$. 
The case $b_j = 0$ for all $j\in \cJ$ is clear. 
Suppose that the image of $\rad^b P_\kappa$ is contained in $\rad^{a+b} P_\sigma$ but not in $\rad^{>a+b}P_\sigma$. 
Let $i\in\cJ$.
Thus, the alcoves of all of the constituents of $\gr^b P_\kappa$ and $\gr^{a+b}P_\sigma$ coincide.
This implies that $\JH(\gr^{b+1_i} P_\kappa)$ and $\JH(\rad^{a+b} P_\sigma/\rad^{a+b+1_i} P_\sigma)$ are disjoint by alcove considerations in embedding $i$.
Hence the map $\gr^{b+1_i}P_\kappa\ra \rad^{a+b} P_\sigma/(\rad^{a+b+1_i} P_\sigma+\textnormal{Im}(\rad^{>b+1_i}P_\kappa))$ is zero. 
Thus the image of $\rad^{b+1_i}P_\kappa$ and $\rad^{>b+1_i}P_\kappa$ in $\rad^{a+b} P_\sigma/\rad^{a+b+1_i} P_\sigma$ coincide.
By Nakayama's lemma these images are zero so that the image of $\rad^{b+1_i}P_\kappa$ is contained in $\rad^{a+b+1_i} P_\sigma$.

It suffices to show that the induced map $\varphi: \gr^{b+1_i} P_\kappa \ra \gr^{a+b+1_i} P_\sigma$ is nonzero. 
In fact, by \cite[Lemma 4.2.2]{LLLM2} this map is the socle of the the map 
\begin{equation}\label{eqn:2steps}
\rad^b P_\kappa/(\rad^{>b+1_i}P_\kappa + \sum_{j\neq i} \rad^{b+1_j}P_\kappa) \ra \rad^{a+b} P_\sigma/(\rad^{>a+b+1_i}P_\sigma + \sum_{j\neq i} \rad^{a+b+1_j}P_\sigma), 
\end{equation}
which is nonzero by the inductive hypothesis. 
If $\varphi=0$, then \eqref{eqn:2steps} factors through a map $\gr^b P_\kappa \ra \gr^{a+b+1_i} P_\sigma$ which must be $0$ by alcove considerations at embedding $i$. 
This is a contradiction. 
\end{proof}

\subsection{The covering property}\label{sec:covering}

For $\lambda \in X_1(\un{T})$, let $A(\lambda)$ be the set $\{j\in \cJ \mid \lambda_j \in A\}$. 
If $\sigma \cong F(\lambda)$, then we define $A(\sigma)$ to be $A(\lambda)$. 

Fix a tame $L$-parameter $\rhobar$ with a lowest alcove presentation $(s,\mu)$ for it. 
The main result of this section is the following result which plays a key role in \S \ref{subsec:proj_and_K1}. 

\begin{prop}\label{prop:W?-cover}
Let $\sigma \in W^?(\rhobar)$ {be $8$-deep.}
Fix $a \in \Z_{\geq 0}^\cJ$ and let $N \subset \rad^a P_\sigma$ be a submodule such that the cokernel of the induced map $N \ra \gr^a P_\sigma$ has no Jordan--H\"older factors in $W(\rhobar)$. 
Then no Jordan--H\"older factors of $\rad^a P_\sigma/N$ are in $W(\rhobar)$. 
\end{prop}

The proof of Proposition \ref{prop:W?-cover} follows inductively from the following lemma. 

\begin{lemma}\label{lemma:W?-cover}
Let $\sigma \in W^?(\rhobar)$ {be $8$-deep.} 
Fix $a,b \in \Z_{\geq 0}^\cJ$ such that for some $i\in \cJ$, $b_i = a_i +1$ and $b_j = a_j$ for all $j \neq i$. 
Let $N \subset \rad^a P_\sigma/\rad^{>b} P_\sigma$ be a submodule such that the cokernel of the induced map $N \ra \gr^a P_\sigma$ has no Jordan--H\"older factors in $W(\rhobar)$. 
Then no Jordan--H\"older factors of $(\rad^a P_\sigma/\rad^{>b} P_\sigma)/N$ are in $W(\rhobar)$. 
\end{lemma}

\begin{proof}[Proof of Proposition \ref{prop:W?-cover}]
We proceed by induction. 
For convenience, for any submodule $M \subset \rad^a P_\sigma$, we write $\ovl{M}$ for the image of $M$ in $\rad^a P_\sigma/N$. 
Similarly, for submodules $M'\subset M \subset \rad^a P_\sigma$, we write $\ovl{M/M'}$ for $\ovl{M}/\ovl{M'}$.
Proposition \ref{prop:W?-cover} holds for $a$ sufficiently large. 
Let $a \in \Z_{\geq 0}^\cJ$ and suppose that Proposition \ref{prop:W?-cover} holds for any $b \in \Z_{\geq 0}^\cJ$ as in Lemma \ref{lemma:W?-cover}. 
Suppose that $N$ is as in Proposition \ref{prop:W?-cover}. 
By the exact sequence
\[
0\ra \sum_{\substack{b>a\\ \text{$b$ as in L.~\ref{lemma:W?-cover}}}}\ovl{\rad^b P_\sigma} \ra \ovl{\rad^a P_\sigma} \ra \ovl{\gr^a P_\sigma} \ra 0, 
\]
it suffices to show that for any $b \in \Z_{\geq 0}^\cJ$ as in Lemma \ref{lemma:W?-cover}, $\ovl{\rad^b P_\sigma}$ contains no Jordan--H\"older factors in $W^?(\rhobar)$. 

Let $b \in \Z_{\geq 0}^\cJ$ be as in the above exact sequence, i.e., as in Lemma \ref{lemma:W?-cover}. 
By Lemma \ref{lemma:W?-cover}, %
$\ovl{\rad^a P_\sigma/\rad^{>b} P_\sigma}$ contains no Jordan--H\"older factors in $W^?(\rhobar)$. 
Thus $\ovl{\gr^b P_\sigma} \subset \ovl{\rad^a P_\sigma/\rad^{>b} P_\sigma}$ contains no Jordan--H\"older factors in $W^?(\rhobar)$. 
By the inductive hypothesis, $\ovl{\rad^b P_\sigma}$ contains no Jordan--H\"older factors in $W^?(\rhobar)$. 
\end{proof}

The proof of Lemma \ref{lemma:W?-cover} requires a series of results. 
Let $\lambda = (\lambda_j)_j \in X_1(\un{T})$ be {$8$-deep} in alcove $A^{\cJ}$. 
If $X = (X_j)_j$ is a dominant alcove, then let $\lambda_X \defeq \sum_j \lambda_{j,X_j}$ (using notation from \S \ref{sec:modrep}).
In particular we have $\lambda_{j,X_j}=\tld{w}_{X_j}\cdot \lambda_j=\lambda_{j,X_j}^0+p\omega_{X_j}$ where $\lambda_{j,X_j}^0\defeq \tld{w}^0_{X_j}\cdot \lambda_j\in X_1^*(T)$.
For the remainder of this section, $M = M_i \otimes \otimes_{j\neq i} L(\lambda_{j,X_j})$ where $X_j \in \{A,B,C,D,E,F,G,H,I,J\}$ for all $j\neq i$ and $M_i$ is a rigid module with Loewy length two. 
Then $M$ is rigid of Loewy length two and we let $M^1$ and $M^0$ be the cosocle and socle, respectively. 
Unless otherwise stated, $M_i$ is a nonsplit extension of $L(\lambda_{i,X^1_i})$ by $L(\lambda_{i,X^0_i})$. 
We will require the following lemmas. 

\begin{lemma}\label{lemma:alcoveconsider}
Let $\sigma \in \JH(M^1|_{\rG})$. 
Let $N \subset N' \subset M|_{\rG}$ be submodules such that the cokernel of the natural map $N \ra \cosoc\, N'$ does not contain $\sigma$ as a Jordan--H\"older factor. 
Then $N'/N$ does not contain $\sigma$ as a Jordan--H\"older factor. 
\end{lemma}
\begin{proof}
If $\varphi: N \ra \cosoc\, N'$ is the natural map, there is an exact sequence 
\[
0 \ra N'' \ra N'/N \ra \coker \varphi \ra 0 
\]
for some subquotient $N''$ of $M^0|_{\rG}$. 
Since $\sigma\notin \JH(M^0|_{\rG})$ by alcove considerations, and $\sigma\notin \JH(\coker \varphi)$ by assumption, $\sigma \notin \JH(N'/N)$. 
\end{proof}

\begin{lemma}
\label{lem:reduction_pres}
There exist a finite set $S$, $6$-deep weights $\lambda^k\in X_1(\un{T})$ in alcove $A^{\cJ}$ for each $k\in S$ with $\lambda^k_{\pi(i)} = \lambda_{\pi(i)}$, and alcoves $Y_j \in \{A,B\}$ for each $j\neq i$ such that with $M_{k,i}$ a nonsplit extension of $L(\lambda^k_{i,X^1_i})$ by $L(\lambda^k_{i,X^0_i})$ and $M_k=M_{k,i} \otimes \otimes_{j\neq i} L(\lambda^k_{j,Y_j})$ for each $k\in S$, $M|_{\rG}\cong\oplus_{k\in S}{M_k}|_{\rG}$. 
\end{lemma}
\begin{proof}
See the proof of \cite[Proposition 4.2.10]{LLLM2}.
\end{proof}

\begin{prop}\label{prop:weylext}
Suppose that $X^1_i,X^0_i\in \{A,B,C,D,E,F,G\}$ or 
\[
(X^1_i,X^0_i) \in \{(C,J), (D,J), (J,C), (J,D), (E,H), (F,I), (H,E), (I,F)\}.
\]
Let $\sigma^n$ be a Jordan--H\"older factor in $M^n|_{\rG}$ with multiplicity one for $n = 0$ and $1$ and fix nonzero maps (unique up to scalar) $P_{\sigma^1} \ra M|_{\rG}$ and $M|_{\rG}\ra P_{\sigma^0}$.
Then the composition $P_{\sigma^1} \ra M|_{\rG}\ra P_{\sigma^0}$ is nonzero if and only if $\Ext^1_{\F[\rG]}(\sigma^1,\sigma^0) \neq 0$.
\end{prop}
\begin{proof}
This follows from the proofs of \cite[Propositions 4.2.10 and 4.2.12]{LLLM2}. 
Indeed, if $X^1_i,X^0_i\in \{A,B,C,D,E,F,G\}$, then the result follows from \cite[Proposition 4.2.10]{LLLM2}. 
We will consider the cases $(X^1_i,X^0_i) \in \{(J,C), (J,D), (H,E), (I,F)\}$. 
The remaining cases follow by duality. 
One reduces to the cases where $X_j\in\{A,B\}$ for $j\neq i$ by Lemma \ref{lem:reduction_pres} (see the proof of \cite[Proposition 4.2.10]{LLLM2}).
The cases $(X^1_i,X^0_i)=(J,C)$ or $(J,D)$ are similar to the cases $(G,E)$ and $(G,F)$ covered in \cite[Proposition 4.2.12]{LLLM2}. 
The cases $(X^1_i,X^0_i)$ is $(H,E)$ or $(I,F)$ are also similar to the cases just mentioned. 
Indeed, in the last step of the proof of \cite[Proposition 4.2.12]{LLLM2} one uses instead that $L(2\omega_{X_i^1})$ is isomorphic to the submodule of $L(\omega_{X_i^1})\otimes L(\omega_{X_i^1})$ on which the involution given by switching tensor factors acts by $1$. 
\end{proof}

\begin{lemma}\label{lemma:directsum}
If $\{X^1_i,X^0_i\} = \{G,J\}$, then $M|_{\rG}$ is a direct sum of indecomposable length two modules.
\end{lemma}
\begin{proof}
Decompose $M^1|_{\rG} = \oplus_k \sigma_k$ (we identify $\sigma_k$ with a submodule of $M^1|_{\rG}$ even though various $\sigma_k$ may be isomorphic). 
Let $N_k \subset M|_{\rG}$ be the image of a map $P_{\sigma_k} \ra M|_{\rG}$ lifting a projective cover $P_{\sigma_k} \onto \sigma_k\subset M^1|_{\rG}$ of $\sigma_k$. 
For each $k$, the length of $N_k$ is at least two since $\sigma_k \notin \JH(\soc\, M|_{\rG})$ by \cite[Lemma 4.2.2]{LLLM2}.
Moreover, since there is a unique $\sigma'_k\in \JH(M^0|_{\rG})$ such that $\Ext^1_{\F[\rG]}(\sigma_k,\sigma'_k)$ is nonzero and $\dim_{\F} \Ext^1_{\F[\rG]}(\sigma_k,\sigma'_k) = 1$ by \cite[Lemma 4.2.6]{LLLM2}, the length of $N_k$ is exactly two. 
The natural map $\oplus_k N_k \ra M|_{\rG}$ is a surjective map between objects of the same length, and is thus an isomorphism.
\end{proof}

We now let $M^A$ and $M^B$ be $L(\lambda_{i,A}) \otimes \otimes_{j\neq i} L(\lambda_{j,X_j})$ and $L(\lambda_{i,B}) \otimes \otimes_{j\neq i} L(\lambda_{j,X_j})$, respectively.
Fix $\sigma \in \JH(M^A|_{\rG})$.

\begin{prop} \label{prop:mult2}
Suppose $(X^1_i,X^0_i) = (G,E)$ or $(G,F)$.
If $N$ is a submodule of $M|_{\rG}$ with $\sigma\in \JH(N)$, then $N$ contains at least two weights in $M^0|_{\rG}$ adjacent to $\sigma$.
\end{prop}
\begin{proof}
Let $N$ be as in the statement.
The projection of $N$ to some ${M_{k}}|_{\rG}$ in Lemma \ref{lem:reduction_pres} contains $\sigma$ as a Jordan--H\"older factor, and so we reduce to the case where $X_j\in\{A,B\}$ for all $j\neq i$.
If $N$ contains only one weight in $M^0|_{\rG}$ adjacent to $\sigma$, then there is a map
$\varphi: P_\sigma\ra N\subset M|_{\rG}$ such that $\textnormal{Im}(\varphi)\cap \soc(M|_{\rG})$ is simple.
Considering the composition of $\varphi$ with the injective envelope of $M|_{\rG}$, the argument of  \cite[Proposition 4.2.12]{LLLM2} implies that there is a nonzero element of the $0$-weight space of $L(\omega_{X_i^0}-w_0\omega_{X_i^0})$ whose image in $L(\omega_{X_i^0}) \otimes L(-w_0\omega_{X_i^0})$ is a pure tensor of weight eigenvectors.
This contradicts Lemma \ref{lemma:2adj}.
\end{proof}

\begin{lemma}\label{lemma:2adj}
Let $\omega \in X^*(T)$ be a fundamental weight. 
Under the inclusion $L(\omega-w_0\omega) \subset L(\omega) \otimes L(-w_0\omega)$, any element in the $0$ weight space of $L(\omega-w_0\omega)$ is not a pure tensor of weight eigenvectors in $L(\omega) \otimes L(-w_0\omega)$.
\end{lemma}
\begin{proof}
The ($2$-dimensional) $0$-weight space of $L(\omega-w_0\omega)$ is stable under the Weyl group symmetry.
However, the Weyl group orbit of a weight $0$ pure tensor of (nonzero) weight eigenvectors in $L(\omega) \otimes L(-w_0\omega)$ spans the entire ($3$-dimensional) $0$-weight space of $L(\omega) \otimes L(-w_0\omega)$.
\end{proof}

\begin{prop} \label{prop:doubleadj}
Suppose that $(X^1_i,X^0_i) = (E,G)$ or $(F,G)$.
If $N$ is a submodule of $M|_{\rG}$ containing two distinct Serre weights adjacent to $\sigma$ as Jordan--H\"older factors, then $N$ contains the $\sigma$-isotypic part of $M^0|_{\rG}$.
\end{prop}
\begin{proof}
Suppose that $N$ is as above. 
Define $N'$ be the kernel of natural surjection $(M|_{\rG})^* \surj N^*$, where $(-)^*$ denotes the contragredient representation.
If $\sigma^*\in \JH(N')$, then two of the weights adjacent to $\sigma^*$ are Jordan--H\"older factors of $N'$ by Proposition \ref{prop:mult2}, and thus not of $N^*$  since $M^1|_{\rG}$ is a multiplicity free representation.
This contradicts the assumption that $N$ contains two distinct Serre weights adjacent to $\sigma$ as Jordan--H\"older factors.
Thus $\sigma^*\notin \JH(N')$ and $N$ contains the $\sigma$-isotypic part of $M^0|_{\rG}$.
\end{proof}

Recall our standing assumption that $\lambda$ is {8-deep} in alcove $A^{\cJ}$.

\begin{prop}\label{prop:extcover}
Suppose that either
\begin{enumerate}
\item \label{item:AB} at least one of $X^1_i$ or $X^0_i$ is in $\{A,B\}$; 
\item \label{item:DE} $(X^1_i,X^0_i) = (C,F)$ or $(D,E)$; or 
\item \label{item:EF} $X^1_i = E$ or $F$. 
\end{enumerate}
Assume further that $F(\lambda)$ is linked to a modular Serre weight, or equivalently, that $F(\sum_j \lambda_{j,B}) \in W^?(\rhobar)$. 
If $X^1_i = A$, assume further that $F(\lambda_{i,A} + \sum_{j\neq i} \lambda_{j,B}) \in W^?(\rhobar)$. 
Suppose that $N$ is a submodule of $M|_{\rG}$ such that the cokernel of the projection of $N$ onto $M^1|_{\rG}$ contains no Serre weights in $W^?(\rhobar)$.
Then $M|_{\rG}/N$ contains no Jordan--H\"older factors in $W^?(\rhobar)$. 
\end{prop}
\begin{proof}
In this proof we repeatedly use the description of $W^?(\rhobar)$ in terms of the extension graph given in \S \ref{subsubsec:combinatorics_TandW}.
With $M|_{\rG} \cong \oplus_{k \in S}M_k|_{\rG}$ as in Lemma \ref{lem:reduction_pres}, $M|_{\rG}/N$ is a quotient of $\oplus_{k \in S} M_k|_{\rG}/(M_k|_{\rG} \cap N)$.  
It suffices to show the following claim: for each $k\in S$, $M_k|_{\rG}/(M_k|_{\rG} \cap N)$ contains no Jordan--H\"older factors in $W^?(\rhobar)$. 
Note that $M_k$ has simple cosocle as a $\un{G}_{/\F}$-module, whose ($6$-deep) highest weight is linked to the weight $\lambda^k\in A^{\cJ}$ appearing in Lemma \ref{lem:reduction_pres}.
Note that if $\lambda=\Trns_{\mu+\eta}(\nu,\un{0})$ then $\lambda^k=\Trns_{\mu+\eta}(\nu+\omega,\un{0})$ for some $\omega \in L(\sum_{j\neq i}\pi\omega_{X_j})$ by Proposition \ref{prop:alg_rep_ext_graph}.
Thus if $F(\lambda^k)$ is not linked to a modular weight then $(\nu_{\pi(j)}+\omega_{\pi(j)},1)$ is not in $s_{\pi(j)}(r(\Sigma_0))$ for some $j\neq i$.
Then $\JH(M_k|_{\rG}) \cap W^?(\rhobar)$ is empty, and the claim holds. 
Assume now that $F(\lambda^k)$ is linked to a modular weight. 
We will show that the remaining hypotheses of the proposition hold for $M_k|_{\rG} \cap N \subset M_k|_{\rG}$ so that we reduce to the case where $X_j\in \{A,B\}$ for $j\neq i$. 
Clearly, the enumerated condition that holds for $M$ holds for any $M_k$. 
It remains to show that the cokernel of the natural map $M_k|_{\rG} \cap N \ra M_k^1|_{\rG} \defeq \cosoc\, M_k|_{\rG}$ contains no Jordan--H\"older factors in $W^?(\rhobar)$. 
Let $\sigma \in \JH(M_k^1|_{\rG}) \cap W^?(\rhobar)$. 
By Lemma \ref{lemma:alcoveconsider} taking $N' = M|_{\rG}$, $\sigma\notin \JH(M|_{\rG}/N)$. 
Thus, in fact, $\sigma\notin \JH(M_k|_{\rG}/(M_k|_{\rG} \cap N))$. 

We now establish the proposition assuming that $X_j\in \{A,B\}$ for $j\neq i$. 
We assume that $F(\lambda_{i,B}+\sum_{j\neq i} \lambda_{j,X_j}) \in W^?(\rhobar)$ since otherwise $\JH(M|_{\rG}) \cap W^?(\rhobar) = \emptyset$, and the proposition follows immediately. 
Our analysis now breaks into the three cases for $X^0_i$ and $X^1_i$ described in the enumerated conditions of the proposition.
If $X^1_i$ or $X^0_i$ is $A$ or $B$, then it is easy to see that $N$ necessarily contains $M^0|_{\rG}$, and the proposition follows. 
(Note that if $X^1_i=A$, then $F(\lambda_{i,A}+\sum_{j\neq i} \lambda_{j,X_j}) \in W^?(\rhobar)$.) 

We next consider case where \eqref{item:DE} holds. 
We assume that $(X^1_i,X^0_i) = (C,F)$ as the case $(D,E)$ is symmetric. 
Suppose $\sigma^0 \subset M^0|_{\rG}$ is simple and in $W^?(\rhobar)$. 
Then one checks that there is a unique $\sigma^1\in \JH(M^1|_{\rG}) \cap W^?(\rhobar)$ and moreover that $\sigma^1$ is adjacent to $\sigma^0$. 
Then $\sigma^0\subset N$ by Proposition \ref{prop:weylext}. 

Finally, we consider the case where \eqref{item:EF} holds. 
We assume that $X^1_i = E$ as the case $X^1_i = F$ is symmetric. 
Then $X_i^0$ is $A,C,D,G,H,$ or $I$. 
The case where $X_i^0=A$ was already considered in \eqref{item:AB}. 
Let $\sigma^0 \subset M^0|_{\rG}$ be simple and in $W^?(\rhobar)$. 
Since $F(\lambda_{i,B}+\sum_{j\neq i} \lambda_{j,X_j}) \in W^?(\rhobar)$ by assumption, it is easy to check that there is $\sigma^1\in \JH(M^1|_{\rG}) \cap W^?(\rhobar)$. 
Suppose first that $X_i^0$ is $D$ or $H$. 
Then $\sigma^0$ and $\sigma^1$ must be adjacent. (The contrapositive is easier to see: If $\sigma^0 = \Trns_{\mu+\eta}(s_j(\omega^0),a^0)$ and $\sigma^1 = \Trns_{\mu+\eta}(s_j(\omega^1),a^1)$ are not adjacent, then $\omega^0_i - \omega^1_i$ is an element of $\Lambda_R$ using that $(\omega^0_j,a^0_j) =(\omega^1_j,a^1_j)$ for all $j\neq i$. This implies that $\omega_E-\omega_X \in \Lambda_R$ for $X = D$ or $H$ which is false.) 
The result then follows from Proposition \ref{prop:weylext}.
Suppose next that $X_i^0$ is $C$ or $I$.
If $\sigma^0$ and $\sigma^1$ are not adjacent, then the weight in $\JH(M^1|_{\rG})$ linked to $\sigma^0$ is also modular and the result follows again from Proposition \ref{prop:weylext} (see Figure \ref{fig-linked}). 
\begin{figure}[htb] 
\centering
\includegraphics[scale=.4]{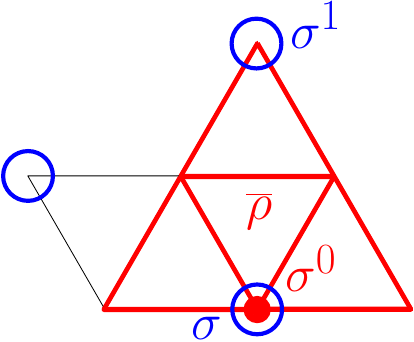}
\caption{$\sigma^0$ and $\sigma^1$ are nonadjacent weights in $W^?(\rhobar)$. The blue circles represent $\JH(M^1|_{\rG})$. 
The weight $\sigma$ linked to $\sigma^0$ is in $\JH(M^1|_{\rG}) \cap W^?(\rhobar)$. }
\label{fig-linked}
\end{figure}

Suppose finally that $X_i^0$ is $G$.
If $\sigma^0 \in \JH(M^A|_{\rG})$ (resp.~$\JH(M^0|_{\rG}) \setminus \JH(M^A|_{\rG})$), then $\sigma^0$ is adjacent to exactly two modular weights (resp.~one modular weight) of $\JH(M^1|_{\rG})$.
The result then follows from Proposition \ref{prop:doubleadj} (resp.~Proposition \ref{prop:weylext}).
\end{proof}

\begin{prop}\label{prop:genJ}
Let $M_i$ be a (unique up to isomorphism) rigid nonsplit extension of $L(\lambda_{C_i}) \oplus L(\lambda_{G_i})$ by $L(\lambda_{J_i})$.
Assume further that $F(\lambda)$ is linked to a modular Serre weight $F(\lambda')$ with $\lambda_i'= \lambda_i$.
If $N$ is a submodule of $M|_{\rG}$ such that the cokernel of the projection of $N$ onto $M^1|_{\rG}$ contains no modular Serre weights, then $M|_{\rG}/N$ contains no modular Serre weights (as Jordan--H\"older factors).
\end{prop}
\begin{proof}
Suppose that $\sigma^0$ is a weight in $W^?(\rhobar) \cap (\JH(M^0|_{\rG}) \setminus \JH(M^B|_{\rG}))$. 
Then, using that $F(\lambda)$ is linked to a modular Serre weight, $L(\lambda_{C_i})\otimes \otimes_{j\neq i} L(\lambda_{j,X_j})|_{\rG}$ contains a modular weight adjacent to $\sigma^0$, 
and $N$ contains $\sigma^0$ by Proposition \ref{prop:weylext} with $(X^1_i,X^0_i) = (C,J)$.
If $\sigma^0$ is a weight in $W^?(\rhobar) \cap \JH(M^B|_{\rG})$, then the Jordan--H\"older factor of $M^A|_{\rG}$ adjacent to $\sigma^0$ is modular (by the assumption that $\lambda_i' = \lambda_i$) and appears with multiplicity two in $M^1|_{\rG}$. 
Then $\sigma^0$ appears in $N$ with multiplicity two by Lemma \ref{lemma:directsum}. 
\end{proof}

\begin{proof}[Proof of Lemma \ref{lemma:W?-cover}]
Let $\sigma$, $a$, $b$, $i$, and $N$ be as in Lemma \ref{lemma:W?-cover}. 
For convenience, for any submodule $M \subset P_\sigma/\rad^{>b}P_\sigma$, we write $\ovl{M}$ for the image of $M$ in $(P_\sigma/\rad^{>b}P_\sigma)/N$. 
Similarly, if $M'\subset M \subset P_\sigma/\rad^{>b}P_\sigma$ are submodules, then we write $\ovl{M/M'}$ for $\ovl{M}/\ovl{M'}$. 
Let $\sigma$ be $F(\mu)$ so that $\rad^a P_\sigma/\rad^{>b} P_\sigma$ is isomorphic to $\rad^a Q_1(\mu)/\rad^{>b} Q_1(\mu)|_{\rG}$ where
\begin{equation}\label{eqn:2layer}
\rad^a Q_1(\mu)/\rad^{>b} Q_1(\mu) \cong \rad^{a_i} Q_1(\mu_i)/\rad^{a_i+2} Q_1(\mu_i) \otimes \otimes_{j\neq i} \gr^{a_j} Q_1(\mu_j). 
\end{equation}
We will show that $\JH(\ovl{\rad^a P_\sigma/\rad^{>b}P_\sigma}) \cap W^?(\rhobar)$ is empty. 
Since $\JH(\ovl{\gr^a P_\sigma}) \cap W^?(\rhobar)$ is empty by assumption, it suffices to show that $\JH(\ovl{\gr^b P_\sigma}) \cap W^?(\rhobar)$ is empty. 
Let $\sigma^0$ be in $\JH(\gr^b P_\sigma)\cap W^?(\rhobar)$. 
We claim that there is a finite, separated, and exhaustive increasing filtration $F^k \gr^b Q_1(\mu)$ on $\gr^b Q_1(\mu)$ and submodules 
\begin{equation} \label{eqn:Mk}
M_k \subset (\rad^a Q_1(\mu)/\rad^{>b} Q_1(\mu))/F^{k-1} \gr^b Q_1(\mu) 
\end{equation}
for each $k$ such that 
\begin{enumerate}
\item \label{item:filtercover} $F^k \gr^b Q_1(\mu)/F^{k-1} \gr^b Q_1(\mu) \subset M_k$; and
\item \label{item:edge} $\sigma^0 \notin \JH(\ovl{M_k|_{\rG}})$.  
\end{enumerate} 
Admitting this claim, we can %
inductively show for all $k$ that $\sigma^0 \notin \JH(\ovl{F^k \gr^b Q_1(\mu)|_{\rG}})$. 
For $k$ sufficiently small, $F^k \gr^b Q_1(\mu) = 0$ and we are done. 
If $\sigma^0 \notin \JH(\ovl{F^{k-1} \gr^b Q_1(\mu)|_{\rG}})$ for some $k$, then $\sigma^0 \notin \JH(\ovl{F^k \gr^b Q_1(\mu)|_{\rG}})$ since 
\[
\ovl{F^k \gr^b Q_1(\mu)/F^{k-1} \gr^b Q_1(\mu)|_{\rG}} \subset \ovl{M_k|_{\rG}}
\]
by \eqref{item:filtercover} and $\sigma^0 \notin \JH(\ovl{M_k|_{\rG}})$ by \eqref{item:edge}. 

It suffices to prove the claim. 
We set $F^0 \gr^b Q_1(\mu) = M_0 = 0$ and $F^1 \gr^b Q_1(\mu) = M_1$ to be the largest submodule $M \subset \gr^b Q_1(\mu)$ such $\sigma^0\notin \JH(M|_{\rG})$. 
Now, by Proposition \ref{prop:filtrations}, we can recursively find $M_k$ as in \eqref{eqn:Mk} to be a subquotient of a graded piece in either the Weyl filtration or the dual Weyl filtration of the form 
\[
M_k = M_{i,k} \otimes \otimes_{j\neq i} L(\lambda_{j,X_j}) 
\]
where $\lambda$ is the unique element in $A^{\cJ} \subset X^*(\un{T})$ linked to $\mu$ and 
\begin{itemize}
\item $\soc\, M_{i,k} \cong L(\lambda_{i,X_i^0})$ and $\cosoc\, M_{k,i} \cong L(\lambda_{i,X_i^1})$ with $(X_i^1,X_i^0)$ in the appropriate column of Table \ref{table:ext_Q1:B} or \ref{table:ext_Q1:A}; or 
\item $a_i=2$, $\mu_i \in A$, $\soc\, M_{i,k} \cong L(\lambda_{i,J})$, and $\cosoc\, M_{k,i} \cong L(\lambda_{i,C}) \oplus L(\lambda_{i,G})$ (as indicated by $(C\oplus G,J)$ in Table \ref{table:ext_Q1:A}). 
\end{itemize}
We define $F^k \gr^b Q_1(\mu)$ to be the preimage of $\soc\, M_k$. 
Then \eqref{item:filtercover} holds by construction, and it suffices to check \eqref{item:edge} for each $M_k$. 
This is clear for $k=0$ and $1$ and follows from Propositions \ref{prop:extcover} and \ref{prop:genJ} for $k>1$. 
\begin{table}[htb]
\footnotesize
\caption{\textbf{Subquotients of $\rad^{b_i-1}Q_1(\lambda_i)/\rad^{b_i+1}Q_1(\mu_i)$, $\mu_i \in B$}}
\label{table:ext_Q1:B}
\centering
\adjustbox{max width=\textwidth}{
\begin{tabular}{| c | c | c | c | c | c| c |}
\hline
$a_i$&0&1&2&3&4&5
\\
\hline
&&&&&&\\
Weyl&&$(C,B),\ (D,B)$&$\begin{aligned} 
&(E,C),\ (E,A),\ (E,D)\\
&(F,C),\ (F,A),\ (F,D)
\end{aligned}$&$(C,B)\times 2$&$(E,A)$&$(C,B)$\\
&&&&&&\\
\hline
&&&&&&\\
Dual Weyl&$(B,C),\ (B,A),\ (B,D)$&$(D,E),\ (C,F)$& $(E,G)$&$(D,E),\ (C,F)$&$(B,C),\ (B,D)$&
\\
&&&&&&\\
\hline
\end{tabular}
}
\caption*{This table records length two subquotients of $\rad^{b_i-1}Q_1(\lambda_i)/\rad^{b_i+1}Q_1(\mu_i)$ for $\mu_i\in B$ which are subquotients of either a Weyl or dual Weyl module in the associated graded of the Weyl or dual Weyl filtration. 
The notation $(X^1,X^0)$ denotes a nonsplit extension of $X^1$ by $X^0$ (unique up to isomorphism).}\end{table}
\begin{table}[htb]
\footnotesize
\caption{\textbf{Subquotients of $\rad^{b_i-1}Q_1(\lambda_i)/\rad^{b_i+1}Q_1(\mu_i)$, $\mu_i \in A$}}
\label{table:ext_Q1:A}
\centering
\adjustbox{max width=\textwidth}{
\begin{tabular}{| c | c | c | c | c | c| c |}
\hline
$a_i$&0&1&2&3&4&5
\\
\hline
&&&&&&\\
Weyl&&$(E,A),\ (B,A),\ (F,A)$&$(A,B)\times 2$&$(E,A)\times 2,\ (F,A)$&&$(E,A)$\\
&&&&&&\\
\hline
&&&&&&\\
Dual Weyl&$(A,E),\ (A,F)$&$\begin{aligned} 
&(B,D),\ (B,C),\ (F,D),\ (E,C)\\
&(F,I),\ (F,G),\ (E,H)
\end{aligned}$& $\begin{aligned} 
&(A,E)\times 2,\ (A,F)\times 2\\
&(C\oplus G,J)
\end{aligned}$&$\begin{aligned} 
&(F,I),\ (F,G),\ (E,H)&\\
&(B,D)\times 2,\ (B,C)\times 2
\end{aligned}$&$(A,F),\ (A,B),\ (A,E)$&
\\
&&&&&&\\
\hline
&&&&&&\\
Other&$(A,B)$&&&&&
\\
&&&&&&\\
\hline
\end{tabular}
}
\caption*{This table records subquotients of $\rad^{b_i-1}Q_1(\lambda_i)/\rad^{b_i+1}Q_1(\mu_i)$ for $\mu_i\in A$ as in Table \ref{table:ext_Q1:B}. The nonsplit extension $(A,B)$ in the last row is the quotient $Q_1(\lambda_i)^B$ though it is not a quotient of either the quotient Weyl or dual Weyl module. $(C\oplus G,J)$ denotes the length $3$ rigid nonsplit extension (unique up to isomorphism) of the direct sum of $C$ and $G$ by $J$. }
\end{table}
\end{proof}

\subsection{Presentations of some quotients of $P_\sigma$}

In this section, we introduce some quotients of $P_\sigma$ and give presentations which will be used in \S \ref{subsec:proj_and_K1}. 

Let $a = (a_j)_{j\in \cJ}$ be a tuple such that $a_j$ is a subset of $\{B,E_o,E_s,F_o,F_s\}$ for $j\in A(\sigma)$ (which we write simply as a string in any order) and $a_j = \widehat{1}$ or $\widehat{2}$ if $j\notin A(\sigma)$. 
Let $b = (b_j)_{j\in \cJ}$ be a tuple such that $b_j = \widehat{2}$ if $j\in A(\sigma)$ and $b_j = a_j$ otherwise. 
We now define a quotient $P_\sigma^a$ of $P_\sigma^b$. 

Write $\sigma=F(\lambda)$ {for a $4$-deep weight $\lambda\in X_1(\un{T})$} so that $P_\sigma^b$ is isomorphic to $Q_1(\lambda)^b|_{\rG}$ {and Proposition \ref{prop:Q1proj} applies.} 
For each $i\in A(\sigma)$, $Q_1(\lambda)^b$ has a submodule 
\begin{align*}
\gr^1 Q_1(\lambda_i) \otimes \otimes_{j\neq i} Q_1(\lambda_j)^{b_j} &\cong (L(\lambda_{i,B}) \oplus L(\lambda_{i,E}) \oplus L(\lambda_{i,F})) \otimes \otimes_{j\neq i} Q_1(\lambda_j)^{b_j} \\
&\cong (L(\lambda^0_{i,B}) \otimes L(p\omega_{i,B}) \oplus L(\lambda^0_{i,E}) \otimes L(p\omega_{i,E})\oplus L(\lambda^0_{i,F})\otimes L(p\omega_{i,F})) \otimes \otimes_{j\neq i} Q_1(\lambda_j)^{b_j} 
\end{align*}
(recall from \S \ref{sec:modrep} that $\lambda_{i,X}=\lambda_{i,X}^0+p\omega_{i,X}$  with $\lambda^0_{i,X} \in X_1(T)$ for $X = B,E,F$). 
Since 
\begin{align*}
&(L(\lambda^0_{i,B}) \otimes L(p\omega_{i,B}) \oplus L(\lambda^0_{i,E}) \otimes L(p\omega_{i,E})\oplus L(\lambda^0_{i,F})\otimes L(p\omega_{i,F})) \otimes \otimes_{j\neq i} Q_1(\lambda_j)^{b_j}|_{\rG} \\
\cong &(L(\lambda^0_{i,B}) \otimes L(\pi\omega_{i,B}) \oplus L(\lambda^0_{i,E}) \otimes L(\pi\omega_{i,E})\oplus L(\lambda^0_{i,F})\otimes L(\pi\omega_{i,F})) \otimes \otimes_{j\neq i} Q_1(\lambda_j)^{b_j}|_{\rG},
\end{align*}
for each $X \in \{B,E,F\}$ and torus weight $\eps_{i,X}$ in $L(\omega_{i,X})$ the translation principle (Proposition \ref{prop:TP}) produces submodules 
\begin{equation}\label{eqn:submodules1}
L(\lambda^0_{i,X}) \otimes Q_1(\lambda_{\pi(i)} + \pi \eps_{i,X})^{b_{\pi(i)}} \otimes \otimes_{j\neq i,\pi(i)} Q_1(\lambda_j)^{b_j}|_{\rG}
\end{equation}
of $P_\sigma^b$ if $\pi(i) \neq i$ and 
\begin{equation}\label{eqn:submodules2}
L(\lambda^0_{i,X}+ \eps_{i,X}) \otimes \otimes_{j\neq i} Q_1(\lambda_j)^{b_j}|_{\rG}
\end{equation}
if $\pi(i) =i$. 
For each $j\in A(\sigma)$, let $X_j \in \{B,E,F\}$, and let 
\[
\lambda^0 = \sum_{j\notin A(\sigma)} \lambda_j + \sum_{j\in A(\sigma)} \lambda^0_{j,X_j}. 
\] 
For $j\in A(\sigma)$, if $X_j = B$ (resp.~$E$ and $F$), let $c_j = B$ (resp.~$c_j \in \{E_o,E_s\}$  and $c_j \in \{F_o,F_s\}$). 
Then for each $j\in A(\sigma)$, there is a unique torus weight $\eps_{j,c_j}$ in $L(\omega_{j,X_j})$ such that if $c_j = E_o$ or $F_o$ (resp.~ $E_s$ or $F_s$), then $F(\lambda^0+\sum_{j\in A(\sigma)} \pi\eps_{j,c_j})$ is obvious (resp.~shadow) in embedding $\pi(j)$. 
Finally, we let $P_\sigma^a$ be the quotient of $P_\sigma^b$ by the sum of the submodules in \eqref{eqn:submodules1} or \eqref{eqn:submodules2} with $i\in A(\sigma)$ and $\eps_i$ a torus weight in $L(\omega_{i,X})$ for $X = B$, $E$, or $F$ but not in $\{\eps_{i,c_i} \mid c_i \in a_i\}$. 
Informally speaking, $P_\sigma^a$ is the non-zero quotient of $P_\sigma^b$ containing precisely $\sigma$ and the Jordan--H\"older factors labelled by the elements in $a_j$ for $j\in A(\sigma)$ %
(see Figure \ref{table:P_sigma^a}). 

\begin{lemma}\label{lemma:wedgeglue}
Let $\cC$ be an abelian category. 
Suppose that we have an exact sequence 
\[
0 \ra \oplus_{i\in I} N_i \ra M \ra N \ra 0
\]
in $\cC$ where $I$ is a finite set. 
Letting $M_i$ be the cokernel of the map $\oplus_{j\in I, j\neq i} N_j \ra M$ for each $i\in I$, there is an exact sequence 
\[
0 \ra M \ra \oplus_{i\in I} M_i \ra \oplus_{i\in I} N/\Delta(N) \ra 0
\]
where the second map is (the composition of the diagonal map and) the sum of the natural projections, the third map is induced by the sum of the natural quotient maps, and $\Delta$ denotes the diagonally embedded copy of $N$ in $\oplus_{i\in I} N$. 
\end{lemma}
\begin{proof}
There is a commutative diagram 
\[
\begin{tikzcd}
0 \arrow[r] \arrow[d]
& \oplus_{i\in I} N_i \arrow[r] \arrow[d]
& M \arrow[r] \arrow[d]
& N \arrow[r] \arrow[d]
& 0 \arrow[d] \\
0 \arrow[r]
& \oplus_{i\in I} N_i \arrow[r]
& \oplus_{i\in I} M_i \arrow[r]
& \oplus_{i\in I} N \arrow[r]
& 0
\end{tikzcd}
\]
where the first row is as in the statement of the lemma, the second row is the direct sum of the exact sequences $0 \ra N_i \ra M_i \ra N \ra 0$, the second vertical map is the identity, the third vertical map is the the composition of the diagonal map and the sum of the natural projections, and the fourth vertical map is the diagonal map. 
The snake lemma gives the desired exact sequence. 
\end{proof}

The following results are special cases of Lemma \ref{lemma:wedgeglue}. 
As usual we assume that $\sigma$ is $8$-deep so that \cite[Lemma 4.2.2]{LLLM2} applies.

\begin{prop}\label{prop:lowerglue}
{Assume that $\sigma$ is $8$-deep.}
Let $a = (a_j)_{j\in \cJ}$ be a tuple such that $a_j$ is a subset of $\{B,E_o,E_s,F_o,F_s\}$ for $j\in A(\sigma)$ and either $a_j = \widehat{1}$ for all $j\notin A(\sigma)$ or $a_j = \widehat{2}$ for all $j\notin A(\sigma)$.
Fix $i \in A(\sigma)$. 
For $X \subset a_i$, let $a_X = (a_{X,j})_{j\in \cJ}$ be the tuple such that $a_{X,j} = a_j$ if $j\neq i$ and $a_{X,i} = X$. 
Then there is an exact sequence
\[
0 \ra P_\sigma^a \ra \oplus_{X\in a_i} P_\sigma^{a_X} \ra (\oplus_{X\in a_i} P_\sigma^{a_\emptyset})/\Delta(P_\sigma^{a_\emptyset})\ra 0
\]
where $\Delta(P_\sigma^{a_\emptyset})$ denotes the diagonally embedded copy of $P_\sigma^{a_\emptyset}$ in $\oplus_{X\in a_i}P_\sigma^{a_\emptyset}$ and the second and third maps are the sums of the natural surjections. 
\end{prop}
\begin{proof}
For $X \in a_i$, let $N_X$ be $\ker(P_\sigma^{a_X} \ra P_\sigma^{a_\emptyset})$ and $L$ be $\ker(P_\sigma^a \ra P_\sigma^{a_\emptyset})$. 
Then there are surjections $L \ra N_X$, and the natural map $L \ra \oplus_{X \in a_i} N_X$ is surjective because the Jordan--H\"older factors of $N_X$ are pairwise disjoint by extension graph considerations at embedding $i$. 
This map $L \ra \oplus_{X \in a_i} N_X$ is then an isomorphism by length considerations. 
The result now follows from Lemma \ref{lemma:wedgeglue} setting $M = P_\sigma^a$ and $N=P_\sigma^{a_\emptyset}$. \end{proof}

Let  
\begin{equation}\label{equation:Pbar}
\ovl{P}_\sigma \defeq P_\sigma/\big(\sum_a \rad^a P_\sigma + \sum_b \rad^b P_\sigma\big)
\end{equation}
where $a$ varies over tuples with 
\[
a_i = \begin{cases}
  2  & \text{ if } i\in A(\sigma) \\
  3 & \text{ if } i\notin A(\sigma)
\end{cases}
\]
for some $i\in \cJ$ and $a_j = 0$ for $j\neq i$ and $b$ varies over tuples with $b_i = b_{i'} = 1$ for some $i \notin A(\sigma)$ and $i' \in\cJ$ distinct from $i$, and $b_j = 0$ for $j\neq i,i'$ (the set of $b$ is empty if $\cJ = A(\sigma)$ or $\#\cJ = 1$). 
The following proposition gives another characterization of $\ovl{P}_\sigma$. 

\begin{prop}\label{prop:finalglue}
{Assume that $\sigma$ is $8$-deep.}
There is an exact sequence 
\begin{equation}\label{eqn:finalglue}
0 \ra \ovl{P}_\sigma \ra \oplus_c P_\sigma^c \ra (\oplus_c \sigma)/\Delta(\sigma) \ra 0
\end{equation}
where $c$ runs over tuples $(c_j)_{j\in \cJ}$ with $c_i = \widehat{3}$ for some $i \notin A(\sigma)$ and $c_j = \widehat{1}$ for all $j \neq i$ and the tuple $(c_j)_{j\in \cJ}$ with $j = \widehat{2}$ for all $j\in A(\sigma)$ and $j = \widehat{1}$ for all $j \notin A(\sigma)$, $\Delta(\sigma) \subset \oplus_c \sigma$ denotes the diagonally embedded copy, and the maps are the natural projections. 
\end{prop}
\begin{proof}
For a tuple $c$ as in \eqref{eqn:finalglue}, let $N_c$ be $\ker(P_\sigma^c \onto \sigma)$, and let $L$ be $\ker(\ovl{P}_\sigma \ra \sigma)$. 
Then we have natural surjections $L \ra N_c$. 
We claim that the natural map $L \ra \oplus_c N_c$ is surjective. 
It suffices to show that it is surjective after taking cosocles. 
This surjectivity follows from the fact that $\cosoc\, L \ra \cosoc\, N_c$ is surjective for each $c$ and that the sets $\JH(\cosoc\, N_c)$ are pairwise disjoint by alcove considerations (as usual, we use \cite[Lemma 4.2.2]{LLLM2}).
Finally, the map $L \ra \oplus_c N_c$ is an isomorphism by length considerations. 
The result now follows from Lemma \ref{lemma:wedgeglue} setting $M = \ovl{P}_\sigma$ and $N=\sigma$. 
\end{proof}

\subsection{Lattices in direct sums of Deligne--Lusztig representations}\label{sec:DL}

Let $\sigma$ be an $8$-deep Serre weight. 
We now study certain quotients of $P_\sigma$ arising from reductions of lattices in direct sums of Deligne--Lusztig representations. 
Recall that $\tld{P}_\sigma \ra \sigma$ is a $\cO[\rG]$-projective cover of $\sigma$. 
Since the multiplicity of an irreducible $E[\rG]$-module $R$ in $\tld{P}_\sigma \otimes_{\cO} E$ is the multiplicity of $\sigma$ in $\ovl{R}$, we have that 
\[
\tld{P}_\sigma \otimes_{\cO} E \cong \oplus_{\sigma\in \JH(\ovl{R})} R
\]
where $R$ runs over Deligne--Lusztig representations. 

Let $T$ be a set of Deligne--Lusztig representations whose reduction contains $\sigma$ as a Jordan--H\"older factor. 
Let $\tld{P}_\sigma^T$ be the quotient of $\tld{P}_\sigma$ which is isomorphic to the image of the composition 
\[
\tld{P}_\sigma \into \tld{P}_\sigma\otimes_{\cO}E \cong \oplus_{\sigma\in \JH(\ovl{R})} R \onto \oplus_{R\in T} R. 
\] 
The isomorphism class of $\tld{P}_\sigma^T$ does not depend on the choice of isomorphism above. 
Let $P_\sigma^T \defeq \tld{P}_\sigma^T  \otimes_{\cO} \F$. 
Then we have natural surjections 
\begin{equation}\label{eqn:projcoverT}
P_\sigma\onto P_\sigma^T \onto P_\sigma^S
\end{equation}
for any subset $S \subset T$. 
If $T$ is the set of all Deligne--Lusztig representations whose reduction contains $\sigma$ as a Jordan--H\"older factor, then $\tld{P}_\sigma^T \cong \tld{P}_\sigma$. 

We now study $P_\sigma^T$ for specific subsets $T$. 
We begin by labelling certain Deligne--Lusztig representations. 
Suppose that $A(\sigma) \neq \cJ$, and fix $i\in \cJ \setminus A(\sigma)$.
Identify $S_3$ with the $i$-th component of $\un{W}$. 
Let $(s,\mu)$ be a lowest alcove presentation for a Deligne--Lusztig representation such that $\sigma \in \JH(\ovl{R}_s(\mu+\eta))$ is an outer weight. 
For $w\in S_3 \subset \un{W}$, let $R_w$ be the Deligne--Lusztig representations $R_{sw}(\mu+\eta)$. 

\begin{lemma}\label{lemma:sigma1}
Fix $i' \neq i$ in $\cJ$. 
Define $a = (a_j)_{j\in \cJ}$ by $a_j = 0$ for all $j\neq i'$ and $a_{i'} = 1$. 
Then the images of the maps 
\begin{equation}\label{eqn:radsigma}
\rad^a P_\sigma \ra P_\sigma^{\{R_{\alpha\beta},R_{w_0}\}} \quad \textrm{and} \quad \rad^aP_\sigma \ra P_\sigma^{\{R_{w_0},R_{\beta\alpha}\}}
\end{equation} 
induced by \eqref{eqn:projcoverT} do not contain $\sigma$ as a Jordan--H\"older factor. 
\end{lemma}
\begin{proof}
Let $S$ be $\{R_{\alpha\beta},R_{w_0}\}$ or $\{R_{w_0},R_{\beta\alpha}\}$. 
Suppose that $\sigma = F(\lambda)$. 
First suppose that $i' \in A(\sigma)$. 
Using Proposition \ref{prop:alg_rep_ext_graph} we see that for any $\nu \in X_1(\un{T})$, $P_\sigma^S$ contains no Jordan--H\"older factors of $L(\lambda_{i',X}) \otimes \otimes_{j\neq i'} L(\nu_j)|_{\rG}$ if $X = B$ and at most one Jordan--H\"older factor if $X = E$ or $F$. 
Then using the Weyl filtration in Proposition \ref{prop:filtrations} and Proposition \ref{prop:weylext} with $(X_{i'}^1,X_{i'}^0) = (B,A), (E,A), (F,A)$, and the dual Weyl filtration in Proposition \ref{prop:filtrations} and the dual of Proposition \ref{prop:mult2}, we see that the kernel of \eqref{eqn:radsigma} contains all Jordan--H\"older factors of the form $L(\lambda_{i',A}) \otimes \otimes_{j\neq i'} L(\nu_j)|_{\rG}$. 

Next suppose that $i' \notin A(\sigma)$. 
Again from Proposition \ref{prop:alg_rep_ext_graph} we see that for any $\nu \in X_1(\un{T})$, $P_\sigma^S$ contains at most two Jordan--H\"older factors of $L(\lambda_{i',X}) \otimes \otimes_{j\neq i'} L(\nu_j)|_{\rG}$ if $X = C$ or $D$. 
Then using the Weyl filtration in Proposition \ref{prop:filtrations} and Proposition \ref{prop:weylext} with $(X_{i'}^1,X_{i'}^0) = (C,B), (D,B)$, we see that the kernel of \eqref{eqn:radsigma} contains all Jordan--H\"older factors of the form $L(\lambda_{i',B}) \otimes \otimes_{j\neq i'} L(\nu_j)|_{\rG}$. 
\end{proof}

\begin{lemma}\label{lemma:sigma2}
Define $a = (a_j)_{j\in \cJ}$ by $a_j = 0$ for all $j\neq i$ and $a_i = 4$. 
The images of the maps 
\begin{equation}\label{eqn:radsigma2}
\rad^a P_\sigma \ra P_\sigma^{\{R_{\alpha\beta},R_{w_0}\}} \quad \textrm{and} \quad \rad^aP_\sigma \ra P_\sigma^{\{R_{w_0},R_{\beta\alpha}\}}
\end{equation} 
induced by \eqref{eqn:projcoverT} do not contain $\sigma$ as a Jordan--H\"older factor. 
\end{lemma}
\begin{proof}
The proof is similar to that of Lemma \ref{lemma:sigma1}. 
For any $\nu \in X_1(\un{T})$ and $S$ equal to $\{R_{\alpha\beta},R_{w_0}\}$ or $\{R_{w_0},R_{\beta\alpha}\}$, $P_\sigma^S$ contains at most two Jordan--H\"older factors of $L(\lambda_{i,X}) \otimes \otimes_{j\neq i} L(\nu_j)|_{\rG}$ if $X = E$ or $F$. 
Then using the Weyl filtration in Proposition \ref{prop:filtrations} and Proposition \ref{prop:weylext} with $(X_{i}^1,X_{i}^0) = (E,A), (F,A),(A,B)$, we see that the kernel of \eqref{eqn:radsigma2} contains all Jordan--H\"older factors of the form $L(\lambda_{i,B}) \otimes \otimes_{j\neq i} L(\nu_j)|_{\rG}$. 
\end{proof}

\begin{lemma}\label{lemma:notsigma}
Define $a = (a_j)_{j\in \cJ}$ by $a_j = 0$ for all $j\neq i$ and $a_i = 2$. 
Let $P$ be a quotient of $P_\sigma$ and let $M$ be the kernel of the map $\rad^a P_\sigma \ra P$. 
Recall that if $\sigma=F(\lambda)$, then $\gr^a P_\sigma$ is isomorphic to 
\begin{align}
\gr^a P_\sigma &\cong \gr^2 Q_1(\lambda_i) \otimes \otimes_{j\neq i} \gr^0 Q_1(\lambda_j)|_{\rG} \\
&\cong \sigma^{\oplus 2} \oplus (L(\lambda_{i,E}) \otimes \otimes_{j\neq i} L(\lambda_j))|_{\rG}  \oplus (L(\lambda_{i,F}) \otimes \otimes_{j\neq i} L(\lambda_j))|_{\rG}  
\end{align}
so that there is a projection map 
\begin{equation} \label{eqn:notsigma}
\gr^a P_\sigma \ra (L(\lambda_{i,E}) \otimes \otimes_{j\neq i} L(\lambda_j))|_{\rG}  \oplus (L(\lambda_{i,F}) \otimes \otimes_{j\neq i} L(\lambda_j))|_{\rG} . 
\end{equation}
If $\kappa$ is a Jordan--H\"older factor of the image of $M$ under \eqref{eqn:notsigma}, then $\kappa$ is not a Jordan--H\"older factor of $P$. 
\end{lemma}
\begin{proof}
Suppose that $\kappa$ is a Jordan--H\"older factor of the image of $M$ under \eqref{eqn:notsigma}. 
Using Proposition \ref{prop:alg_rep_ext_graph}, there is a $\rhobar$ such that
\[
\JH(\gr^a P_\sigma) \cap W^?(\rhobar) = \{\kappa\}. 
\] 
As the LHS of \eqref{eqn:notsigma} is multiplicity free, we conclude that the cokernel of $M\ra \gr^a P_\sigma$ does not contain any elements of $W^?(\rhobar)$. 
By Proposition \ref{prop:W?-cover}, $\kappa \notin \JH(\rad^a P_\sigma/M)$. 
Since $\kappa \notin \JH(P_\sigma/\rad^a P_\sigma)$ (again using Proposition \ref{prop:alg_rep_ext_graph}), $\kappa$ is not a Jordan--H\"older factor of the image of the surjective map $P_\sigma \ra P$ and thus not in $\JH(P)$. 
\end{proof}

For $n\geq 0$ define 
\[
\Fil^{n_i}P_\sigma=\rad^{n_i}P_\sigma + \sum_{j\neq i} \rad^{1_j}P_\sigma.
\] 
Note that $\Fil^{n_i}P_\sigma/\Fil^{(n+1)_i}P_\sigma$ is $\gr^{n_i}P_\sigma$.
\begin{lemma}\label{lemma:kergr}
The kernel of the map $\Fil^{2_i} P_\sigma \ra P_\sigma^{\{R_{\alpha\beta},R_{w_0}\}} \oplus P_\sigma^{\{R_{w_0},R_{\beta\alpha}\}} \oplus P_\sigma^{R_{\Id}}$ induced by \eqref{eqn:projcoverT} is contained in $\Fil^{3_i}P_\sigma$.
\end{lemma}
\begin{proof}
Suppose that $\kappa$ is a simple submodule of the projection of the kernel of 
\[
\Fil^{2_i} P_\sigma \ra P_\sigma^{\{R_{\alpha\beta},R_{w_0}\}} \oplus P_\sigma^{\{R_{w_0},R_{\beta\alpha}\}} \oplus P_\sigma^{R_{\Id}} 
\]
to $\Fil^{2_i} P_\sigma/ \Fil^{3_i}P_\sigma\cong\gr^{2_i} P_\sigma$. 

Let $\sigma=F(\lambda)$. 
First suppose that $\kappa$ is a Jordan--H\"older factor of 
\begin{equation}\label{eqn:EF}
(L(\lambda_{i,E}) \otimes \otimes_{j\neq i} L(\lambda_j))|_{\rG}  \oplus (L(\lambda_{i,F}) \otimes \otimes_{j\neq i} L(\lambda_j))|_{\rG}.
\end{equation}
Then Lemma \ref{lemma:notsigma} with $P = P_\sigma^{\{R_{\alpha\beta},R_{w_0}\}}, P_\sigma^{\{R_{w_0},R_{\beta\alpha}\}},$ and $P_\sigma^{R_{\Id}}$ implies that $\kappa$ is not a Jordan--H\"older factor of $P_\sigma^{\{R_{\alpha\beta},R_{w_0}\}} \oplus P_\sigma^{\{R_{w_0},R_{\beta\alpha}\}} \oplus P_\sigma^{R_{\Id}}$. 
However, every Jordan--H\"older factor of \eqref{eqn:EF} is in 
\[
\JH(P_\sigma^{\{R_{\alpha\beta},R_{w_0}\}} \oplus P_\sigma^{\{R_{w_0},R_{\beta\alpha}\}} \oplus P_\sigma^{R_{\Id}}) = \JH(\ovl{R}_{\Id}\oplus\ovl{R}_{\alpha\beta}\oplus\ovl{R}_{w_0}\oplus\ovl{R}_{\beta\alpha}), 
\]
which is a contradiction. 

We conclude that $\kappa$ must isomorphic to $\sigma$. 
Let $P_\sigma^b$ be the quotient of $P_\sigma/\Fil^{3_i}P_\sigma$ by \eqref{eqn:EF}. 
Then there is a natural injection $\kappa \into P_\sigma^b$, and we identify $\kappa$ with its image. 
Recall from the Weyl filtration of Proposition \ref{prop:filtrations} that $P_\sigma^b$ contains the direct sum of two submodules $M_C$ and $M_D$ which are the restrictions to $\rG$ of extensions of $L(\lambda_{i,C}) \otimes \otimes_{j\neq i} L(\lambda_j)$ and $L(\lambda_{i,D}) \otimes \otimes_{j\neq i} L(\lambda_j)$, respectively, by $L(\lambda)$. 
Then $\kappa$ is in the direct sum $M_C\oplus M_D$. 
Suppose without loss of generality that the projection of $\kappa$ to $M_C$ is nonzero. 
One of $P_\sigma^{\{R_{\alpha\beta},R_{w_0}\}}$ and $P_\sigma^{\{R_{w_0},R_{\beta\alpha}\}}$, say $P_\sigma^S$, does not contain all of the Jordan--H\"older factors of $L(\lambda_{i,D}) \otimes \otimes_{j\neq i} L(\lambda_j)|_{\rG}$. 
We conclude from Proposition \ref{prop:weylext} with $(X_{i}^1,X_{i}^0) = (D,B)$ that the image of the map $\Fil^{2_i} P_\sigma \ra P_\sigma^S$ does not contain $\sigma$ as a Jordan--H\"older factor. 
Combined with Lemmas \ref{lemma:sigma1} and \ref{lemma:sigma2}, we see that the image of the surjection $P_\sigma \onto P_\sigma^S$ contains $\sigma$ as a Jordan--H\"older factor with multiplicity at most $1$. 
This contradicts the fact that $P_\sigma^S$  contains $\sigma$ as a Jordan--H\"older factor with multiplicity $2$. 
\end{proof}

\begin{prop}\label{prop:multiDL}
The kernel of the map 
\begin{equation} \label{eqn:4DL}
P_\sigma \ra P_\sigma^{\{R_{\alpha\beta},R_{w_0}\}} \oplus P_\sigma^{\{R_{w_0},R_{\beta\alpha}\}} \oplus P_\sigma^{R_{\Id}} 
\end{equation}
induced by \eqref{eqn:projcoverT} is contained in $\Fil^{3_i}P_\sigma$.
\end{prop}
\begin{proof}
We claim that the map 
\begin{equation}\label{eqn:modFil3}
P_\sigma/\Fil^{3_i} P_\sigma \ra (P_\sigma^{\{R_{\alpha\beta},R_{w_0}\}} \oplus P_\sigma^{\{R_{w_0},R_{\beta\alpha}\}} \oplus P_\sigma^{R_{\Id}})/\mathrm{im}(\Fil^{3_i} P_\sigma)
\end{equation}
is injective, where $\mathrm{im}(\Fil^{3_i} P_\sigma)$ denotes the image of $\Fil^{3_i} P_\sigma$ under \eqref{eqn:4DL}. 
Note that $\Fil^{2_i} P_\sigma/\Fil^{3_i} P_\sigma \subset P_\sigma/\Fil^{3_i} P_\sigma$ is the socle by \cite[Lemma 4.2.2]{LLLM2}. 
It thus suffices to show that the restriction of \eqref{eqn:modFil3} to $\Fil^{2_i} P_\sigma/\Fil^{3_i} P_\sigma$ is injective. 
Suppose that $a+\Fil^{3_i} P_\sigma$ is in the kernel of \eqref{eqn:modFil3} for $a\in \Fil^{2_i} P_\sigma$. 
Then $a-b$ is in the kernel of \eqref{eqn:4DL} for some $b\in \Fil^{3_i} P_\sigma$. 
By Lemma \ref{lemma:kergr}, $a\in \Fil^{3_i} P_\sigma$, and the claim follows. 

Suppose now that $a$ is in the kernel of \eqref{eqn:4DL}. 
Then $a \in \Fil^{3_i} P_\sigma$ by the injectivity of \eqref{eqn:modFil3}. 
\end{proof}

\begin{figure}[htb] 
\centering
\includegraphics[scale=.1]{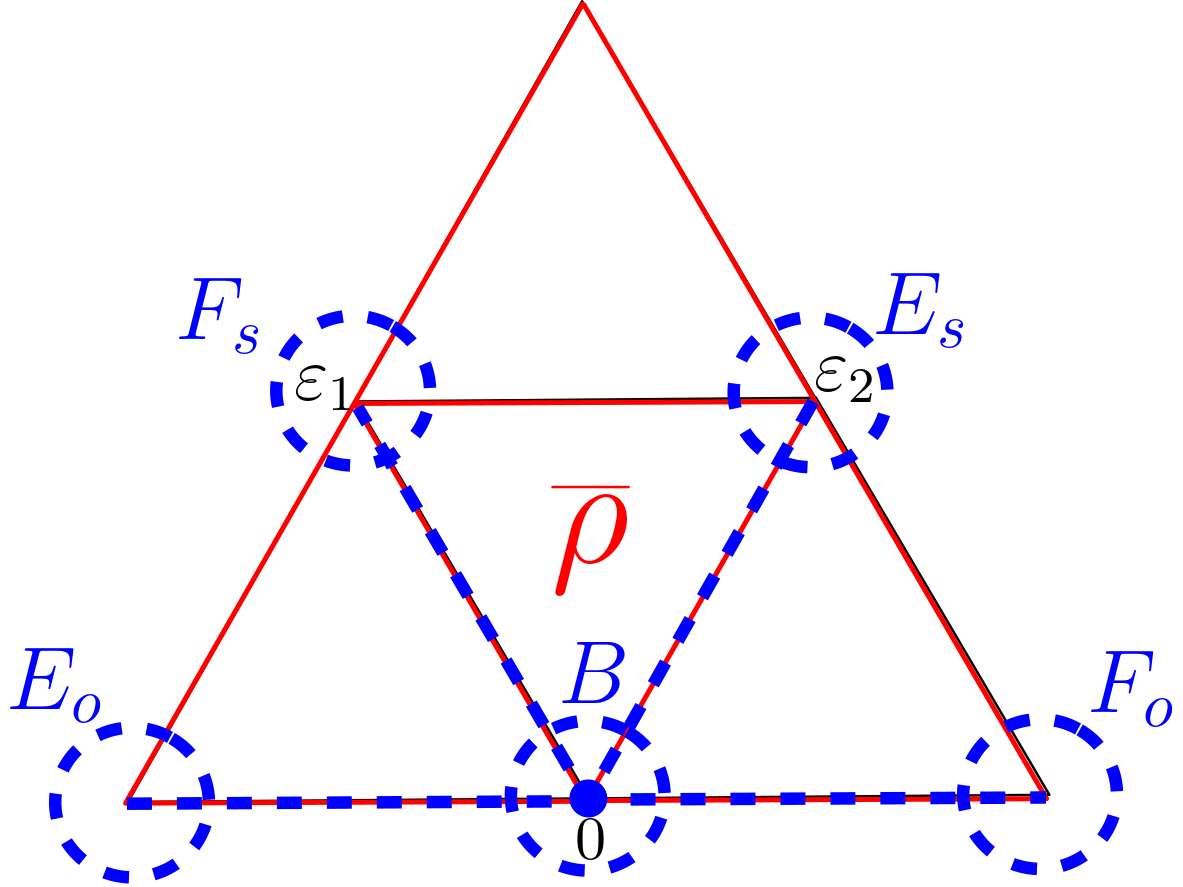}
\caption{
Using the extension graph with origin in $\lambda_{\un{A}}+\eta$ (where $\lambda\in X_1(\un{T})$ is such that $\sigma=F(\lambda)$) we  describe the Jordan--H\"older factors of $P_\sigma^a$ at $j\in\cJ$, for different configurations of $a_j\subset\{B,E_o,E_s,F_o,F_s\}$: given $b\in \{B,E_o,E_s,F_o,F_s\}$ the dashed circle appears if and only if $b\in a_j$. 
}
\label{table:P_sigma^a}
\end{figure}

\section{Patched modules}
\label{sec:patch}

\subsection{Patching axioms}
\label{sub:patch:axioms}
Recall from \S \ref{sec:notation} that $\cO_p=\prod\limits_{v\in S_p} \cO_v$ where $S_p$ is a finite set and for each $v\in S_p$, $\cO_v$ is the ring of integers in a finite unramified extension $F^+_v$ of $\Q_p$.

Let $\rhobar$ be an $\F$-valued $L$-homomorphism and let $(\rhobar_v)_{v \in S_p}$ be the collection of continuous Galois representations corresponding to it (see \S \ref{subsubsec:L_parameters}).
Let $R_\infty\defeq R_{\rhobar} \widehat{\otimes}_{\cO} R^p$ where
\[
R_{\rhobar} \defeq \widehat{\bigotimes}_{v\in S_p,\cO} R_{\rhobar_v}^\square
\]
and $R^p$ is a (nonzero) complete local Noetherian equidimensional flat $\cO$-algebra with residue field $\F$ such that each irreducible component of $\Spec R^p$ and of $\Spec \overline{R}^p$ is geometrically irreducible. 
A finite set $T$ of tame inertial $L$-parameters gives rise to a collection $(T_v)_{v\in S_p}$ of finite sets of tame inertial types and we let $R_\infty(T)\defeq R_\infty \otimes_{R_{\rhobar}^\Box} R_{\rhobar}^{\eta,T}$ 
where
\[
R_{\rhobar}^{\eta,T} \defeq \widehat{\bigotimes}_{v\in S_p,\cO} R_{\rhobar_{v}}^{\eta,T_{v}}.%
\]
Write $X_\infty$, $X_\infty(T)$, and $\ovl{X}_\infty(T)$ for $\Spec R_\infty$, $\Spec R_\infty(T)$, and $\Spec \ovl{R}_\infty(T)$ respectively. 
Finally, let $\Mod(X_\infty)$ be the category of coherent sheaves over $X_\infty$, and $\Rep_{\cO}(\GL_3(\cO_p))$ the category of topological $\cO[\GL_3(\cO_p)]$-modules which are finitely generated over $\cO$. 

\begin{defn}\label{minimalpatching}
A \emph{weak patching functor} for an $L$-homomorphism $\rhobar: G_{\Q_p} \ra$ $^L \un{G}(\F)$ is a nonzero covariant exact functor $M_\infty:\Rep_{\cO}(\GL_3(\cO_p))\ra \Coh(X_{\infty})$ satisfying the following axioms: if $\tau$ is an inertial $L$-parameter and $\sigma^\circ(\tau)$ is an $\cO$-lattice in $\sigma(\tau)$ then
\begin{enumerate}
\item 
\label{it:minimalpatching:1}
$M_\infty(\sigma^\circ(\tau))$ is {either zero or} a maximal Cohen--Macaulay sheaf on $X_\infty(\tau)$; and
\item 
\label{it:minimalpatching:2}
for all $\sigma \in \JH(\ovl{\sigma}^\circ(\tau))$, $M_\infty(\sigma)$ is a maximal Cohen--Macaulay sheaf on $\ovl{X}_\infty(\tau)$ (or is $0$).
\end{enumerate}
(We identify the tame inertial $L$-parameter with the finite set $T=\{\tau\}$ in the above definition.)
A weak patching functor $M_\infty$ is \emph{minimal} if $R^p$ is formally smooth over $\cO$ and whenever $\tau$ is an inertial $L$-parameter, $M_\infty(\sigma^\circ(\tau))[p^{-1}]$, which is locally free over the regular scheme $\Spec R_\infty(\tau)[p^{-1}]$, has rank at most one on each connected component.
\end{defn}
\begin{rmk}
In \S \ref{sub:abstract:locality} we will further assume that $M_\infty$ has the form $\Hom_{\GL_3(\cO_p)}(-,M_{\infty}^\vee)^\vee$ for some pseudocompact $\cO[\![\GL_3(\cO_p)]\!]$-module $M_\infty$.
Note that the $M_\infty$ that arise from Taylor--Wiles patching always have this property.
\end{rmk}

Given a finite set $T$ of tame inertial $L$-parameters we have a surjective homomorphism $R_{\rhobar}^\Box\onto R_{\rhobar}^{\eta,T}$.
In particular, for any $T'\subseteq T$ we have a surjective homomorphism
\[
R_\infty\onto R_\infty(T)  \onto R_\infty(T')
\]
whose kernel (in either $R_\infty$ or $R_\infty(T)$) will be denoted as $I_\infty(T')$.
(This abuse of notation will not create confusion in what follows, since it will always be clear from the context which ring the ideal $I_\infty(T')$ is in.)

\subsection{Minimal number of generators}
\label{subsec:proj_and_K1}

As in \S \ref{subsec:Gproj}, $P_\sigma$ denotes a $\F[\rG]$-projective cover of a Serre weight $\sigma$, and we have fixed $\lambda\in X_1(\un{T})$ such that $\sigma\cong F(\lambda)$. 
We will use notation from \S \ref{subsec:Gproj} in what follows. 
Recall that $A(\sigma) \defeq \{j\in \cJ\mid \lambda_j \in A\}$. 
The following theorem is the main result of the section. 

\begin{thm}\label{thm:mingen}
Let $\rhobar$ be a $11$-generic tame $L$-homomorphism and let $\sigma \in W^?(\rhobar)$.

Then $M_\infty(P_\sigma)$ is minimally generated by $3^{\#A(\sigma)}$ elements. 
\end{thm}

One can reduce this theorem to Lemma \ref{lemma:broom} below. 
Recall from \eqref{equation:Pbar} that
\[
\ovl{P}_\sigma \defeq P_\sigma/\big(\sum_a \rad^a P_\sigma + \sum_b \rad^b P_\sigma\big)
\]
where $a$ varies over tuples with 
\[
a_i = \begin{cases}
  2  & \text{ if } i\in A(\sigma) \\
  3 & \text{ if } i\notin A(\sigma)
\end{cases}
\]
for some $i\in \cJ$ and $a_j = 0$ for $j\neq i$ and $b$ varies over tuples with $b_i = b_{i'} = 1$ for some $i \notin A(\sigma)$ and $i' \in\cJ$ distinct from $i$, and $b_j = 0$ for $j\neq i,i'$ (the set of $b$ is empty if $\cJ = A(\sigma)$ or $\#\cJ = 1$). 

\begin{lemma}\label{lemma:broom}
Let $\rhobar$ be a $11$-generic tame $L$-homomorphism and let $\sigma \in W^?(\rhobar)$.
\begin{enumerate}
\item \label{item:broom2} Let $a= (a_j)_{j\in \cJ}$ be a tuple such that $a_j = \widehat{2}$  if $j\in A(\sigma)$ and $a_j = \widehat{1}$ or $\widehat{2}$ if $j\notin A(\sigma)$. 
Then $M_\infty(P_\sigma^a)$ is minimally generated by $3^{\#A(\sigma)}$ elements. 
\item \label{item:bigbroom} $M_\infty(\ovl{P}_\sigma)$ is minimally generated by $3^{\#A(\sigma)}$ elements. 
\end{enumerate}
\end{lemma}

The following is a consequence of Lemma \ref{lemma:broom}. 

\begin{cor}\label{cor:noness}
Let $\rhobar$ be a $11$-generic tame $L$-homomorphism and let $\sigma \in W^?(\rhobar)$.

Let $a$ be a tuple as in Lemma \ref{lemma:broom}\eqref{item:broom2} and $b$ be a tuple with $b_j > 0$ for some $j \notin A(\sigma)$. 
Then the following hold.
\begin{enumerate}
\item\label{item:noness:broom}
The image of the natural map
\[
M_\infty(\rad^b P_\sigma) \ra M_\infty(P_\sigma^a)
\]
is contained in $\fm M_\infty(P_\sigma^a)$. 
\item \label{item:noness:Bigbroom}
The image of the natural map
\[
M_\infty(\rad^b P_\sigma) \ra M_\infty(\ovl{P}_\sigma)
\]
is contained in $\fm M_\infty(\ovl{P}_\sigma)$. 
\end{enumerate}
\end{cor}
\begin{proof}
We use a similar argument as in the proof of \cite[Lemma 4.5]{DanWild}.
Let $a_{\min}$ be the tuple $(a_{\min,j})_{j\in \cJ}$ with $a_{\min,j} = \widehat{2}$  if $j\in A(\sigma)$ and $a_{\min,j} = \widehat{1}$ if $j\notin A(\sigma)$. 
The main observation is that the images of $\rad^b P_\sigma$ in $P_\sigma^a$ and $\ovl{P}_\sigma$ are contained in the kernels of the natural surjections to $P_\sigma^{a_{\min}}$, respectively. 
Since these surjections are isomorphisms after applying $M_\infty(-)\otimes_{R_\infty} R_\infty/\fm$ by dimension considerations using Lemma \ref{lemma:broom}, the result follows. 
\end{proof}

Before proving Lemma \ref{lemma:broom}, we first use it to prove Theorem \ref{thm:mingen}. 
We first introduce some notation. 
\begin{defn}
\label{defn:Agr}
For a tuple $a = (a_j)_{j\in \cJ}$ with $0\leq a_j \leq 7$, let 
\[
A(\gr^a P_\sigma) = \{j\in \cJ \mid (j\notin A(\sigma) \textrm{ and } 2\nmid a_j) \textrm{ or } (j\in A(\sigma), a_j \geq 2, \textrm{ and } 2\mid a_j) \}.
\]
The definition of $A(\gr^a P_\sigma)$ means that for $j\in \cJ$, $j\in A(\gr^a P_\sigma)$ if and only if $a_j>0$ and for some (equivalently for all) $\kappa\in \JH(\gr^a P_\sigma)$, $j\in A(\kappa)$. 
\end{defn}

\begin{lemma}\label{lemma:noness}
In the setup of Theorem \ref{thm:mingen} suppose that $c = (c_j)_{j\in \cJ}$ with $0\leq c_i \leq 7$ and 
\begin{equation}\label{eqn:noness}
c_i \geq \begin{cases}
  2  & \text{ if } i\in A(\sigma) \\
  1 & \text{ if } i\notin A(\sigma)
\end{cases}
\end{equation}
for some $i\in \cJ$. 
Then the injection
\begin{equation}\label{eqn:EGS}
M_\infty(\gr^c P_\sigma) \into M_\infty\Big(P_\sigma/\big(\rad^{|c|+1}P_\sigma + \sum_{\substack{|c|=|d| \\ \#A(\gr^dP_\sigma) > \#A(\gr^cP_\sigma)}}\rad^d P_\sigma\big)\Big).
\end{equation}
factors through $\fm M_\infty\Big(P_\sigma/\big(\rad^{|c|+1}P_\sigma + \sum_{\substack{|c|=|d| \\ \#A(\gr^dP_\sigma) > \#A(\gr^cP_\sigma)}}\rad^d P_\sigma\big)\Big)$.
\end{lemma}
\begin{proof}
First suppose that $A(\gr^cP_\sigma) = \emptyset$. 
Then we claim that $c_\ell \geq 2$ for some $\ell\in \cJ$. 
Indeed, if $c_j < 2$ for all $j\in \cJ$, then \eqref{eqn:noness} implies that for some $j\notin A(\sigma)$, $c_j=1$. 
Then $j\in A(\gr^cP_\sigma)$ which is a contradiction. 

Suppose that $c_\ell\geq 2$. 
Let $I$ be the image of $M_\infty(\rad^{c-2_\ell}P_\sigma)$ in the codomain of \eqref{eqn:EGS}. 
Then $I$ contains (the image under \eqref{eqn:EGS} of) $M_\infty(\gr^c P_\sigma)$.
We will show that $\fm I$ contains $M_\infty(\gr^c P_\sigma)$. 

Let $N \subset \rad^{c-2_\ell}P_\sigma$ be the minimal submodule for which the cokernel of the map $N \ra \gr^{c-2_\ell}P_\sigma$ contains no Jordan--H\"older factors in $W(\rhobar)$. 
Let $M'$ be the image of $M_\infty(N)$ in the codomain of \eqref{eqn:EGS}. 
Then $M' = I$ by Proposition \ref{prop:W?-cover}. 
(Note that $\sigma$ is $9$-deep by \cite[Lemma 4.2.13]{LLLM2}.)
We will show that $M_\infty(\gr^c P_\sigma)\subset \fm M'$. 

Let $P \defeq \oplus_{\kappa \in \JH(\gr^{c-2_\ell} P_\sigma) \cap W(\rhobar)} P_{\kappa}^{\oplus[\gr^{c-2_\ell} P_\sigma:\kappa]}$. 
By projectivity, the natural map $P\ra\gr^{c-2_\ell}P_\sigma$ lifts to a map
$P \ra N$ which we now fix.
Since the cokernel of the composite $P\ra N\into \rad^{c-2_\ell}P_\sigma\onto \gr^{c-2_\ell}P_\sigma$ contains no Jordan--H\"older factors in $W(\rhobar)$, by Proposition \ref{prop:W?-cover} and minimality of $N$  we conclude that the map $P\ra N$ is surjective.
Moreover, the preimage of $\gr^c P_\sigma$ in $P$ under the map
\begin{equation}\label{eqn:projcover}
P \ra N \ra P_\sigma/\big(\rad^{|c|+1}P_\sigma + \sum_{\substack{|d|=|c| \\ \#A(\gr^dP_\sigma) > \#A(\gr^cP_\sigma)}}\rad^d P_\sigma\big)
\end{equation}
is contained in 
\[
\rad^1 P = \oplus_{\kappa \in \JH(\gr^{c-2_\ell}  P_\sigma) \cap W(\rhobar)} \rad^1 P_{\kappa}^{\oplus[\gr^{c-2_\ell} P_\sigma:\kappa]}. 
\] 
Even more, by alcove considerations, the preimage of $\gr^c P_\sigma$ in $P$ is contained in 
\begin{equation}\label{eqn:twosteps}
\oplus_{\kappa \in \JH(\gr^{c-2_\ell} P_\sigma) \cap W(\rhobar)} \sum_{j\in \cJ} \rad^{2_j} P_{\kappa}^{\oplus[\gr^{c-2_\ell} P_\sigma:\kappa]}. 
\end{equation}
It suffices to show that the image of $M_\infty\eqref{eqn:twosteps}$ in $M'$ is contained in $\fm M'$. 

We claim that the map \eqref{eqn:projcover} factors through 
\begin{equation}\label{eqn:Pbar}
\oplus_{\kappa \in \JH(\gr^{c-2_\ell} P_\sigma) \cap W(\rhobar)} \ovl{P}_{\kappa}^{\oplus[\gr^{c-2_\ell} P_\sigma:\kappa]}
\end{equation}
in the notation of \eqref{equation:Pbar}. 
Let $\kappa \in \JH(\gr^{c-2_\ell} P_\sigma) \cap W(\rhobar)$. 
Suppose that $a = (a_j)_{j\in \cJ}$ is a tuple as in the definition of \eqref{equation:Pbar}, i.e., $a = 2_j$ for $j\in A(\kappa)$ or $a = 3_j$ for $j\notin A(\kappa)$. 
If $a = 3_j$ for $j\notin A(\kappa)$, then the image of $\rad^a P_{\kappa}$ in $P_\sigma$ is contained in $\rad^{c-2_\ell+3_j}P_\sigma \subset \rad^{|c|+1}P_\sigma$ by Proposition \ref{prop:filtrationcompatible} (note that $\kappa$ is $9$-deep by \cite[Lemma 4.2.13]{LLLM2}). 
If $a = 2_j$ for $j\in A(\kappa)$, then the image of $\rad^a P_{\kappa}$ in $P_\sigma$ is contained in $\rad^{c-2_\ell+2_j} P_\sigma$. 
Since $|c-2_\ell+2_j| = |c|$ and $j\in A(\rad^{c-2_\ell+2_j} P_\sigma)$, we conclude that $\rad^a P_{\kappa}$ maps to $0$ in the codomain of \eqref{eqn:projcover}. 

Now suppose that $b = (b_j)_{j\in \cJ}$ is a tuple as in the definition of \eqref{equation:Pbar}. 
In particular, $b_i = 1$ for some $i\notin A(\kappa)$. 
Then $A(\gr^{c-2_\ell+b} P_\sigma) \neq \emptyset$ (in fact $i \in A(\gr^{c-2_\ell+b} P_\sigma)$) so that the image of $\rad^b P_{\kappa}$ in $P_\sigma$ is contained in \[
\sum_{\substack{|d|=|c| \\ \#A(\gr^dP_\sigma) > \#A(\gr^cP_\sigma)}}\rad^d P_\sigma. 
\]
This establishes the claim. 

Now for any generic Serre weight $\kappa$ and $i\in A(\kappa)$, the image of $\rad^{2_i} P_{\kappa}$ in $\ovl{P}_{\kappa}$ is $0$. 
Thus the image of \eqref{eqn:twosteps} in \eqref{eqn:Pbar} is equal to the the image of \eqref{eqn:twosteps} where $j$ instead ranges over $j\notin A(\kappa)$. 
Thus the image of $M_\infty\eqref{eqn:twosteps}$ in $M_\infty\eqref{eqn:Pbar}$ is contained in $\fm M_\infty\eqref{eqn:Pbar}$ by Corollary \ref{cor:noness}\eqref{item:noness:Bigbroom}. 
We conclude that the image of $M_\infty\eqref{eqn:twosteps}$ in $M'$ is contained in $\fm M'$. 
This finishes the proof of the lemma when $A(\gr^c P_\sigma) = \emptyset$. 

For $A(\gr^c P_\sigma) \neq \emptyset$, we employ a similar argument using Corollary \ref{cor:noness}\eqref{item:noness:broom} instead of Corollary \ref{cor:noness}\eqref{item:noness:Bigbroom}. 
If $A(\gr^c P_\sigma) \neq \emptyset$, let $\ell \in A(\gr^c P_\sigma)$. 
Then we replace $c-2_\ell$ in the above argument with $c-1_\ell$. 
We otherwise define $I$, $N$, $M'$, and $P \ra N$ as before. 
In this case, the map \eqref{eqn:projcover} factors as 
\begin{equation}\label{eqn:onestep}
P \ra P/\rad^2 P \ra \rad^{|c|-1} P_\sigma/\big(\rad^{|c|+1}P_\sigma + \sum_{\substack{|d|=|c| \\ \#A(\gr^d P_\sigma) > \#A(\gr^c P_\sigma)}}\rad^d P_\sigma\big)
\end{equation}
since the codomain of \eqref{eqn:onestep} has a length two semisimple filtration. 
By alcove considerations, the preimage of $\gr^c P_\sigma$ in $P/\rad^2 P$ under the map in \eqref{eqn:onestep} is contained in 
\begin{equation}\label{eqn:onestepb}
\oplus_{\kappa \in \JH(\gr^{c-1_\ell}  P_\sigma) \cap W(\rhobar)} \gr^{1_\ell} P_{\kappa}^{\oplus[\gr^{c-1_\ell} P_\sigma:\kappa]} 
\end{equation}
(recall that $\ell\in A(c)$ was fixed above). 
As $\ell\notin A(\kappa)$, $M_\infty\eqref{eqn:onestepb}$ is contained in $\fm M_\infty(P/\rad^2 P)$ by Corollary \ref{cor:noness}\eqref{item:noness:broom}, noting that $P_{\kappa}/\rad^2 P_{\kappa}$ is a quotient of $P_{\kappa}^a$ with $a = (\widehat{2})_{j\in \cJ}$. 
\end{proof}

\begin{proof}[Proof of Theorem \ref{thm:mingen}]
First, we claim that for a tuple $c = (c_j)_{j\in \cJ}$ with $0\leq c_j \leq 7$ and 
\[
c_i \geq \begin{cases}
  2  & \text{ if } i\in A(\sigma) \\
  1 & \text{ if } i\notin A(\sigma)
\end{cases}
\]
for some $i\in \cJ$, 
\[
M_\infty(\rad^c P_\sigma) \subset \fm M_\infty(P_\sigma).
\]
Indeed, we can partially order ($\succ$) such tuples $c$ lexicographically using $|\bullet|$ and $\# A(\gr^{\bullet}P_\sigma)$ and proceed by induction. 
If $c_j = 7$ for all $j\in \cJ$, then the result follows from Lemma \ref{lemma:noness} (as the right hand side of \eqref{eqn:EGS} in this case is simply $M_\infty(P_\sigma)$). 
In general,  
\[
M_\infty(\rad^c P_\sigma) \subset \fm M_\infty(P_\sigma) + \sum_{c'\succ c} M_\infty(\rad^{c'} P_\sigma) \subset \fm M_\infty(P_\sigma)
\] 
where the first inclusion follows from Lemma \ref{lemma:noness} and the second inclusion follows by the inductive hypothesis. 

Now let $a = (a_j)_{j\in \cJ}$ be the tuple with 
\[
a_j = \begin{cases}
  \widehat{2}  & \text{ if } j\in A(\sigma) \\
  \widehat{1} & \text{ if } j\notin A(\sigma).
\end{cases}
\]
By the previous paragraph, the map $P_\sigma \ra P_\sigma^a$ induces an isomorphism $M_\infty(P_\sigma)/\fm \ra M_\infty(P_\sigma^a)/\fm$. 
By Lemma \ref{lemma:broom}\eqref{item:broom2}, $M_\infty(P_\sigma^a)/\fm$ is minimally generated by $3^{\# A(\sigma)}$ elements. 
\end{proof}

\subsection{Proof of Lemma \ref{lemma:broom}}
\label{sec:pf:lemma:broom}

We prove Lemma \ref{lemma:broom} using intersection computations in multitype deformation spaces. 
Throughout this section we are in the setup of Lemma \ref{lemma:broom}.
Thus, $\rhobar$ is a $11$-generic tame $L$-homomorphism and $\lambda\in X_1(\un{T})$ is such that $\sigma\defeq F(\lambda) \in W^?(\rhobar)$.
We fix a point $\ovl{x}_{\rhobar}\in M^T(\F)$ corresponding to $\rhobar$ (see \S \ref{subsub:gpd}) and we fix a weak minimal patching functor $M_\infty$ for $\rhobar$.

\subsubsection{Preliminaries}
\label{sec:glue:prel}
We introduce some notation that will be used in the proofs appearing in this section. 
Let $T$ be a finite set of tame inertial types satyisfying Hypothesis \ref{hypothesis:T} and recall from \S \ref{subsub:gpd} that $\rhobar$ gives rise to a point $\ovl{x}_{\rhobar}\in M^T(\F)$ which we now fix.
By Definition \ref{def:M_infty} and \eqref{eq:morp} we have
\begin{equation}
\label{eq:cl:imm}
M^{T,\nabla_\infty}_{\ovl{x}_{\rhobar}}\cdot \tld{w}^{*,T}(\rhobar)\into M^{T}_{\ovl{x}_{\rhobar}}\cdot \tld{w}^{*,T}(\rhobar)\ra{\Phi\text{-}\Mod^{\textnormal{\'et},3}_{K,\rhobar_\infty}}. 
\end{equation}
Recall that $\tld{S} = \widehat{\otimes}_{j\in\cJ,\cO}\tld{S}^{(j)}$ is the formal power series ring on the natural coordinates of $M^T_{\ovl{x}_{\rhobar}}\cdot \tld{w}^{*,T}(\rhobar)$ and hence $M^{T,\nabla_\infty}_{\ovl{x}_{\rhobar}}\cdot \tld{w}^{*,T}(\rhobar)$ corresponds to a quotient $\tld{S}^{\nabla_\infty}$ of $\tld{S}$. 
Pulling back to $\Spf\,R^\Box_{\rhobar_\infty}$ along \eqref{eq:cl:imm} 
gives a formally smooth map $R_{\rhobar}^{\leqeta,T} \ra \tld{S}^{\nabla_\infty,\Box}$, where $\tld{S}^\Box$ is a suitable power series ring over $\tld{S}$ and $\tld{S}^\Box \onto \tld{S}^{\nabla_\infty,\Box}$ is the pullback of $\tld{S} \onto \tld{S}^{\nabla_\infty}$. 

If $V$ is a finite $\cO[\![K]\!]$-module such that the $R^{\Box}_{\rhobar}$-action on $M_\infty(V)$ factors through $R^{\leqeta,T}_{\rhobar}$, then we define 
\begin{equation}\label{eqn:Minfty'}
M_\infty'(V) \defeq M_\infty(V) \widehat{\otimes}_{R^{\leqeta,T}_{\rhobar}} \tld{S}^{\nabla_\infty,\Box}. 
\end{equation}
Let $R'_\infty(T) \defeq R_\infty(T) \widehat{\otimes}_{R^{\leqeta,T}_{\rhobar}} \tld{S}^{\nabla_\infty,\Box}$. 
If $T \subset T'$ are as in Proposition \ref{prop:big_diagram} and the $R_\infty$-action on $M_\infty(V)$ factors through $R_\infty(T)$ then the $\tld{S}^{\nabla_\infty,\Box}$-actions defined with respect to $T$ and $T'$ on $M'_\infty(V)$ are compatible by Remark \ref{rmk:compatibility:multitype}.

As $R'_\infty(T)$ is a formally smooth $\tld{S}^{\nabla_\infty,\Box}$-algebra, we can choose an isomorphism $\ovl{R}_\infty'(T) \cong {S}^{\nabla_\infty,\Box} \widehat{\otimes}_{\F} \cA$ where  ${S}^{\nabla_\infty,\Box}\defeq \tld{S}^{\nabla_\infty,\Box}\otimes_{\cO}\F$ and  $\cA$ denotes a formally smooth $\F$-algebra. 
If $V$ is moreover an $\F$-vector space, then $M_\infty'(V)$ is a module for the formally smooth ${S}$-algebra ${S}^\Box_\infty\defeq{S}^\Box \widehat{\otimes}_{\F} \cA$ where ${S}^\Box\defeq \tld{S}^\Box \otimes_{\cO}\F$. 
When $M_\infty'(V)$ is obtained from an $S$-module by extension of scalars along $S\ra S^\Box_\infty\onto\ovl{R}_\infty'(T)$, calculations in ${S}$ can be substituted for those in ${S}^\Box_\infty$. 
We caution that we have suppressed the dependence on the set $T$ in the notation $\tld{S}$, $\tld{S}^{\Box}$, $S^{\nabla_\infty,\Box}$, $S_\infty^\Box$, etc. 
In practice, the set $T$ will change from from proof to proof. 
When this notation is used, we will indicate which $T$ we take. 

Finally, note that the genericity assumption on $\rhobar$ implies that any $\tau\in T$ is $9$-generic (see for instance \cite[\S 5.5]{MLM}) and any $\kappa\in W^?(\rhobar)$ is $9$-deep (\cite[Lemma 4.2.13]{LLLM2}), so that in particular \cite[Proposition 5.3.2]{GL3Wild} (which improves \cite[Theorem 4.1.9]{LLLM2}) and Proposition \ref{prop:W?-cover} apply.

Recall that we have fixed $\sigma=F(\lambda)\in W^?(\rhobar)$.
Let $\tau$ be a tame inertial type such that $\sigma\in \JH(\ovl{\sigma}(\tau))$ and let $\ell\notin A(\sigma)$.
For $w \in S_3$, let $\tau_{w,\ell}$ be the tame inertial $L$-parameter such that $\tld{w}(\rhobar,\tau_{w,\ell})_j = \tld{w}(\rhobar,\tau)_j$ for all $j \neq \ell$ and $\tld{w}(\rhobar,\tau_{w,\ell})_\ell = ww_0\tld{w}(\rhobar,\tau)_\ell$. 
Given a subset $\Sigma \subset S_3$, let $T_{\Sigma,\ell}$ be the set $\{\tau_{w,\ell} \mid w\in \Sigma\}$ and $\sigma(T_{\Sigma,\ell})$ the set $\{\sigma(\tau_{w,\ell}) \mid w\in \Sigma\}$ of Deligne--Lusztig representations. 
Note that we suppressed the type $\tau$ from the notation of $T_{\Sigma,\ell}$ and $\sigma(T_{\Sigma,\ell})$.

\subsubsection{Cyclic patched modules}\label{subsec:cyclic}

In this subsection we assume that $A(\sigma) \subsetneq \cJ$ and fix $i\notin A(\sigma)$. 
The main result of this subsection is the following proposition. 

\begin{prop}\label{prop:3layer}
Let $a = (a_j)_{j\in \cJ}$ be the tuple with $a_i = \widehat{3}$ and $a_j = \widehat{1}$ for all $j\neq i$. 
Then $M_\infty(P^a_\sigma)$ is a cyclic $R_\infty$-module. 
\end{prop}

Let $\tau_{\min}$ be the minimal tame inertial $L$-parameter of $\sigma$ with respect to $\rhobar$ in the sense of \cite[Remark 3.5.10]{LLLM2}. 
Up to changing lowest alcove presentations of $\rhobar$, we assume that $t_{-\un{1}}\tld{w}(\rhobar,\tau_{\min})_i = t_{w_0\eta}$ or $w_0$. 
In the remainder of this subsection, given a subset $\Sigma\subset S_3$ and $w\in \Sigma$, define $\tau_w$, $T_\Sigma$ and $\sigma(T_\Sigma)$  as in \S \ref{sec:glue:prel} with respect to $\tau=\tau_{\min}$ and  $\ell=i$, {omitting the subscript $i$ from the notation for readability.}
Let $W$ be the Weyl module over $\F$ with cosocle isomorphic to $\sigma$. 
Then there are surjections $P_\sigma^{\sigma(\tau_w)} \onto W$ for all $w\in S_3$ by \cite[Proposition 5.3.2]{GL3Wild}. 
These surjections are unique up to scalar. 
This induces a unique up to scalar surjection $P_\sigma^{\sigma(T_\Sigma)} \ra W$ for any nonempty $\Sigma\subset S_3$. 

We will repeatedly use the following result.

\begin{prop}\label{lemma:fusion}
Let $\tld{P}_\sigma \onto P_1$ and $\tld{P}_\sigma \onto P_2$ be nonzero surjections. 
Then there is an $\cO[\rG]$-module $C$ and an exact sequence 
\[
\tld{P}_\sigma \ra P_1 \oplus P_2 \ra C \ra 0
\]
where the restriction of the second map to each summand is surjective. 

Moreover, $C$ satisfies the following property. 
If $D$ is an $\cO[\rG]$-module such that $\sigma$ is a Jordan--H\"older factor of $D$ with multiplicity one and there are surjections $P_1 \onto D$ and $P_2 \onto D$, then there is a surjection $C \onto D$. 
\end{prop}
\begin{proof}
We let $C$ be the cokernel of the induced map $\tld{P}_\sigma \ra P_1 \oplus P_2$. 
The induced maps $P_1\ra C$ and $P_2\ra C$ are surjective since they are surjective after passing to cosocles (using that formation of cosocle is right exact). 

Now suppose that we have $D$ and surjections $P_1 \onto D$ and $P_2 \onto D$ as in the statement of the proposition. 
After rescaling the surjection $P_1 \ra D$, we can assume that the two compositions $\tld{P}_\sigma \ra D$ induce the same maps on cosocles. 
Thus their difference factors through $\rad\, D$. 
Since $\sigma$ is not a Jordan--H\"older factor of $\rad\, D$, this difference is $0$. 
In other words, taking a difference, we get a surjective map $\delta: P_1 \oplus P_2 \ra D$ whose composition with the map $\tld{P}_\sigma \ra P_1 \oplus P_2$ is $0$. 
This gives the desired surjection $C\onto D$. 
\end{proof}

\begin{lemma}\label{lemma:2type}
With $\rhobar$, $\sigma$, and $T_\Sigma$ as above, $M_\infty(P_\sigma^{\sigma(T_{w_0,\alpha\beta})})$ and $M_\infty(P_\sigma^{\sigma(T_{w_0,\beta\alpha})})$ are cyclic $R_\infty$-modules. 
\end{lemma}
\begin{proof}
As the two cases are similar, we prove the proposition for $T\defeq T_{w_0,\alpha\beta}$. 
We claim that $M_\infty(P_\sigma^{\sigma(T)})$ is isomorphic to $R_\infty/I_\infty(\tau_{\alpha\beta}) \cap I_\infty(\tau_{w_0})$. 
Proposition \ref{lemma:fusion} gives an exact sequence 
\[
0\ra \tld{P}_\sigma^{\sigma(T)} \ra \tld{P}_\sigma^{\sigma(\tau_{\alpha\beta})} \oplus \tld{P}_\sigma^{\sigma(\tau_{w_0})} \ra C \ra 0. 
\]
Indeed, since the maps $\tld{P}_\sigma^{\sigma(\tau_{\alpha\beta})} \ra \tld{P}_\sigma^{\sigma(\tau_{\alpha\beta})}\otimes_{\cO} E$ and $\tld{P}_\sigma^{\sigma(\tau_{w_0})} \ra \tld{P}_\sigma^{\sigma(\tau_{w_0})}\otimes_{\cO} E$ are injective, the image of $\tld{P}_\sigma$ in $\tld{P}_\sigma^{\sigma(\tau_{\alpha\beta})} \oplus \tld{P}_\sigma^{\sigma(\tau_{w_0})}$ is isomorphic to its image in $\tld{P}_\sigma^{\sigma(\tau_{\alpha\beta})}\otimes_{\cO} E \oplus \tld{P}_\sigma^{\sigma(\tau_{w_0})}\otimes_{\cO} E$. 
Applying the exact functor $M_\infty$, we have the exact sequence
\[
0 \ra M_\infty(\tld{P}_\sigma^{\sigma(T)}) \ra M_\infty(\tld{P}_\sigma^{\sigma(\tau_{\alpha\beta})}) \oplus M_\infty(\tld{P}_\sigma^{\sigma(\tau_{w_0})}) \ra M_\infty(C) \ra 0. 
\]
Moreover, the restriction of the third map to each summand is a surjection. 
By \cite[Theorem 5.3.1]{GL3Wild} (which improves the genericity condition in \cite[Theorem 5.1.1]{LLLM2}), $M_\infty(\tld{P}_\sigma^{\sigma(\tau_w)})$ is isomorphic to $R_\infty(\tau_w)$ for all $w\in S_3$. 
We can choose isomorphisms to obtain an exact sequence 
\[
0 \ra M_\infty(\tld{P}_\sigma^{\sigma(T)}) \ra R_\infty/I_\infty(\tau_{\alpha\beta}) \oplus R_\infty/I_\infty(\tau_{w_0}) \ra R_\infty/I_\infty(C) \ra 0 
\]
for some ideal $I_\infty(C)$ with $I_\infty(\tau_{\alpha\beta})+I_\infty(\tau_{w_0})\subset I_\infty(C)$ where the third map is the difference of the natural surjections. 
By Lemma \ref{lemma:CAfusion}, it suffices to show that $I_\infty(C) \subset I_\infty(\tau_{\alpha\beta}) + I_\infty(\tau_{w_0})$. 

Then there are surjections $P_\sigma^{\sigma(\tau_{\alpha\beta})} \onto W$ and $P_\sigma^{\sigma(\tau_{w_0})} \onto W$ and thus a surjection $C \onto W$ by the second part of Proposition \ref{lemma:fusion}. 
The induced map $P_\sigma^{\sigma(\tau_{w_0})} \onto W$ becomes an isomorphism after applying $M_\infty(-)$ by \cite[Proposition 5.3.2]{GL3Wild} and \cite[Corollary 2.3.11, Theorem 3.5.2]{LLLM2}.
We conclude that $I_\infty(C) \subset (p) + I_\infty(\tau_{w_0})$. 
By Proposition \ref{prop:ideal2type}, $(p) + I_\infty(\tau_{w_0}) \subset I_\infty(\tau_{\alpha\beta}) + I_\infty(\tau_{w_0})$ so that $I_\infty(C) \subset I_\infty(\tau_{\alpha\beta}) + I_\infty(\tau_{w_0})$. 
\end{proof}
Let $\ovl{P}^{\sigma(T_{w_0,\alpha\beta,\beta\alpha})}_\sigma$ be the image of the map $P_\sigma \ra P^{\sigma(T_{w_0,\alpha\beta})}_\sigma \oplus P^{\sigma(T_{w_0,\beta\alpha})}_\sigma$ induced by the natural surjections $P_\sigma \ra  P^{\sigma(T_{w_0,\beta\alpha})}_\sigma$ and $P_\sigma \ra P^{\sigma(T_{w_0,\alpha\beta})}_\sigma $. 

\begin{lemma}\label{lemma:3type}
With $\rhobar$, $\sigma$, $\ovl{P}^{\sigma(T_{w_0,\alpha\beta,\beta\alpha})}_\sigma$ as above, $M_\infty(\ovl{P}^{\sigma(T_{w_0,\alpha\beta,\beta\alpha})}_\sigma)$ is a cyclic $R_\infty$-module. 
\end{lemma}
\begin{proof}
We have an exact sequence
\[
0 \ra \ovl{P}^{\sigma(T_{w_0,\alpha\beta,\beta\alpha})}_\sigma \ra P^{\sigma(T_{w_0,\alpha\beta})}_\sigma \oplus P^{\sigma(T_{w_0,\beta\alpha})}_\sigma \ra C \ra 0
\]
for some $C$ as in Proposition \ref{lemma:fusion}. 
Since there are surjective maps $P^{\sigma(T_{w_0,\alpha\beta})}_\sigma\ra W$ and $P^{\sigma(T_{w_0,\beta\alpha})}_\sigma\ra W$, 
the second part of Proposition \ref{lemma:fusion} gives a surjective map $C \ra W$. 
This gives an exact sequence 
\[
0 \ra \ovl{P}^{\prime}_\sigma \ra P^{\sigma(T_{w_0,\alpha\beta})}_\sigma \oplus P^{\sigma(T_{w_0,\beta\alpha})}_\sigma \ra W \ra 0
\]
where the third map is a difference of the surjections and $\ovl{P}^{\prime}_\sigma$ is the kernel of this map. 
Then we have $\ovl{P}^{\sigma(T_{w_0,\alpha\beta,\beta\alpha})}_\sigma\subset \ovl{P}^{\prime}_\sigma$. 

Applying $M_\infty$ to the second exact sequence, we have
\[
0 \ra M_\infty(\ovl{P}^{\prime}_\sigma) \ra M_\infty(P^{\sigma(T_{w_0,\alpha\beta})}_\sigma) \oplus M_\infty(P^{\sigma(T_{w_0,\beta\alpha})}_\sigma) \ra M_\infty(W) \ra 0. 
\]
Let $\Sigma$ be $\{\alpha\beta,\beta\alpha,w_0, \Id\}$ and define $M_\infty'(-)$ using $T = T_\Sigma$ (see \eqref{eqn:Minfty'}). 
By Lemma \ref{lemma:2type} and its proof (which shows that $M_\infty(P_\sigma^{\sigma(\tau_{w_0})})\cong M_\infty(W)$), we can and do choose isomorphisms to get the exact sequence 
\[
0 \ra M_\infty'(\ovl{P}^{T,\prime}_\sigma) \ra R_\infty'(T)/\big(\tld{I}_{\tau_{w_0},\nabla_\infty}\cap \tld{I}_{\tau_{\alpha\beta},\nabla_\infty},p\big) \oplus R_\infty'(T)/\big(\tld{I}_{\tau_{w_0},\nabla_\infty}\cap \tld{I}_{\tau_{\beta\alpha},\nabla_\infty},p\big) \ra R_\infty'(T)/\big(\tld{I}_{\tau_{w_0},\nabla_\infty},p\big) \ra 0. 
\]
As the union of the images of $\Tor^{S^\Box_\infty}_1\Big(\F,R_\infty'(T)/\big(\tld{I}_{\tau_{w_0},\nabla_\infty}\cap \tld{I}_{\tau_{\alpha\beta},\nabla_\infty},p\big)\Big)$ and $\Tor^{S^\Box_\infty}_1\Big(\F,R_\infty'(T)/\big(\tld{I}_{\tau_{w_0},\nabla_\infty}\cap \tld{I}_{\tau_{\beta\alpha},\nabla_\infty},p\big)\Big)$ in $\Tor^{S^\Box_\infty}_1\Big(\F,R_\infty'(T)/\big(\tld{I}_{\tau_{w_0},\nabla_\infty},p\big)\Big)$ is spanning by Lemma \ref{lemma:ideal3types}, Lemma \ref{lemma:CAfusion} implies that $M_\infty'(\ovl{P}^{\prime}_\sigma)$ is a cyclic $R_\infty'(T)$-module so that $M_\infty(\ovl{P}^{\prime}_\sigma)$ is a cyclic $R_\infty$-module. 

Finally, we claim that the inclusion $M_\infty(\ovl{P}^{\sigma(T_{w_0,\alpha\beta,\beta\alpha})}_\sigma) \into M_\infty(\ovl{P}^{\prime}_\sigma)$ is an isomorphism. 
As $M_\infty(\ovl{P}^{\prime}_\sigma)$ is a cyclic $R_\infty$-module, it suffices to show that the map is nonzero after applying $-\otimes_{R_\infty} \F$ by Nakayama's lemma. 
This follows from the fact that applying the functor $M_\infty(-) \otimes_{R_\infty} \F$ to the composition $\ovl{P}^{\sigma(T_{w_0,\alpha\beta,\beta\alpha})}_\sigma \subset \ovl{P}^{\prime}_\sigma \onto P_\sigma^{\sigma(T_{w_0,\alpha\beta})}\onto \sigma$ (with all maps the natural ones) induces a surjection and that $M_\infty(\sigma)\neq 0$. 
\end{proof}

Let $\ovl{P}_\sigma^{\sigma(T_{\Id,w_0,\alpha\beta,\beta\alpha})}$ be the image of the map $P_\sigma \ra P_\sigma^{\sigma(\tau_{\Id})} \oplus \ovl{P}_\sigma^{\sigma(T_{w_0,\alpha\beta,\beta\alpha})}$ induced by the natural surjections $P_\sigma \ra  P_\sigma^{\sigma(\tau_{\Id})}$ and $P_\sigma \ra \ovl{P}_\sigma^{\sigma(T_{w_0,\alpha\beta,\beta\alpha})}$. 
Equivalently, $\ovl{P}_\sigma^{\sigma(T_{\Id,w_0,\alpha\beta,\beta\alpha})}$ is the image of the map 
\[
P_\sigma \ra P_\sigma^{\{R_{\alpha\beta},R_{w_0}\}} \oplus P_\sigma^{\{R_{w_0},R_{\beta\alpha}\}} \oplus P_\sigma^{R_{\Id}} 
\]
appearing in \eqref{eqn:4DL}. 

\begin{lemma}\label{lemma:4type}
With $\rhobar$, $\sigma$, $\ovl{P}_\sigma^{\sigma(T_{\Id,w_0,\alpha\beta,\beta\alpha})}$ as above, $M_\infty(\ovl{P}_\sigma^{\sigma(T_{\Id,w_0,\alpha\beta,\beta\alpha})})$ is a cyclic $R_\infty$-module. 
\end{lemma}
\begin{proof}
The proof is similar to that of Lemma \ref{lemma:3type}. 
We have an exact sequence
\[
0 \ra \ovl{P}_\sigma^{\sigma(T_{\Id,w_0,\alpha\beta,\beta\alpha})} \ra P_\sigma^{\sigma(\tau_{\Id})} \oplus \ovl{P}_\sigma^{\sigma(T_{w_0,\alpha\beta,\beta\alpha})} \ra C \ra 0
\]
for some $C$ as in Proposition \ref{lemma:fusion}. 
Let $\lambda \in X_1(\un{T})$ such that $\sigma = F(\lambda)$. 
By \cite[Proposition 5.3.2]{GL3Wild}, there exists a quotient $\Lambda$ of $P_\sigma^{\sigma(\tau_{\Id})}$ whose Jordan--H\"older factors are precisely $\kappa = \Trns_{\lambda_A+\eta}(\omega,a) \in P_\sigma^{\sigma(\tau_{\Id})}$ with $(\omega_i,a_i) \in \{(0,0),(\eps_1,0),(\eps_2,0),(\eps_1-\eps_2,0),(\eps_2-\eps_1,0),(0,1)\}$. 
Though we will not use it, one can check that $\Lambda$ is the cokernel of the composition
\[
\rad^{2_i} P_\sigma\subset P_\sigma \ra P_\sigma^{\sigma(\tau_{\Id})}. 
\] 
We will show that the natural surjection $P_\sigma\onto \Lambda$ factors through $\ovl{P}_\sigma^{\sigma(T_{w_0,\alpha\beta,\beta\alpha})}$. 
Again by \cite[Proposition 5.3.2]{GL3Wild}, for each $w \in \{w_0,\alpha\beta,\beta\alpha\}$ there is a unique quotient $Q_w$ of $P_\sigma^{\sigma(\tau_w)}$ whose Jordan--H\"older factors are precisely $\JH(P_\sigma^{\sigma(\tau_w)}) \cap \JH(\Lambda)$. 
Moreover, there is a surjection $\Lambda \ra Q_w$ whose kernel we denote $K_w$. 
Then the image of $N_w\defeq\ker(P_\sigma\onto P_\sigma^{\sigma(\tau_w)})$ in $\Lambda$ is contained in $K_w$. 
As $\JH(K_{\alpha\beta}) \cap \JH(K_{\beta,\alpha})$ is empty, the intersection of the images of $N_{\alpha\beta}$ and $N_{\beta\alpha}$ in $\Lambda$ is $0$. 
Since each $N_w$ contains the kernel of the map $P_\sigma\onto \ovl{P}_\sigma^{\sigma(T_{w_0,\alpha\beta,\beta\alpha})}$, we obtain the desired factorization. 
Then the second part of Proposition \ref{lemma:fusion} gives a surjective map $C \ra \Lambda$ and an exact sequence 
\begin{equation}
\label{eq:exact_seq_broom}
0 \ra \ovl{P}'_\sigma \ra P_\sigma^{\sigma(\tau_{\Id})} \oplus \ovl{P}_\sigma^{\sigma(T_{w_0,\alpha\beta,\beta\alpha})} \ra \Lambda \ra 0
\end{equation}
where the third map is a difference of the surjections and $\ovl{P}'_\sigma$ is the kernel of this map (which is different from what is denoted $\ovl{P}'_\sigma$ in the proof of Lemma \ref{lemma:3type}). 
Then we have $\ovl{P}_\sigma^{\sigma(T_{\Id,w_0,\alpha\beta,\beta\alpha})} \subset \ovl{P}'_\sigma$. 

Applying $M_\infty$ to \eqref{eq:exact_seq_broom}, we have
\[
0 \ra M_\infty(\ovl{P}'_\sigma) \ra M_\infty(P_\sigma^{\sigma(\tau_{\Id})}) \oplus M_\infty(\ovl{P}_\sigma^{\sigma(T_{w_0,\alpha\beta,\beta\alpha})}) \ra M_\infty(\Lambda) \ra 0.
\]
By \cite[Theorem 5.3.1]{GL3Wild}, $M_\infty(P_\sigma^{\sigma(\tau_{\Id})})$ is a cyclic $R_\infty$-module. 
This implies that $M_\infty(\Lambda)$ is a cyclic $R_\infty$-module as well.  	
Let $T$ and $M_\infty'$ as in the proof of Lemma \ref{lemma:3type}. 
By the proof of \cite[Lemma 3.6.2]{LLLM2} the annihilator of $M'_\infty(\Lambda)$ corresponds to the ideal $(\tld{I}_\Lambda,p)\subseteq \tld{S}$, with $\tld{I}_\Lambda$ defined as in \S \ref{subsub:4types}.
Let $I'_{\infty}(T_{w_0,\alpha\beta,\beta\alpha})$ be the image of $
(\tld{I}_{\tau_{w_0},\nabla_\infty}\cap \tld{I}_{\tau_{\alpha\beta},\nabla_\infty},p)\cap(\tld{I}_{\tau_{w_0},\nabla_\infty}\cap\tld{I}_{\tau_{\beta\alpha},\nabla_\infty},p)\subseteq \tld{S}$ in $\tld{S}^{\nabla_\infty,\Box}$.
By Lemma \ref{lemma:CAfusion} $M'_\infty(\ovl{P}^{\sigma(T_{w_0,\alpha\beta,\beta\alpha})}_\sigma)$ is isomorphic to $R_\infty'(T)/I'_{\infty}(T_{w_0,\alpha\beta,\beta\alpha})$ and we can and do choose isomorphisms so that we have the exact sequence 
\begin{align*}
0 \ra M_\infty'(\ovl{P}'_\sigma) \ra R_\infty'(T)/\big(\tld{I}_{\tau_{\Id}\nabla_\infty},p\big) \oplus R_\infty'(T)/I'_{\infty}(T_{w_0,\alpha\beta,\beta\alpha})
\ra R_\infty'(T)/\big(\tld{I}_{\Lambda},p\big) \ra 0
\end{align*}
where the third map is the difference of the natural surjections. 
Since the union of the images of $\Tor^{{S}^\Box_\infty}_1(\F,R_\infty'(T)/\big(\tld{I}_{\tau_{\Id}\nabla_\infty},p\big))$ and $\Tor^{{S}^\Box_\infty}_1(\F,R_\infty'(T)/I'_{\infty}(T_{w_0,\alpha\beta,\beta\alpha}))$ in $\Tor^{{S}^\Box_\infty}_1(\F,R_\infty'(T)/(\tld{I}_\Lambda,p))$ is spanning by Lemma \ref{lemma:ideal4types}, Lemma \ref{lemma:CAfusion} implies that $M_\infty(\ovl{P}'_\sigma)$ is a cyclic $R_\infty$-module. 
Then one shows that $M_\infty(\ovl{P}_\sigma^{\sigma(T_{\Id,w_0,\alpha\beta,\beta\alpha})}) = M_\infty(\ovl{P}'_\sigma)$ as in the proof of Lemma \ref{lemma:3type}. 
\end{proof}

\begin{proof}[Proof of Proposition \ref{prop:3layer}]
By Proposition \ref{prop:multiDL}, with $a$ as in the statement of Proposition \ref{prop:3layer} (note that $P^a_\sigma = P_\sigma/\Fil^{3_i}P_\sigma$), the surjection $P_\sigma \ra P^a_\sigma$ factors through $\ovl{P}_\sigma^{\sigma(T_{\Id,w_0,\alpha\beta,\beta\alpha})}$. 
Since $M_\infty(\ovl{P}_\sigma^{\sigma(T_{\Id,w_0,\alpha\beta,\beta\alpha})})$ is a cyclic $R_\infty$-module by Lemma \ref{lemma:4type}, we conclude that $M_\infty(P^a_\sigma)$ is as well. 
\end{proof}

\subsubsection{A $\Tor$ computation}\label{sec:tor}

In this subsection, fix $i\notin A(\sigma)$ and let $a = (a_j)_{j\in \cJ}$ be the tuple with $a_i = \widehat{3}$ and $a_j = \widehat{1}$ for all $j\neq i$. 
In order to prove Lemma \ref{lemma:broom}, we will need to find a lower bound for the image of the map
\[
\Tor_1^{S_\infty^\Box}(\F,M_\infty'(P_\sigma^a)) \ra \Tor_1^{S_\infty^\Box}(\F,M_\infty'(\sigma))
\]
induced by the surjection $P_\sigma^a \onto \sigma$ (see Lemma \ref{lemma:4type}). 
Let $P_\sigma^{\sigma(T_{w_0,\alpha\beta}),a}$ be the cokernel of the composition
\[
\Fil^{3_i}P_\sigma\subset P_\sigma \onto P_\sigma^{\sigma(T_{w_0,\alpha\beta})}
\]
where $P_\sigma^{\sigma(T_{w_0,\alpha\beta})}$ is defined as in \S \ref{subsec:cyclic}.
Similarly, we define $P_\sigma^{\sigma(\tau_\Id),a}$, $\ovl{P}_\sigma^{\sigma(T_{w_0,\alpha\beta,\beta\alpha}),a}$, $\ovl{P}_\sigma^{\sigma(T_{\Id,w_0,\alpha\beta,\beta\alpha}),a}$, and $\Lambda^a$ (with $\Lambda$ defined as in the proof of Lemma \ref{lemma:4type}).

\begin{lemma}\label{lemma:aexact}
We have an exact sequence 
\[
0 \ra M_\infty(\ovl{P}_\sigma^{\sigma(T_{\Id,w_0,\alpha\beta,\beta\alpha}),a}) \ra M_\infty(P_\sigma^{\sigma(\tau_\Id),a}) \oplus M_\infty(\ovl{P}_\sigma^{\sigma(T_{w_0,\alpha\beta,\beta\alpha}),a}) \ra M_\infty(\Lambda^a) \ra 0.
\]
\end{lemma}
\begin{proof}

We have a commutative diagram 
\[
\begin{tikzcd}
0 \arrow[r] 
& \Fil^{3_i}P_\sigma\arrow[r] \arrow[d]
& \Fil^{3_i}P_\sigma\oplus \Fil^{3_i}P_\sigma \arrow[r] \arrow[d]
& \Fil^{3_i}P_\sigma \arrow[r] \arrow[d]
& 0 \\
0 \arrow[r]
& \ovl{P}_\sigma^{\sigma(T_{\Id,w_0,\alpha\beta,\beta\alpha})} \arrow[r]
&P_\sigma^{\sigma(\tau_\Id)} \oplus \ovl{P}_\sigma^{\sigma(T_{w_0,\alpha\beta,\beta\alpha})} \arrow[r]
& C \arrow[r]
& 0
\end{tikzcd}
\]
where the rows are exact, the bottom row is as in Lemma \ref{lemma:4type}, and the top row has nonzero maps given by diagonal and difference maps. 
Since $\ker(\Fil^{3_i}P_\sigma\oplus \Fil^{3_i}P_\sigma)\ra \ker(\Fil^{3_i}P_\sigma)$ is surjective (as both $P_\sigma^{\sigma(\tau_\Id)}\ra C$, $\ovl{P}_\sigma^{\sigma(T_{w_0,\alpha\beta,\beta\alpha})}\ra C$ are factorizations of $P_\sigma\onto C$), the snake lemma furnishes an exact sequence 
\[
0 \ra \ovl{P}_\sigma^{\sigma(T_{\Id,w_0,\alpha\beta,\beta\alpha}),a} \ra P_\sigma^{\sigma(\tau_\Id),a} \oplus \ovl{P}_\sigma^{\sigma(T_{w_0,\alpha\beta,\beta\alpha}),a} \ra C^a \ra 0
\]
where $C^a$ denotes the cokernel of the natural map $\Fil^{3_i}P_\sigma \ra C$. 
The result follows by applying $M_\infty(-)$ and noting that the natural map $M_\infty(C) \ra M_\infty(\Lambda)$ is an isomorphism by the proof of Lemma \ref{lemma:4type}. 
\end{proof}

The following is the main result of the subsection. 

\begin{lemma}\label{lemma:tor}
Recall from \eqref{eqn:Minfty'} the definition of $M_\infty'(-)$ with respect to $T_{\Id,w_0,\alpha\beta,\beta\alpha}$. 
The image of $\Tor_1(\F,M_\infty'(\ovl{P}_\sigma^{\sigma(T_{\Id,w_0,\alpha\beta,\beta\alpha}),a})) \ra \Tor_1(\F,M_\infty'(\Lambda^a))$ is the intersection of the images of the maps
\[
\Tor_1^{S_\infty^\Box}(\F,M_\infty'(\ovl{P}_\sigma^{\sigma(T_{w_0,\alpha\beta,\beta\alpha}),a})) \ra \Tor_1^{S_\infty^\Box}(\F,M_\infty'(\Lambda^a))
\]
and
\[
\Tor_1^{S_\infty^\Box}(\F,M_\infty'(P_\sigma^{\sigma(\tau_\Id),a})) \ra \Tor_1^{S_\infty^\Box}(\F,M_\infty'(\Lambda^a)).
\]
\end{lemma}
\begin{proof}
The result follows from combining Lemmas \ref{lemma:aexact} and \ref{lemma:CAfusion} using that $M_\infty(\ovl{P}_\sigma^{\sigma(T_{\Id,w_0,\alpha\beta,\beta\alpha}),a})$ is a cyclic $R_\infty$-module by Lemma \ref{lemma:4type} so that $M_\infty'(\ovl{P}_\sigma^{\sigma(T_{\Id,w_0,\alpha\beta,\beta\alpha}),a})$ is a cyclic $S_\infty^\Box$-module. 
\end{proof}

\subsubsection{Noncyclic patched modules}
\label{subsub:non-cyclic}
Recall that at the beginning of \S \ref{sec:pf:lemma:broom} we have fixed $\lambda\in X_1(\un{T})$ such that $\sigma=F(\lambda)\in W^?(\rhobar)$.
Throughout this subsection we let $\tau^{\textnormal{\tiny{$B$}}}_{\min}$ be the minimal tame inertial $L$-parameter of $F(\sum_{j}\lambda_{j,B})$ with respect to $\rhobar$ in the sense of \cite[Remark 3.5.10]{LLLM2}.
Note that, up to changing lowest alcove presentation of $\rhobar$, we can and do assume that $t_{-\un{1}}\tld{w}(\rhobar,\tau^{\textnormal{\tiny{$B$}}}_{\min})_j=t_{w_0\eta}$ or $t_{-\un{1}}\tld{w}(\rhobar,\tau^{\textnormal{\tiny{$B$}}}_{\min})_j=w_0$, for all $j\in\cJ$.

In this subsection, let $T$ be the set of tame inertial types $\tau$ with $\tld{w}(\rhobar,\tau)_j = w_jw_0\tld{w}(\rhobar,\tau^{\textnormal{\tiny{$B$}}}_{\min})_j$ with 
{$w_j \in \{\Id, \alpha\beta,\beta\alpha,w_0, t_{-\un{1}}t_{w_0\eta} \}$ if $t_{-\un{1}}\tld{w}(\rhobar,\tau^{\textnormal{\tiny{$B$}}}_{\min})_j=w_0$, and  $w_j \in \{\Id, \alpha\beta,\beta\alpha,w_0\}$ if $t_{-\un{1}}\tld{w}(\rhobar,\tau^{\textnormal{\tiny{$B$}}}_{\min})_j=t_{w_0\eta}$.
}
With this choice of $T$, we define $M_\infty'(-)$ as in \eqref{eqn:Minfty'} (with the implicit assertion that the $R_\infty$-action on $M_\infty(-)$ factors through $R_\infty(T)$). 
We also let $a$ denote a tuple $(a_j)_{j\in \cJ}$ with $a_j \subset \{B,E_o,F_o,E_s,F_s\}$ if $j\in A(\sigma)$ and $a_j = \widehat{1}$ or $\widehat{2}$ if $j\notin A(\sigma)$. 
For $j \in A(\sigma)$, $b_j$ will denote an element of $\{B,E_o,F_o,E_s,F_s\}$, $I_j^{b_j}$ is as defined in \S \ref{subsec:specialfiber}, and $M_j^{b_j} \defeq S^{(j)}/I_j^{b_j}$. 
From \cite[Lemma 5.3.3]{GL3Wild}, \cite[Lemma 3.6.2]{LLLM2} we deduce that $M_\infty'(\sigma) \cong \ovl{R}_\infty'(T)/(\fP_\sigma)$ where $\fP_\sigma = \sum_j\fP_\sigma^{(j)}S$ for prime ideals $\fP_\sigma^{(j)} \subset S^{(j)}$, and $\fP_\sigma$ is the pullback to $S$ of the ideal $\ovl{\fP}_{\sigma}\subset S/I_{T,\nabla_\infty}$ defined in \S \ref{sub:SF}. 
For $j\notin A(\sigma)$, let $M_j^{\widehat{1}} = S^{(j)}/\fP_\sigma^{(j)}$. 

\begin{prop}
\label{prop:bristle}
\begin{enumerate}
\item \label{item:bristle1} Let $b=(b_j)_{j\in \cJ}$ with $b_j \in \{B,E_o,F_o,E_s,F_s\}$ if $j\in A(\sigma)$ and $b_j = \widehat{1}$ if $j\notin A(\sigma)$. 
Then $M_\infty'(P_\sigma^b) \cong (\widehat{\otimes}_j M_j^{b_j}) \widehat{\otimes}_{\tld{S}} R_\infty'(T)$ with $M_j^{b_j}$ defined as above. 
\item \label{item:bristle2} For $j \notin A(\sigma)$, there are ideals $I_j^{\widehat{2}}$ such that for any $b=(b_j)_{j\in \cJ}$ with $b_j \in \{B,E_o,F_o,E_s,F_s\}$ if $j\in A(\sigma)$ and $b_j = \widehat{2}$ if $j\notin A(\sigma)$, $M_\infty'(P_\sigma^b) \cong (\widehat{\otimes}_j M_j^{b_j}) \widehat{\otimes}_{\tld{S}} R_\infty'(T)$ where $M_j^{b_j} \defeq S^{(j)}/I_j^{\widehat{2}}$ if $j \notin A(\sigma)$. 
\end{enumerate}
\end{prop}
\begin{proof}
We prove \eqref{item:bristle2} as \eqref{item:bristle1} is similar but easier. 
There exists $\tau\in T$ with $\tld{w}(\rhobar,\tau)_j = w_0\tld{w}(\rhobar,\tau^{\textnormal{\tiny{$B$}}}_{\min})_j$ for $j\notin A(\sigma)$ and $\tld{w}(\rhobar,\tau)_j \neq w_0\tld{w}(\rhobar,\tau^{\textnormal{\tiny{$B$}}}_{\min})_j$ for $j\in A(\sigma)$ such that $\JH(P_\sigma^b) \cap W^?(\rhobar) \subset \JH(\ovl{\sigma}(\tau))$. 
By \cite[Proposition 5.3.2]{GL3Wild}, there is a quotient $Q$ of $P_\sigma^{\sigma(\tau)}$ such that $\JH(Q) = \JH(P_\sigma^b) \cap W^?(\rhobar)$. 
By Corollary \ref{cor:uniquewedge}, there is a surjection $P_\sigma^b \onto Q$. 
Since $\JH(Q) = \JH(P_\sigma^b) \cap W^?(\rhobar)$ and $P_\sigma^b$ is multiplicity free, the induced map $M_\infty(P_\sigma^b) \onto M_\infty(Q)$ is an isomorphism. 
In particular the $R_\infty$-action on $M_\infty(P_\sigma^b)$ factors through $R_\infty(T)$. 
As $M_\infty'(Q)$ is a cyclic $R_\infty'(T)$-module by \cite[Theorem 5.3.1]{GL3Wild}, so is $M_\infty'(P_\sigma^b)$. 
The result now follows from \cite[Lemma 3.6.2]{LLLM2}. 
\end{proof}

If $j\in A(\sigma)$, and $\emptyset\neq a_j \subset \{B,E_o,F_o,E_s,F_s\}$, define an $S^{(j)}$-module $M_j^{a_j}$ by the exact sequence 
\begin{equation}\label{eqn:wedgepresentation}
0 \ra M_j^{a_j} \ra \oplus_{b_j \in a_j} S^{(j)}/I_j^{b_j} \ra  (\oplus_{b_j \in a_j} S^{(j)}/\fP_\sigma^{(j)})/\Delta(S^{(j)}/\fP_\sigma^{(j)}) \ra 0,
\end{equation} 
where $I_j^{b_j}$ is as in \S \ref{subsec:specialfiber}, the third map is induced by the sum of the natural projections, and $\Delta(S^{(j)}/\fP_\sigma^{(j)})$ denotes the diagonally embedded submodule.

\begin{lemma}\label{lemma:gluelower}
Let $a$ be $(a_j)_{j\in \cJ}$ with $a_j \subset \{B,E_o,F_o,E_s,F_s\}$ if $j\in A(\sigma)$ and $a_j = \widehat{1}$ for all $j\notin A(\sigma)$ or $a_j = \widehat{2}$ for all $j\notin A(\sigma)$.
Then $M_\infty'(P^a_\sigma) \cong (\widehat{\otimes}_j M_j^{a_j}) \widehat{\otimes}_{\tld{S}} R_\infty'(T)$.
\end{lemma}
\begin{proof}
We induct on $k\defeq \#\{j\in A(\sigma): \#a_j > 1\}$.
The case $k=0$ follows from Proposition \ref{prop:bristle}.
Now suppose that $k>0$ and $\#a_i > 1$ for $i \in A(\sigma)$.
For $X \in a_i$, let $a_X$ be the tuple with $a_{X,i} = X$ and $a_{X,j} = a_j$ for $j\neq i$.
Then Proposition \ref{prop:lowerglue} gives an exact sequence
\[0\ra P^a_\sigma \ra \oplus_{X\in a_i} P^{a_X}_\sigma \ra (\oplus_{X\in a_i} P^{a_\emptyset}_\sigma)/\Delta(P^{a_\emptyset}_\sigma) \ra 0,\]
where the third map is induced by the sum of the natural projections.
Since the $R_\infty$-action on $M_\infty(P_\sigma^{a_X})$  factors through $R_\infty(T)$ for each $X\in a_i$ by the inductive hypothesis, the same is true for $M_\infty(P_\sigma^a)$. 
This induces the exact sequence 
\[0\ra M_\infty'(P^a_\sigma) \ra \oplus_{X\in a_i} M_\infty'(P^{a_X}_\sigma) \ra (\oplus_{X\in a_i} M_\infty'(P^{a_\emptyset}_\sigma))/\Delta(M_\infty'(P^{a_\emptyset}_\sigma)) \ra 0.\]
By the induction hypothesis, the modules $M_\infty'(P^{a_X}_\sigma)$ and $M_\infty'(P^{a_\emptyset}_\sigma)$ are isomorphic to modules $(\widehat{\otimes}_j M^{a_{X,j}}_j)\widehat{\otimes}_{\tld{S}} R_\infty'(T)$ and $(\widehat{\otimes}_j M^{a_{\emptyset,j}}_j)\widehat{\otimes}_{\tld{S}} R_\infty'(T)$, respectively. 
Thus we have an exact sequence
\begin{equation}\label{eqn:gluelower}
0 \ra M_\infty'(P^a_\sigma) \ra \oplus_{X\in a_i}(\widehat{\otimes}_j M^{a_{X,j}}_j)\widehat{\otimes}_{\tld{S}} R_\infty'(T) 
\ra (\oplus_{X\in a_i} (\widehat{\otimes}_j M^{a_{\emptyset,j}}_j) \widehat{\otimes}_{\tld{S}} R_\infty'(T))/\Delta \ra 0
\end{equation}
where $\Delta$ is short for $\Delta((\widehat{\otimes}_j M^{a_{\emptyset,j}}_j) \widehat{\otimes}_{\tld{S}} R_\infty'(T))$. 

The third map of \eqref{eqn:gluelower} is induced by a sum of surjective maps 
\begin{equation}\label{eqn:summandgluelower}(\widehat{\otimes}_j M^{a_{X,j}}_j)\widehat{\otimes}_{\tld{S}} R'_\infty(T) 
\onto (\widehat{\otimes}_j M^{a_{\emptyset,j}}_j) \widehat{\otimes}_{\tld{S}} R'_\infty(T)
\end{equation} 
for each $X\in a_i$. 
By consideration of scheme-theoretic supports, \eqref{eqn:summandgluelower} factors through the surjection 
\[(\widehat{\otimes}_j M^{a_{X,j}}_j)\widehat{\otimes}_{\tld{S}} R'_\infty(T) \onto (\widehat{\otimes}_j M^{a_{\emptyset,j}}_j)\widehat{\otimes}_{\tld{S}} R'_\infty(T)\] 
induced by $M_i^{a_{X,i}} \onto M_i^{a_{\emptyset,i}}$. 
The resulting surjective endomorphism of the Cohen--Macaulay module $(\widehat{\otimes}_j M^{a_{\emptyset,j}}_j)\widehat{\otimes}_{\tld{S}} R'_\infty(T)$ must be an isomorphism since the kernel of the map must have support of smaller dimension by cycle considerations and the Cohen--Macaulay module $(\widehat{\otimes}_j M^{a_{\emptyset,j}}_j)\widehat{\otimes}_{\tld{S}} R'_\infty(T)$ cannot have embedded primes by \cite[Theorem 17.3]{matsumura}. 
We conclude that the map \eqref{eqn:summandgluelower} is, up to postcomposing with an automorphism of the codomain, induced by a surjection $M_i^{a_{X,i}} \onto M_i^{a_{\emptyset,i}}$. 
By Lemma \ref{lemma:unitlift} applied to $M^{a_{X,i}}_i \widehat{\otimes} \widehat{\otimes}_{j\neq i} M^{a_{X,j}}_j$ (using that $M_i^{a_{X,i}}$ and $M_i^{a_{\emptyset,i}}$ are cyclic), up to precomposing with an automorphism of the domain, \eqref{eqn:summandgluelower} is the map induced by a surjection $M_i^{a_{X,i}} \onto M_i^{a_{\emptyset,i}}$. 
Fixing an isomorphism $M_i^{a_{\emptyset,i}}\cong S^{(i)}/\fP_\sigma^{(i)}$, we can choose isomorphisms $M_i^{a_{X,i}} \cong S^{(i)}/I_i^X$ for each $X\in a_i$ so that the surjection $S^{(i)}/I_i^X \cong M_i^{a_{X,i}} \onto M_i^{a_{\emptyset,i}}\cong S^{(i)}/\fP_\sigma^{(i)}$ takes $1$ to $1$. 
\eqref{eqn:gluelower} is then obtained from \eqref{eqn:wedgepresentation} by taking completed tensor products over $\F$ and then applying $-\otimes_{\tld{S}} R_\infty'(T)$ which is exact here. 
(Each module is the completion of a module over a polynomial ring over $\F$. These completed tensor products are obtained by usual tensor product of these decompleted modules over $\F$ and then completion---each step is exact. Applying $-\otimes_{\tld{S}} R_\infty'(T)$ simply has the effect of adding formal variables.) 
The result now follows. 
\end{proof}

\begin{proof}[Proof of Lemma \ref{lemma:broom}\eqref{item:broom2}]
Let $(\widehat{2})$ be the tuple $(\widehat{2})_{j\in \cJ}$ and $b$ be the tuple $(b_j)_{j\in \cJ}$ with $b_j = \{B,E_o,E_s,F_o,F_s\}$ if $j\in A(\sigma)$ and $b_j = \widehat{2}$ otherwise. 
Then the kernel of the surjective map $P_\sigma^{(\widehat{2})} \ra P_\sigma^b$ has no modular Serre weights. 
This induces an isomorphism $M_\infty(P_\sigma^{(\widehat{2})}) \ra M_\infty(P_\sigma^b)$. 
By Lemma \ref{lemma:gluelower} and Proposition \ref{prop:lowermult}, $M_\infty'(P^b_\sigma)$ and thus $M_\infty(P^{(\widehat{2})}_\sigma)$ is minimally generated by $3^{\#A(\sigma)}$ elements. 

Let $c = (c_j)_{j\in \cJ}$ with $c_j = \widehat{2}$ if $j\in A(\sigma)$ and $c_j = \widehat{1}$ otherwise. 
A similar argument as in the previous paragraph implies that $M_\infty(P^c_\sigma)$ is minimally generated by $3^{\#A(\sigma)}$ elements. 
In particular, the natural surjection $M_\infty(P^{(\widehat{2})}_\sigma)/\fm \ra M_\infty(P^c_\sigma)/\fm$ is an isomorphism. 
For a tuple $a$ as in Lemma \ref{lemma:broom}\eqref{item:broom2}, the surjection $P^{(\widehat{2})}_\sigma \ra P^c_\sigma$ factors through $P^a_\sigma$ and Lemma \ref{lemma:broom}\eqref{item:broom2} follows. 
\end{proof}

We record the following result for use in \S \ref{sec:locality}. 

\begin{prop}
Let $a = (a_j)_{j\in \cJ}$ be the tuple with $a_j = \{B,E_s,F_s\}$ if $j\in A(\sigma)$ and $a_j = \widehat{1}$ otherwise. 
Then $M_\infty(P_\sigma^a)$ is minimally generated by $3^{\#A(\sigma)}$ elements. 
\end{prop}
\begin{proof}
This follows from Lemma \ref{lemma:gluelower} and Proposition \ref{prop:lowermult}. 
\end{proof}

\begin{proof}[Proof of Lemma \ref{lemma:broom}\eqref{item:bigbroom}]
If $A(\sigma) = \cJ$, then the result follows from Lemma \ref{lemma:broom}\eqref{item:broom2}. 
We now assume that $A(\sigma) \neq \cJ$. 
Recall from Proposition \ref{prop:finalglue} the exact sequence \eqref{eqn:finalglue}:
\[
0 \ra \ovl{P}_\sigma \ra \oplus_c P_\sigma^c \ra (\oplus_c \sigma)/\Delta(\sigma) \ra 0 
\]
where $c$ runs over tuples $(c_j)_{j\in \cJ}$ with $c_i = \widehat{3}$ for some $i \notin A(\sigma)$ and $c_j = \widehat{1}$ for all $j \neq i$ and the tuple $(c_j)_{j\in \cJ}$ with $j = \widehat{2}$ for all $j\in A(\sigma)$ and $j = \widehat{1}$ for all $j \notin A(\sigma)$, $\Delta(\sigma) \subset \oplus_c \sigma$ denotes the diagonally embedded copy, and the maps are the natural projections. 
Since the $R_\infty$-action on each $M_\infty(P_\sigma^c)$ factors through $R_\infty(T)$, the same is true for $M_\infty(\ovl{P}_\sigma)$.
Using Remark \ref{rmk:compatibility:multitype} we have the exact sequence
\begin{equation}\label{eqn:patchglue}
0 \ra M_\infty'(\ovl{P}_\sigma) \ra \oplus_c M_\infty'(P_\sigma^c) \ra (\oplus_c M_\infty'(\sigma))/\Delta(M_\infty'(\sigma)) \ra 0.
\end{equation}
For the tuples $c$ with $c_i = \widehat{3}$ for some $i \notin A(\sigma)$ and $c_j = \widehat{1}$ for all $j \neq i$, $M_\infty'(P_\sigma^c)$ is a cyclic module, i.e.~isomorphic to ${{S}^\Box_\infty}/I_i$ for some ideal $I_i$, by Proposition \ref{prop:3layer}. 
Let $M$ be $M_\infty'(P_\sigma^c)$ for the tuple $(c_j)_{j\in \cJ}$ with $j = \widehat{2}$ for all $j\in A(\sigma)$ and $j = \widehat{1}$ for all $j \notin A(\sigma)$. 
We can choose isomorphisms so that \eqref{eqn:patchglue} becomes
\[
0 \ra M_\infty'(\ovl{P}_\sigma) \ra M \oplus \oplus_{j\notin A(\sigma)} {{S}^\Box_\infty}/I_j \ra (\oplus_c {{S}^\Box_\infty}/(\fP_\sigma))/\Delta({{S}^\Box_\infty}/(\fP_\sigma)) \ra 0. 
\] 
We claim that the hypotheses of Lemma \ref{lemma:CAfinalglue} hold, from which we deduce that $M_\infty(\ovl{P}_\sigma)$ is minimally generated by the same number of elements as $M$ which is $3^{\#A(\sigma)}$ by Lemma \ref{lemma:broom}\eqref{item:broom2}. 

We now verify the hypotheses of Lemma \ref{lemma:CAfinalglue}. 
Let $V_j$ (resp.~$U$) be the image of the induced map
\[
\Tor^{{S}^\Box_\infty}_1(\F,{{S}^\Box_\infty}/I_j)\ra\Tor^{{S}^\Box_\infty}_1(\F,{{S}^\Box_\infty}/(\fP_\sigma))
\]
for $j \notin A(\sigma)$ (resp.~$\Tor^{{S}^\Box_\infty}_1(\F,M)\ra\Tor^{{S}^\Box_\infty}_1(\F,{{S}^\Box_\infty}/(\fP_\sigma))$). 
We need to show that for all $i\notin A(\sigma)$, 
\begin{equation}\label{eqn:bigglue}
V_i+U\cap \cap_{j\neq i} V_j = \Tor^{{S}^\Box_\infty}_1(\F,{{S}^\Box_\infty}/(\fP_\sigma)).
\end{equation} 
As the natural map $\oplus_{j\in \cJ} \Tor^{{S}^\Box_\infty}_1(\F,{S}^\Box_\infty/({\fP}^{(j)}_\sigma)) \ra \Tor^{{S}^\Box_\infty}_1(\F,{S}^\Box_\infty/({\fP}_\sigma))$ is surjective, it is enough to show for all $i\notin A(\sigma)$ and $j_0\in \cJ$, the LHS of \eqref{eqn:bigglue} contains the image of the natural map 
\begin{equation}\label{eqn:jcompsigma}
\Tor^{{S}^\Box_\infty}_1(\F,{S}^\Box_\infty/({\fP}^{(j_0)}_\sigma)) \ra \Tor^{{S}^\Box_\infty}_1(\F,{S}^\Box_\infty/({\fP}_\sigma)). 
\end{equation}
In what follows, for an ideal $I \subset {S}$, we identify $\Tor^{{S}^\Box_\infty}_1(\F,{S}^\Box_\infty/(I))$ with $I\otimes_S \F$ (recall that $\F=S/\fm_S$).

Fix $i\notin A(\sigma)$ and $j_0\in \cJ$. 
We claim that $U$ and $V_{j}$ for each $j \notin A(\sigma)$ with $j \neq j_0$ contain the image of \eqref{eqn:jcompsigma}. 
This would imply that \eqref{eqn:bigglue} contains the image of \eqref{eqn:jcompsigma}. 
(If $j_0\neq i$ then take $j=i$ and $V_i$ would contain the image of \eqref{eqn:jcompsigma}. If $j_0=i$ then $U\cap \cap_{j\neq i} V_j$ would contain the image of \eqref{eqn:jcompsigma}.)

We first consider $U$. 
By Lemma \ref{lemma:gluelower}, $M$ is isomorphic to $M_\infty'(P^a_\sigma) \cong (\widehat{\otimes}_{\jmath\in \cJ} M_{\jmath}^{a_{\jmath}}) \widehat{\otimes}_S R_\infty'(T)$ where $a_{\jmath} = \{B,E_o,F_o,E_s,F_s\}$ if $\jmath\in A(\sigma)$ and $a_{\jmath}= \widehat{1}$ for all $\jmath\notin A(\sigma)$. 
Since $\sigma$ occurs as a Jordan--H\"older factor in $P_\sigma^a$ with multiplicity one, the cycle of $M$ is multiplicity free.
Corollary \ref{cor:disjointass} implies that the kernel of the map $M \ra M_\infty'(\sigma)$ is the kernel of $M\ra M_{\fP_\sigma}$.
Thus, up to postcomposing with an automorphism, the map $M \ra M_\infty'(\sigma)$ is unique and thus 
induced by $S^{(\jmath)}$-module surjections $M_{\jmath}^{a_{\jmath}}\onto S^{(\jmath)}/\fP_\sigma^{(\jmath)}$. 
As $M_{j_0}^{a_{j_0}}$ is isomorphic to $S^{(j_0)}/\fP_\sigma^{(j_0)}$, the surjection $M_{j_0}^{a_{j_0}}\onto S^{(j_0)}/\fP_\sigma^{(j_0)}$ is an isomorphism. 
We conclude that $U$ contains the image of \eqref{eqn:jcompsigma}.

We next turn to $V_j$. 
Let $j \in \cJ$ with $j \notin A(\sigma)$ and $j\neq j_0$. 
Let $\tau_{\min}$ be the minimal type with respect to $\sigma$ as in \cite[Remark~3.5.10]{LLLM2}.
For the rest of this proof, for a subset $\Sigma\subset S_3$ and an element $w\in\Sigma$ define $\tau^{\textnormal{\tiny{$B$}}}_{w,j}$, $T^{\textnormal{\tiny{$B$}}}_{\Sigma,j}$ (resp.~$\sigma(T_{\Sigma,j})$) as in \ref{sec:glue:prel} with respect to $\tau^{\textnormal{\tiny{$B$}}}_{\min}$ (resp.~$\tau_{\min}$) and $\ell=j$.

We first suppose that $t_{-\un{1}}\tld{w}(\rhobar,\tau^{\textnormal{\tiny{$B$}}}_{\mathrm{min}})_{j_0} = t_{w_0\eta}$. 
For each $\tau^{\textnormal{\tiny{$B$}}}_{w,j}$ with $w\in \{\Id,\alpha\beta,\beta\alpha,w_0\}$ and each generator $c^{(j_0)}$ in the $(\eps_1+\eps_2,1)$-entry of Table \ref{Table:components:F2}, one sees from the $t_{w_0\eta}$-entries in Table \ref{Table_Ideals_2} that $c^{(j_0)}+p\eps_{w,j,c^{(j_0)}} \in \tld{I}_{\tau_{w,j},\nabla_\infty}$ for some $\eps_{w,j,c^{(j_0)}} \in \fm_{\tld{S}}$. 
By Lemmas \ref{lem:distortion} and \ref{lemma:p3type}, each generator $c^{(j_0)}$ is in $(\tld{I}_{\tau_{w_0,j},\nabla_\infty} \cap \tld{I}_{\tau_{\alpha\beta,j},\nabla_\infty} \cap \tld{I}_{\tau_{\beta\alpha,j},\nabla_\infty},p) + \fm_{{S}^\Box_\infty} {\fP}_\sigma$. 
Thus the surjective map ${S}^\Box_\infty/(\tld{I}_{\tau_{w_0,j},\nabla_\infty} \cap \tld{I}_{\tau_{\alpha\beta,j},\nabla_\infty} \cap \tld{I}_{\tau_{\beta\alpha,j},\nabla_\infty},p) \ra {S}^\Box_\infty/(\fP_\sigma)$ induces a map on $\Tor^{{S}^\Box_\infty}_1(\F,-)$ whose image contains each generator $c^{(j_0)}$. 
Let $a = (a_\ell)_{\ell\in \cJ}$ be the tuple with $a_j = \widehat{3}$ and $a_\ell = \widehat{1}$ for each $\ell \neq j$ as in \S \ref{sec:tor}. 
Since the action of $S_\infty^\Box$ on $M_\infty'(\ovl{P}^{\sigma(T_{w_0,\alpha\beta,\beta\alpha,j}),a}_\sigma)$ factors through $(\tld{I}_{\tau_{w_0,j},\nabla_\infty} \cap \tld{I}_{\tau_{\alpha\beta,j},\nabla_\infty} \cap \tld{I}_{\tau_{\beta\alpha,j},\nabla_\infty},p)$, we conclude by Lemma \ref{lemma:tor} that $V_j$ contains each generator $c^{(j_0)}$.

Finally, we suppose that $t_{-\un{1}}\tld{w}(\rhobar,\tau^{\textnormal{\tiny{$B$}}}_{\mathrm{min}})_{j_0} = w_0 $. 
Similar arguments as before (using the $\alpha\beta\alpha t_{\un{1}}$-entry of Table \ref{Table_Ideals}) show that, for $j\neq j_0$, $V_{j}$ contains the images of the generators of the ideal in Lemma \ref{lem:intsc}\eqref{eq:lem:intsc:1}. 
Looking at the $(0,1)$-entry of Table \ref{Table:components} if $j_0\notin A(\sigma)$ (resp.~the $(0,0)$-entry of Table \ref{Table:components} if $j_0\in A(\sigma)$), it suffices to show that the image of $c^{(j_0)} \defeq (b-c)d_{21}^{(j_0)}d_{32}^{(j_0)}-(a-c)d_{31}^{(j_0)}d_{22}^{*,(j_0)}$ (resp.~$c_{13}^{(j_0)}$) in $\Tor^{{S}^\Box_\infty}_1(\F,{S}^\Box_\infty/({\fP}_\sigma))$ is in $V_{j}$ if $j_0 \notin A(\sigma)$ (resp.~$j_0\in A(\sigma)$; note in this case that if $V_j$ contains $c_{13}^{(j_0)}$ \emph{and} the images of the generators of the ideal in Lemma \ref{lem:intsc}\eqref{eq:lem:intsc:1}, then $V_j$ contains the images of the generators of the ideal in the $(0,0)$-entry of Table \ref{Table:components}).  
Let $a=(a_{\ell})_{\ell\in\cJ}$  be the tuple with $a_{j}=\widehat{3}$ and $a_{\ell}=\widehat{1}$ for $\ell\neq j$. 
The module $M_\infty'(\ovl{P}^{\sigma(\tau_{\Id,j}),a}_\sigma)$ is cyclic by \cite[Theorem 5.1.1]{LLLM2} with scheme-theoretic support determined by \cite[Lemma 3.6.2]{LLLM2}, from which we see that $c^{(j_0)}$ (resp.~$c_{13}^{(j_0)}$) is in the image of $\Tor^{{S}^\Box_\infty}_1(\F,M_\infty'(\ovl{P}^{\sigma(\tau_{\Id,j}),a}_\sigma)) \ra \Tor^{{S}^\Box_\infty}_1(\F,M_\infty'(\sigma))$, and so it suffices by Lemma \ref{lemma:tor} to show that the image of $c^{(j_0)}$ (resp.~$c_{13}^{(j_0)}$) is contained in the image of 
\begin{equation}\label{eqn:truncate}
\Tor^{{S}^\Box_\infty}_1(\F,M_\infty'(\ovl{P}^{\sigma(T_{w_0,\alpha\beta,\beta\alpha,j}),a}_\sigma)) \ra \Tor^{{S}^\Box_\infty}_1(\F,M_\infty'(\sigma)). 
\end{equation}
Thus, we are left to show that $c^{(j_0)}$ (resp.~$c_{13}^{(j_0)}$) annihilates $M_\infty'(\ovl{P}^{\sigma(T_{w_0,\alpha\beta,\beta\alpha,j}),a}_\sigma)$, since this module is cyclic over $S^\Box_\infty$ by Lemma \ref{lemma:3type}.

For each $w \in \{\Id,\alpha\beta,\beta\alpha,w_0\}$, 
\[
(\tld{z}_w^*)^{-1}c_{13}^{(j_0)}((\tld{b}_w-\tld{c}_w)d_{21}^{(j_0)}d_{32}^{(j_0)}-(\tld{a}_w-\tld{c}_w)d_{31}^{(j_0)}d_{22}^{*,(j_0)}) -p\in \tld{I}_{\tau_{w,j},\nabla_\infty}
\]
where $\tld{m}_w \in \Z$ is a specific lift of $m$ for $m = a,b,c$ and $\tld{z}_w^*\in\tld{S}^\times$ is a specific unit (all depending \emph{a priori} on $w$) by \cite[\S 5.3.1]{LLLM}. 
Fixing lifts $\tld{m}\in\Z$ of $m$ for $m=a,b,c$, and a lift $\tld{z}^*$ of the reductions of $\tld{z}^*_w$ modulo $p$ (note that $\tld{z}^*_w$ modulo $p$ is independent of $w$), we have for each $w\in \{\Id,\alpha\beta,\beta\alpha,w_0\}$ that 
\[
(\tld{z}^*)^{-1}c_{13}^{(j_0)}((\tld{b}-\tld{c})d_{21}^{(j_0)}d_{32}^{(j_0)}-(\tld{a}-\tld{c})d_{31}^{(j_0)}d_{22}^{*,(j_0)}) -p+p\eps_{w,j}\in \tld{I}_{\tau_{w,j},\nabla_\infty}
\]
for some $\eps_{w,j}\in c_{13}^{(j_0)}\fm_{\tld{S}}$.
Taking $f\defeq (\tld{z}^*)^{-1}c_{13}^{(j_0)}((\tld{b}-\tld{c})d_{21}^{(j_0)}d_{32}^{(j_0)}-(\tld{a}-\tld{c})d_{31}^{(j_0)}d_{22}^{*,(j_0)}) -p$, we conclude from Lemmas \ref{lem:distortion} and \ref{lemma:p3type} that 
\[
c_{13}^{(j_0)}\tld{c}^{(j_0)}-p \in \tld{I}_{\tau_{w_0,j},\nabla_\infty} \cap \tld{I}_{\tau_{\alpha\beta,j},\nabla_\infty} \cap \tld{I}_{\tau_{\beta\alpha,j},\nabla_\infty} 
\]
for some $\tld{c}^{(j_0)} \in (\tld{b}-\tld{c})d_{21}^{(j_0)}d_{32}^{(j_0)}-(\tld{a}-\tld{c})d_{31}^{(j_0)}d_{22}^{*,(j_0)}+\fm_{\tld{S}}$. 
Filtering $\ovl{P}^{\sigma(T_{w_0,\alpha\beta,\beta\alpha,j}),a}_\sigma$ with irreducible subquotients $\kappa$ induces a filtration on $M_\infty'(\ovl{P}^{\sigma(T_{w_0,\alpha\beta,\beta\alpha,j}),a}_\sigma)$ with subquotients $M_\infty'(\kappa)$. 
For each $\kappa$ and $m\in M_\infty'(\kappa)$, the support of $\tld{c}^{(j_0)}m$ (resp.~$c_{13}^{(j_0)}m$) has positive codimension in the scheme-theoretic support of $M_\infty'(\kappa)$ as $c^{(j_0)}_{13}$ (resp.~$\tld{c}^{(j_0)}$) is $M_\infty'(\kappa)$-regular. 
Thus the same is true for $\tld{c}^{(j_0)}m$ (resp.~$c_{13}^{(j_0)}m$) with $m \in M_\infty'(\ovl{P}^{\sigma(T_{w_0,\alpha\beta,\beta\alpha,j}),a}_\sigma)$. 
Since $M_\infty'(\ovl{P}^{\sigma(T_{w_0,\alpha\beta,\beta\alpha,j}),a}_\sigma)$ is maximal Cohen--Macaulay over its support, it has no embedded primes from which we conclude that it is annihilated by $\tld{c}^{(j_0)}$ (resp.~$c_{13}^{(j_0)}$). 
\end{proof}

\section{Locality results}\label{sec:locality}

\subsection{Subquotients of Deligne--Lusztig representations and presentations}\label{sec:DLpresent}

We begin by defining some quotients of reductions of generic Deligne--Lusztig representations of $\rG$. 
Fix a Deligne--Lusztig representation $R_s(\mu-\un{1})$ with $\mu-\eta\in\un{C}_0$ being $9$ deep, so that $\JH(\ovl{R}_s(\mu-\un{1})) = F(\Trns_{\mu}(s(\Sigma)))$ and \cite[Proposition 5.3.2]{GL3Wild} applies to $R_s(\mu-\un{1})$.

Given two $\cO$-lattices $\Lambda_1,\Lambda_2 \subset R_s(\mu-\un{1})$, there is a unique $n\in \Z$ so that $p^n\Lambda_1 \subset \Lambda_2$ and $p^{n-1} \Lambda_1 \not\subset \Lambda_2$. 
We denote by $\iota$ the composition 
\begin{equation}\label{eqn:iota}
\iota: \Lambda_1 \overset{\times p^n} \risom p^n\Lambda_1\subset \Lambda_2. 
\end{equation}
By construction $\iota \otimes_{\cO} \F: \Lambda_1 \otimes_{\cO} \F \ra \Lambda_2 \otimes_{\cO} \F$ is nonzero. 

Recall from \cite[Lemma 4.1.1]{EGS} that, as $R_s(\mu-\un{1})$ is residually multiplicity free, given $\sigma \in \JH(\ovl{R}_s(\mu-\un{1}))$, there is a unique up to scaling $\cO$-lattice in $R_s(\mu-\un{1})^\sigma \subset R_s(\mu-\un{1})$ with cosocle isomorphic to $\sigma$. 
We will define various quotients of reductions of the lattices $R_s(\mu-\un{1})^\sigma$ whose structure is given by \cite[Proposition 5.3.2]{GL3Wild}. 

Recall from \S \ref{subsub:surgeries} that we defined a \emph{path} to be a sequence of elements 
\[
\gamma = (\gamma_k)_{k\geq 1}^{\ell(\gamma)}\in
\{(0,0), (\eps_1,0), (\eps_2,0),(0,1), (\eps_1,1), (\eps_2,1)\}^{\ell(\gamma)}
\] 
satisfying certain properties.
We also defined subsets $\Sigma_\gamma\subset \Sigma_0$ for each path $\gamma$ and a partial ordering $\leq$ on the set of paths.

Fix a subset $J \subset \cJ$. 
For each $j\in \cJ \setminus J$, fix $(\omega_j,a_j) \in \Sigma_0$. 
For a tuple $\gamma_J= (\gamma^{(j)})_{j\in J}$ of paths, let 
\[
\ell(\gamma_J) = \underset{j\in J}{\sum} \ell(\gamma^{(j)}) 
\]
and $\sigma(\gamma_J) = F(\Trns_{\mu}(s\omega,a))$ where $(\omega_j,a_j)$ is the fixed element in $\Sigma_0$ if $j\notin J$ and is $\gamma^{(j)}_{\ell(\gamma^{(j)})}$ otherwise. 
By \cite[Proposition 5.3.2]{GL3Wild}, there is a unique quotient $Q_{\gamma_J}$ of $R_s(\mu-\un{1})^{\sigma(\gamma_J)} \otimes_{\cO} \F$ such that 
\[
\JH(Q_{\gamma_J}) = \underset{j\in J}{\times} \Sigma_{\gamma^{(j)}} \times \underset{j\not\in J}{\times} \{(\omega_j,a_j)\}. 
\]

We define a complex using two choices. 
First, choose a complete ordering of $J$. 
Second, for each path $\gamma$ of length $3$ and $j\in J$, let $\kappa_\gamma^{(j)}$ be an element in $\F^\times$ with $\kappa_\gamma^{(j)} = 1$ if $\gamma_2 = (0,0)$ or $\gamma_3 = (0,1)$. 
(There are 12 paths of length $3$ and 8 of these paths have the property that $\gamma_2 = (0,0)$ or $\gamma_3 = (0,1)$.) 
It is the second choice that is more consequential. 
Consider the complex 
\begin{equation}\label{eqn:exactQ}
0 \ra \bigoplus_{\ell(\gamma_J) = 3\#J} Q_{\gamma_J} \ra \bigoplus_{\ell(\gamma_J) = 3\#J-1} Q_{\gamma_J} \ra \cdots \ra \bigoplus_{\ell(\gamma_J) = 2\#J} Q_{\gamma_J}, 
\end{equation}
where each map is a direct sum of maps $Q_{\gamma_J} \ra Q_{\beta_J}$ which is nonzero if and only if $\beta^{(j)} \leq \gamma^{(j)}$ for all $j\in J$. 
If $\beta^{(j)} \leq \gamma^{(j)}$ for all $j\in J$, $\ell(\beta_J)=\ell(\gamma_J)-1$ and $i\in\cJ$ is defined by $\beta^{(i)} \neq \gamma^{(i)}$, then the map $Q_{\gamma_J} \ra Q_{\beta_J}$ is induced by 
\[
\Big(\underset{j\in J,\ell(\gamma^{(j)}) = 3}{\prod} (-1)^{\delta_{j<i}}\kappa_{\gamma^{(j)}}^{(j)}\Big) \iota
\]
with $\iota$ in \eqref{eqn:iota} and $\delta_{j<i} = 1$ if $j<i$ and $0$ if $j\geq i$. 
Though we will not use it, one can check that \eqref{eqn:exactQ} is exact. 

We now define a subcomplex of \eqref{eqn:exactQ}. 
For each $2\# J \leq \ell \leq 3\#J$, let 
\[
\Big(\bigoplus_{\ell(\gamma_J) = \ell} Q_{\gamma_J}\Big)^0 \subset \bigoplus_{\ell(\gamma_J) = \ell} Q_{\gamma_J}
\]
be the submodule of elements $(a_{\gamma_J})_{\ell(\gamma_J) = \ell}$ such that for any tuple $\beta_J$ with $\ell(\beta_J) = \ell$ and $i \in J$ such that $\ell(\beta^{(i)}) = 3$, we have 
\[
\underset{\substack{\gamma^{(j)} = \beta^{(j)} \, \forall \, j\neq i \\ \ell(\gamma^{(i)}) = 3, \gamma^{(i)}_3 = \beta^{(i)}_3}}{\sum} a_{\gamma_J} = 0. 
\]
(Note that $Q_{\gamma_J}$ are the same quotient for all $\gamma_J$ in the sum above, and the sum takes place in this same quotient.)
Then 
\begin{equation}\label{eqn:exactQ0}
0 \ra \Big(\bigoplus_{\ell(\gamma_J) = 3\#J} Q_{\gamma_J}\Big)^0 \ra \Big(\bigoplus_{\ell(\gamma_J) = 3\#J-1} Q_{\gamma_J}\Big)^0 \ra \cdots \ra \Big(\bigoplus_{\ell(\gamma_J) = 2\#J} Q_{\gamma_J}\Big)^0 
\end{equation}
is a subcomplex of \eqref{eqn:exactQ}. 
Again, though we will not use it, one can check that \eqref{eqn:exactQ0} is exact. 

\subsection{Patching and presentations}\label{sec:patchingpresentation}

In this section, we fix an $11$-generic $L$-homomorphism $\rhobar $ and a patching functor $M_\infty$ for it. 
Let $\nu$ be a weight in $W(\eta)$. 
Then $\nu_j \in S_3(\eta_j) \cup \{(1,1,1)_j\}$ for each $j\in \cJ$. 
Let $J = J(\nu) \defeq \{j\in \cJ\mid \nu_j = (1,1,1)_j\}$. 
For $j\notin J$, if $\nu_j = w_j \eta_j$ for $w_j \in S_3$, then we let $(\omega(\nu)_j,a(\nu)_j) = (-\eps_{w_j},\delta_{(-1)^{w_j}=-1})$ where $\eps_{w_j} = \eps'_{w_j}+X^0(T) \in X^*(T)/X^0(T)$ and $w_j t_{\eps'_{w_j}} \in \tld{W}^+_1$ and $\delta_{(-1)^{w_j}=-1} = 1$ if $w_j$ is an odd permutation and is $0$ if $w_j$ is an even permutation. 

Let $\tld{w} \in \Adm(\eta)$ be such that $\tld{w}_j = t_{\un{1}}$ for all $j\in J$ and $(\omega(\nu)_j,a(\nu)_j) \in \tld{w}_j(\Sigma_0)$ and $\ell(\tld{w}_j) \geq 2$ for all $j\notin J$.
Let $\tau$ be the tame inertial $L$-parameter such that $\tld{w}(\rhobar,\tau) = \tld{w}$. 
Given $\sigma(\tau) = R_s(\mu-\un{1})$, $J$, and $(\omega_j,a_j)\defeq (\tld{w}_j^{-1}\omega(\nu)_j,a(\nu)_j) \in\Sigma_0$ for $j\notin J$, we have the subquotients $Q_{\gamma_J}$ of $\sigma(\tau)$ for each tuple $\gamma_J$ of paths as defined in \S \ref{sec:DLpresent}. 
(Note that $\mu-\eta$ can be chosen to be $9$-deep in alcove $\un{C}_0$.)

We will use the notation $M_\infty'$ as in \eqref{eqn:Minfty'} with $T = \{\tau\}$. 
If $(\omega,a)=(\omega_j,a_j)_{j\in\cJ}$ with $(\omega_j,a_j)$ as above if $j\notin J$ and $(\omega_j,a_j)\in \{(\eps_1,1), (\eps_2,1), (0,1)\}$ if $j\in J$, and letting 
$\sigma \defeq F(\Trns_{\mu}(s\omega,a)) \in \JH(\ovl{R}_s(\mu-\un{1}))$, then \cite[Theorem 5.1.1]{LLLM2} implies that $M_\infty'(R_s(\mu-\un{1})^\sigma)$ is a free $R_\infty'(\tau)$-module of rank one. 
Let $M^{(j)}_{(\omega_j,a_j)}$ be $\tld{M}^{(j)}_{(\omega_j,a_j)} \otimes_{\cO} \F$ where $\tld{M}^{(j)}_{(\omega_j,a_j)}$ is defined in \S \ref{subsub:surgeries}. 
For a pair $(\omega_j,a_j)$ and $(\omega'_j,a'_j)$ in $\{(0,0), (\eps_1,0), (\eps_2,0),(0,1), (\eps_1,1), (\eps_2,1)\}$, there is a unique $n((\omega_j,a_j),(\omega'_j,a'_j)) \defeq n \in \Z$ such that $p^n\tld{M}^{(j)}_{(\omega'_j,a'_j)} \subset \tld{M}^{(j)}_{(\omega_j,a_j)}$ and $p^{n-1}\tld{M}^{(j)}_{(\omega'_j,a'_j)} \not\subset \tld{M}^{(j)}_{(\omega_j,a_j)}$. 
By \cite[Theorem 5.2.3]{LLLM2}, %
the finite category consisting of maps
\[
M_\infty'(\ovl{R}_s(\mu-\un{1})^{F(\Trns_{\mu}(s\omega',a'))}) \ra M_\infty'(\ovl{R}_s(\mu-\un{1})^{F(\Trns_{\mu}(s\omega,a))})
\]
induced by $\iota$ in \eqref{eqn:iota} is identified with the maps
\begin{equation}\label{eqn:factorlattice}
\widehat{\otimes}_{j\in \cJ} M^{(j)}_{(\omega_j',a_j')} \widehat{\otimes}_{{S}} \ovl{R}_\infty'(\tau) \ra \widehat{\otimes}_{j\in \cJ} M^{(j)}_{(\omega_j,a_j)} \widehat{\otimes}_{{S}} \ovl{R}_\infty'(\tau) 
\end{equation}
given by maps $M^{(j)}_{(\omega_j',a_j')} \ra M^{(j)}_{(\omega_j,a_j)}$ which are induced by multiplication by $p^{n((\omega_j,a_j),(\omega'_j,a'_j))}$. 
(Tensor products without a subscript are taken over $\F$.) %

Now let {$\gamma_J$ be a tuple of paths and} $\sigma$ be $F(\Trns_{\mu}(s\omega,a))$ with $(\omega_j,a_j)$ the fixed choice for $j\notin J$ and $(\omega_j,a_j) = \gamma_1^{(j)}$ for $j\in J$. 
Then, up to scaling by a unique power of $p$, each $Q_{\gamma_J}$ is a subquotient of $R_s(\mu-\un{1})^\sigma$. 
Moreover, by \cite[Lemma 3.6.2]{LLLM2} or by Corollary \ref{cor:disjointass}, there is an ideal $I_{\gamma^{(j)}} \subset S^{(j)}$ for each $j\in J$ and path $\gamma^{(j)}$ and a prime ideal $\fP^{(j)}_{(\omega_j,a_j)}$ for each $j\notin J$ so that $M_\infty'(Q_{\gamma_J})$ is identified with
\[
(\underset{j\notin J}{\widehat{\otimes}} M^{(j)}_{(\omega_j,a_j)}/\fP^{(j)}_{(\omega_j,a_j)}M^{(j)}_{(\omega_j,a_j)}) \widehat{\otimes} (\underset{j\in J}{\widehat{\otimes}} I_{\gamma^{(j)}}M_{\gamma_1^{(j)}}^{(j)}) \widehat{\otimes}_{{S}} \ovl{R}_\infty'(\tau) 
\]
and $M_\infty'(\iota): M_\infty'(Q_{\gamma_J}) \ra M_\infty'(Q_{\beta_J})$ for $\beta \leq \gamma$ is induced by \eqref{eqn:factorlattice}. 

With the above identifications, applying $M_\infty'$ to \eqref{eqn:exactQ0} yields the complex 
\begin{equation}\label{eqn:totalcomplex}
(\underset{j\notin J}{\widehat{\otimes}} M^{(j)}_{(\omega_j,a_j)}/\fP^{(j)}_{(\omega_j,a_j)}M^{(j)}_{(\omega_j,a_j)}) \widehat{\otimes} (\underset{j\in J}{\widehat{\otimes}} C^{(j)}_{\kappa^{(j)}}) \widehat{\otimes}_{{S}} \ovl{R}_\infty'(\tau) 
\end{equation}
where $C^{(j)}_{\kappa^{(j)}}$ is the complex \eqref{eq:map:surgery} and the signs for the tensor product of the $C^{(j)}_{\kappa^{(j)}}$ are given by the complete ordering on $J$. 
Let $Q_{\nu,\kappa}$ be the cokernel of the last map in \eqref{eqn:exactQ0}. 
By exactness of $M_\infty'$, we have an identification 
\begin{equation}
\label{eq:identify_M_infty}
M_\infty'(Q_{\nu,\kappa})\cong
(\widehat{\otimes}_{j\in \cJ} M^{(j)}_{\kappa^{(j)}}) \widehat{\otimes}_{{S}} \ovl{R}_\infty'(\tau) 
\end{equation}
where $M^{(j)}_{\kappa^{(j)}}$ is the cokernel of \eqref{eq:map:surgery} as defined in \S \ref{subsub:surgeries}. 
(To see that \eqref{eqn:totalcomplex} is exact, note that each tensor factor is the completion of a complex of modules over a polynomial ring. Then \eqref{eqn:totalcomplex} is the completion of a tensor product of these complexes. The exactness of \eqref{eqn:totalcomplex} follows from the exactness of completion for modules over a Noetherian ring.)

\begin{prop}\label{prop:Qmult}
A Serre weight $F(\Trns_{\mu}(s\omega,a))$ is a Jordan--H\"older factor of $Q_{\nu,\kappa}$ if and only if $(\omega_j,a_j)$ is the fixed element for $j\notin J$ and $(\omega_j,a_j)\in \{(0,0), (\eps_1,0), (\eps_2,0),(0,1), (\eps_1,1), (\eps_2,1)\}$ for $j\in J$. 
In this case, its multiplicity is $2^{\#\{j\in J \mid a_j = 0\}}$. 
\end{prop}
\begin{proof}
By the exactness of $M_\infty'(-)$ and  the fact that $\JH(Q_{\nu,\kappa})\subseteq W^?(\rhobar)$ we conclude that
\[
Z_d(M'_\infty(Q_{\nu,\kappa}))=\sum_{\sigma\in \JH(Q_{\nu,\kappa})}[Q_{\nu,\kappa}:\sigma]\fP_\sigma \ovl{R}_\infty'(\tau).
\]
The result now follows from the identification \eqref{eq:identify_M_infty}, Proposition \ref{prop:special_fiber} and Corollary \ref{cor:cycle:Mkappa}.
\end{proof}

For each $j\in J$ let $(a_j,b_j,c_j)\in\Fp^3$ be the mod $p$ reduction of $-s_j^{-1}(\mu_j)$. 
From now on we take $\kappa$ to be $\kappa(\rhobar,\nu)\defeq \Big(\kappa_{\min}(a_j,b_j,c_j)\Big)^J$ (see \S \ref{subsub:surgeries} for the definition of the tuple $\kappa_{\min}(a_j,b_j,c_j)$).

\begin{prop}\label{prop:wedgeinj}
If $\ell(\gamma_J) = 2\#J$, then the composition
\begin{equation}\label{eqn:wedgeinj}
Q_{\gamma_J} \ra \bigoplus_{\ell(\gamma_J) = 2\#J} Q_{\gamma_J} = \Big(\bigoplus_{\ell(\gamma_J) = 2\#J} Q_{\gamma_J}\Big)^0 \ra Q_{\nu,\kappa}
\end{equation}
is injective. 
\end{prop}
\begin{proof}
It suffices to show that for path $\beta_J \geq \gamma_J$ with $\ell(\beta_J) = 3\#J$, the composition of \eqref{eqn:wedgeinj} and the natural inclusion $Q_{\beta_J} \into Q_{\gamma_J}$ is nonzero as $Q_{\gamma_J}$ is multiplicity free and the map $\oplus_{\beta} Q_{\beta_J}\into Q_{\gamma_J}$, summing over all such $\beta_J$, identifies the domain with the socle of the codomain. 
We will show that $M_\infty'(Q_{\beta_J}) \ra M_\infty'(Q_{\gamma_J}) \ra M_\infty'(Q_{\nu,\kappa})$ is nonzero. 

Applying $M_\infty'$ to \eqref{eqn:wedgeinj} and using the above identifications yields a map 
\begin{align*}
&(\underset{j\notin J}{\widehat{\otimes}} M^{(j)}_{(\omega_j,a_j)}/\fP^{(j)}_{(\omega_j,a_j)}M^{(j)}_{(\omega_j,a_j)}) \widehat{\otimes} (\underset{j\in J}{\widehat{\otimes}} I_{\gamma^{(j)}}M_{\gamma_1^{(j)}}^{(j)}) \widehat{\otimes}_{{S}} \ovl{R}_\infty'(\tau) \\
\ra &(\underset{j\notin J}{\widehat{\otimes}} M^{(j)}_{(\omega_j,a_j)}/\fP^{(j)}_{(\omega_j,a_j)}M^{(j)}_{(\omega_j,a_j)}) \widehat{\otimes} (\underset{j\in J}{\widehat{\otimes}} M^{(j)}_{\kappa^{(j)}}) \widehat{\otimes}_{{S}} \ovl{R}_\infty'(\tau) 
\end{align*}
induced by the compositions 
\begin{equation}\label{eqn:wedgeinjideal}
(I_{\gamma^{(j)}}M_{\gamma_1^{(j)}}^{(j)}) \ra \bigoplus_{\ell(\beta^{(j)}) = 2} I_{\alpha^{(j)}}M_{\beta_1^{(j)}}^{(j)} \ra M^{(j)}_{\kappa^{(j)}}
\end{equation}
for each $j\in J$ where the first map is the natural inclusion and the second map is the natural projection. 
Similarly, the map $M_\infty'(Q_{\beta_J}) \into M_\infty'(Q_{\gamma_J})$ is given by 
\begin{align*}
&(\underset{j\notin J}{\widehat{\otimes}} M^{(j)}_{(\omega_j,a_j)}/\fP^{(j)}_{(\omega_j,a_j)}M^{(j)}_{(\omega_j,a_j)}) \widehat{\otimes} (\underset{j\in J}{\widehat{\otimes}} I_{\beta^{(j)}}M_{\gamma_1^{(j)}}^{(j)}) \widehat{\otimes}_{{S}} \ovl{R}_\infty'(\tau) \\
\ra &(\underset{j\notin J}{\widehat{\otimes}} M^{(j)}_{(\omega_j,a_j)}/\fP^{(j)}_{(\omega_j,a_j)}M^{(j)}_{(\omega_j,a_j)}) \widehat{\otimes} (\underset{j\in J}{\widehat{\otimes}}  I_{\gamma^{(j)}}M_{\gamma_1^{(j)}}^{(j)}) \widehat{\otimes}_{{S}} \ovl{R}_\infty'(\tau) 
\end{align*}
The desired nonvanishing now follows from Corollary \ref{cor:surgery}. 
\end{proof}

\begin{prop}\label{prop:soclength}
A Serre weight $F(\Trns_{\mu}(s\omega,a))$ is a Jordan--H\"older factor of $\soc \, Q_{\nu,\kappa}$ if and only if $(\omega_j,a_j)$ is the fixed element for $j\notin J$ and $(\omega_j,a_j)\in \{(0,1), (\eps_1,1), (\eps_2,1)\}$ for $j\in J$. 
Morever, $\soc \, Q_{\nu,\kappa}$ is multiplicity free. 
\end{prop}
\begin{proof}
If $(\omega_j,a_j)$ is the fixed element for $j\notin J$ and $(\omega_j,a_j)\in \{(0,1), (\eps_1,1), (\eps_2,1)\}$ for $j\in J$, then $\sigma \defeq F(\Trns_{\mu}(s\omega,a)) \in \JH(\soc\, Q_{\gamma_J})$ for some tuple of paths $\gamma_J$ with $\ell(\gamma_J) = 2\#J$. 
By Proposition \ref{prop:wedgeinj}, there an inclusion $\soc\, Q_{\gamma_J} \into \soc\, Q_{\nu,\kappa}$ so that $\sigma \in \JH(\soc \, Q_{\nu,\kappa})$. 

Now let $\theta \defeq F(\Trns_{\mu}(s\omega,a)) \in \JH(Q_{\nu,\kappa})$. 
Then $(\omega_j,a_j)$ is the fixed element for $j\notin J$ and $(\omega_j,a_j)\in \{(0,0), (\eps_1,0), (\eps_2,0),(0,1), (\eps_1,1), (\eps_2,1)\}$ for $j\in J$. 
Suppose that $a_i = 0$ for some $i\in J$ and that we have a nonzero map $P_\theta \ra Q_{\nu,\kappa}$. 
Consider a lift 
\[
\begin{tikzcd}
& \bigoplus_{\ell(\gamma_J) = 2\#J} Q_{\gamma_J} \arrow[d] \\
P_\theta \arrow[ur]\arrow[r] & Q_{\nu,\kappa} 
\end{tikzcd}
\]
where the vertical map is the natural quotient map. 
The composition 
\[
P_\theta \ra \bigoplus_{\ell(\gamma_J) = 2\#J} Q_{\gamma_J} \onto \bigoplus_{\substack{\ell(\gamma_J) = 2\#J \\ \gamma^{(i)}_1 = (\eps_k,1)}} Q_{\gamma_J},
\]
where the second map is the natural projection, is nonzero for $k= 1$ and $2$. 
Moreover, this composition factors as the composition
\begin{equation}\label{eqn:Qprojection}
P_\theta \ra \bigoplus_{\substack{\ell(\gamma_J) = 2\#J \\ \gamma^{(i)} = ((\eps_k,1),(\omega_i,0))}} Q_{\gamma_J} \subset \bigoplus_{\substack{\ell(\gamma_J) = 2\#J \\ \gamma^{(i)}_1 = (\eps_k,1)}} Q_{\gamma_J}
\end{equation}
as $\theta$ is not a Jordan--H\"older factor of the cokernel of the inclusion in \eqref{eqn:Qprojection}. 
Let $\theta' = F(\Trns_{\mu}(s\omega',a'))$ where $(\omega'_j,a'_j) = (\omega_j,a_j)$ for all $j \neq i$ and $(\omega'_i,a'_i) = (\eps_{3-k},1)$. 
Then $\theta'$ is a Jordan--H\"older factor of the image of \eqref{eqn:Qprojection}---indeed the composition of \eqref{eqn:Qprojection} with the projection to some $Q_{\gamma_J}$ is surjective (choosing $\gamma_J$ so that $\cosoc\,Q_{\gamma_J} \cong \theta$, and using that $Q_{\gamma_J}$ is multiplicity free) and $\theta'\in\JH(Q_{\gamma_J})$. 
On the other hand, we claim that the intersection of  
\[
\bigoplus_{\substack{\ell(\gamma_J) = 2\#J \\ \gamma^{(i)} = ((\eps_k,1),(\omega_i,0))}} Q_{\gamma_J}
\]
and the image of 
\begin{equation}
\label{eq:second:image}
\left(\bigoplus_{\ell(\gamma_J) = 2\#J+1} Q_{\gamma_J}\right)^0 \ra \bigoplus_{\ell(\gamma_J) = 2\#J} Q_{\gamma_J} \onto \bigoplus_{\substack{\ell(\gamma_J) = 2\#J\\ \gamma^{(i)}_1 = (\eps_k,1)}} Q_{\gamma_J}, 
\end{equation}
(where the first map in \eqref{eq:second:image} comes from \eqref{eqn:exactQ0}) does not contain $\theta'$ as a Jordan--H\"older factor. 
Indeed, the cokernel of 
\[
\bigoplus_{\substack{\ell(\gamma_J) = 2\#J+1 \\ \gamma^{(i)} = ((\eps_k,1),(\omega_i,0),(\eps_{3-k},1))}} Q_{\gamma_J}\subset
\bigoplus_{\substack{\ell(\gamma_J) = 2\#J \\ \gamma^{(i)} = ((\eps_k,1),(\omega_i,0))}} Q_{\gamma_J}
\]
does not contain $\theta'$ as a Jordan--H\"older factor.
It suffices to show that the intersection of the images of 
\[
\bigoplus_{\substack{\ell(\gamma_J) = 2\#J+1 \\ \gamma^{(i)} = ((\eps_k,1),(\omega_i,0),(\eps_{3-k},1))}} Q_{\gamma_J}
\]
and
\[
\left(\bigoplus_{\ell(\gamma_J) = 2\#J+1} Q_{\gamma_J}\right)^0
\]
in 
\[
\bigoplus_{\substack{\ell(\gamma_J) = 2\#J\\ \gamma^{(i)}_1 = (\eps_k,1)}} Q_{\gamma_J}
\]
is zero.
An element 
\[
(a_{\gamma_J})_{\gamma_J}\in \bigoplus_{\substack{\ell(\gamma_J)=2\#J+1\\ \ell(\gamma^{(i)})=3,\, \gamma^{(i)}_1 = (\eps_k,1)}}Q_{\gamma_J}\subset \bigoplus_{\substack{\ell(\gamma_J) = 2\#J\\ \gamma^{(i)}_1 = (\eps_k,1)}} Q_{\gamma_J}
\] 
in the intersection of these images satisfies the condition $a_{\gamma_J}=0$ if $\gamma_2^{(i)}\neq (\omega_i,0)$ and for each $\gamma_J$ as above $\sum_{\beta_J} a_{\beta_J}=0$ where the sum runs over $\beta_J$ with $\beta^{(j)}=\gamma^{(j)}$ for all $j\neq i$ and $\beta^{(i)}_3=(\eps_{3-k},1)$.
These conditions imply that $a_{\gamma_J}=0$ for all $\gamma_J$.

Thus $\theta'$ is a Jordan--H\"older factor of the image of the original map $P_\theta \ra Q_{\nu,\kappa}$. 
We conclude that $\theta$ is not a Jordan--H\"older factor of $\soc\, Q_{\nu,\kappa}$. 

Finally, the Jordan--H\"older factors of $\soc\, Q_{\nu,\kappa}$ appear in a composition series of $Q_{\nu,\kappa}$ with multiplicity one by Proposition \ref{prop:Qmult}. 
In particular, $\soc\, Q_{\nu,\kappa}$ is multiplicity free. 
\end{proof}

\begin{prop}\label{prop:indecomp}
The module $Q_{\nu,\kappa}$ is indecomposable. 
\end{prop}
\begin{proof}
Suppose that $Q_{\nu,\kappa} = Q \oplus Q'$ where $Q$ is nonzero. 
Let $F(\Trns_{\mu}(\omega,a)) \cong \sigma \subset Q$ be a simple submodule. 
Let $\sigma'\subset Q_{\nu,\kappa}$ be a simple submodule. 
We will show that $\sigma' \subset Q$ which implies that $Q' = 0$ and thus that $Q_{\nu,\kappa}$ is indecomposable. 

For a tuple of paths $\beta_J$ of length $2\#J$, we identify $Q_{\beta_J}$ with its image in $Q_{\nu,\kappa}$ using Proposition \ref{prop:wedgeinj}. 
We claim that if $Q_{\beta_J} \cap Q$ is nonzero, then $Q_{\beta_J} \subset Q$. 
Indeed, the splitting of the inclusion $Q \into Q_{\nu,\kappa}$ gives a splitting of the (nonzero) map $Q_{\beta_J} \cap Q \into Q_{\beta_J}$. 
On the other hand, $Q_{\beta_J}$ is indecomposable since $\cosoc\, Q_{\beta_J}$ is simple. 
We conclude that $Q_{\beta_J} \cap Q = Q_{\beta_J}$. 

Now there are tuples of paths $\gamma_J$ and $\gamma'_J$ of length $2\#J$ such that $\sigma \subset Q_{\gamma_J}$ and $\sigma' \subset Q_{\gamma'_J}$. 
Since $\sigma \subset Q$, $Q_{\gamma_J} \cap Q \neq 0$ so that $Q_{\gamma_J} \subset Q$ by the previous paragraph. 
On the other hand, $\JH(Q_{\gamma_J}) \cap \JH(Q_{\gamma'_J})$ contains $\sigma_0 \defeq F(\Trns_{\mu}(\omega,a))$ where $(\omega_j,a_j)$ is the fixed element in $\Sigma_0$ for $j\notin J$ and is $(0,1)$ for all $j\in J$. 
As $\sigma_0$ is a Jordan--H\"older factor of $Q_{\nu,\kappa}$ with multiplicity one by Proposition \ref{prop:Qmult}, $Q_{\gamma_J} \cap Q_{\gamma'_J} \neq 0$. 
We conclude from the previous paragraph that $\sigma' \subset Q_{\gamma'_J} \subset Q$. 
\end{proof}

\subsection{Distinguished presentations and locality}
\label{sub:abstract:locality}

We maintain the notation (and assumptions on) $\rhobar$, $M'_\infty$, $\nu$, and $\tau$ from the last section. 
In particular, throughout this section $\rhobar$ is $11$-generic.
Let $\kappa(\rhobar,\nu) \defeq \kappa$ be as in \ref{sec:patchingpresentation} so that the map $M^{(j)}_{(\omega_j,1)} \ra M^{(j)}_{\gamma^{(j)}} \ra M^{(j)}_{\kappa^{(j)}}$ is nonzero for $(\omega_j,1) \neq \gamma^{(j)}_1$. 
Let $D_{\mathrm{m},\nu}(\rhobar) \defeq Q_{\nu,\kappa(\rhobar,\nu)}$ and $D_{\mathrm{m}}(\rhobar)$ be 
\[
\bigoplus_{\nu \in W(\eta)} D_{\mathrm{m},\nu}(\rhobar). 
\]

\begin{rmk}
While the definition of $D_{\mathrm{m},\nu}(\rhobar)$ depends on a choice of tame inertial $L$-parameter (see \S \ref{sec:patchingpresentation}), we will see \emph{a posteriori} that $D_{\mathrm{m},\nu}(\rhobar)$ does not depend on this choice (see Remark \ref{rmk:tauind}). 
\end{rmk}

\begin{prop}\label{prop:Dm}
\begin{enumerate}
\item \label{item:modular} $\JH(D_{\mathrm{m}}(\rhobar)) \subset W^?(\rhobar)$. 
\item \label{item:Dmult} If $\sigma \in W^?(\rhobar)$, then $\sigma$ is a Jordan--H\"older factor of $D_{\mathrm{m}}(\rhobar)$ with multiplicity $3^{\#A(\sigma)}$. 
\item \label{item:Dsoc}The socle of $D_{\mathrm{m}}(\rhobar)$ is multiplicity free. 
\end{enumerate}
\end{prop}
\begin{proof}
\eqref{item:modular} follows from the fact that for each $Q_{\gamma_J}$ appearing in \S \ref{sec:patchingpresentation}, $\JH(Q_{\gamma_J}) \subset W^?(\rhobar)$. 

For each $\sigma = F(\Trns_{\mu}(\omega,a)) \in W^?(\rhobar)$, $\sigma \in \JH(D_{\mathrm{m},\nu}(\rhobar))$ if and only if for all $j\in \cJ$, $(-\eps_{w_j},\delta_{(-1)^{w_j}=-1}) = (\tld{w}_j(\omega),a)$ if $\nu_j = w_j \eta_j$ and 
\[
(\omega_j,a_j) \in \{(0,0),(\eps_1,0),(\eps_2,0),(0,1),(\eps_1,1),(\eps_2,1)\} 
\]
if $\nu_j = (1,1,1)_j$. 
In particular, there are $2^{\# A(\sigma)}$ choices of $\nu\in W(\eta)$ so that $\sigma \in \JH(D_{\mathrm{m},\nu}(\rhobar))$. 
Moreover, if $\sigma$ appears as a Jordan--H\"older factor of $D_{\mathrm{m},\nu}(\rhobar)$, then it appears with multiplicity $2^{\#\{j\in J(\nu)\mid a_j = 0\} }$. 
The binomial expansion of $(1+2)^{\# A(\sigma)}$ yields \eqref{item:Dmult}. 

\eqref{item:Dsoc} follows easily from Proposition \ref{prop:soclength}. 
\end{proof}

\begin{prop}\label{prop:essential}
For each irreducible submodule $\sigma \subset D_{\mathrm{m},\nu}(\rhobar)$, $M_\infty(\sigma) \notin \fm M_\infty(D_{\mathrm{m},\nu}(\rhobar))$. 
\end{prop}
\begin{proof}
Suppose that $\sigma \defeq F(\Trns_{\mu}(\omega,a))$ is in the image of the map $Q_{\gamma_J} \ra D_{\mathrm{m},\nu}(\rhobar)$ for some $\gamma_J$ of length $2\#J$.
By \S \ref{sec:patchingpresentation}, it suffices to show that the image of the composition $M^{(j)}_{(\omega_j,a_j)} \ra M^{(j)}_{\gamma^{(j)}} \ra M^{(j)}_{\kappa(\rhobar,\nu)^{(j)}}$ is not contained in $\fm^{(j)} M^{(j)}_{\kappa(\rhobar,\nu)^{(j)}}$ for all $j\in J$. 
This follows from Corollary \ref{cor:surgery}. 
\end{proof}

We now assume that $M_\infty$ has the form $\Hom_{\GL_3(\cO_p)}^{\textrm{cont}}(-,M_\infty^\vee)^\vee$ for a pseudocompact $\cO[\![\GL_3(\cO_p)]\!]$-module also denoted $M_\infty$. 
Let $\pi = (M_\infty/\fm)^\vee$. 
Then there is a natural transformation $\Hom_K(-,\pi) \cong (M_\infty(-)/\fm)^\vee$. 

\begin{prop}\label{prop:Dm-ess}
There is an injection $D_{\mathrm{m}}(\rhobar) \into \pi^{K_1}$. 
Moreover, this induces an isomorphism on socles. 
\end{prop}
\begin{proof}
First, we show that there is an injection $D_{\mathrm{m},\nu}(\rhobar) \into \pi$ for each $\nu\in W(\eta)$. 
Recall that $\soc\, D_{\mathrm{m},\nu}(\rhobar)$ is multiplicity free and has length $3^{\#J}$ by Proposition \ref{prop:soclength}. 
For cardinality reasons, there is an element of $(M_\infty(D_{\mathrm{m},\nu}(\rhobar))/\fm)^\vee$ which is not in the kernel of any map $(M_\infty(D_{\mathrm{m},\nu}(\rhobar))/\fm)^\vee \ra M_\infty(\sigma)/\fm$ for any irreducible submodule $\sigma \subset D_{\mathrm{m},\nu}(\rhobar)$ by Proposition \ref{prop:essential}. 
This element corresponds to an injective homomorphism $D_{\mathrm{m},\nu}(\rhobar) \ra \pi$. 

Taking a direct sum of such injections gives a map $D_{\mathrm{m}}(\rhobar) \ra \pi^{K_1}$. 
Any irreducible submodule of the kernel is a submodule of $D_{\mathrm{m},\nu}(\rhobar)$ for some $\nu\in W(\eta)$ since $\soc\, D_{\mathrm{m}}(\rhobar)$ is multiplicity free by Proposition \ref{prop:Dm}\eqref{item:Dsoc}. 
We conclude that the kernel is $0$. 
To show that this injection induces an isomorphism on socles, it suffices to show that the $K$-socle of $\pi$ is isomorphic to $\oplus_{\sigma\in W^?(\rhobar)} \sigma$. 
This follows from \cite[Lemma 5.3.3]{GL3Wild}. 
\end{proof}

\begin{prop}\label{prop:maxext}
Let $\cW$ be a set of Serre weights and $V\subset Q$ be an $\F[\rG]$-submodule of an $\F[\rG]$-module. 
Then there is a unique maximal submodule $U \subset Q$ such that $\JH(U/(V\cap U)) \cap \cW = \emptyset$. 
Moreover, $V \subset U$. 
\end{prop}
\begin{proof}
Suppose that $U,U' \subset Q$ are maximal submodules such that $\JH(U/(V\cap U)) \cap \cW = \JH(U'/(V\cap U')) \cap \cW = \emptyset$. 
Then there is a surjection $U/(V\cap U) \oplus U'/(V\cap U') \onto (U+U')/(V\cap (U+U'))$ from which we conclude that $\JH((U+U')/(V\cap (U+U'))) \cap \cW = \emptyset$. 
By maximality of $U$ and $U'$, we conclude that $U = U'$. 
Moreover, maximality implies that $V \subset U$. 
\end{proof}

Fix injections $D_{\mathrm{m},\nu}(\rhobar) \into \oplus_{\sigma \in \soc\, D_{\mathrm{m},\nu}(\rhobar)} P_\sigma$ for each $\nu \in W(\eta)$. 
Taking a direct sum gives an injection $D_{\mathrm{m}}(\rhobar) \into \oplus_{\sigma \in W^?(\rhobar)} P_\sigma$. 
We identify the domains with their images in $\oplus_{\sigma \in W^?(\rhobar)} P_\sigma$. 
Using Proposition \ref{prop:maxext}, we let $D_0(\rhobar)$ (resp.~$D_{0,\nu}(\rhobar)$ for each $\nu \in W(\eta)$) be the maximal submodule $U\subset \oplus_{\sigma \in W^?(\rhobar)} P_\sigma$ such that $\JH(U/(D_{\mathrm{m}}(\rhobar)\cap U)) \cap W^?(\rhobar) = \emptyset$ (resp.~$\JH(U/(D_{\mathrm{m},\nu}(\rhobar)\cap U)) \cap W^?(\rhobar) = \emptyset$).

\begin{prop}
We have 
\[
D_0(\rhobar) \cong \bigoplus_{\nu \in W(\eta)} D_{0,\nu}(\rhobar). 
\]
\end{prop}
\begin{proof}
Let $\nu \in W(\eta)$. 
Then $D_{0,\nu}(\rhobar)/(D_{\mathrm{m}}(\rhobar) \cap D_{0,\nu}(\rhobar))$ is a quotient of $D_{0,\nu}(\rhobar)/(D_{\mathrm{m},\nu}(\rhobar) \cap D_{0,\nu}(\rhobar))$ from which we conclude that $\JH(D_{0,\nu}(\rhobar)/(D_{\mathrm{m}}(\rhobar) \cap D_{0,\nu}(\rhobar))) \cap W^?(\rhobar) = \emptyset$ and thus that $D_{0,\nu}(\rhobar)\subset D_0(\rhobar)$. 
Since $\soc\, D_{\mathrm{m},\nu}(\rhobar) \subset \soc\, D_{0,\nu}(\rhobar) \subset \soc\, \oplus_{\sigma \in W^?(\rhobar)} P_\sigma \cong \oplus_{\sigma \in W^?(\rhobar)} \sigma$ and $\JH(\soc\, D_{0,\nu}(\rhobar)/\soc\, D_{\mathrm{m},\nu}(\rhobar)) \cap W^?(\rhobar) = \emptyset$, we have that $\soc\, D_{\mathrm{m},\nu}(\rhobar) = \soc\, D_{0,\nu}(\rhobar)$. 
Thus $\JH(\soc\, D_{0,\nu}(\rhobar))$ are pairwise disjoint for $\nu \in W(\eta)$, and the natural map 
\[
\bigoplus_{\nu \in W(\eta)} D_{0,\nu}(\rhobar) \ra D_0(\rhobar)
\]
is injective. 

It suffices to show that for each $\nu \in W(\eta)$, the image, denoted $D_{0,\nu}$, of the projection 
\[
D_0(\rhobar) \ra \oplus_{\sigma \in \JH(\soc\, D_{0,\nu}(\rhobar))} P_\sigma
\]
is contained in $D_{0,\nu}(\rhobar)$. 
The image of the restriction of this projection to $D_{\mathrm{m}}(\rhobar)$ is $D_{\mathrm{m},\nu}(\rhobar)$. 
Thus we have submodules
\[
D_{\mathrm{m},\nu}(\rhobar) \subset D_{0,\nu} \subset \oplus_{\sigma \in \JH(\soc\, D_{0,\nu}(\rhobar))} P_\sigma
\]
with $\JH(D_{0,\nu}/D_{\mathrm{m},\nu}(\rhobar)) \cap W^?(\rhobar) \subset \JH(D_0(\rhobar)/D_{\mathrm{m}}(\rhobar)) \cap W^?(\rhobar) = \emptyset$. 
Maximality implies that $D_{0,\nu} \subset D_{0,\nu}(\rhobar)$. 
\end{proof}

\begin{thm}\label{thm:K1}
There is an isomorphism $\pi^{K_1} \cong D_0(\rhobar)$. 
\end{thm}
\begin{proof}
The inclusion $D_{\mathrm{m}}(\rhobar) \subset D_0(\rhobar)$ induces an isomorphism after applying $M_\infty$. 
Thus we get an injection $D_0(\rhobar) \into \pi^{K_1}$ extending the injection $D_{\mathrm{m}}(\rhobar) \into \pi^{K_1}$. 
We can extend the map $D_0(\rhobar) \subset \oplus_{\sigma \in W^?(\rhobar)} P_\sigma$ to a map 
\[
\pi^{K_1} \into \oplus_{\sigma \in W^?(\rhobar)} P_\sigma
\]
by injectivity of $P_\sigma$ and Proposition \ref{prop:Dm-ess}. 
We claim that $\JH(\pi^{K_1}/D_{\mathrm{m}}(\rhobar)) \cap W^?(\rhobar) = \emptyset$. 
Then by maximality of $D_0(\rhobar)$, we conclude that $\pi^{K_1} \cong D_0(\rhobar)$. 

It suffices to prove the claim. 
Suppose that $\sigma \in W^?(\rhobar)$. 
The multiplicity of $\sigma$ as a Jordan--H\"older factor of $\pi^{K_1}$ is $\dim_{\F} \Hom_{\rG}(P_\sigma,\pi^{K_1}) = \dim_{\F} M_\infty(P_\sigma)/\fm = 3^{\#A(\sigma)}$ by Theorem \ref{thm:mingen}. 
This is precisely the multiplicity of $\sigma$ as a Jordan--H\"older factor of $D_{\mathrm{m}}(\rhobar)$ by Proposition \ref{prop:Dm}\eqref{item:Dmult}. 
The claim follows. 
\end{proof}

\begin{rmk}\label{rmk:tauind}
For $\nu\in W(\eta)$, the $\F[\rG]$-modules $D_{\mathrm{m},\nu}(\rhobar)$ (see \S \ref{sec:patchingpresentation}), and thus the modules $D_{0,\nu}(\rhobar)$, depend on $\rhobar|_{I_{\Q_p}}$ (to define $W^?(\rhobar)$), $\nu$, a choice of a tame inertial $L$-parameter $\tau$, and a choice of lowest alcove presentation for $\rhobar$ (to define $\kappa$ and to parametrize the obvious weights in $W^?(\rhobar)$). 
However, by Proposition \ref{prop:indecomp} and Theorem \ref{thm:K1}, $D_{0,\nu}(\rhobar)$ is an indecomposable summand of $\pi^{K_1}$ which is determined by its socle and is independent of $\tau$. 
By the Krull--Schmidt theorem, we conclude that both $D_{\mathrm{m},\nu}(\rhobar)$ and $D_{0,\nu}(\rhobar)$ depend only on $\rhobar|_{I_{\Q_p}}$, a lowest alcove presentation of $\rhobar$, and $\nu \in W(\eta)$. 
Moreover, since $\pi^{K_1}$ is independent of a lowest alcove presentation of $\rhobar$ and $\nu \in W(\eta)$, $D_0(\rhobar)$ depends only on $\rhobar|_{I_{\Q_p}}$. 
\end{rmk}

\subsection{Jordan--H\"older factors of $D_0(\rhobar)$}

We describe the Jordan--H\"older factors of $D_0(\rhobar)$ with multiplicity. 
Fix $\rhobar$, $\nu\in W(\eta)$, and $\tau = \tau(s,\mu-\un{1})$ with $\tld{w} = \tld{w}(\rhobar,\tau)$ as in \S \ref{sec:patchingpresentation}. 
We define a set $\Sigma_{\nu_j}$ for each $j\in \cJ$. 
Let 
\[
\Sigma_{w_j\eta_j} \defeq \{(\tld{w}(\omega-(w_0w_j)^{-1}(\eps_1+\eps_2)),a) \mid (\omega,a) \in \Sigma_0\} 
\]
for $w_j \in S_3$ and
\[
\Sigma_{(1,1,1)_j} \defeq \Sigma_0 \cup \{(-\eps_1,0),(-\eps_2,0),(2\eps_1,0),(2\eps_1-\eps_2,0),(2\eps_2,0),(2\eps_2-\eps_1,0)\}. 
\]

\begin{thm}\label{thm:multiplicity}
Let $\nu \in W(\eta)$. Then 
\[
\JH(D_{0,\nu}(\rhobar)) = F(\Trns_\mu(s\prod_{j\in \cJ}\Sigma_{\nu_j})). 
\]
Moreover, if $\sigma = F(\Trns_\mu(s\omega,a))$ and 
\[
n(\sigma) = \# \{j\in \cJ \mid \nu_j = (1,1,1)_j,\, a_j = 0, \textrm{ and } (\omega,a)\in \Sigma_0\}, 
\]
then $\sigma$ appears in $\JH(D_{0,\nu}(\rhobar))$ with multiplicity $2^{n(\sigma)}$. 
\end{thm}
\begin{proof}
Fix an injective envelope 
\[
\iota: D_{0,\nu}(\rhobar) \into \bigoplus_{\sigma \in \JH(\soc D_{0,\nu}(\rhobar))} P_\sigma. 
\]
Let $V_\sigma \subset P_\sigma$ be the Weyl submodule in the Weyl filtration in Proposition \ref{prop:filtrations}. 
If $\sigma\in \JH(\soc D_{0,\nu}(\rhobar))$ and $\sigma \uparrow \kappa$ ($\sigma$ and $\kappa$ are linked Serre weights with $\kappa$ in a higher alcove), then $\sigma$ appears with multiplicity one in $\JH(D_{0,\nu}(\rhobar))$ and $\kappa\notin \JH(D_{0,\nu}(\rhobar))$ if $\kappa \not\cong \sigma$. 
Using this, it is easy to show using the dual Weyl filtration in Proposition \ref{prop:filtrations} and \cite[Lemma 4.2.2]{LLLM2} that the $\F$-dual $\iota^\vee$ of $\iota$ factors through the dual Weyl quotients $V_\sigma^\vee$ of $P_{\sigma^\vee}$. 
For a tuple $a$, let $\rad^a V_\sigma^\vee$ be the image of $\rad^a P_{\sigma^\vee}$ in $V_\sigma^\vee$. 

Let $J = J(\nu) = \{j\in \cJ\mid \nu_j = (1,1,1)_j\}$. 
If $i \in J$, we claim that $\rad^{2_i} V_\sigma^\vee \subset \ker \iota^\vee$. 
It suffices to show that the induced map 
\begin{equation}\label{eqn:2i0}
\gr^{2_i} V_\sigma^\vee \ra D_{0,\nu}(\rhobar)^\vee/\iota^\vee(\rad^{>2_i}V_\sigma^\vee) 
\end{equation} 
is $0$ as $\gr^{2_i} V_\sigma^\vee$ is the cosocle of $\rad^{2_i} V_\sigma^\vee$ by \cite[Lemma 4.2.2]{LLLM2}. 
We first claim that the induced map 
\begin{equation}\label{eqn:2i}
\gr^{2_i} V_\sigma^\vee \ra D_{\mathrm{m},\nu}(\rhobar)^\vee/\iota^\vee(\rad^{>2_i}V_\sigma^\vee) 
\end{equation}
is $0$. 
Indeed, \eqref{eqn:2i} factors through $\rad D_{\mathrm{m},\nu}(\rhobar)^\vee/\iota^\vee(\rad^{>2_i}V_\sigma^\vee)$, but $\JH(\rad D_{\mathrm{m},\nu}(\rhobar)^\vee) \cap \JH(\gr^{2_i} V_\sigma^\vee) = \emptyset$ by alcove considerations. 
Thus, if $\kappa^\vee\subset \gr^{2_i} V_\sigma^\vee$ with $\kappa \in W^?(\rhobar)$, then $\kappa^\vee$ maps to $0$ under \eqref{eqn:2i0}. 

Now suppose that $\kappa^\vee\subset \gr^{2_i} V_\sigma^\vee$ is a simple $\rG$-submodule with $\kappa \notin W^?(\rhobar)$. 
(For example, if $\sigma \cong \Trns_\mu(s\omega,a)$ with $(\omega_i,a_1) = (0,1)$ and $\kappa \cong \Trns_\mu(s\xi,a)$, then $(\xi_i,a_i) = (-\eps,1)$ for $\eps = \eps_1$ or $\eps_2$.) 
To show that $\kappa^\vee$ maps to $0$ under \eqref{eqn:2i0}, it suffices to show that $\kappa$ is not a Jordan--H\"older factor of $\iota(D_{0,\nu}(\rhobar))$. 
Suppose otherwise. 
Note that $\kappa$ is a Jordan--H\"older factor of $P_\sigma$ for a unique $\sigma \in \JH(\soc D_{0,\nu}(\rhobar))$ which we now fix. 
Thus, $\kappa$ is a Jordan--H\"older factor of $\iota(D_{0,\nu}(\rhobar)) \cap V_\sigma$. 
Propositions \ref{prop:filtrations} and \ref{prop:weylext} imply that $\iota(D_{\mathrm{m},\nu}(\rhobar)) \cap V_\sigma$, and in particular $D_{\mathrm{m},\nu}(\rhobar)$, contains an extension of $\sigma'$ by $\sigma$ where $\sigma'$ is the weight linked to $\sigma$ with alcove differing precisely at $i \in \cJ$. 
This contradicts the fact that $D_{\mathrm{m},\nu}(\rhobar)$ does not contain the extension of two simple modules in $W^?(\rhobar)$. 

To summarize, there is a surjection 
\begin{equation}\label{eqn:weyltrancate}
\iota^\vee: \bigoplus_{\sigma \in \JH(\soc D_{0,\nu}(\rhobar))} (V_\sigma^\vee/\sum_{i\in J}\rad^{2_i} V_\sigma^\vee) \onto D_{0,\nu}(\rhobar)^\vee. 
\end{equation}
Let $V^\vee$ be the domain of \eqref{eqn:weyltrancate}. 
Let $D_{0,\nu}^\vee$ denote the quotient of $V^\vee$ by the image of 
\[
\bigoplus_{\sigma \in W^?(\rhobar)} \ker(\Hom(P_\sigma,V^\vee)\ra \Hom(P_\sigma,D_{0,\nu}(\rhobar)^\vee)) \otimes P_\sigma
\]
under the evaluation map. 
Then \eqref{eqn:weyltrancate} induces an isomorphism $D_{0,\nu}^\vee \risom D_{0,\nu}(\rhobar)^\vee$ by the maximality property of $D_{0,\nu}(\rhobar)$. 
Let $V$ and $D_{0,\nu}$ be the $\F$-duals of $V^\vee$ and $D_{0,\nu}^\vee$, respectively. 
It suffices to show that $D_{0,\nu}$ has the properties asserted in the theorem for $D_{0,\nu}(\rhobar)$. 

First, if $\kappa \in \JH(V) \setminus F(\Trns_\mu(s\prod_{j\in \cJ}\Sigma_{\nu_j}))$, then Propositions \ref{prop:weylext} and \ref{prop:doubleadj} imply that there exists $\sigma \in W^?(\rhobar) \setminus \JH(D_{\mathrm{m},\nu})$ such that the map $\Hom_{\rG}(P_{\kappa^\vee},P_{\sigma^\vee})\otimes \Hom_{\rG}(P_{\sigma^\vee},V^\vee)\ra \Hom_{\rG}(P_{\kappa^\vee},V^\vee)$ induced by composition is surjective. 
Since the induced map $\Hom_{\rG}(P_{\sigma^\vee},V^\vee) \ra \Hom_{\rG}(P_{\sigma^\vee},D_{0,\nu}^\vee) \risom \Hom_{\rG}(P_{\sigma^\vee},D_{\mathrm{m},\nu}^\vee)$ is $0$, we have that $\Hom_{\rG}(P_{\kappa^\vee},D_{0,\nu}^\vee) = 0$. 
We conclude that 
\[
\JH(D_{0,\nu}^\vee) \subset F(\Trns_\mu(s\prod_{j\in \cJ}\Sigma_{\nu_j})). 
\]

Now suppose that $\kappa \in F(\Trns_\mu(s\prod_{j\in \cJ}\Sigma_{\nu_j}))$. 
Let $N(\kappa)$ denote the set of weights in $\JH(D_{\mathrm{m},\nu})$ nearest to $\kappa$ in the metric defined in \cite[Definition 2.1.8]{LLLM2}. 
(The set $N(\kappa)$ may have more than one element.) 
Let $d$ be the distance of $\kappa$ to elements in $N(\kappa)$. 
It suffices to show that $[D_{0,\nu}^\vee:\kappa^\vee] = \sum_{\sigma \in N(\kappa)} [D_{0,\nu}^\vee:\sigma^\vee]$. 
For each $\sigma \in N(\kappa)$, fix a lift $P_{\kappa^\vee} \ra P_{\sigma^\vee}$ of a nonzero map $P_{\kappa^\vee} \ra P_{\sigma^\vee}/\rad^{d+1} P_{\sigma^\vee}$. 
Then the induced map 
\[
\bigoplus_{\sigma \in N(\kappa)} \Hom_{\rG}(P_{\sigma^\vee},V^\vee) \ra \Hom_{\rG}(P_{\kappa^\vee},V^\vee) 
\]
is an isomorphism. 
Thus, $[D_{0,\nu}^\vee:\kappa^\vee] \leq \sum_{\sigma \in N(\kappa)} [D_{0,\nu}^\vee:\sigma^\vee]$. 
If $\tau \in \JH(D_{\mathrm{m},\nu})$ and a composition $P_{\kappa^\vee} \ra P_{\tau^\vee} \ra V^\vee$ is nonzero, then this composition can be written as a composition $P_{\kappa^\vee} \ra P_{\sigma^\vee} \ra P_{\tau^\vee} \ra V^\vee$ where the first map is the one fixed above. 
We conclude that $[D_{0,\nu}^\vee:\kappa^\vee] = \sum_{\sigma \in N(\kappa)} [D_{0,\nu}^\vee:\sigma^\vee]$. 
\end{proof}

\subsection{Global applications}

We now apply the results of \S \ref{sub:abstract:locality} to obtain instances of local--global compatibility in the mod-$p$ Langlands correspondence for $\GL_3$.
We follow the setup (and most of the notation) of \cite[\S 5.3]{LLLM2}. 
In particular $F/F^+$ is a CM extension which is unramified at all finite places.

We fix a totally definite outer form $H_{/F^+}$ of $\GL_3$ which splits over $F$, and a compact open subgroup $U^p\leq H(\bA_{F^+}^{\infty p})$.
Given a finite smooth $\F[U^p]$-module $W$ we have the space of mod $p$ algebraic automorphic forms \[
S(U^p,W) \defeq \left\{f:\,H(F^{+})\backslash H(\A^{\infty}_{F^{+}})\rightarrow W\,|\, f(gu)=u^{-1}f(g)\,\,\forall\,\,g\in G(\A^{\infty}_{F^{+}}), u\in U^p\right\}.
\]
This space carries commuting actions of a Hecke algebra $\bT^{\textnormal{univ}}$ and $H(F^+_p)$.
We fix a maximal ideal $\fm\subset\bT^{\textnormal{univ}}$ in the support of $S(U^p,W)$ giving rise to a Galois representation $\rbar:G_F\ra\GL_3(\F)$.
We now make the same assumptions as in \cite[\S 5.3]{LLLM2}, namely:
\begin{enumerate}
\item 
\label{eq:glob:1}
$p$ is unramified in $F^+$ and all places in $F^+$ above $p$ split in $F$;
\item $H_{/F^+}$ is quasi-split at all finite places;
\item $U^p$ is as in \cite[\S 5.3, (1)-(3)]{LLLM2} (so that $U^p$ is hyperspecial at all but one auxiliary place);
\item $W=W_{\Sigma_{0}^+}\otimes_{\cO} \F$ where $W_{\Sigma_{0}^+}$ is obtained using minimally ramified types away from $p$;
\item $\rbar$ satisfies the Taylor--Wiles conditions of \cite[Definition 7.3]{LLLM}; and
\item 
\label{eq:glob:5}
if $\rbar$ is ramified at a finite place $w$ of $F$ then $w|_{F^+}$ splits in $F$.
\end{enumerate}

We set $\pi(\rbar)\defeq S(U^p,W)[\fm]$.
Let $S_p$ be the set of places of $F^+$ above $p$.
For each $v\in S_p$ fix a place $w|v$ of $F$, and isomorphisms $F^+_v\cong F_w$,  $H(F^+_v)\cong \GL_3(F_w)$ (see \cite[\S 5.3]{LLLM2}).
As explained in \S \ref{subsubsec:L_parameters} the collection $\{\rbar|_{G_{F_v^+}}\}_{v\in S_p}$ gives rise to an $L$-homomorphism $\rhobar:G_{\Qp}\ra{}^L\un{G}(\F)$.

\begin{thm}
\label{main:glob:app}
Let $\rbar:G_F\ra\GL_3(\F)$ a continuous Galois representation satisfying items \eqref{eq:glob:1}--\eqref{eq:glob:5}, and let $\rhobar$ the $L$-homomorphism corresponding to $\{\rbar|_{G_{F_v^+}}\}_{v\in S_p}$.
We further assume that $\rhobar$ is tame and $11$-generic.
Then
\[
\pi(\rbar)^{K_1}\cong D_0(\rhobar)
\]
where $D_0(\rhobar)$ is as in \S \ref{sub:abstract:locality}.
\end{thm}
\begin{proof}
As in \cite[Theorem 5.2.1]{HLM}, we can and do choose a weak minimal patching functor $M_\infty$ of the form $\Hom_{\GL_3(\cO_p)}^{\textrm{cont}}(-,M_\infty^\vee)^\vee$ for a pseudocompact $\cO[\![\GL_3(\cO_p)]\!]$-module also denoted $M_\infty$ satisfying Definition \ref{minimalpatching} so that $\pi(\rhobar) \cong (M_\infty/\fm)^\vee$. 
The result follows now from Theorem \ref{thm:K1}.
\end{proof}
\vfill

\appendix
\section{Tables for multi-type deformation rings}
\begin{table}[H]
\caption[Foo content]{\textbf{The ring $\tld{S}$ and its ideals $\tld{I}^{(j)}_{\tau,\nabla_\infty}$.
}
}
\label{Table_Ideals}
\centering
\tiny
\adjustbox{max width=\textwidth}{
\begin{tabular}{| c | c | c | c |}
\hline
&\multicolumn{3}{c |}{}\\
$A^{T,(j)}\tld{w}^{*,T}(\rhobar)_j$ &
\multicolumn{3}{c |}{
$\begin{pmatrix}
d_{11}^{*}(v+p)+c_{11}+\frac{e_{11}}{v}&c_{12}&c_{13}\\
d_{21}(v+p)+c_{21}&d^{*}_{22}(v+p)+c_{22}&c_{23}\\
d_{31}(v+p)+c_{31}&d_{32}(v+p)+c_{32}
&d_{33}^{*}(v+p)+c_{33}
\end{pmatrix}\cdot s_j^{-1}v^{\mu_j+(1,0,-1)}$}\\
&\multicolumn{3}{c |}{}\\
\hline
&\multicolumn{3}{c |}{}\\
$\tld{S}^{(j)}$& \multicolumn{3}{c |}{$\cO[\![c_{11},x^*_{11},e_{11},c_{12},c_{13},d_{21},c_{21},c_{22},x_{22}^*,c_{23},d_{31},c_{31},d_{32},c_{32},c_{33},x^*_{33}]\!]$}\\
&\multicolumn{3}{c |}{}\\
\hline
\hline
$\tld{w}^*(\rhobar,\tau)_j$&\multicolumn{3}{c |}{}\\
\hline
&&&\\
\multirow{8}{*}{$\alpha\beta\alpha \gamma t_{\un{1}}$}&
\multirow{8}{*}{$\tld{I}^{(j)}_{\tau,\nabla_{\infty}}$}&\multirow{5}{*}{$\tld{I}^{(j)}_{\tau}$}&
$\begin{aligned}
&c_{13},\ c_{23}\\
&c_{32}+pd_{32},\ c_{33}+pd_{33}^*
\end{aligned}$\\
&&&\\
\cline{4-4}
&&&\\
&&&$\begin{aligned}
&c_{21},\,c_{22}\\
&c_{11}-pd_{11}^*,\, e_{11}-p^2d_{11}^*
\end{aligned}$\\
&&&\\
\cline{3-4}
&&&\\
&&\multirow{2}{*}{$\textnormal{Mon}_{\tau}$}&
$\begin{aligned}
&(b_{\tau,1}-b_{\tau,3}-1)d^*_{22}c_{31}+(b_{\tau,2}-b_{\tau,3}-1)d_{21}d_{32}+pd_{31}d_{22}^*+O(p^{N-4})
\end{aligned}$
\\
&&&\\
\hline
&&&\\
\multirow{8}{*}{$\alpha\beta\alpha t_{\un{1}}$}&
\multirow{8}{*}{$\tld{I}^{(j)}_{\tau,\nabla_{\infty}}$}&\multirow{5}{*}{$\tld{I}^{(j)}_{\tau}$}&
$\begin{aligned}
&e_{11},\,c_{21},\ c_{22},\ c_{23}\\
&c_{32}+pd_{32},\ c_{33}+pd_{33}^*
\end{aligned}$\\
&&&\\
\cline{4-4}
&&&\\
&&&$\begin{aligned}
&c_{13}d_{32}-c_{12}d_{33}^*,\\
&c_{13}d_{31}-c_{11}d_{33}^*+pd_{11}^*d_{33}^*\\
&c_{13}c_{31}+pc_{11}d_{33}^*
\end{aligned}$\\
&&&\\
\cline{3-4}
&&&\\
&&\multirow{2}{*}{$\textnormal{Mon}_{\tau}$}&
$\begin{aligned}
&(b_{\tau,2}-b_{\tau,3})d_{21}c_{12}+(b_{\tau,3}-b_{\tau,1})c_{11}d_{22}^*-pd_{11}^*d_{22}^*+O(p^{N-4})
\\
&p(b_{\tau,2}-b_{\tau,3})d_{21}d_{32}+(b_{\tau,1}-b_{\tau,3}+1)c_{31}d_{22}^*+pd_{31}d_{22}^*+O(p^{N-4})
\end{aligned}$
\\
&&&\\
\hline
&&&\\
\multirow{8}{*}{$\alpha\beta t_{\un{1}}$}&
\multirow{8}{*}{$\tld{I}^{(j)}_{\tau,\nabla_{\infty}}$}&\multirow{5}{*}{$\tld{I}^{(j)}_{\tau}$}&
$\begin{aligned}
&e_{11},\,c_{22}+pd_{22}^*,\ c_{32}+pd_{32},\ c_{33}+pd_{33}^*
\end{aligned}$\\
&&&\\
\cline{4-4}
&&&\\
&&&$\begin{aligned}
&c_{13}d_{32}-c_{12}d_{33}^*,\\
&c_{31}d_{22}^*-d_{32}c_{21}\\
&c_{23}d_{32}+pd_{22}^*d_{33}^*\\
&c_{13}c_{21}-c_{23}c_{11}\\
&c_{11}d_{33}^*-d_{31}c_{13}-pd_{11}^*d_{33}^*
\end{aligned}$\\
&&&\\
\cline{3-4}
&&&\\
&&\multirow{2}{*}{$\textnormal{Mon}_{\tau}$}&
$\begin{aligned}
&(b_{\tau,2}-b_{\tau,3})d_{21}c_{12}+(b_{\tau,3}-b_{\tau,1})c_{11}d_{22}^*+p(b_{\tau,2}-b_{\tau,3}-1)d_{11}^*d_{22}^*+O(p^{N-4})
\\
&(b_{\tau,2}-b_{\tau,3}-1)c_{23}d_{31}+(b_{\tau,1}-b_{\tau,2}+1)c_{21}d_{33}^*+p(b_{\tau,2}-b_{\tau,3})d_{21}d_{33}^*+O(p^{N-4})
\end{aligned}$
\\
&&&\\
\hline
\end{tabular}}
\caption*{\footnotesize{
We give a presentation of $\cO(M_{\rhobar}^{\nabla_{T,\infty}})$ when $\#T=1$.
Let $z t_\nu\defeq\tld{w}^*(\rhobar,\tau)$ and define $b^{(j)}_{\tau}\in T(\Zp)$ by $z_j^{-1}(b^{(j)}_{\tau})=(s'_{\orient,j})^{-1}(\bf{a}^{\prime(j)})/(p^{f'}-1)$, where $\bf{a}^{\prime(j)}$ is defined with respect to $(sz,\mu-sz(\nu))\in \un{W}\times X^*(\un{T})$ (see \S \ref{subsubsec:TIT}).
For readability, we write $b_{\tau},c_{ik}$, etc.~instead of $b^{(j)}_{\tau},c^{(j)}_{ik}$, etc.
We assume that $\mu$ is $N$-deep in $\un{C}_0$ (hence $\mu-sz(\nu)$ is $(N-2)$-deep in $\un{C}_0$).
Note that $D^{\tau, \overline{\beta}}_{\overline{\fM}_\tau}\into M_{\rhobar}^{T}\cdot\tld{w}^{*, T}(\rhobar)$ (Proposition \ref{prop:factor_mono}) is given by $A^{(j)}_{\fM,\beta}\mapsto A^{(j)}_{\fM,\beta}\cdot z^{-1}_j\cdot \tld{w}^{*,T}(\rhobar)_j$ when {$T^{(j)}$ is as in \ref{it:prop:T:1}-\ref{it:prop:T:1'}}.
Finally $(a,b,c)\defeq b_{\tau}$ modulo $\varpi$ and note that $(a,b,c)\equiv -\big(s_j^{-1}(\mu_j+\eta_j)-z(\nu)\big)$ modulo $\varpi$.}
}
\end{table}

\begin{table}[H]
\captionsetup{justification=centering}
\caption[Foo content]{\textbf{Further cases of the ideals $\tld{I}^{(j)}_{\tau}$, $\tld{I}^{(j)}_{\tau,\nabla_\infty}$ and ${I}^{(j)}_{\tau,\nabla_{\textnormal{alg}}}$ when $\tld{w}^*(\rhobar,\tau)_j\in\{\beta\alpha t_{\un{1}}, t_{\un{1}}\}$.
}
}
\label{Table_Ideals_1}
\centering
\adjustbox{max width=\textwidth}{
\begin{tabular}{| c | c | c | c |}
\hline
\hline
$\tld{w}^*(\rhobar,\tau)_j$&\multicolumn{3}{c |}{}\\
\hline
&&&\\
\multirow{8}{*}{$\beta\alpha t_{\un{1}}$}&
\multirow{8}{*}{$\tld{I}^{(j)}_{\tau,\nabla_{\infty}}$}&\multirow{5}{*}{$\tld{I}^{(j)}_{\tau}$}&
$\begin{aligned}
&e_{11},\,c_{21},\ c_{22},\ c_{23}\\
&c_{31}+pd_{31},\ c_{33}+pd_{33}^*
\end{aligned}$\\
&&&\\
\cline{4-4}
&&&\\
&&&$\begin{aligned}
&c_{11}d_{33}^*-c_{13}d_{31},\\
&c_{13}c_{32}+pd_{33}^*c_{12}\\
&d_{21}(c_{13}d_{32}-c_{12}d_{33}^*)-pd_{11}^*d_{22}^*d_{33}^*
\end{aligned}$\\
&&&\\
\cline{3-4}
&&&\\
&&\multirow{2}{*}{$\textnormal{Mon}_{\tau}$}&
$\begin{aligned}
&(b_{\tau,2}-b_{\tau,3})d_{21}c_{12}+(b_{\tau,3}-b_{\tau,1})c_{11}d_{22}^*-pd_{11}^*d_{22}^*+O(p^{N-4})
\\
&(b_{\tau,1}-b_{\tau,3})c_{12}d_{31}+(b_{\tau,3}-b_{\tau,1})c_{11}d_{32}+(b_{\tau,3}-b_{\tau,2}-1)c_{32}d_{11}^*-pd_{32}d_{11}^*+O(p^{N-4})
\end{aligned}$
\\
&&&\\
\hline
&&&\\
\multirow{9}{*}{$t_{\un{1}}$}&
\multirow{9}{*}{$\tld{I}^{(j)}_{\tau,\nabla_\infty}$}&\multirow{5}{*}{$\tld{I}^{(j)}_\tau$}&
$\begin{aligned}
&e_{11},\,c_{21}+pd_{21},\ c_{31}+pd_{31},\ c_{32}+pd_{32}
\end{aligned}$\\
&&&\\
\cline{4-4}
&&&\\
&&&$\begin{aligned}
&c_{23}d_{31} - d_{21}c_{33},\ c_{12}d_{31} - c_{11}d_{32}, \\
&c_{13}c_{22} - c_{12}c_{23},\  pc_{13}d_{32} + c_{12}c_{33},\\
& c_{22}d_{31} + pd_{21}d_{32},\  pc_{13}d_{31} + c_{11}c_{33},\\
&pc_{13}d_{21} + c_{11}c_{23},\  pc_{12}d_{21} + c_{11}c_{22},\\
& c_{13}d_{21}d_{32} - c_{12}d_{21}d_{33}^* - c_{13}d_{31}d_{22}^* - c_{23}d_{32}d_{11}^* + c_{11}d_{22}^* d_{33}^*+ d_{11}^*c_{22}d_{33}^* + d_{11}^*d_{22}^*c_{33},\\
&c_{12}d_{21}c_{33} - c_{11}c_{22}d_{33}^* - c_{11}d_{22}^*c_{33} - d_{11}^*c_{22}c_{33} - p(c_{11}d_{22}^*d_{33}^*  +d_{11}^*c_{22}d_{33}^* +d_{11}^*d_{22}^*c_{33})
\end{aligned}$\\
&&&\\
\cline{3-4}
&&&\\
&&\multirow{3}{*}{$\textnormal{Mon}_{\tau}$}&
$\begin{aligned}
&(b_{1,\tau}-b_{3,\tau}-1)(c_{23}d_{32}-c_{33}d_{22}^*)-(b_{1,\tau}-b_{2,\tau}-1)c_{22}d_{33}^*+pd_{11}^*d_{22}^*d_{33}^*+O(p^{N-4}),\\
&(b_{1,\tau}-b_{2,\tau})(c_{13}d_{31}-c_{11}d_{33}^*)+(b_{2,\tau}-b_{3,\tau}-1)c_{33}d_{11}^*-pd_{11}^*d_{22}^*d_{33}^*+O(p^{N-4}),\\
&(b_{2,\tau}-b_{3,\tau})(c_{12}d_{21}-c_{11}d_{11}^*)-(b_{1,\tau}-b_{3,\tau})c_{11}d_{22}^*-pd_{11}^*d_{22}^*d_{33}^*+O(p^{N-4}),
\end{aligned}$\\
&&&\\
\hline
\hline
\end{tabular}}
\captionsetup{justification=raggedright,
singlelinecheck=false
}
\caption*{\footnotesize{
}
}
\end{table}

\begin{table}[H]
\captionsetup{justification=centering}
\caption[Foo content]{\textbf{The ring $\tld{S}$ and its ideals $\tld{I}^{(j)}_{\tau,\nabla_\infty}$.
}
}
\label{Table_Ideals_2}
\centering
\adjustbox{max width=\textwidth}{
\begin{tabular}{| c | c | c | c |}
\hline
&\multicolumn{3}{c |}{}\\
$A^{T,(j)}\tld{w}^{*,T}(\rhobar)_j$ &
\multicolumn{3}{c |}{
$\begin{pmatrix}
d_{11}^{*}&c_{12}&c_{13}(v+p)+e_{13}\\
d_{21}&d^{*}_{22}(v+p)+c_{22}&c_{23}(v+p)+e_{23}\\
d_{31}&d_{32}(v+p)+c_{32}
&d_{33}^{*}(v+p)^2+c_{33}(v+p)+e_{33}
\end{pmatrix}\cdot s_j^{-1}v^{\mu_j}$}\\
&\multicolumn{3}{c |}{}\\
\hline
&\multicolumn{3}{c |}{}\\
$\tld{S}^{(j)}$& \multicolumn{3}{c |}{$\cO[\![x^*_{11},c_{12},c_{13},e_{13},d_{21},c_{22},x_{22}^*,c_{23},e_{23},d_{31},d_{32},c_{32},c_{33},e_{33},x^*_{33}]\!]$}\\
&\multicolumn{3}{c |}{}\\
\hline
\hline
$\tld{w}^*(\rhobar,\tau)_j$&\multicolumn{3}{c |}{}\\
\hline
&&&\\
\multirow{8}{*}{$t_{w_0(\eta)}$}&
\multirow{8}{*}{$\tld{I}^{(j)}_{\tau,\nabla_{\infty}}$}&\multirow{5}{*}{$\tld{I}^{(j)}_{\tau}$}&
$\begin{aligned}
&d_{21},\ d_{31},\ c_{32}+pd_{32}
\end{aligned}$\\
&&&\\
\cline{4-4}
&&&\\
&&&$\begin{aligned}
&e_{33},\  c_{33},\  d_{32},\  e_{23},\  c_{22}
\end{aligned}$\\
&&&\\
\cline{3-4}
&&&\\
&&$\textnormal{Mon}_\tau$&
$\begin{aligned}
&(b_{\tau,1} - b_{\tau,2})c_{12}c_{23} + pc_{13}d_{22}^* - (b_{\tau,1} - b_{\tau,3})e_{13}d_{22}^*+O(p^{N-4})
\end{aligned}$
\\
&&&\\
\hline
&&&\\
\multirow{8}{*}{$t_{w_0(\eta)}\alpha$}&
\multirow{8}{*}{$\tld{I}^{(j)}_{\tau,\nabla_{\infty}}$}&\multirow{5}{*}{$\tld{I}^{(j)}_{\tau}$}&
$\begin{aligned}
&c_{22}+pd_{22}^*,\ c_{32}+pd_{32},\ d_{31}
\end{aligned}$\\
&&&\\
\cline{4-4}
&&&\\
&&&$\begin{aligned}
&e_{33}, c_{33}, d_{32}, \\
&e_{13}d_{21} - e_{23}d_{11}^*, c_{12}d_{21} + pd_{11}^*d_{22}^*
\end{aligned}$\\
&&&\\
\cline{3-4}
&&&\\
&&$\textnormal{Mon}_\tau$&
$\begin{aligned}
&
(b_{\tau,2}-b_{\tau,1})c_{12}c_{23}+p(b_{\tau,2}-b_{\tau,1}-1)c_{13}d_{22}^*+(b_{\tau,1}-b_{\tau,3})e_{13}d_{22}^*+O(p^{N-1})
\end{aligned}$
\\
&&&\\
\hline
&&&\\
\multirow{8}{*}{$t_{w_0(\eta)}\beta$}&
\multirow{8}{*}{$\tld{I}^{(j)}_{\tau,\nabla_{\infty}}$}&\multirow{5}{*}{$\tld{I}^{(j)}_{\tau}$}&
$\begin{aligned}
&d_{21},\ d_{31},\ e_{33}+pc_{33}+p^2d_{33}^*
\end{aligned}$\\
&&&\\
\cline{4-4}
&&&\\
&&&$\begin{aligned}
&c_{32},\ e_{23},\ c_{22},\ c_{33} + pd_{33}^*,\ c_{23}d_{32} + pd_{22}^*d_{33}^*
\end{aligned}$\\
&&&\\
\cline{3-4}
&&&\\
&&$\textnormal{Mon}_\tau$&
$\begin{aligned}
&(b_{\tau,1}-b_{\tau,2})c_{12}c_{23}+pc_{13}d_{22}^*-(b_{\tau,1}-b_{\tau,3})e_{13}d_{22}^*+O(p^{N-4})
\end{aligned}$
\\
&&&\\
\hline
&&&\\
\multirow{8}{*}{$t_{w_0(\eta)}w_0$}&
\multirow{8}{*}{$\tld{I}^{(j)}_{\tau,\nabla_{\infty}}$}&\multirow{5}{*}{$\tld{I}^{(j)}_{\tau}$}&
$\begin{aligned}
&pd_{32}+c_{32},\ pc_{23}+e_{23},\ 
 e_{33}+pc_{33}+p^2d_{33}^*
\end{aligned}$\\
&&&\\
\cline{4-4}
&&&\\
&&&$\begin{aligned}
&e_{13}d_{32}-c_{12}c_{33}-pc_{12}d_{33}^*,\ c_{23}d_{31}-d_{21}c_{33}-pd_{21}d_{33}^*,\\
&c_{12}d_{31}+pd_{32}d_{11}^*,\ e_{13}d_{21}+pc_{23}d_{11}^*,\ c_{12}d_{21}-c_{22}d_{11}^*,\\
& c_{13}d_{21}d_{32}-c_{13}d_{31}d_{22}^*-c_{23}d_{32}d_{11}^*+c_{33}d_{11}^*d_{22}^*,\ e_{13}d_{31}+pc_{33}d_{11}^*+p^2d_{33}^*d_{11}^*
\end{aligned}$\\
&&&\\
\cline{3-4}
&&&\\
&&$\textnormal{Mon}_\tau$&
$\begin{aligned}
&(b_{\tau,2}-b_{\tau,1})
      (c_{13}c_{22}-c_{12}c_{23})+pc_{13}d_{22}^*+(b_{\tau,3}-b_{\tau,1})e_{13}d_{22}^*+O(p^{N-4}),
      \\
      &(b_{\tau,3}-b_{\tau,1}-1)(c_{23}d_{32}-c_{33}d_{11}^*)-(b_{\tau,2}-b_{\tau,1})c_{22}d_{33}^*-p(b_{\tau,3}-b_{\tau,1})d_{11}^*d_{33}^*+O(p^{N-4}),
      \\
      &(b_{\tau,2}-b_{\tau,1}-1)c_{13}
      d_{31}+(b_{\tau,3}-b_{\tau,2})c_{33}d_{11}^*+p(b_{\tau,3}-b_{\tau,1})d_{11}^*d_{33}^*+O(p^{N-4})
\end{aligned}$
\\
&&&\\
\hline
\end{tabular}}
\caption*{\footnotesize{
}
}
\end{table}

\begin{table}[H]
\caption{\textbf{Special fiber for the algebraic multi-type deformation ring when $T^{(j)}=\{\alpha\beta t_{\un{1}},\alpha\beta\alpha t_{\un{1}}\}$ or $T^{(j)}=\{\beta\alpha t_{\un{1}},\alpha\beta\alpha t_{\un{1}}\}$.
}
}
\label{Table_Ideals_mod_p}
\centering
\adjustbox{max width=\textwidth}{
\begin{tabular}{| c | c |}
\hline
&\\
$\ovl{A}^{T,(j)}\tld{w}^*(\rhobar)$ &
$\begin{pmatrix}
d_{11}^{*}v+c_{11}+\frac{e_{11}}{v}&c_{12}&c_{13}\\
d_{21}v+c_{21}&d^{*}_{22}v+c_{22}&c_{23}\\
d_{31}v+c_{31}&d_{32}v+c_{32}
&d_{33}^{*}v+c_{33}
\end{pmatrix}\cdot s_j^{-1}v^{\mu_j+(1,0,-1)}$\\
&\\
\hline
&\\
${S}^{(j)}$&$\F[\![c_{11},x^*_{11},e_{11},c_{12},c_{13},d_{21},c_{21},c_{22},x_{22}^*,c_{23},d_{31},c_{31},d_{32},c_{32},c_{33},x^*_{33}]\!]$\\
&\\
\hline
\hline
$T^{(j)}$&some elements of ${I}^{(j)}_{T,\nabla_{\textnormal{alg}}}$\\
\hline
&\\
$\alpha\beta\alpha t_{\un{1}}, \alpha\beta t_{\un{1}}$&
$\begin{aligned}
&e_{11},\,c_{33},\ c_{32},\ c_{23}d_{32}-c_{22}d_{33}^*,\ c_{13}d_{32}-c_{12}d_{33}^*,\ c_{23}c_{31},\ c_{13}c_{31},
\\ 
&c_{13}d_{31}-c_{11}d_{33}^*,\ c_{22}c_{23},\ c_{21}c_{22},\ c_{13}c_{21}-c_{11}c_{23},\ c_{21}^2d_{32}-c_{21}c_{31}d_{22}^*,\\ 
&(b-c)d_{21}c_{22}d_{33}^*+(-b+c+1)c_{23}d_{31}d_{22}^*+(-a+b-1)c_{21}d_{22}^*d_{33}^*,\\ 
&(b-c)c_{12}d_{21}+\ovl{x}c_{11}c_{22}-(a-c)c_{11}d_{22}^*-(b-c)c_{22}d_{11}^*\quad\exists\ \ovl{x}\in \Fp,
\\ 
&
(b-c)d_{21}c_{22}d_{32}+c_{22}d_{31}d_{22}^*+c_{21}c_{32}d_{22}^*+(a-c+1)d_{22}^*(c_{31}d_{22}^*-c_{21}d_{32})
\end{aligned}$\\
&\\
\hline
\hline
&\\
$\alpha\beta\alpha t_{\un{1}}, \beta\alpha t_{\un{1}}$&
$\begin{aligned}
&e_{11},\,c_{33},\ c_{23},\ c_{22},\ c_{21},\ c_{31}c_{32},\ c_{13}c_{32},\ c_{11}c_{32},\ c_{13}c_{31},\ c_{12}c_{31},\\ 
&c_{13}d_{21}d_{32}-c_{12}d_{21}d_{33}^*-c_{13}d_{31}d_{22}^*+c_{11}d_{22}^*d_{33}^*,\\ 
&c_{13}^2d_{31}d_{32}-c_{12}c_{13}d_{31}d_{33}^*-c_{11}c_{13}d_{32}d_{33}^*+c_{11}c_{12}(d_{33}^*)^2,\\
&(a-b)c_{13}d_{31}d_{32}-(a-c)c_{12}d_{31}d_{33}^*+(b-c)c_{11}d_{32}d_{33}^*+(b-c+1)c_{32}d_{11}^*d_{33}^*,
\\ 
&(b-c)c_{12}d_{21}d_{31}-(a-b)c_{13}d_{31}^2-(b-c)c_{11}d_{21}d_{32}+\\
&\quad+(a-b)c_{11}(d_{31}d_{22}^*-d_{21}c_{32})-c_{31}d_{22}^*((a-c)c_{11}-(a-c+1)d_{11}^*),\\ 
&\ovl{z}'(c_{12}c_{13}d_{21}d_{31}-c_{11}c_{12}d_{21}d_{33}^*)+\ovl{z}''(c_{11}^2d_{22}^*d_{33}^*-c_{11}c_{13}d_{31}d_{11}^*)+\\
&\quad+(b-c)c_{12}d_{21}d_{11}^*d_{33}^*-(a-c)c_{11}d_{11}^*d_{22}^*d_{33}^*
\qquad \exists\ \ovl{z}',\ovl{z}''\in\F
\end{aligned}$\\
&\\
\hline
\hline
&some elements of ${I}^{(j)}_{\{w_0,\alpha\beta\},\nabla_{\textnormal{alg}}}\cap{I}^{(j)}_{\{w_0,\beta\alpha\},\nabla_{\textnormal{alg}}}$\\
\hline
&\\
&
$\begin{aligned}
&e_{11},\,c_{33},\ c_{31}c_{32},\ c_{23}c_{32},\ c_{21}c_{32},\ c_{13}c_{32},\ c_{11}c_{32},\ c_{23}d_{32}-c_{22}d_{33}^*,\ c_{23}c_{31},\\
&c_{13}c_{31},\ c_{12}c_{31},\ c_{22}c_{23},\ c_{21}c_{22},\ c_{13}c_{22}-c_{12}c_{23},\ c_{13}c_{21}-c_{11}c_{23},\ c_{12}c_{21}-c_{11}c_{22},\\ 
&c_{21}^2d_{32}-c_{21}c_{31}d_{22}^*,\ c_{13}d_{21}d_{32}-c_{12}d_{21}d_{33}^*-c_{13}d_{31}d_{22}^*+c_{11}d_{22}^*d_{33}^*,\\
&c_{13}^2d_{31}d_{32}-c_{12}c_{13}d_{31}d_{33}^*-c_{11}c_{13}d_{32}d_{33}^*+c_{11}c_{12}(d_{33}^*)^2,\\
&(b-c)d_{21}c_{22}d_{33}^*-(b-c-1)c_{23}d_{31}d_{22}^*-(a-b+1)c_{21}d_{22}^*d_{33}^*,\\ 
&(a-b)c_{13}d_{31}d_{32}-(a-c)c_{12}d_{31}d_{33}^*+(b-c)c_{11}d_{32}d_{33}^*+(b-c+1)c_{32}d_{11}^*d_{33}^*,
\\ 
&(b-c)(c_{12}d_{21}d_{31}-c_{11}d_{21}d_{32}-c_{22}d_{31}d_{11}^*-c_{21}d_{32}d_{11}^*)+(a-b)d_{31}(c_{11}d_{33}^*-c_{13}d_{31})+\\
&\quad+(-a+c-1)c_{31}d_{11}^*d_{33}^*,
\\
& \ovl{z}'c_{12}d_{21}(c_{13}d_{31}-c_{11}d_{33}^*)+ \ovl{z}''c_{11}d_{22}^*(c_{11}d_{33}^*-c_{13}d_{31})+\\
&\quad-(b-c)(c_{12}d_{21}-c_{22}d_{11}^*)d_{11}^*d_{33}^*-\ovl{x}c_{11}c_{22}d_{11}^*d_{33}^*+(a-c)c_{11}d_{11}^*d_{22}^*d_{33}^*
\end{aligned}$
\\
&\\
\hline
\hline
\end{tabular}}
\caption*{\footnotesize{In this table we give an explicit presentation of the ring ${S}/{I}_{T,\nabla_{\textnormal{alg}}}$ in some cases when $\#T^{(j)}=2$ (see Proposition \ref{prop:special_fiber}, and Lemma \ref{lemma:ideal3types}), and of the ideal $\ovl{S^{(j)}}/{I}^{j}_{\{w_0,\alpha\beta\},\nabla_{\textnormal{alg}}}\cap{I}^{j}_{\{w_0,\beta\alpha\},\nabla_{\textnormal{alg}}}$ needed for Lemma \ref{lemma:ideal4types}.
We have set $(a,b,c)\defeq -\big(s_j^{-1}(\mu_j+\eta_j)-(1,1,1)\big)$.
It can be shown that $\ovl{x}\equiv\frac{1}{p}\frac{(b_{\tau_{\alpha\beta},2}-b_{\tau_{\alpha\beta},3})(b_{\tau_{w_0},3}-b_{\tau_{w_0},1})-(b_{\tau_{\alpha\beta},3}-b_{\tau_{\alpha\beta},1})(b_{\tau_{w_0},2}-b_{\tau_{w_0},3})}{(b_{\tau_{w_0},2}-b_{\tau_{w_0},3})}$, $\ovl{z}'=\frac{(b_{\tau_{w_0},2}-b_{\tau_{w_0},3})-(b_{\tau_{\beta\alpha},2}-b_{\tau_{\beta\alpha},3})}{p}+(b_{\tau_{w_0},2}-b_{\tau_{w_0},3})$ and $\ovl{z}''\equiv \frac{(b_{\tau_{w_0},3}-b_{\tau_{w_0},1})-(b_{\tau_{\beta\alpha},3}-b_{\tau_{\beta\alpha},1})}{p}+(b_{\tau_{w_0},3}-b_{\tau_{w_0},1})$ modulo $p$, but we will not need this fact.
}
}
\end{table}

\begin{table}[H]
\caption{\textbf{Special fiber for the algebraic multi-type deformation ring when $\#T^{(j)}=2$.
}
}
\label{Table_Ideals_mod_p_F2}
\centering
\adjustbox{max width=\textwidth}{
\begin{tabular}{| c | c |}
\hline
&\\
$\ovl{A}^{T,(j)}\tld{w}^*(\rhobar)$ &
$\begin{pmatrix}
d_{11}^{*}&c_{12}&c_{13}v+e_{13}\\
d_{21}&d^{*}_{22}v+c_{22}&c_{23}v+e_{23}\\
d_{31}&d_{32}v+c_{32}
&d_{33}^{*}v^2+c_{33}v+e_{33}
\end{pmatrix}\cdot s_j^{-1}v^{\mu_j+(1,0,-1)}$\\
&\\
\hline
&\\
${S}^{(j)}$&$\F[\![x_{11}^*,c_{12},c_{13},e_{13},d_{21},x_{22}^*,c_{22},c_{23},e_{23},d_{31},d_{32},c_{32},x_{33}^*,c_{33},e_{33}]\!]$\\
&\\
\hline
\hline
$T^{(j)}$&some elements of ${I}^{(j)}_{T,\nabla_{\textnormal{alg}}}$\\
\hline
&\\
$t_{w_0(\eta)}, t_{w_0(\eta)}\alpha$&
$\begin{aligned}
&e_{33},\ c_{33},\ c_{32},\ d_{32},\ d_{31},\ d_{21}c_{22},\ e_{13}d_{21}-e_{23}d_{11}^*,\ c_{12}d_{21}-c_{22}d_{11}^*,\\
&(a-b)(c_{13}c_{22}-c_{12}c_{23})-\ovl{x}c_{12}e_{23}+(a-c)e_{13}d_{22}^*,\ \exists \ovl{x}\in \Fp
\end{aligned}$\\
&\\
\hline
\hline
&\\
$t_{w_0(\eta)}, t_{w_0(\eta)}\beta$&
$\begin{aligned}
&e_{33},\ c_{32},\ d_{31},\ e_{23},\ c_{22},\ d_{21},\ d_{32}c_{33},\ c_{23}d_{32}-c_{33}d_{22}^*,\\
&\ovl{z}'c_{33}c_{12}c_{23}+\ovl{z}''c_{33}e_{13}d_{22}^*-(a-b)c_{12}c_{23}d_{33}^*+(a-c)e_{13}d_{22}^*d_{33}^*,\ \exists \ovl{z}', \ovl{z}''\in \Fp
\end{aligned}$\\
&\\
\hline
\hline
&some elements of ${I}^{(j)}_{\{t_{w_0(\eta),t_{w_0(\eta)}\alpha}\},\nabla_{\textnormal{alg}}}\cap {I}^{(j)}_{\{t_{w_0(\eta),t_{w_0(\eta)}\beta}\},\nabla_{\textnormal{alg}}}$\\
\hline
&\\
&
$\begin{aligned}
&e_{33},\ c_{32},\ d_{31},\ d_{32}c_{33},\ c_{23}d_{32}-c_{33}d_{22}^*,\\
&d_{21}d_{32},\ d_{21}c_{22},\ e_{13}d_{21}-e_{23}d_{11}^*,\ c_{12}d_{21}-c_{22}d_{11}^*,\\
&\ovl{z}'c_{33}c_{12}c_{23}+\ovl{z}''c_{33}e_{13}d_{22}^*+(a-b)d_{33}^*(c_{13}c_{22}-c_{12}c_{23})+\\
&\quad-\ovl{x}c_{12}e_{23}d_{33}^*+(a-c)e_{13}d_{22}^*d_{33}^*
\end{aligned}$\\
&\\
\hline
\hline
\end{tabular}}
\caption*{\footnotesize{
In this table we give an explicit presentation of the ring ${S}/{I}_{T,\nabla_{\textnormal{alg}}}$ in some cases when $\#T^{(j)}=2$ (see Proposition \ref{prop:special_fiber}, and Lemma \ref{lemma:ideal3types:prime}) and of the ideal ${I}^{(j)}_{\{t_{w_0(\eta),t_{w_0(\eta)}\alpha}\}\nabla_{\textnormal{alg}}}\cap {I}^{(j)}_{\{t_{w_0(\eta),t_{w_0(\eta)}\beta}\}\nabla_{\textnormal{alg}}}$ needed for Lemma \ref{lemma:ideal4types:prime}.
We have set $(a,b,c)\defeq -\big(s_j^{-1}(\mu_j+\eta_j)-w_0(\eta_j)\big)$.
It can be shown that $\ovl{x}=\frac{1}{p}\frac{(b_{\alpha,1}-b_{\alpha,3})(b_{\Id,1}-b_{\Id,2})-(b_{\Id,1}-b_{\Id,3})(b_{\alpha,1}-b_{\alpha,2})}{(b_{\Id,1}-b_{\Id,2})}$,
$\ovl{z}'=\frac{1}{p}(b_{\beta,1}-b_{\beta,2}-(p+1)(b_{\Id,1}-b_{\Id,2}))$ and
$\ovl{z}''=\frac{1}{p}((p+1)(b_{\Id,1}-b_{\Id,3})-(b_{\beta,1}-b_{\beta,3}))$, where we have set $b_{\Id,i}\defeq b_{\tau_{t_{w_0(\eta)}},i}$, $b_{\alpha,i}\defeq b_{\tau_{t_{w_0(\eta)}\alpha},i}$  and $b_{\beta,i}\defeq b_{\tau_{t_{w_0(\eta)}\beta},i}$ for readability ($i\in\{1,2,3\}$).
}
}
\end{table}

\begin{table}[H]
\caption{\textbf{Minimal prime ideals of ${S}/ {I}_{T,\nabla_{\textnormal{alg}}}$ when $\#T\leq 4$ and $T^{(j)}\subseteq \{\alpha\beta\alpha\gamma t_{\un{1}},\alpha\beta\alpha t_{\un{1}},\beta\alpha t_{\un{1}},\alpha\beta t_{\un{1}}, t_{\un{1}}\}$
}
}
\label{Table:components}
\centering
\adjustbox{max width=\textwidth}{
\begin{tabular}{| c | c |}
\hline
&\\
$(\eta_j,a_j)\in r(\Sigma_0)$&${\fP}_{(\eps_j, a_i)}$\\
&\\
\hline
&\\
$(0,0)$&$e_{11},\,c_{33},\  c_{32},\  c_{31},\  c_{23},\  c_{22},\  c_{21},\  c_{13},\  c_{12},\  c_{11}$\\
&\\
\hline&
\\
$(\eps_1,0)$&$e_{11},\,c_{33},\  c_{32},\  d_{32},\  c_{31},\  d_{31},\  c_{22},\  c_{21},\  c_{12},\  c_{11}$\\
&\\
\hline&
\\
$(\eps_2,0)$&$e_{11},\,c_{33},\  c_{32},\  c_{31},\  d_{31},\  c_{23},\  c_{22},\ c_{21},\  d_{21},\  c_{11}$\\
&\\
\hline&
\\
$(0,1)$&$\begin{aligned}
&e_{11},\,c_{33},\  c_{32},\  c_{31},\  c_{23},\  c_{22},\  c_{21},\  c_{13}d_{32} - c_{12}d_{33}^*,\  c_{13}d_{31} - c_{11}d_{33}^*,\\  
&c_{12}d_{31} - c_{11}d_{32},\  (b-c)d_{21}d_{32} - (a-c)d_{31}d_{22}^*,\  (b-c)c_{12}d_{21} - (a-c)c_{11}d_{22}^*
\end{aligned}$\\
&\\
\hline&
\\
$(\eps_1,1)$&$\begin{aligned}
&e_{11},\,c_{33},\  c_{32},\  c_{31},\   d_{31},\  c_{21},\  c_{11},\  c_{13}c_{22} - c_{12}c_{23},\  c_{13}d_{21} - c_{23}d_{11}^*,\  c_{12}d_{21} - c_{22}d_{11}^*,\\  
&(a - c-1)c_{23}d_{32} - ( a - b - 1)c_{22}d_{33}^*,\  (a -c- 1)c_{13}d_{32} -( a - b - 1)c_{12}d_{33}^* \end{aligned}$
\\
&\\
\hline&
\\
$(\eps_2,1)$&$\begin{aligned}
&e_{11},\,c_{32},\  c_{31},\  c_{22},\  c_{21},\  c_{12},\  c_{11},\  c_{23}d_{32} - c_{33}d_{22}^*,\  d_{21}d_{32} -  d_{31}d_{22}^*,\  c_{23}d_{31} - d_{21}c_{33},\\  
&(a - b)c_{13}d_{31} + (b -c- 1)c_{33}d_{11}^*,\  (a - b)c_{13}d_{21} + (b -c- 1)c_{23}d_{11}^*\end{aligned}$\\
&\\
\hline&
\\
$(\eps_2-\eps_1,1)$&$\begin{aligned}
&e_{11},\,c_{33},\  c_{32},\  d_{32},\  c_{31},\  c_{22},\  c_{13},\  c_{12},\  c_{11},\\  
&(b-c- 1)c_{23}d_{31} + (a - b + 1)c_{21}d_{33}^*\end{aligned}$\\
&\\
\hline&
\\
$(\eps_1-\eps_2,1)$&$\begin{aligned}
&e_{11},\,c_{33},\  c_{31},\  c_{23},\ c_{22},\  c_{21},\  d_{21},\  c_{13},\ c_{11},\\  
&(a-c)c_{12}d_{31} - (b -c+ 1)c_{32}d_{11}^*\end{aligned}$\\
&\\
\hline
\hline
\hline
\hline
\end{tabular}}
\caption*{
We give the presentation of the minimal prime ideals for ${S}/{I}_{T,\nabla_{\textnormal{alg}}}$ when $T^{(j)}=\{\alpha\beta\alpha\gamma t_{\un{1}},\alpha\beta\alpha t_{\un{1}},\beta\alpha t_{\un{1}},\alpha\beta t_{\un{1}}, t_{\un{1}}\}$, using the parametrization of Proposition \ref{prop:special_fiber}.
We have set $(a,b,c)\defeq -\big(s_j^{-1}(\mu_j+\eta_j)-(1,1,1)\big)$.}
\end{table}

\begin{table}[H]
\caption{\textbf{Minimal prime ideals of ${S}/ {I}_{T,\nabla_{\textnormal{alg}}}$ when $\#T\leq 4$ and $T^{(j)}\subseteq \{t_{w_0(\eta)}, t_{w_0(\eta)}\alpha ,  t_{w_0(\eta)}\beta, t_{w_0(\eta)}w_0\}$
}
}
\label{Table:components:F2}
\centering
\adjustbox{max width=\textwidth}{
\begin{tabular}{| c | c |}
\hline
&\\
$(\eta_j,a_j)\in r(\Sigma_0)$&${\fP}_{(\eps_j, a_i)}$\\
&\\
\hline
&\\
$(0,0)$&$e_{33},\ c_{33},\ c_{32},\ e_{23},\ c_{23},\ c_{22},\ e_{13},\ c_{13},\ c_{12}$\\
&\\
\hline&
\\
$(\eps_1,0)$&$e_{33},\ c_{32},\ d_{31},\ d_{32},\ c_{33},\ c_{22},\ c_{12},\ e_{13},\ e_{23}$\\
&\\
\hline&
\\
$(\eps_2,0)$&$e_{33},\ c_{33},\ c_{32},\ d_{31},\ e_{23},\ c_{23},\ c_{22},\ d_{21},\ e_{13}$\\
&\\
\hline&
\\
$(\eps_1,1)$&$\begin{aligned}
&e_{33},\ c_{32},\ e_{23},\ c_{22},\ e_{13},\ c_{12},\ d_{21}c_{33}-c_{23}d_{31},\\  
&d_{21}d_{32}-d_{31}d_{22}^*,\  c_{23}d_{32}-c_{33}d_{22}^*\\
& (a - b + 1)c_{13}d_{31} + (b - c)c_{33}d_{11}^*,\ (a - b + 1)c_{13}d_{21} + (b - c)c_{23}d_{11}^* \end{aligned}$
\\
&\\
\hline&
\\
$(\eps_2,1)$&$\begin{aligned}
& e_{33},\ c_{33},\ c_{32},\ d_{31},\ e_{23},\ e_{13},\ c_{12}c_{23} - c_{13}c_{22}\\  
&c_{12}d_{21} - c_{22}d_{11}^*,\ c_{13}d_{21} - c_{23}d_{11}^*\\
&(a - c +1)c_{23}d_{32} - (a - b)c_{22}d_{33}^*,\ (a - c + 1)c_{13}d_{32} - (a - b)c_{12}d_{33}^*\end{aligned}$\\
&\\
\hline&
\\
$(\eps_1+\eps_2,1)$&$\begin{aligned}
&e_{33},\ c_{32},\ d_{31},\ e_{23},\ d_{32},\ c_{33},\ d_{21},\ c_{22}, \\  
&(a - b)c_{12}c_{23} - (a-c)e_{13}d_{22}^*\end{aligned}$\\
&\\
\hline
\hline
\end{tabular}}
\caption*{We give the presentation of the minimal prime ideals for ${S}/{I}_{T,\nabla_{\textnormal{alg}}}$ when $T^{(j)}=\{t_{w_0(\eta)}, t_{w_0(\eta)}\alpha ,  t_{w_0(\eta)}\beta, t_{w_0(\eta)}w_0\}$, using the parametrization of Proposition \ref{prop:special_fiber}.
We have set $(a,b,c)\defeq -\big(s_j^{-1}(\mu_j+\eta_j)-w_0(\eta_j)\big)$.}
\end{table}

\section{Ideal computations}
\label{appendix:IC}
\setcounter{subsection}{1}

\subsubsection{Ideal intersections in the special fiber of ${S}^{(j)}/{I}^{(j)}_{T,\nabla_{\textnormal{alg}}}$.}
\label{appendix:multi:sp:fiber}

\begin{proof}[Proof of Lemma \ref{lem:intsc}]
We first observe that there exists $\tau\in T$ and $(\omega',a')\in r(\Sigma_0)$ such that both $\fP^{(j)}_{(\omega,a)}{S}+\sum_{j'\in\cJ\setminus\{j\}}\fP^{(j)}_{(\omega',a')}{S}$ and $\fP^{(j)}_{(0,0)}{S}+\sum_{j'\in\cJ\setminus\{j\}}\fP^{(j)}_{(\omega',a')}{S}$ are the pullback, via \eqref{eq:T=1}, of a minimal prime ideal of ${S}/{I}_{\tau,\nabla_\infty}$.
In particular, by the explicit description of ${S}/{I}_{\tau,\nabla_\infty}$ appearing in Tables \ref{Table_Ideals}, \ref{Table_Ideals_1}, the ring ${S}^{(j)}/I^{b_j}_j$ is equidimensional of dimension six, and has $2$ minimal primes.

We prove item \eqref{eq:lem:intsc:1}.
From Table \ref{Table:components} one immediately checks that
\begin{equation}
\label{pf:eq:lem:intsc:1}
(c_{33},c_{32},c_{31},c_{23},c_{22},c_{21},c_{13}d_{32}-c_{12}d_{33}^*,c_{13}d_{31}-c_{11}d_{33}^*,c_{12}d_{31}-c_{11}d_{32},(b-c)c_{12}d_{21}-(a-c)c_{11}d_{22}^*))\subseteq {I}^{b_j}_j
\end{equation}
In particular, we obtain a surjection
\begin{equation}
\label{pf:eq:lem:intsc:1:surj}
{S}^{(j)}/I^{\prime,b_j}_j\onto {S}^{(j)}/I^{b_j}_j
\end{equation}
where we have indicated by $I^{\prime,b_j}_j$ the left hand side of \eqref{pf:eq:lem:intsc:1}.
Moreover
\[
{S}^{(j)}/I^{\prime,b_j}_j\cong \frac{\F[\![c_{13},d_{21},d_{31},d_{32},x_{11}^*,x_{22}^*,x_{33}^*]\!]}{c_{13}((a-c)d_{31}d_{22}^*-(b-c)d_{32}d_{21})}
\]
which is evidently reduced, equidimensional of dimension six, and has two minimal prime ideals.
We conclude by \cite[Lemma 3.6.11]{LLLM2} that the surjection \eqref{pf:eq:lem:intsc:1:surj} is an isomorphism, hence that the inclusion \eqref{pf:eq:lem:intsc:1} is an equality.

The proofs of items \eqref{eq:lem:intsc:2}--\eqref{eq:lem:intsc:5} are analogous.
\end{proof}

\begin{proof}[Proof of Lemma \ref{lem:broom:other:prime}]
The proof is analogous to that of Lemma \ref{lem:intsc}.
From Table \ref{Table:components:F2} we have an evident inclusion of ideals of ${S}^{(j)}$:
\begin{equation}
\label{pf:eq:lem:Lambda2:1:surj}
(c_{22},c_{33},c_{32},e_{33},e_{23}, d_{31},(a-b)c_{12}c_{23}-(a-c)e_{13}d_{22}^*,d_{21}d_{32},c_{23}d_{32},d_{21}c_{12})\subseteq 
{I}^{(j)}_{\Lambda}
\end{equation}
hence a surjection 
\begin{equation}
\label{pf:eq:lem:Lambda2:2:surj}
{S}^{(j)}/{I}^{\prime(j)}_{\Lambda}\onto {S}^{(j)}/{I}^{(j)}_{\Lambda}
\end{equation}
(where we have indicated by ${I}^{\prime(j) }_{\Lambda}$ the left hand side of \eqref{pf:eq:lem:Lambda2:1:surj}).
An direct computation shows that
\[
{S}^{(j)}/{I}^{\prime (j)}_{\Lambda}\cong \frac{\F[\![c_{12},d_{21},d_{32},c_{13},c_{23},x_{11}^*,x_{22}^*,x_{33}^*]\!]}{(d_{21}d_{32},c_{23}d_{32},d_{21}c_{12})}
\]
and the latter ring is evidently reduced, equidimensional of dimension six and has three minimal prime ideals.
\end{proof}

\begin{proof}[Proof of Proposition \ref{prop:ideals_coprimes}]

In the following computations, we work in $\tld{S}^{(j)}/(\tld{I}^{(j)}_{\tau}+\tld{I}^{(j)}_{\tau'})$.

\emph{Case $\tld{w}^*(\rhobar,\tau)_j=\alpha\beta\alpha t_{\un{1}}$ and $\tld{w}^*(\rhobar,\tau')_j=t_{\un{1}}$.}
Using the relations $c_{22}\equiv0$, $c_{33}\equiv-pd_{33}^*$ coming from $\tld{I}^{(j)}_{\alpha\beta\alpha t_{\un{1}}}$, the last listed equation in $\tld{I}^{(j)}_{t_{\un{1}}}$ becomes:
\begin{equation}
\label{eq:case:ABAetId}
-pc_{12}d_{21}d_{33}^*  +pc_{11}d_{22}^*d_{33}^* - p(c_{11}d_{22}^*d_{33}^*-pd_{11}^*d_{22}^*d_{33}^*)
\end{equation}
On the other hand  the relations $c_{21}\equiv0$ and  $c_{21}\equiv-pd_{21}$ coming from $\tld{I}^{(j)}_{\alpha\beta\alpha t_{\un{1}}}$ and $\tld{I}^{(j)}_{ t_{\un{1}}}$ respectively give $-pc_{12}d_{21}d_{33}^*\equiv 0$, hence \eqref{eq:case:ABAetId} becomes $pc_{11}d_{22}^*d_{33}^* - p(c_{11}d_{22}^*d_{33}^*-pd_{11}^*d_{22}^*d_{33}^*)$ yelding $p^2d_{11}^*d_{22}^*d_{33}^*\equiv 0$.

\emph{Case $\tld{w}^*(\rhobar,\tau)_j=\beta\alpha t_{\un{1}}$ and $\tld{w}^*(\rhobar,\tau')_j=t_{\un{1}}$.}
Using the relations $c_{22}\equiv0$ coming from $\tld{I}^{(j)}_{\beta\alpha t_{\un{1}}}$, the last listed equation in $\tld{I}^{(j)}_{t_{\un{1}}}$ becomes:
\begin{equation}
\label{eq:case:BAetId:1}
c_{12}d_{21}c_{33}  -c_{11}d_{22}^*c_{33} - p(c_{11}d_{22}^*d_{33}^*+d_{11}^*d_{22}^*c_{33})
\end{equation}
and, using $d_{21}c_{33}\equiv c_{23}d_{31}$, $c_{11}c_{33}\equiv -pc_{13}d_{31}$ coming from  $\tld{I}^{(j)}_{ t_{\un{1}}}$, equation \eqref{eq:case:BAetId:1} becomes 
\[
c_{12}c_{23}d_{31}  +pd_{22}^*c_{13}d_{31} - p(c_{11}d_{22}^*d_{33}^*+c_{11}^*d_{22}^*c_{33})
\] which, using $c_{23}\equiv 0$, $c_{11}d_{33}^*-c_{13}d_{31}$ and $c_{33}=-pd_{33}^*$ coming from $\tld{I}^{(j)}_{\beta\alpha t_{\un{1}}}$, yields $p^2d_{11}^*d_{22}^*d_{33}^*\equiv 0$.

\emph{Case $\tld{w}^*(\rhobar,\tau)_j=\beta t_{\un{1}}$ and $\tld{w}^*(\rhobar,\tau')_j=\alpha t_{\un{1}}$.}
Using the relation $c_{11}d_{33}^*\equiv c_{13}d_{31}$ coming from $\tld{I}^{(j)}_{\beta t_{\un{1}}}$, the last listed equation in $\tld{I}^{(j)}_{\alpha t_{\un{1}}}$ becomes:
\begin{equation}
\label{eq:case:BetA:1}
c_{13}d_{21}d_{32}  -c_{12}d_{21}d_{33}^*-pd_{11}^*d_{22}^*d_{33}^*.
\end{equation}
Multiplying \eqref{eq:case:BetA:1} by $pd_{33}^*$, and using $pd_{21}d_{33}^*\equiv -c_{23}d_{31}$, $pc_{12}d_{33}^*\equiv -c_{13}c_{32}$ coming from  $\tld{I}^{(j)}_{ \beta t_{\un{1}}}$, equation \eqref{eq:case:BAetId:1} becomes 
\[
-c_{23}d_{31}c_{13}d_{32}  +c_{13}c_{32}d_{21}d_{33}^*-p^2d_{11}^*d_{22}^*(d_{33}^*)^2.
\] which, using $c_{32}\equiv -pd_{32}$ coming from $\tld{I}^{(j)}_{\alpha t_{\un{1}}}$, yields 
\[
-c_{23}d_{31}c_{13}d_{32}  -pd_{32}c_{13}d_{21}d_{33}^*-p^2d_{11}^*d_{22}^*(d_{33}^*)^2
\]
hence $-p^2d_{11}^*d_{22}^*(d_{33}^*)^2$ noting again that $pd_{21}d_{33}^*\equiv -c_{23}d_{31}$.

\emph{Case $\tld{w}(\rhobar,\tau)_j=\beta t_{w_0(\eta)}$ and $\tld{w}(\rhobar,\tau')_j=w_0t_{w_0(\eta)}$.}
Multiplying by $-p$ the relation $-p d_{22}^*d_{33}^*\equiv c_{23}d_{32}$ coming from $\tld{I}^{(j)}_{\beta t_{w_0(\eta)}}$, and using $-pc_{32}\equiv e_{23}$ coming from $\tld{I}^{(j)}_{w_0t_{w_0(\eta)}}$ we obtain $p^2 d_{22}^*d_{33}^*\equiv e_{23}d_{32}$, and the latter expression is zero in the quotient ring since $e_{23}\in \tld{I}^{(j)}_{\beta t_{w_0(\eta)}}$.

\emph{Case $\tld{w}(\rhobar,\tau)_j=\alpha t_{w_0(\eta)}$ and $\tld{w}(\rhobar,\tau')_j=w_0t_{w_0(\eta)}$.}
Noting that $c_{33},e_{33}\equiv 0$ (relation coming from $\tld{I}^{(j)}_{\alpha t_{w_0(\eta)}}$) and $-e_{33}-pc_{33}\equiv p^2d_{33}^*$ (relation coming from $\tld{I}^{(j)}_{w_0t_{w_0(\eta)}}$) we obtain $p^2d_{33}^*\equiv 0$.

\emph{Case $\tld{w}(\rhobar,\tau)_j= t_{w_0(\eta)}$ and $\tld{w}(\rhobar,\tau')_j=w_0t_{w_0(\eta)}$.}
It is exactly as above noticing that  $c_{33},e_{33}\equiv 0$ modulo $\tld{I}^{(j)}_{t_{w_0(\eta)}}$).

The cases where both $\tld{w}^*(\rhobar,\tau)_j$, $\tld{w}^*(\rhobar,\tau')_j$ have length at least 2 are much easier, and give the stronger result $p\in \tld{I}^{(j)}_{\tau}+\tld{I}^{(j)}_{\tau'}$. 
(For instance, if $\tld{w}(\rhobar,\tau)_j= t_{w_0(\eta)}$ and $\tld{w}(\rhobar,\tau')_j=\alpha t_{w_0(\eta)}$ then $c_{22}\equiv 0$ and $c_{22}\equiv -pd_{22}^*$, relations coming from $\tld{I}^{(j)}_{t_{w_0(\eta)}}$ and $\tld{I}^{(j)}_{\alpha t_{w_0(\eta)}}$ respectively.)
\end{proof}

\subsubsection{Ideal intersections for ${S}^{(j)}/{I}^{(j)}_{\tau,\nabla_{\textnormal{alg}}}$, $\tld{w}(\rhobar,\tau)_j=t_{\un{1}}$.}
\label{subsub:surgery}
In this section we work in the ring ${S}^{(j)}/{I}^{(j)}_{\tau,\nabla_{\textnormal{alg}}}$. 
By abuse of notation we will consider ${S}^{(j)}$ to be the ring $\F[\![c_{11},x^*_{11},c_{12},c_{13},d_{21},c_{22},x_{22}^*,c_{23},d_{31},d_{32},c_{33},x^*_{33}]\!]$, and ${I}^{(j)}_{\tau,\nabla_{\textnormal{alg}}}$, $\fP_{(\omega,a)}$ (for $\omega\in\{0,\eps_1,\eps_2\}$, $a\in\{0,1\}$) ideals of $\F[\![c_{11},x^*_{11},c_{12},c_{13},d_{21},c_{22},x_{22}^*,c_{23},d_{31},d_{32},c_{33},x^*_{33}]\!]$.
(In other words, we abuse notation and ``neglect the variables $c_{21}, c_{31}, c_{32}$'').

We  now remark that the assignement $c_{i,j}\mapsto c_{(132)(i),(132)(j)}$, $a\mapsto c+1$, $b\mapsto a$, $c\mapsto b$ induces an automorphism of $\F$-algebras on ${S}^{(j)}/{I}^{(j)}_{\tau,\nabla_{\textnormal{alg}}}$, which moreover sends $\fP_{(0,0)}$ to $\fP_{(\eps_1,0)}$ (resp.~$\fP_{(0,1)}$ to $\fP_{(\eps_2,1)}$), $\fP_{(\eps_1,0)}$ to $\fP_{(\eps_2,0)}$ (resp.~$\fP_{(\eps_2,1)}$ to $\fP_{(\eps_1,1)}$) and $\fP_{(\eps_2,0)}$ to $\fP_{(0,0)}$ (resp.~$\fP_{(\eps_1,1)}$ to $\fP_{(0,1)}$).

\begin{proof}[Proof of Lemma \ref{lemma:id:intersect}]
By the remark at the beginning of \S \ref{subsub:surgery} it is enough to prove the statements for the ideals $I_\gamma$ with $\ell(\gamma)=2$, $\gamma_1=(\eps_2,1)$, and for $I_\beta$, $\ell(\beta)=3$, $\beta_3=(0,1)$.

\emph{The ideal $I_\beta$, $\ell(\beta)=3$, $\beta_3=(0,1)$.}
A direct inspection of Table \ref{Table:components} gives an inclusion
\[
c_{11}\in \fP^{(j)}_{(0,0)}\cap\fP^{(j)}_{(\eps_2,0)}\cap\fP^{(j)}_{(\eps_2,1)}\cap\fP^{(j)}_{(\eps_1,0)}\cap\fP^{(j)}_{(\eps_1,1)}.
\]
Thus, we have a surjection
\begin{equation}
\label{eq:ring:id:intersect:40}
{S}^{(j)}/(c_{11},{I}^{(j)}_{\tau,\nabla_{\textnormal{alg}}})\onto {S}^{(j)}/(\fP^{(j)}_{(0,0)}\cap\fP^{(j)}_{(\eps_2,0)}\cap\fP^{(j)}_{(\eps_2,1)}\cap\fP^{(j)}_{(\eps_1,0)}\cap\fP^{(j)}_{(\eps_1,1)})
\end{equation}
and a direct computation using Table \ref{Table_Ideals_1} gives
\[
{S}^{(j)}/(c_{11},{I}^{(j)}_{\tau,\nabla_{\textnormal{alg}}})\cong \frac{\F[\![c_{12},c_{13},c_{23},d_{21},d_{31},d_{32},x_{11}^*,x_{22}^*,x_{33}^*]\!]}{J}
\]
where $J$ is the ideal generated by 
\begin{align*}
&c_{12}d_{31},\, c_{12}(c_{23}d_{11}^*-d_{21}c_{13}),\,d_{31}(c_{23}d_{11}^*-(a-b)d_{21}c_{13}),\\
&d_{32}(c_{23}d_{22}^*-d_{21}c_{13})-(a-b-1)c_{13}d_{31},
\\
&(a-c-1)c_{23}d_{32}-(a-b)(a-c-1)c_{13}d_{31}-(a-b-1)c_{12}d_{21}
\end{align*}
The latter ring is reduced (since the initial ideal of $J$ is generated by squarefree monomial, as it can be checked by considering a suitable Groebner basis) and has $5$ minimal primes each of dimension $6$.
Thus, by \cite[Lemma 3.6.11]{LLLM2}, the surjection \eqref{eq:ring:id:intersect:40} is an isomorphism.

\emph{Case  $\gamma=((\eps_2,1),(\eps_1,0))$.}
A direct inspection of Table \ref{Table:components} gives an inclusion
\begin{equation}
\label{eq:Igamma:1}
\underbrace{(c_{22},c_{11},(a-b)c_{13}d_{21}+(b-c-1)c_{23}d_{11}^*)}_{\defeq J}\subseteq \fP^{(j)}_{(0,0)}\cap\fP^{(j)}_{(\eps_2,0)}\cap\fP^{(j)}_{(\eps_2,1)}.
\end{equation}
Thus, we have a surjection
\begin{equation}
\label{eq:ring:id:intersect:4}
{S}^{(j)}/(J+{I}^{(j)}_{\tau,\nabla_{\textnormal{alg}}})\onto {S}^{(j)}/(\fP^{(j)}_{(0,0)}\cap\fP^{(j)}_{(\eps_2,0)}\cap\fP^{(j)}_{(\eps_2,1)})
\end{equation}
and a direct computation gives
\[
{S}^{(j)}/(J+{I}^{(j)}_{\tau,\nabla_{\textnormal{alg}}})\cong \frac{\F[\![c_{12},c_{13},d_{21},d_{31},d_{32},x_{11}^*,x_{22}^*,x_{33}^*]\!]}{c_{13}(d_{31}d_{22}^*-d_{32}d_{21}),c_{12}d_{21},c_{12}d_{31}}.
\]
The latter ring is reduced, equidimensional of dimension $6$ and has $3$ minimal primes, and hence, by \cite[Lemma 3.6.11]{LLLM2}, the surjection \eqref{eq:ring:id:intersect:4} is an isomorphism.
We conclude that \eqref{eq:Igamma:1} is an equality.

We now prove the assertion on the minimal number of generators.
First of all we note that $c_{11}\in (c_{22}, (b-c-1)c_{23}d_{11}^*+(a+b)(c_{13}d_{21}))$.
Indeed, using the three equations in row $\textnormal{Mon}_\tau$, and the equation in line $4$ of row $I^{(j)}_{\tau,\nabla_{\textnormal{alg}}}$ (the ``determinant'' equation) in Table \ref{Table_Ideals_1}, we have 
\begin{align*}
f\defeq &(a-b)(b-c)(a-c-1)c_{13}d_{21}d_{32}-(a-c)(a-b)(a-c-1)c_{11}d_{22}^*d_{33}^*+\\
&\qquad+(b-c)(a-c)(b-c-1)c_{33}d_{22}^*d_{11}^*+(a-b-1)(a-b)(b-c)c_{22}d_{11}^*d_{33}^*\in I^{(j)}_{\tau,\nabla_{\textnormal{alg}}}
\end{align*}
and, on the other hand,
\[
xc_{11}d_{22}^*d_{33}^*+yc_{22}d_{11}^*d_{33}^*+zd_{32}\left(c_{23}d_{11}^*+\frac{a+b}{b-c-1}c_{13}d_{21}\right)+\kappa d_{11}^*\textnormal{Mon}_{\tau,1}=f
\]
where $z\defeq (b-c-1)(b-c)(a-c-1)$, $\kappa\defeq -(b-c-1)(b-c)$, $y=\kappa(a-b-1)$ and $x=-(a-c)(a-b)(a-c-1)$.

Hence $I_\gamma=(c_{22}, (b-c-1)c_{23}d_{11}^*+(a+b)(c_{13}d_{21}))$ and we now prove that $\ovl{I}_\gamma$ has dimension $2$.
We first note that $\ovl{c}_{22}\neq0$, as $c_{22}\notin \fm_{S^{(j)}}\cdot (J+I^{(j)}_{\tau,\nabla_{\textnormal{alg}}})$ (alternatively, one can check on the explicit equations of Table \ref{Table_Ideals_1} that $c_{22}\neq 0$ in $S^{(j)}/(\fm_{S^{(j)}}^2+I^{(j)}_{\tau,\nabla_{\textnormal{alg}}})$).
Now, if we have a relation of the form $\ovl{c}_{2}+\kappa \ovl{(a-b)c_{13}d_{21}+(b-c-1)c_{23}d_{11}^*)}$ in $\ovl{I}_\gamma$, for some $\kappa\in \F^\times$, this would imply that the natural inclusion $(c_{22})\subseteq I_\gamma$ induces an isomorphism of $1$-dimensional $\F$-vector spaces, and hence that $(c_{22})\subseteq I_\gamma$ is in fact an equality.
This is impossible since $(c_{22})=I_\beta$, $I_\beta\neq I_\gamma$ (e.g.~by looking at the number of minimal primes of $S^{(j)}/I^{(j)}_{\tau,\nabla_{\textnormal{alg}}}$ above them which has been computed along the proof).
\end{proof}

\begin{lemma}
\label{lemma:id:relations}
We have the following equalities in ${S}^{(j)}/{I}^{(j)}_{\tau,\nabla_{\textnormal{alg}}}$
\begin{enumerate}
\item
\label{it:id:relations:1}
$-\big(\frac{a-b}{b-c}\big)c_{11}d^*_{22}d^*_{33}-\big(\frac{1-a+b}{1-a+c}\big)d^*_{11}c_{22}d^*_{33}-c_{13}(d^*_{22}d_{31}-d_{21}d_{32})\equiv0$;
\item
\label{it:id:relations:3}
$(b-c)d^*_{11}c_{22}+(a-c)c_{11}d^*_{22}-(b-c)c_{12}d_{21}\equiv0$;
\item
\label{it:id:relations:2}
$-\big(\frac{a-c}{b-c}\big)c_{11}d^*_{22}d^*_{33}+\big(\frac{a-b-1}{a-b}\big)d^*_{11}c_{22}d^*_{33}+\big(\frac{b-c-1}{a-b}\big)d^*_{11}c_{23}d_{32}+d_{21}d_{32}c_{13}\equiv0$;
\item
\label{it:id:relations:4}
$-\big(\frac{c+1-a}{a-b}\big)c_{33}d^*_{11}d^*_{22}-\big(\frac{-c+a}{-c+b}\big)d^*_{33}c_{11}d^*_{22}-d_{32}(d^*_{11}c_{23}-c_{13}d_{21})\equiv0$;
\item
\label{it:id:relations:5}
$(a-b)d^*_{33}c_{11}+(c+1-b)c_{33}d^*_{11}-(a-b)d_{31}c_{13}\equiv0$;
\item
\label{it:id:relations:6}
$-\big(\frac{c+1-b}{a-b}\big)c_{33}d^*_{11}d^*_{22}+\big(\frac{c-a}{c+1-a}\big)d^*_{33}c_{11}d^*_{22}+\big(\frac{a-b-1}{c+1-a}\big)d^*_{22}c_{12}d_{21}+d_{13}d_{21}c_{32}\equiv0$.
\end{enumerate}
\end{lemma}
\begin{proof}
By the remark at the beginning of \S \ref{subsub:surgery} it is enough to prove the statements for items \eqref{it:id:relations:1}, \eqref{it:id:relations:3} and \eqref{it:id:relations:2}.

Let $\ovl{\textnormal{Mon}}_{\tau,1}$, $\ovl{\textnormal{Mon}}_{\tau,2}$, $\ovl{\textnormal{Mon}}_{\tau,3}$ denote the mod $p$-reduction of the first, second and third equations in row $\textnormal{Mon}_{\tau}$ of Table \ref{Table_Ideals_1}. In particular,  item \eqref{it:id:relations:3} is $\ovl{\textnormal{Mon}}_{\tau,3}$.

Item \eqref{it:id:relations:1} is deduced from $\frac{d_{11}^*}{a-c-1}\ovl{\textnormal{Mon}}_{\tau,1}\equiv 0$, using the relations
\begin{align*}
d_{11}^*(c_{23}d_{32}-c_{33}d_{22}^*)&\equiv
c_{13}d_{21}d_{32} - c_{12}d_{21}d_{33}^* - c_{13}d_{31}d_{22}^* + c_{11}d_{22}^* d_{33}^*+ d_{11}^*c_{22}d_{33}^* \\
c_{12}d_{21}&\equiv d_{11}^*c_{22}+\frac{(a-b)}{(b-c)}c_{11}d_{22}^*
\end{align*}
(the first relation comes from the sixth equation in row $I^{(j)}_\tau$ in Table \ref{Table_Ideals_1}, and the second relation from $\ovl{\textnormal{Mon}}_{\tau,3}$).

Item \eqref{it:id:relations:2} is deduced from $\frac{d_{22}^*}{a-b}\ovl{\textnormal{Mon}}_{\tau,2}\equiv 0$, using the relations
\begin{align*}
d_{22}^*(c_{23}d_{32}-c_{22}d_{33}^*)&\equiv
c_{13}d_{21}d_{32} - c_{12}d_{21}d_{33}^* - c_{23}d_{32}d_{11}^* + c_{33}d_{22}^* d_{11}^*+ d_{11}^*c_{22}d_{33}^* 
 \\
c_{33}d_{22}^*&\equiv c_{23}d_{32}-\frac{(a-b-1)}{(a-c-1)}c_{22}d_{33}^*
\\
c_{12}d_{21}&\equiv d_{11}^*c_{22}+\frac{(a-b)}{(b-c)}c_{11}d_{22}^*
\end{align*}
(the first relation comes from the sixth equation in row $I^{(j)}_\tau$ in Table \ref{Table_Ideals_1}, and the second and third relation from $\ovl{\textnormal{Mon}}_{\tau,2}$ and $\ovl{\textnormal{Mon}}_{\tau,3}$).
\end{proof}

\subsubsection{Justification for Table \ref{Table_Ideals_mod_p}}
\label{subsub:Just:T1}
The justification is a direct computation, performed by exhibiting elements in $\tld{I}^{(j)}_{\{w_0,\alpha\beta\},\nabla_{\textnormal{alg}}}\defeq \tld{I}^{(j)}_{\tau_{w_0},\nabla_{\textnormal{alg}}}\cap \tld{I}^{(j)}_{\tau_{\alpha\beta},\nabla_{\textnormal{alg}}}$ (resp.~in $\tld{I}^{(j)}_{\{w_0,\beta\alpha\},\nabla_{\textnormal{alg}}}=\tld{I}^{(j)}_{\tau_{w_0},\nabla_{\textnormal{alg}}}\cap \tld{I}^{(j)}_{\tau_{\beta\alpha},\nabla_{\textnormal{alg}}}$), and taking their mod-$\varpi$ reduction.

We mention that these computation can ultimately be checked by exhibiting a Groebner basis for the ideals $\tld{I}^{(j)}_{\tau_{w_0},\nabla_{\textnormal{alg}}}$, $\tld{I}^{(j)}_{\tau_{\alpha\beta},\nabla_{\textnormal{alg}}}$ and $\tld{I}^{(j)}_{\tau_{\beta\alpha},\nabla_{\textnormal{alg}}}$ (for the monomial ordering on $\tld{S}^{(j)}$ given by $c_{11}>c_{12}>c_ {13}>d_{21}>c_{21}>c_{22}>c_{23}>c_{31}>d_{31}>d_{32}>c_{32}>c_{33}$), and give full detail for the most complicated equations (namely, those involving the structure constants from the monodromy).

\begin{proof}[Study of $\tld{I}^{(j)}_{\{w_0,\alpha\beta\},\nabla_{\textnormal{alg}}}\defeq \tld{I}^{(j)}_{\tau_{w_0},\nabla_{\textnormal{alg}}}\cap \tld{I}^{(j)}_{\tau_{\alpha\beta},\nabla_{\textnormal{alg}}}$]
We claim that the element $f\defeq (b_{\tau_{\alpha\beta},2}-b_{\tau_{\alpha\beta},3})d_{21}c_{22}d^*_{33}-(b_{\tau_{\alpha\beta},2}-b_{\tau_{\alpha\beta},3}-1)c_{23}d_{31}d^*_{22}+(b_{\tau_{\alpha\beta},,2}-b_{\tau_{\alpha\beta},1}-1)c_{21}d^*_{33}d^*_{22}$ (whose mod $\varpi$-reduction gives the element in the third line of row $\alpha\beta\alpha t_{\un{1}},\alpha\beta t_{\un{1}}$ in Table \ref{Table_Ideals_mod_p}) is in $\tld{I}^{(j)}_{\{w_0,\alpha\beta\},\nabla_{\textnormal{alg}}}$.
Indeed, on the one hand $c_{21}, c_{22}, c_{23}$ all belong to $\tld{I}^{(j)}_{\tau_{w_0},\nabla_{\textnormal{alg}}}$, and on the other hand
\[
f=(c_{22}+pd_{22}^*)(b_{\tau_{\alpha\beta},2}-b_{\tau_{\alpha\beta},3})d_{33}^*d_{21}-d_{22}^*(\textnormal{Mon}_{\tau_{\alpha\beta},2})
\]
where $\textnormal{Mon}_{\tau_{\alpha\beta},2}$ denotes the second equation in row $(\alpha\beta t_{\un{1}},  \tld{I}^{(j)}_{\tau,\nabla_{\infty}},\textnormal{Mon}_{\tau})$ in Table \ref{Table_Ideals}.

In a similar fashion, we have the equality
\begin{align*}
&((b_{\tau_{w_0},2}-b_{\tau_{w_0},3})d_{21}d_{32}+d_{31})c_{22}+((b_{\tau_{w_0},3}-b_{\tau_{w_0},1}-1+p)d_{32}+c_{32})c_{21}+\textnormal{Mon}_{\tau_{w_0},2}=\\
&\qquad=
((b_{\tau_{w_0},2}-b_{\tau_{w_0},3})d_{21}d_{32}+d_{31})(c_{22}+pd_{22}^*)+(b_{\tau_{w_0},3}-b_{\tau_{w_0},1}-1)(c_{21}d_{32}-c_{31}d_{22}^*)+c_{21}(c_{32}+pd_{32})
\end{align*}
hence obtaining an element in $\tld{I}^{(j)}_{\{w_0,\alpha\beta\},\nabla_{\textnormal{alg}}}$ which reduces modulo $p$ to the last equation in row $(\alpha\beta\alpha t_{\un{1}},\alpha\beta t_{\un{1}})$ of Table \ref{Table_Ideals_mod_p}.

Finally, we check that 
\begin{equation}
\label{eq:in:inter}
c_{22}(xc_{11}+yd_{11}^*)+z \textnormal{Mon}_{\tau_{w_0},1}=(c_{22}+pd_{22}^*)(xc_{11}+yd_{11}^*)+\textnormal{Mon}_{\tau_{\alpha\beta},1}
\end{equation}
where 
\[
\begin{cases}
z\defeq\frac{(b_{\tau_{\alpha\beta},2}-b_{\tau_{\alpha\beta},3})}{(b_{\tau_{w_0},2}-b_{\tau_{w_0},3})},&\\
y\defeq 1-(b_{\tau_{\alpha\beta},2}-b_{\tau_{\alpha\beta},3})-\frac{(b_{\tau_{\alpha\beta},2}-b_{\tau_{\alpha\beta},3})}{(b_{\tau_{w_0},2}-b_{\tau_{w_0},3})},&\\
x\defeq\frac{1}{p}\frac{(b_{\tau_{\alpha\beta},2}-b_{\tau_{\alpha\beta},3})(b_{\tau_{w_0},3}-b_{\tau_{w_0},1})-(b_{\tau_{\alpha\beta},3}-b_{\tau_{\alpha\beta},1})(b_{\tau_{w_0},2}-b_{\tau_{w_0},3})}{(b_{\tau_{w_0},2}-b_{\tau_{w_0},3})}&\\
\end{cases}
\]
(note that $x,y,z\in\ \Zp$ by the genericity assumption on $\mu_j+\eta_j$ and fact that $b_{\tau_{\alpha\beta},i}\equiv b_{\tau_{w_0},i}$ for $i=1,2,3$) and where $\textnormal{Mon}_{\tau_{w_0},1}$ (resp.~$\textnormal{Mon}_{\tau_{\alpha\beta},1}$) denotes the first equation in row $(\alpha\beta \alpha t_{\un{1}},  \tld{I}^{(j)}_{\tau,\nabla_{\infty}},\textnormal{Mon}_{\tau})$ (resp.~$(\alpha\beta t_{\un{1}},  \tld{I}^{(j)}_{\tau,\nabla_{\infty}},\textnormal{Mon}_{\tau})$) in Table \ref{Table_Ideals}.
Observing that $z\equiv 1$ and $y\equiv -(b-c)$  modulo $p$, equation \eqref{eq:in:inter} justifies the fourth line in row $(\alpha\beta\alpha t_{\un{1}},\alpha\beta t_{\un{1}})$ of Table \ref{Table_Ideals_mod_p}.

The computations for the elements in the first two lines in row $\alpha\beta\alpha t_{\un{1}},\alpha\beta t_{\un{1}}$ are easier (only involving finite height equations).
For instance the element $c_{23}d_{32}-c_{22}d^*_{33}$ is evidently in ${I}^{(j)}_{\{w_0,\alpha\beta\},\nabla_{\textnormal{alg}}}$ since on the one hand $c_{22},c_{23}\in \tld{I}^{(j)}_{\tau_{w_0},\nabla_{\textnormal{alg}}}$ and on the other hand $c_{22}+pd_{22}^*, c_{23}d_{32}+pd_{33}^*d_{22}^*\in \tld{I}^{(j)}_{\tau_{w_0},\nabla_{\textnormal{alg}}}$.
\end{proof}

\begin{proof}[Study of $\tld{I}^{(j)}_{\tau_{w_0},\nabla_{\textnormal{alg}}}\cap \tld{I}^{(j)}_{\tau_{\beta\alpha},\nabla_{\textnormal{alg}}}$]
We explicitly construct elements in $\tld{I}^{(j)}_{\tau_{w_0},\nabla_{\textnormal{alg}}}\cap \tld{I}^{(j)}_{\tau_{\beta\alpha},\nabla_{\textnormal{alg}}}$ and compute their mod $p$-reductions.

The elements 
\begin{align}
\label{eq:4th:ABA:BA:1}
&(b_{\tau_{\beta\alpha},2}-b_{\tau_{\beta\alpha},3}+1)(c_{32+pd_{32}})+(b_{\tau_{\beta\alpha},1}-b_{\tau_{\beta\alpha},3})d_{31}(c_{13}d_{32}-c_{12}d_{33}^*)+\\
&\qquad+(b_{\tau_{\beta\alpha},3}-b_{\tau_{\beta\alpha},2})d_{32}(c_{13}d_{31}-c_{11}d_{33}^*+pd_{11}d_{33}^*)
\nonumber
\end{align}
and 
\begin{equation}
\label{eq:4th:ABA:BA:2}
(b_{\tau_{\beta\alpha},1}-b_{\tau_{\beta\alpha},2})(c_{13}d_{31}-c_{11}d_{33}^*)-\textnormal{Mon}_{\tau_{\beta\alpha}}.
\end{equation}
are equal, and, by direct inspection of the finite height equations in Table \ref{Table_Ideals}, equation \eqref{eq:4th:ABA:BA:1} defines an element in $\tld{I}^{(j)}_{\tau_{w_0},\nabla_{\textnormal{alg}}}$ and equation \eqref{eq:4th:ABA:BA:2} defines an element in $\tld{I}^{(j)}_{\tau_{\beta\alpha},\nabla_{\textnormal{alg}}}$.
This equation reduces mod $p$ to the fourth line of row $\alpha\beta\alpha t_{\un{1}}, \beta\alpha t_{\un{1}}$ of Table \ref{Table_Ideals_mod_p}.

Let $x'\defeq \frac{b_{\tau_{w_0},1}-b_{\tau_{w_0},2}+p(b_{\tau_{w_0},1}-b_{\tau_{w_0},3})}{1+p}$.
A direct computation shows that the expressions
\begin{align}
\label{eq:4th:ABA:BA:1'}
&d_{11}^*d_{33}^*(\textnormal{Mon}_{\alpha\beta\alpha,2})+x'c_{11}d_{21}d_{33}^*(c_{32}+pd_{32})+(b_{\tau_{w_0},2}-b_{\tau_{w_0},3})d_{21}d_{31}(c_{13}d_{32}-c_{12}d_{33}^*)+
\\
&\qquad+\Big(x'd_{31}d_{22}^*+(b_{\tau_{w_0},1}-b_{\tau_{w_0},3})c_{31}d_{22}^*-(b_{\tau_{w_0},2}-b_{\tau_{w_0},3})d_{21}d_{32}\Big)\Big(c_{13}d_{31}-c_{11}d_{33}^*+pd_{11}^*d_{33}^*\Big)+\nonumber\\
&\qquad\qquad-(b_{\tau_{w_0},1}-b_{\tau_{w_0},3})d_{31}d_{22}^*(c_{13}c_{31}+pc_{11}d_{33})
\nonumber
\end{align}
and 
\begin{align}
\label{eq:4th:ABA:BA:2'}
&\Big((b_{\tau_{w_0},1}-b_{\tau_{w_0},3})(p-c_{11})+(b_{\tau_{w_0},1}-b_{\tau_{w_0},3}+1)\Big)(c_{31}+pd_{31})d_{11}^*d_{22}^*d_{33}^*+
\\
&\qquad+\Big(x'(d_{31}d_{22}^*-d_{21}c_{32})-(b_{\tau_{w_0},2}-b_{\tau_{w_0},3}+px')d_{21}d_{32}\Big)(c_{13}c_{31}-c_{11}d_{33}^*)
+\nonumber
\\
&\qquad\qquad+x'd_{21}d_{31}(c_{13}c_{32}+pc_{12}d_{33}^*)+(b_{\tau_{w_0},2}-b_{\tau_{w_0},3}+px')d_{31}(c_{13}d_{21}d_{32}-c_{12}d_{21}d_{33}^*-pd_{11}^*d_{22}^*d_{33}^*)
\nonumber
\end{align}
are equal, and, as above, equation \eqref{eq:4th:ABA:BA:1'} defines an element in $\tld{I}^{(j)}_{\tau_{w_0},\nabla_{\textnormal{alg}}}$ and \eqref{eq:4th:ABA:BA:2'} defines an element in $\tld{I}^{(j)}_{\tau_{\beta\alpha},\nabla_{\textnormal{alg}}}$.
The mod $p$-reduction of such element justifies the equation in the fifth line of row $\alpha\beta\alpha t_{\un{1}}, \beta\alpha t_{\un{1}}$ of Table \ref{Table_Ideals_mod_p}.

Finally, let
\begin{align*}
z'&\defeq \frac{(b_{\tau_{w_0},2}-b_{\tau_{w_0},3})-(b_{\tau_{\beta\alpha},2}-b_{\tau_{\beta\alpha},3})}{p}+(b_{\tau_{w_0},2}-b_{\tau_{w_0},3})\\
z''&\defeq \frac{(b_{\tau_{w_0},3}-b_{\tau_{w_0},1})-(b_{\tau_{\beta\alpha},3}-b_{\tau_{\beta\alpha},1})}{p}+(b_{\tau_{w_0},3}-b_{\tau_{w_0},1}).
\end{align*}
Again, a direct computation shows that
\begin{align*}
&(z'c_{12}d_{21}+z''c_{11}d_{22}^*-pd_{11}d_{22}^*)(c_{13}d_{31}-c_{11}d_{33}^*+pd_{11}^*d_{33}^*)-(p+1)d_{11}^*d_{33}^*(\textnormal{Mon}_{\tau_{w_0},1})=\\
&\qquad=
(z'c_{12}d_{21}+z''c_{11}d_{22}^*-pd_{11}d_{22}^*)(c_{13}d_{31}-c_{11}d_{33}^*)-d_{11}^*d_{33}^*(\textnormal{Mon}_{\tau_{\beta\alpha},1})
\end{align*}
which, similarly as in the previous cases, defines an element in $\tld{I}^{(j)}_{\tau_{w_0},\nabla_{\textnormal{alg}}}\cap \tld{I}^{(j)}_{\tau_{\beta\alpha},\nabla_{\textnormal{alg}}}$ whose mod $p$-reduction justifies the last equation in row $\alpha\beta\alpha t_{\un{1}}, \beta\alpha t_{\un{1}}$ of Table \ref{Table_Ideals_mod_p}.
\end{proof}

\subsubsection{Justification for Table \ref{Table_Ideals_mod_p_F2}}
As for Table \ref{Table_Ideals_mod_p}, the justification is a direct computation (cf.~\S \ref{subsub:Just:T1}).
For $i\in\{1,2,3\}$ we set $b_{\Id,i}\defeq b_{\tau_{t_{w_0(\eta)}},i}$, $b_{\alpha,i}\defeq b_{\tau_{t_{w_0(\eta)}\alpha},i}$  and $b_{\beta,i}\defeq b_{\tau_{t_{w_0(\eta)}\beta},i}$ for readability in what follows.

\begin{proof}[Study of $\tld{I}^{(j)}_{\tau_{t_{w_0(\eta)}},\nabla_{\textnormal{alg}}}\cap \tld{I}^{(j)}_{\tau_{t_{w_0(\eta)\alpha}},\nabla_{\textnormal{alg}}}$]

Define $z\defeq\frac{b_{\alpha,1}-b_{\alpha,2}}{b_{\Id,1}-b_{\Id,2}}$, $y\defeq b_{\alpha,2}-b_{\alpha,1}-1+z$ and 
\[
x=\frac{1}{p}\big(b_{\alpha,1}-b_{\alpha,3}-z(b_{\Id,1}-b_{\Id,3})\big)
\]
(note that $z\in\Zp$ as $b_{\Id,1}-b_{\Id,2}\not\equiv0$ mod $p$ and that $x\in \Zp$ as $b_{\alpha,1}-b_{\alpha,3}-z(b_{\Id,1}-b_{\Id,3})\equiv 0$ mod $p$).
A direct computation shows that the expressions
\[
d_{11}^*(xc_{12}e_{23}+yc_{13}c_{22}+z\textnormal{Mon}_{t_{w_0(\eta)}})
\]
and 
\[
xe_{13}(c_{12}d_{21}+pd_{11}^*d_{22}^*)-xc_{12}(e_{13}d_{21}-e_{23}d_{11}^*)+yd_{11}^*c_{13}(c_{22}+pd_{22})-d_{11}^*\textnormal{Mon}_{t_{w_0(\eta)}\alpha}
\]
are equal (where we denoted by $\textnormal{Mon}_{t_{w_0(\eta)}}$ and $\textnormal{Mon}_{t_{w_0(\eta)}\alpha}$ the last equation in row $t_{w_0(\eta)}$ and $t_{w_0(\eta)}\alpha$ respectively).
These expressions define an element in the intersection $\tld{I}^{(j)}_{\tau_{t_{w_0(\eta)}},\nabla_{\textnormal{alg}}}\cap \tld{I}^{(j)}_{\tau_{t_{w_0(\eta)\alpha}},\nabla_{\textnormal{alg}}}$, whose mod $p$ reduction explains the second line in row $t_{w_0(\eta)}, t_{w_0(\eta)}\alpha$ of Table \ref{Table_Ideals_mod_p_F2}.
\end{proof}

\begin{proof}[Study of $\tld{I}^{(j)}_{\tau_{t_{w_0(\eta)}},\nabla_{\textnormal{alg}}}\cap \tld{I}^{(j)}_{\tau_{t_{w_0(\eta)\beta}},\nabla_{\textnormal{alg}}}$]

Define%
\begin{align*}
z'&=\frac{1}{p}\big(b_{\beta,1}-b_{\beta,2}-(p+1)(b_{\Id,1}-b_{\Id,2})\big)\\
z''&=\frac{1}{p}\big((p+1)(b_{\Id,1}-b_{\Id,3})-(b_{\beta,1}-b_{\beta,3})\big)
\end{align*}
(note that $z'$ and $z''$ are elements of $\Zp$ as $b_{\beta,1}-b_{\beta,j}-(p+1)(b_{\Id,1}-b_{\Id,j})\equiv 0$ mod $p$ for $j\in\{2,3\}$).
Again a direct computation shows that the expressions
\[
(z''e_{13}d_{22}^*-pc_{13}d_{22}^*+z'c_{12}c_{23})c_{33}+(p+1)d_{33}^*\textnormal{Mon}_{t_{w_0(\eta)}}
\]
and 
\[
(z''e_{13}d_{22}^*-pc_{13}d_{22}^*+z'c_{12}c_{23})(c_{33}+p)-d_{33}^*\textnormal{Mon}_{t_{w_0(\eta)}\beta}
\]
are equal (where again we denoted by $\textnormal{Mon}_{t_{w_0(\eta)}}$ and $\textnormal{Mon}_{t_{w_0(\eta)}\beta}$ the last equation in row $t_{w_0(\eta)}$ and $t_{w_0(\eta)}\beta$ respectively).
These expressions define an element in the intersection $\tld{I}^{(j)}_{\tau_{t_{w_0(\eta)}},\nabla_{\textnormal{alg}}}\cap \tld{I}^{(j)}_{\tau_{t_{w_0(\eta)\beta}},\nabla_{\textnormal{alg}}}$, whose mod $p$ reduction explains the second line in row $t_{w_0(\eta)}, t_{w_0(\eta)}\beta$ of Table \ref{Table_Ideals_mod_p_F2}.
\end{proof}

\subsubsection{Computations on $\Tor^{{S}^{(j)}}_1(\F,(\tld{S}^{(j)}/\tld{I}^{(j)}_{T,\nabla_{\textnormal{alg}}})\otimes\F)$}
\label{subsubsec:Tor:cmpt}
We provide details for the computations of the maps between various $\Tor^{{S}^{(j)}}_1(\F,(\tld{S}^{(j)}/\tld{I}^{(j)}_{T,\nabla_{\textnormal{alg}}})\otimes\F)$ appearing in the proofs of Lemmas \ref{lemma:ideal3types}, \ref{lemma:ideal3types:prime}, \ref{lemma:ideal4types} and \ref{lemma:ideal4types:prime}.
In the following computation, given an ideal $I\subseteq {S}^{(j)}$ we write elements of $\Tor_1^{{S}^{(j)}}(\F,{S}^{(j)}/I)$ in terms of generators of $I$, by virtue of the canonical isomorphism $\Tor_1^{{S}^{(j)}}(\F,{S}^{(j)}/I)\cong I/(\fm_{{S}^{(j)}}\cdot I)$.
\begin{proof}[Complements in the proof of Lemma \ref{lemma:ideal3types}]
We need to prove that the union of the images of the canonical maps
\begin{align}
\label{eq:first:map}
\Tor_1(\F,(\tld{S}^{(j)}/\big( \tld{I}^{(j)}_{\tau_{\alpha\beta},\nabla_{\textnormal{alg}}}\cap \tld{I}^{(j)}_{\tau_{w_0},\nabla_{\textnormal{alg}}}\big)) \otimes \F)
&\rightarrow
\Tor_1(\F,(\tld{S}/ \tld{I}^{(j)}_{\tau_{w_0},\nabla_{\infty}}) \otimes \F)
\\
\label{eq:second:map}
\Tor_1(\F,(\tld{S}^{(j)}/\big( \tld{I}^{(j)}_{\tau_{\beta\alpha},\nabla_{\textnormal{alg}}}\cap \tld{I}^{(j)}_{\tau_{w_0},\nabla_{\textnormal{alg}}}\big)) \otimes \F)
&\rightarrow
\Tor_1(\F,(\tld{S}/ \tld{I}^{(j)}_{\tau_{w_0},\nabla_{\infty}}) \otimes \F)
\end{align}
generates a spanning set for $\Tor_1(\F,(S^{(j)}/ I^{(j)}_{\tau_{w_0},\nabla_{\infty}}) \otimes \F)$, e.g.~using Table \ref{Table_Ideals}, the set given by the images of the elements
\begin{align*}
&c_{21},\ c_{22},\, c_{23},\, c_{32},\, c_{33}\\
&c_{13}d_{32}-c_{12}d_{33}^*,\, 
c_{13}d_{31}-c_{11}d_{33}^*,\, c_{13}c_{31}\\
&(b-c)c_{21}d_{12}+(c-a)c_{11}d_{22}^*,\, c_{31}
\end{align*}
(where $(a,b,c)\defeq s_j^{-1}(\mu_j+\eta_j)-(1,1,1)\equiv b_{\tau_{w_0}}$ modulo $\varpi$).
We immediately see from row $\alpha\beta\alpha t_{\un{1}},\alpha\beta t_{\un{1}}$ in Table \ref{Table_Ideals_mod_p} that the elements $c_{32}, c_{33}, c_{13}c_{31}, c_{13}d_{32}-c_{12}d_{33}^*$ are in the image of \eqref{eq:first:map}.
Similarly the elements $c_{21}, c_{22}, c_{23}$ are in the image of \eqref{eq:second:map}.

Writing 
\[
c_{13}d_{31}-c_{11}d_{33}^*=\underbrace{c_{13}d_{21}d_{32}-c_{12}d_{21}d_{33}^*-c_{13}d_{31}d_{22}^*+c_{11}d_{22}^*d_{33}^*}_{\in\text{ image of \eqref{eq:second:map}}}-d_{21}\underbrace{(c_{13}d_{32}-c_{12}d_{33}^*)}_{\in\text{ image of \eqref{eq:first:map}}}
\]
we conclude that $c_{13}d_{31}-c_{11}d_{33}^*$ is in the $\F$-span of the union of the images of \eqref{eq:first:map},\eqref{eq:second:map}.

Similarly, 
\begin{align*}
&(b-c)c_{21}d_{12}+(c-a)c_{11}d_{22}^*=\underbrace{(b-c)c_{12}d_{21}+\ovl{x}c_{11}c_{22}-(a-c)c_{11}d_{22}^*-(b-c)c_{22}d_{11}^*}_{\in\text{ image of \eqref{eq:first:map}}}+\\&\qquad-(\ovl{x}c_{11}-(b-c)d_{11}^*)\underbrace{c_{22}}_{\in\text{ image of \eqref{eq:second:map}}}
\end{align*}
so that $(b-c)c_{21}d_{12}+(c-a)c_{11}d_{22}^*$ is in the $\F$-span of the union of the images of \eqref{eq:first:map},\eqref{eq:second:map}.

Finally, note that $c_{22}(d_{21}d_{32}),c_{22}(d_{21}c_{32}), c_{22}(d_{31}d_{22}^*), c_{21}(d_{32}d_{22}^*)\in {I}^{(j)}_{\tau_{w_0},\nabla_{\textnormal{alg}}}\cdot \fm_{\tld{S}^{(j)}}$
so that the last equation  in row $\alpha\beta\alpha t_{\un{1}}, \alpha\beta t_{\un{1}}$ in Table \ref{Table_Ideals_mod_p} is sent by the map \eqref{eq:first:map} to $(a-c+1)c_{31}(d_{22}^*)^2$ and in particular $c_{31}$ is in the image of the map \eqref{eq:first:map}.
\end{proof}

\begin{proof}[Complements in the proof of Lemma \ref{lemma:ideal3types:prime}]
The argumen is similar to that for Lemma \ref{lemma:ideal3types}.
Consider the natural maps
\begin{align}
\label{eq:first:map:prime}
\Tor^{{S}}_1(\F,(\tld{S}/\big( \tld{I}_{\tau_{\alpha t_{w_0(\eta)}},\nabla_{\infty}}\cap \tld{I}_{\tau_{t_{w_0(\eta)}},\nabla_{\infty}}\big)) \otimes \F)
&\rightarrow
\Tor^{{S}}_1(\F,(\tld{S}/\tld{I}_{\tau_{t_{w_0(\eta)}},\nabla_{\infty}}) \otimes \F)
\\
\label{eq:second:map:prime}
\Tor_1^{{S}}(\F,(\tld{S}/\big( \tld{I}_{\tau_{t_{w_0(\eta)}},\nabla_{\infty}}\cap \tld{I}_{\tau_{\beta t_{w_0(\eta)}},\nabla_{\infty}}\big)) \otimes \F)
&\rightarrow
\Tor^{{S}}_1(\F,(\tld{S}/\tld{I}_{\tau_{t_{w_0(\eta)}},\nabla_{\infty}}) \otimes \F).
\end{align}
A spanning set for $\Tor^{{S}}_1(\F,(\tld{S}/\tld{I}_{\tau_{t_{w_0(\eta)}},\nabla_{\infty}}) \otimes \F)$ (using Table \ref{Table_Ideals_mod_p_F2}) is  given by the images of the elements
\begin{align*}
&d_{21},\,d_{32},\,c_{32},\,e_{33},\,c_{33},\,d_{32},\,e_{23},\,c_{22}\\
&(a-c)e_{23}d_{22}^*-(a-b)c_{12}c_{23}.
\end{align*}
We immediately see from row $t_{w_0(\eta)},t_{w_0(\eta)}\alpha$ in Table \ref{Table_Ideals_mod_p_F2} that the elements $\,d_{32},\,c_{32},\,e_{33},\,c_{33},\,d_{32},$ are in the image of \eqref{eq:first:map:prime}, and, from row $t_{w_0(\eta)},t_{w_0(\eta)}\beta$, that the element $d_{21}$ is in the image of \eqref{eq:second:map:prime}.

Moreover, noting that $c_{12}d_{21}, e_{13}d_{21}, c_{13}c_{22}, c_{12}e_{23}\in \fm_{{S}^{(j)}}\cdot{I}^{(j)}_{\tau_{t_{w_0(\eta)}},\nabla_{\textnormal{alg}}}$ we conclude that \eqref{eq:first:map:prime} maps the elements $c_{12}d_{21}-c_{22}d_{11}^*$, $e_{13}d_{21}-e_{23}d_{11}^*$ and $(a-b)(c_{13}c_{22}-c_{12}c_{23})-\ovl{x}c_{12}e_{23}+(a-c)e_{23}d_{22}^*$ to $c_{22}d_{11}^*$, $e_{23}d_{11}^*$ and $(a-c)e_{23}d_{22}^*-(a-b)c_{12}c_{23}$ respectively.
\end{proof}

\begin{proof}[Complements in the proof of Lemma \ref{lemma:ideal4types}]

We check that the union of the images of the canonical maps 
\begin{align}
\label{eq:ideal4types:Map1}
\Tor_1^{{S}}\bigg(\F,S/\big( \tld{I}^{(j)}_{\tau_{\alpha\beta},\nabla_{\textnormal{alg}}}\cap \tld{I}^{(j)}_{\tau_{w_0},\nabla_{\infty}},p\big)\cap (\tld{I}^{(j)}_{\tau_{w_0},\nabla_{\textnormal{alg}}}\cap \tld{I}^{(j)}_{\tau_{\beta\alpha},\nabla_{\infty}},p)\bigg)
&\rightarrow
\Tor^{{S}}_1(\F,S/ {I}^{(j)}_{\Lambda})
\\
\label{eq:ideal4types:Map2}
\Tor_1^{{S}}(\F,S/\big( \tld{I}^{(j)}_{\tau_{\Id},\nabla_{\textnormal{alg}}},p))
&\rightarrow
\Tor^{{S}}_1(\F,S/ {I}^{(j)}_{\Lambda})
\end{align}
generates a spanning set for the target, i.e.~by Lemma \ref{lem:broom:other}, the set given by the image of the elements $c_{33}, d_{32}c_{23}-c_{22}d_{33}^*, c_{22}, c_{11}d_{33}^*-c_{13}d_{31}$ of ${I}^{(j)}_{\Lambda}$.
From the last row of Table \ref{Table_Ideals_mod_p} we immediately see that the elements $c_{33}, d_{32}c_{23}-c_{22}d_{33}^*$ are in the image of the map \eqref{eq:ideal4types:Map1}.
Moreover, by Table \ref{Table_Ideals_1}, the image of the map \eqref{eq:ideal4types:Map2} contains the elements 
\begin{eqnarray*}
(a-c-1)(c_{23}d_{32}-c_{33}d_{22}^*)-(a-b-1)c_{22}d_{33}^*\\
(a-b)(c_{13}d_{31}-c_{11}d_{33}^*)-(b-c-1)c_{33}d_{11}^*.
\end{eqnarray*}
In particular, as $a-b\neq 0\neq b-c$, the union of the images of \eqref{eq:ideal4types:Map1},\eqref{eq:ideal4types:Map2} contains the elements $c_{13}d_{31}-c_{11}d_{33}^*$ and $c_{22}$.
\end{proof}

\begin{proof}[Complements in the proof of Lemma \ref{lemma:ideal4types:prime}]

We check that the union of the images of the canonical maps 
\begin{align}
\label{eq:ideal4types:Map1:prime}
\Tor^{{S}^{(j)}}_1(\F,S^{(j)}/\big( \tld{I}^{(j)}_{\tau_{t_{w_0(\eta)}\alpha },\nabla_{\infty}}\cap \tld{I}^{(j)}_{\tau_{t_{w_0(\eta)}},\nabla_{\infty}},p\big)\cap (\tld{I}_{\tau_{t_{w_0(\eta)}},\nabla_{\infty}}\cap \tld{I}_{\tau_{t_{w_0(\eta)}\beta },\nabla_{\infty}},p)
&\rightarrow
\Tor_1^{{S}^{(j)}}(\F,{S}^{(j)}/I^{(j)}_\Lambda)
\\
\label{eq:ideal4types:Map2:prime}
\Tor^{{S}^{(j)}}_1(\F,(\tld{S}/\tld{I}_{\tau_{t_{w_0(\eta)}w_0},\nabla_{\infty}}) \otimes \F)
&\rightarrow
\Tor_1^{{S}^{(j)}}(\F,{S}^{(j)}/I^{(j)}_\Lambda)
\end{align}
is a spanning set for the target.
By Lemma \ref{lem:broom:other} a spanning set for the target is given by 
\begin{align}
&c_{32}, e_{33}, d_{31}, d_{21}d_{32},
\label{eq:easy:1}
\\ 
&e_{23}, (a-b)c_{12}c_{23}-(a-c)e_{13}d_{22}^*, c_{33}
\label{eq:easy:2}\\
\label{eq:linear:combos}
&c_{23}d_{32},c_{22},c_{12}d_{21}.
\end{align}
By the last row in Table \ref{Table_Ideals_mod_p_F2} the elements in \eqref{eq:easy:2} are immediately checked to be in the image of \eqref{eq:ideal4types:Map1:prime}.
By row $t_{w_0(\eta)}w_0$ in Table \ref{Table_Ideals_2}, and noting further that $c_{13}c_{22}, c_{13}d_{31}\in \fm_{{S}^{(j)}}{I}^{(j)}_{\Lambda}$ we immediately see that the elements in \eqref{eq:easy:2} are in the image of \eqref{eq:ideal4types:Map2:prime}.

As $c_{23}d_{32}-c_{33}d_{22}^*$ is in the image of \eqref{eq:ideal4types:Map1:prime} by the last row of Table \ref{Table_Ideals_mod_p_F2}, we conclude from the above that $c_{23}d_{32}$ is in the linear span of the union of the images of \eqref{eq:ideal4types:Map1:prime} and \eqref{eq:ideal4types:Map2:prime}.
Moreover, as $(c-a-1)(c_{23}d_{32}-c_{33}d_{22}^*)+(a-b)c_{22}d_{33}^*$ is in the image of \eqref{eq:ideal4types:Map1:prime} by  row $t_{w_0(\eta)}w_0$ in Table \ref{Table_Ideals_2}, we conclude by the above $c_{22}$ is also in the linear span of the union of the images of \eqref{eq:ideal4types:Map1:prime} and \eqref{eq:ideal4types:Map2:prime}.
Finally, as $c_{12}d_{21}-c_{22}d_{1}^*$ is in the image of \eqref{eq:ideal4types:Map1:prime} by the last row of Table \ref{Table_Ideals_mod_p_F2}, we conclude from the above that $c_{12}d_{21}$ is also in the linear span of the union of the images of \eqref{eq:ideal4types:Map1:prime} and \eqref{eq:ideal4types:Map2:prime}.
\end{proof}
 
\newpage
\bibliography{Biblio}
\bibliographystyle{amsalpha}

\end{document}